\documentclass[a4paper,11pt,oneside]{amsart}

\usepackage{etex}
\usepackage[T1]{fontenc}%
\usepackage[latin1]{inputenc}%
\usepackage{xspace}%
\usepackage{amsfonts,amsmath,amssymb,amsthm}%
\usepackage{chngcntr}%
\usepackage{graphicx,mathdots}%
\usepackage{enumerate}%
\usepackage{multirow}%

\input{xy}
\xyoption{all}
\objectmargin={3mm}

\theoremstyle{plain}
\newtheorem{theorem}{Theorem}
\newtheorem{proposition}[theorem]{Proposition}
\newtheorem{corollary}[theorem]{Corollary}
\newtheorem{lemma}[theorem]{Lemma}
\newtheorem{fact}[theorem]{Fact}

\newtheorem{example}[theorem]{Example}
\newtheorem{setting}[theorem]{Basic setting}

\theoremstyle{definition}
\newtheorem{definition}[theorem]{Definition}
\newtheorem{remark}[theorem]{Remark}

\newcommand{\C}{\mathbb{C}}
\newcommand{\R}{\mathbb{R}}
\newcommand{\Z}{\mathbb{Z}}
\newcommand{\N}{\mathbb{N}}
\newcommand{\tilG}{\widetilde{G}}
\newcommand{\tilH}{\widetilde{H}}
\newcommand{\GL}{\mathrm{GL}}
\newcommand{\SL}{\mathrm{SL}}

\newcommand{\OO}{\mathrm{O}}
\newcommand{\SO}{\mathrm{SO}}
\newcommand{\PSO}{\mathrm{PSO}}
\newcommand{\U}{\mathrm{U}}
\newcommand{\SU}{\mathrm{SU}}
\newcommand{\Sp}{\mathrm{Sp}}
\newcommand{\Spin}{\mathrm{Spin}}
\newcommand{\tilg}{\widetilde{\mathfrak{g}}}
\newcommand{\tilh}{\widetilde{\mathfrak{h}}}
\newcommand{\tila}{\widetilde{\mathfrak{a}}}
\newcommand{\g}{\mathfrak{g}}
\newcommand{\h}{\mathfrak{h}}
\newcommand{\kk}{\mathfrak{k}}
\newcommand{\p}{\mathfrak{p}}
\newcommand{\q}{\mathfrak{q}}
\newcommand{\aaa}{\mathfrak{a}}

\newcommand{\m}{\mathfrak{m}}
\newcommand{\n}{\mathfrak{n}}
\newcommand{\jj}{\mathfrak{j}}
\newcommand{\ttt}{\mathfrak{t}}
\newcommand{\llll}{\mathfrak{l}}
\newcommand{\gl}{\mathfrak{gl}}
\newcommand{\ssl}{\mathfrak{sl}}
\newcommand{\so}{\mathfrak{so}}
\newcommand{\spin}{\mathfrak{spin}}

\newcommand{\ssp}{\mathfrak{sp}}

\newcommand{\Hom}{\mathrm{Hom}}

\newcommand{\D}{\mathbb{D}}
\newcommand{\Diag}{\mathrm{Diag}}

\newcommand{\M}{\mathcal{M}}

\newcommand{\W}{\mathcal{W}}
\newcommand{\Spec}{\mathrm{Spec}}

\newcommand{\Ad}{\operatorname{Ad}}
\newcommand{\ad}{\operatorname{ad}}
\newcommand{\HH}{\mathbb{H}}
\newcommand{\PP}{\mathbb{P}}

\newcommand{\rank}{\operatorname{rank}}
\newcommand{\V}{\mathcal{V}}

\newcommand{\Disc}{\mathrm{Disc}}
\newcommand{\Supp}{\mathrm{Supp}}
\newcommand{\Rep}{\mathrm{Rep}}
\newcommand{\Ker}{\mathrm{Ker}}
\newcommand{\DD}{\mathcal{D}}
\newcommand{\ii}{\mathbf{i}}
\newcommand{\dd}{\mathrm{d}}
\newcommand{\resp}{resp.\ }
\newcommand{\ie}{i.e.\ }
\newcommand{\eg}{e.g.\ }
\newcommand{\sumplus}[1]{\underset{\hspace{-0.25cm}#1}{{\sum}^{\oplus}}}
\newcommand{\nnu}{\boldsymbol\nu}
\newcommand{\llambda}{\boldsymbol\lambda}

\newcommand*{\longhookrightarrow}{\ensuremath{\lhook\joinrel\relbar\joinrel\rightarrow}}

\newcommand*{\leftlongmapsto}{\mathrel{\reflectbox{\ensuremath{\longmapsto}}}}

\DeclareRobustCommand{\SkipTocEntry}[4]{}

\newenvironment{changemargin}[2]{\begin{list}{}{%
\setlength{\topsep}{0pt}%
\setlength{\leftmargin}{0pt}%
\setlength{\rightmargin}{0pt}%
\setlength{\listparindent}{\parindent}%
\setlength{\itemindent}{\parindent}%
\setlength{\parsep}{0pt plus 1pt}%
\addtolength{\leftmargin}{#1}%
\addtolength{\rightmargin}{#2}%
}\item }{\end{list}}

\title[Invariant differential operators on spherical spaces]{Invariant differential operators on spherical homogeneous spaces with overgroups}

\author{Fanny Kassel}
\address{CNRS and Institut des Hautes \'Etudes Scientifiques, Laboratoire Alexander Grothendieck, 35 route de Chartres, 91440 Bures-sur-Yvette, France}
\email{kassel@ihes.fr}

\author{Toshiyuki Kobayashi}
\address{Graduate School of Mathematical Sciences and Kavli Institute for the Physics and Mathematics of the Universe (WPI), The University of Tokyo, 3-8-1 Komaba, Tokyo, 153-8914 Japan}
\email{toshi@ms.u-tokyo.ac.jp}

\thanks{This project received funding from the European Research Council (ERC) under the European Union's Horizon 2020 research and innovation programme (ERC starting grant DiGGeS, grant agreement No 715982).
FK was partially supported by the Agence Nationale de la Recherche through the grant DiscGroup (ANR-11-BS01-013) and the Labex CEMPI (ANR-11-LABX-0007-01).
TK was partially supported by the JSPS under the Grant-in-Aid for Scientific Research (A) (18H03669).}

\begin{document}

\maketitle
\numberwithin{equation}{section}
\numberwithin{table}{section}
\numberwithin{theorem}{section}

\begin{abstract}
We investigate the structure of the ring $\D_G(X)$ of $G$-invariant differential operators on a reductive spherical homogeneous space $X=G/H$ with an overgroup~$\tilG$.
We consider three natural subalgebras of $\D_G(X)$ which are polynomial algebras with explicit generators, namely the subalgebra $\D_{\tilG}(X)$ of $\tilG$-invariant differential operators on~$X$ and two other subalgebras coming from the centers of the enveloping algebras of $\g$ and~$\kk$, where $K$ is a maximal proper subgroup of $G$ containing~$H$.
We show that in most cases $\D_G(X)$ is generated by any two of these three subalgebras, and analyze when this may fail.
Moreover, we find explicit relations among the generators for each possible triple $(\tilG,G,H)$, and describe \emph{transfer maps} connecting eigenvalues for $\D_{\tilG}(X)$ and for the center of the enveloping algebra of~$\g_{\C}$.
\end{abstract}

\tableofcontents

\section{Introduction}\label{sec:intro}

Let $X$ be a manifold with a transitive action of a compact Lie group~$G$.
The ring $\D_G(X)$ of $G$-invariant differential operators on~$X$ is commutative if and only if the complexification $X_{\C}$ is $G_{\C}$-spherical, \ie $X_{\C}$ admits an open orbit of a Borel subgroup of~$G_{\C}$.
This is the case for instance if $X$ is a symmetric space of~$G$, but there are also spherical homogeneous spaces that are not symmetric, \eg $X_{\C}=\SO(2n+1,\C)/\GL_n(\C)$ or $\SL_{2n+1}(\C)/\Sp(n,\C)$; they were classified in \cite{kra79,bri87,mik87}.
Knop \cite{kno94} proved that if $X_{\C}$ is $G_{\C}$-spherical then $\D_G(X)$ is actually a polynomial ring; the number of algebrai-

\pagebreak
\noindent
cally independent generators of $\D_G(X)$ is equal to the rank of the spherical space $X=G/H$ (see Section~\ref{subsec:DGH-spherical}).
A more explicit structure is known in some special settings:
\begin{enumerate}
  \item 
  If $X$ is a reductive symmetric space, then $\D_G(X)$ is naturally isomorphic to the ring of invariant polynomials for the (little) Weyl group, by work of Harish-Chandra \cite{har58}.
  \item 
  If $X_{\C}$ appears as an open $G_{\C}$-orbit in a prehomogeneous vector space and if the center $Z(\g_{\C})$ of the enveloping algebra $U(\g_{\C})$ surjects onto $\D_G(X)$, then explicit generators in $\D_G(X)$ were given by Howe--Umeda \cite{hu91} as a generalization of the classical Capelli identity.
\end{enumerate}

In this paper we consider the situation where the spherical homogeneous space~$X$ admits an \emph{overgroup}, \ie there is a Lie group $\tilG$ containing~$G$ and acting (transitively) on~$X$.
In this situation there are two natural subalgebras of $\D_G(X)$, obtained from the centers of $Z(\tilg_{\C})$ and $Z(\g_{\C})$, which play an important role in the global analysis by means of representation theory of $\tilG$ and~$G$.
We investigate the structure of the ring $\D_G(X)$ by using these two subalgebras as well as a third one, induced from a certain $G$-equivariant fibration of~$X$ (see Section~\ref{subsec:intro-three-subalg}).

More precisely, the setting of the paper is the following.

\begin{setting} \label{setting}
We consider a connected compact Lie group~$\tilG$ and two connected proper closed subgroups $G$ and~$\tilH$ of~$\tilG$ such that the complexified homogeneous space $\tilG_{\C}/\tilH_{\C}$ is $G_{\C}$-spherical.
The embedding $G\hookrightarrow\tilG$ then induces a diffeomorphism
\begin{equation} \label{eqn:overgp}
X:=G/H \overset{\sim}{\longrightarrow} \tilG/\tilH,
\end{equation}
where we set $H:=\tilH\cap G$.
\end{setting}

In most of the paper, we furthermore assume that $\tilG$ is simple.
A classification of such triples $(\tilG,\tilH,G)$ up to a covering is given in Table~\ref{table1}; it is obtained from Oni\v{s}\v{c}ik's infinitesimal classification \cite{oni69} of triples $(\tilG,\tilH,G)$ with $\tilG$ compact simple and $\tilG=\tilH G$, and from the classification of spherical homogeneous spaces \cite{kra79,bri87,mik87}.
In this setting $G/H$ is never a symmetric space.

\subsection{Three subalgebras of $\D_G(X)$} \label{subsec:intro-three-subalg}

Let $\D(X)$ be the full $\C$-algebra of differential operators on~$X$.
The differentiations of the left and right regular representations of $G$ on $C^{\infty}(G)$ induce a $\C$-algebra homomorphism
\begin{equation}\label{eqn:dl-dr}
\dd\ell \otimes \dd r : U(\g_{\C}) \otimes U(\g_{\C})^H \longrightarrow \D(X),
\end{equation}
where $U(\g_{\C})^H$ is the subalgebra of $H$-invariant elements in the enveloping algebra $U(\g_{\C})$ (see Section~\ref{subsec:DGH}).
It is known (see \eg \cite[Ch.\,II, Th.\,4.6]{hel00}) that the image $\dd r(U(\g_{\C})^H)$ coincides with $\D_G(X)$.
However, the ring $U(\g_{\C})^H$ is noncommutative and difficult to understand in general.
Instead, we analyze $\D_G(X)$ in terms of three well-understood subalgebras.

The first subalgebra is the image $\dd\ell(Z(\g_{\C}))$ of the center $Z(\g_{\C})$ of $U(\g_{\C})$.
The ring $Z(\g_{\C})$ is well-understood (it is isomorphic to a polynomial ring, described by the Harish-Chandra isomorphism), but its image $\dd\ell(Z(\g_{\C}))$ is typically smaller than $\D_G(X)$ in our setting, in contrast with the case where $X=G/H$ is a symmetric space \cite{hel92}.

The second subalgebra of $\D_G(X)$ we consider is $\dd\ell(Z(\tilg_{\C}))=\dd r(Z(\tilg_{\C}))$, where we regard $X$ as a $\tilG$-space and consider the map
$$\dd\ell \otimes \dd r : U(\tilg_{\C}) \otimes U(\tilg_{\C})^{\tilH} \longrightarrow \D(X)$$
similar to \eqref{eqn:dl-dr}.
In our setting, $\dd\ell(Z(\tilg_{\C}))$ is always equal to the full subalgebra $\D_{\tilG}(X)\subset\D_G(X)$ of $\tilG$-invariant differential operators on~$X$ (Lemma~\ref{lem:surj-dl-drF}).

Finally, the third subalgebra is $\dd r(Z(\kk_{\C}))$ for some subgroup $K$ of $G$ containing~$H$.
This algebra is zero if $K=H$, and equal to $\dd\ell(Z(\g_{\C}))=\dd r(Z(\g_{\C}))$ if $K=G$.
However, it may yield new nontrivial $G$-invariant differential operators on~$X$ if $H\subsetneq K\subsetneq G$.
We shall choose $K$ to be a maximal connected proper subgroup of~$G$ containing~$H$ (this is possible by Proposition~\ref{prop:H-subset-K}.(2)).
The geometric meaning of the algebra $\dd r(Z(\kk_{\C}))$ will be explained in Section~\ref{subsec:dliota}, in terms of the fibration of $X=G/H$ over $G/K$ with fiber $F:=K/H$: namely, there are natural maps $\dd r_{\scriptscriptstyle F} : U(\kk_{\C})^H\twoheadrightarrow\D_K(F)$ (similar to $\dd r$ in \eqref{eqn:dl-dr}) and $\iota : \D_K(F)\hookrightarrow\D_G(X)$ such that the following diagram commutes.
\begin{equation}\label{eqn:diagram}
\xymatrix{Z(\tilg_{\C}) \ar@{->>}[d]_{\dd\ell} & Z(\g_{\C}) \ar[d]_{\dd\ell} & Z(\kk_{\C}) \ar[dl]_{\dd r} \ar@{->>}[d]^{\dd r_{_ F}}\\
\D_{\tilG}(X) \ar@{}[r]|{\textstyle\subset} & \D_G(X) & \D_K(F) \ar@{_{(}->}[l]_{\iota}}
\end{equation}
In our setting, $\dd r_{_F}(Z(\kk_{\C}))$ is also equal to the full algebra $\D_K(F)$ (Lemma~\ref{lem:surj-dl-drF}), and in particular,
\begin{equation}\label{eqn:DKF-drZk}
\iota(\D_K(F)) = \dd r(Z(\kk_{\C})).
\end{equation}

\begin{remark}
The number of connected components of $H = \tilH \cap G$ may vary under taking a covering of~$\tilG$, but the algebra $\D_G(X)$ and its subalgebras $\D_{\tilG}(X)=\dd\ell(Z(\tilg_{\C}))$, $\dd r(Z(\kk_{\C}))$, and $\dd\ell(Z(\g_{\C}))$ do not, see Theorem~\ref{thm:InvDiff-covering}.
\end{remark}

We prove the following.

\begin{theorem} \label{thm:main}
In the setting~\ref{setting}, suppose that $\tilG$ is simple.
If $\tilh\cap\g$ is not a maximal proper subalgebra of~$\g$, then there is a unique maximal connected proper subgroup $K$ of~$G$ containing $H$, and
\begin{enumerate}
  \item $\D_G(X)$ is generated by $\D_{\tilG}(X)$ and $\dd r(Z(\kk_{\C}))$;
  \item $\D_G(X)$ is generated by $\dd\ell(Z(\g_{\C}))$ and $\dd r(Z(\kk_{\C}))$;
  \item $\D_G(X)$ is generated by $\D_{\tilG}(X)$ and $\dd\ell(Z(\g_{\C}))$, except if we are in case (ix) of Table~\ref{table1} up to a covering of~$\tilG$.
\end{enumerate}
If $\tilh\cap\g$ is a maximal proper subalgebra of~$\g$, then
$$\D_G(X) = \D_{\tilG}(X) = \dd\ell(Z(\g_{\C})).$$
\end{theorem}

The complete list of triples $(\tilG,\tilH,G)$ in Theorem~\ref{thm:main}, up to a covering of~$\tilG$, is given in Table~\ref{table1}.
\begin{table}[h!]
\centering
\begin{changemargin}{-1.8cm}{0cm}
\begin{tabular}{|p{0.8cm}|p{1.8cm}|p{2.4cm}|p{2.6cm}|p{3.1cm}|p{3.55cm}|}
\hline
& & & & & \tabularnewline [-0.4cm]
& \centering $\tilG$ & \centering $\tilH$ & \centering $G$ & \centering $H$ & \centering $K$ \tabularnewline
\hline
\centering (i) & \centering $\SO(2n+2)$ & \centering $\SO(2n+1)$ & \centering $\U(n+1)$ & \centering $\U(n)$ & \centering $\U(n)\times\U(1)$\tabularnewline
\centering (i)$'$ & \centering $\SO(2n+2)$ & \centering $\U(n+1)$ & \centering $\SU(n+1)$ & \centering $\SU(n)$ & \centering $\U(n)$\tabularnewline
\centering (ii) & \centering $\SO(2n+2)$ & \centering $\U(n+1)$ & \centering $\SO(2n+1)$ & \centering $\U(n)$ & \centering $\SO(2n)$\tabularnewline
\centering (iii) & \centering $\SU(2n+2)$ & \centering $\U(2n+1)$ & \centering $\Sp(n+1)$ & \centering $\Sp(n)\times\U(1)$ & \centering $\Sp(n)\times\Sp(1)$\tabularnewline
\centering (iv) & \centering $\SU(2n+2)$ & \centering $\Sp(n+1)$ & \centering $\U(2n+1)$ & \centering $\Sp(n)\times\U(1)$ & \centering $\U(2n)\times\U(1)$\tabularnewline
\centering (v) & \centering $\SO(4n+4)$ & \centering $\SO(4n+3)$ & \centering $\Sp(n+1)\cdot\Sp(1)$ & \centering $\Sp(n)\cdot\Diag(\Sp(1))$ & \centering $(\Sp(n)\times\Sp(1))\cdot\Sp(1)$\tabularnewline
\centering (v)$'$ & \centering $\SO(4n+4)$ & \centering $\SO(4n+3)$ & \centering $\Sp(n+1)\cdot\U(1)$ & \centering $\Sp(n)\cdot\Diag(\U(1))$ & \centering $(\Sp(n)\times\Sp(1))\cdot\U(1)$\tabularnewline
\centering (vi) & \centering $\SO(16)$ & \centering $\SO(15)$ & \centering $\Spin(9)$ & \centering $\Spin(7)$ & \centering $\Spin(8)$\tabularnewline
\centering (vii) & \centering $\SO(8)$ & \centering $\Spin(7)$ & \centering $\SO(5)\times\SO(3)$ & \centering $\iota_7(\SO(4))$ & \centering $\SO(4)\times\SO(3)$\tabularnewline
\centering (viii) & \centering $\SO(7)$ & \centering $G_{2(-14)}$ & \centering $\SO(5)\times\SO(2)$ & \centering $\iota_8(\U(2))$ & \centering $\SO(4)\times\SO(2)$\tabularnewline
\centering (ix) & \centering $\SO(7)$ & \centering $G_{2(-14)}$ & \centering $\SO(6)$ & \centering $\SU(3)$ & \centering $\U(3)$\tabularnewline
\centering (x) & \centering $\SO(7)$ & \centering $\SO(6)$ & \centering $G_{2(-14)}$ & \centering $\SU(3)$ & \centering $\SU(3)$\tabularnewline
\centering (xi) & \centering $\SO(8)$ & \centering $\Spin(7)$ & \centering $\SO(7)$ & \centering $G_{2(-14)}$ & \centering $G_{2(-14)}$\tabularnewline
\centering (xii) & \centering $\SO(8)$ & \centering $\SO(7)$ & \centering $\Spin(7)$ & \centering $G_{2(-14)}$ & \centering $G_{2(-14)}$\tabularnewline
\centering (xiii) & \centering $\SO(8)$ & \centering $\Spin(7)$ & \centering $\SO(6)\times\SO(2)$ & \centering $\iota_{13}\big(\widetilde{\U(3)}\big)$ & \centering $\U(3)\times\SO(2)$\tabularnewline
\centering (xiii)$'$ & \centering $\SO(8)$ & \centering $\Spin(7)$ & \centering $\SO(6)$ & \centering $\SU(3)$ & \centering $\U(3)$\tabularnewline
\centering (xiv) & \centering $\SO(8)$ & \centering $\SO(6)\times\SO(2)$ & \centering $\Spin(7)$ & \centering $\iota_{14}\big(\widetilde{\U(3)}\big)$ & \centering $\Spin(6)$\tabularnewline
\hline
\end{tabular}
\end{changemargin}
\caption{Complete list of triples $(\tilG,\tilH,G)$ in the setting~\ref{setting} with $\tilG$ simple, up to a covering of~$\tilG$. We also indicate $H:=\tilH\cap G$ and the maximal connected proper subgroup $K$ of~$G$ containing~$H$. In case (i)$'$ we require $n\geq 2$.}
\label{table1}
\end{table}
In that table we use the notation $H_1\cdot H_2$ for the almost product of two subgroups $H_1$ and~$H_2$ (meaning there is a surjective homomorphism with finite kernel from $H_1\times H_2$ to $H_1\cdot H_2$).
We also use the notation $\Diag$ to indicate a diagonal embedding.
Here $\iota_7,\iota_8,\iota_{13} : H\to K$ and $\iota_{14} : H\to\tilH$ are nontrivial embeddings described in Sections \ref{subsec:ex(vii)}, \ref{subsec:ex(viii)}, and~\ref{subsec:ex-triality}.
We denote by $\widetilde{\U(3)}$ the double covering of $\U(3)$, see Section~\ref{subsec:ex-triality}.

The main case is when $\tilh\cap\g$ is not a maximal proper subalgebra of~$\g$: the only exceptions in Table~\ref{table1} are (x), (xi), and~(xii).

The condition that $\tilG_{\C}/\tilH_{\C}$ be $G_{\C}$-spherical depends only on the triples of complex Lie algebras $(\tilg_{\C},\tilh_{\C},\g_{\C})$.
The pair $(\tilg,\tilh)$ is always a symmetric pair except in cases (viii) and~(ix); the pair $(\g,\h)$ is never symmetric.

\begin{remark} \label{rem:nonsimple-case}
By using the triality of~$D_4$ for the realization of $G$ in~$\tilG$, we see that the triple
\begin{equation} \label{eqn:nonsimple-case}
(\tilG,\tilH,G) = \big(\Spin(8)\times\Spin(8), \Spin(7)\times\Spin(7), \Spin(8)\big)
\end{equation}
satisfies all the conditions of Theorem~\ref{thm:main} with
$$H := \tilH \cap G = G_{2(-14)}$$
except that $\tilG$ is not simple.
This case arises as a compact real form of the complexification of the isomorphism
$$\Spin(1,7)/G_{2(-14)} \simeq \Spin(8,\C)/\Spin(7,\C).$$
In this case, there is a unique maximal connected proper subgroup $K$ of~$G$ containing~$H$, with $K\simeq\Spin(7)$, and most (but not all) of our main results hold.
We discuss the case \eqref{eqn:nonsimple-case} separately in Section~\ref{sec:ex(*)} (see also Remark~\ref{rem:main-can-fail}).
\end{remark}

\subsection{Generators and relations for $\D_G(X)$}

We prove Theorem~\ref{thm:main} by finding explicit relations among the three subalgebras $\D_{\tilG}(X)=\dd\ell(Z(\tilg_{\C}))$, $\iota(\D_K(F))=\dd r(Z(\kk_{\C}))$ (see \eqref{eqn:DKF-drZk}), and $\dd\ell(Z(\g_{\C}))$ of $\D_G(X)$.
In particular, we find explicit algebraically independent generators of $\D_G(X)$ chosen from any two of the three subalgebras, as follows.

\begin{theorem} \label{thm:main-explicit}
In the setting~\ref{setting}, suppose that $\tilG$ is simple.
Let $K$ be a maximal connected proper subgroup of~$G$ containing~$H$ if $\h$ is not a maximal proper subalgebra of~$\g$, and $K=H$ otherwise.
Let $F=K/H$.

(1) There exist elements $P_k$ of $\D_{\tilG}(X)$, elements $Q_k$ of $\D_K(F)$, elements $R_k$ of $Z(\g_{\C})$, and integers $m,n,s,t\in\N$ with $m+n=s+t$ such that
\begin{itemize}
  \item $\D_{\tilG}(X) = \C[P_1,\dots,P_m]$ is a polynomial ring in the~$P_k$;
  \item $\D_K(F) = \C[Q_1,\dots,Q_n]$ is a polynomial ring in the~$Q_k$;
  \item $\D_G(X) = \C[P_1,\dots,P_m,\iota(Q_1),\dots,\iota(Q_n)]$
    
  \hspace{0.73cm} $= \C[\iota(Q_1),\dots,\iota(Q_s),\dd\ell(R_1),\dots,\dd\ell(R_t)]$\\
is a polynomial ring in the $P_k$ and~$\iota(Q_k)$, as well as in the $\iota(Q_k)$ and~$\dd\ell(R_k)$.
\end{itemize}

(2) The $P_k,Q_k,R_k$ can be chosen in such a way that for any~$k$ there exist constants $a_k,b_k,c_k\in\C$ with
\begin{equation}\label{eqn:rel-diff-op}
a_k\,P_k + b_k\,\iota(Q_k) = c_k\,\dd\ell(R_k).
\end{equation}

(3) The $P_k,Q_k,R_k$ can always be chosen in such a way that $\D_G(X) = \C[P_1,\dots,P_m,\dd\ell(R_1),\dots,\dd\ell(R_n)]$ is also a polynomial ring in the $P_k$ and $\dd\ell(R_k)$, unless we are in case (ix) of Table~\ref{table1} up to a covering of~$\tilG$.
\end{theorem}

Theorem~\ref{thm:main-explicit}.(1) gives an algebra isomorphism
\begin{equation} \label{eqn:InvOpTensor}
\D_G(X) \simeq \D_{\tilG}(X) \otimes \D_K(F).
\end{equation}
In this setting there are two expressions of $X$ as a homogeneous space, namely $\tilG/\tilH$ and $G/H$.
Both are spherical, but their ranks are different in general; as an immediate consequence of \eqref{eqn:InvOpTensor}, we obtain the following relation with the rank of the spherical homogeneous space $K/H$.

\begin{corollary} \label{cor:rank}
In the setting of Theorem~\ref{thm:main-explicit}, we have
$$\rank \tilG/\tilH + \rank K/H = \rank G/H.$$
\end{corollary}

Corollary~\ref{cor:rank} also holds in the case \eqref{eqn:nonsimple-case} by Proposition~\ref{prop:ex(*)}.
Table~\ref{table2} gives the ranks of $\tilG/\tilH$, $K/H$, and $G/H$ in each case.

\begin{table}[h!]
\centering
\begin{tabular}{|p{2cm}|p{2cm}|p{2cm}|p{2cm}|}
\hline
& & & \tabularnewline [-0.4cm]
& \centering $\rank\tilG/\tilH$ & \centering $\rank K/H$ & \centering $\rank G/H$ \tabularnewline
\hline
\centering (i), (i)$'$ & \centering $1$ & \centering $1$ & \centering $2$\tabularnewline
\centering (ii) & \centering $\lfloor\frac{n+1}{2}\rfloor$ & \centering $\lfloor\frac{n}{2}\rfloor$ & \centering $n$\tabularnewline
\centering (iii) & \centering $1$ & \centering $1$ & \centering $2$\tabularnewline
\centering (iv) & \centering $n$ & \centering $n$ & \centering $2n$\tabularnewline
\centering (v), (v)$'$ &  \centering $1$ & \centering $1$ & \centering $2$\tabularnewline
\centering (vi) &  \centering $1$ & \centering $1$ & \centering $2$\tabularnewline
\centering (vii) &  \centering $1$ & \centering $1$ & \centering $2$\tabularnewline
\centering (viii) &  \centering $1$ & \centering $2$ & \centering $3$\tabularnewline
\centering (ix) &  \centering $1$ & \centering $1$ & \centering $2$\tabularnewline
\centering (x) & \centering $1$ & \centering $0$ & \centering $1$\tabularnewline
\centering (xi) & \centering $1$ & \centering $0$ & \centering $1$\tabularnewline
\centering (xii) & \centering $1$ & \centering $0$ & \centering $1$\tabularnewline
\centering (xiii), (xiii)$'$ & \centering $1$ & \centering $1$ & \centering $2$\tabularnewline
\centering (xiv) & \centering $2$ & \centering $1$ & \centering $3$\tabularnewline
\centering \eqref{eqn:nonsimple-case} & \centering $2$ & \centering $1$ & \centering $3$\tabularnewline
\hline
\end{tabular}
\caption{Ranks of $\tilG/\tilH$, $K/H$, and $G/H$ in each case of Table~\ref{table1} and in case \eqref{eqn:nonsimple-case}}
\label{table2}
\end{table}

The closed formulas \eqref{eqn:rel-diff-op} in Theorem~\ref{thm:main-explicit} for explicit generators $P_k,Q_k,R_k$ are given in Section~\ref{sec:computations} for each triple $(\tilG,\tilH,G)$ according to the classification of Table~\ref{table1}.
These formulas imply the following.

\begin{corollary} \label{cor:rel-Lapl}
In the setting of Theorem~\ref{thm:main-explicit}, let $C_{\tilG}\in Z(\tilg_{\C})$ and $C_G\in Z(\g_{\C})$ be the respective Casimir elements of the complex reductive Lie algebras $\tilg_{\C}$ and~$\g_{\C}$.
Then there exists a nonzero $a\in\R$ such that
\begin{equation} \label{eqn:rel-Lapl}
\dd\ell(C_{\tilG}) \in a\,\dd\ell(C_G) + \dd r(Z(\kk_{\C}))
\end{equation}
as a holomorphic differential operator on $\tilG_{\C}/\tilH_{\C}$.
\end{corollary}

\begin{remark}
In cases (v), (vii), (viii), and~(xiii) of Table~\ref{table1}, the group $\tilG$ is simple but $G$ is not.
The formulas that we compute show for \emph{any} choice of an $\Ad(G)$-invariant nondegenerate symmetric bilinear form on~$\g_{\C}$, the corresponding Casimir element $C_G\in Z(\g_{\C})$ satisfies \eqref{eqn:rel-Lapl} for some nonzero $a\in\R$.
\end{remark}

The formulas \eqref{eqn:rel-diff-op} play a fundamental role in constructing a ``transfer map'' relating the eigenvalues of $Z(\g_{\C})$ and $\D_{\tilG}(X)$ (see Theorem~\ref{thm:transfer}), providing some interaction between the representation of~$\tilg$ and of its subalgebra~$\g$.

\begin{remark} \label{rem:main-can-fail}
An analogue of Theorems \ref{thm:main} and~\ref{thm:main-explicit} may fail in the setting \ref{setting} in the following situations:
\begin{itemize}
  \item Theorem~\ref{thm:main}.(2) may fail if $\tilG$ is only assumed to be semisimple, not simple: this happens in the case \eqref{eqn:nonsimple-case} (see Proposition~\ref{prop:ex(*)-RDGX}).
  \item Theorem~\ref{thm:main}.(1)--(2) may fail if $K$ is not maximal: this happens for
  $$(\tilG,\tilH,G,H) = \big(\SU(2n+2), \Sp(n+1), \U(2n+1), \Sp(n)\times\U(1)\big)$$
  and $K=\Sp(n)\times\U(1)\times\U(1)$ (see Remark~\ref{rem:K-not-max}.(2)).
  \item Theorem~\ref{thm:main}.(1)--(2)--(3) may fail if $X_{\C}$ is not $G_{\C}$-spherical: this happens for
  $$(\tilG,\tilH,G,H) = \big(\SO(4n+4), \SO(4n+3), \Sp(n+1), \Sp(n)\big)$$
  (see Remark~\ref{rem:Spnonsph}).
\end{itemize}
\end{remark}

In the case \eqref{eqn:nonsimple-case}, we shall prove that the ``transfer map'' still exists even though Theorem~\ref{thm:main}.(2) fails, and we shall find a closed formula for it in Proposition~\ref{prop:nu-ex(*)}.
This will be used in the forthcoming paper \cite{kkII} for analysis on the locally symmetric space $\Gamma\backslash\SO(8,\C)/\SO(7,\C)$, see Section~\ref{subsec:intro-applic} below.

\subsection{Transfer maps} \label{subsec:intro-transfer}

We now explain how eigenvalues of the two algebras $Z(\g_{\C})$ and $\D_{\tilG}(X)$ are related through a ``transfer map''.
In the whole section, we work in the setting of Theorem~\ref{thm:main-explicit}.

\subsubsection{Localization}

We start with some general formalism.
Let $\mathcal{I}$ be a maximal ideal of $Z(\kk_{\C})$ and $\langle\mathcal{I}\rangle$ the ideal generated by $\dd r(\mathcal{I})$ in the commutative algebra $\D_G(X)$.
Let
\begin{equation} \label{eqn:qI}
q_{\mathcal{I}} : \D_G(X) \longrightarrow \D_G(X)_{\mathcal{I}} := \D_G(X)/\langle\mathcal{I}\rangle
\end{equation}
be the quotient homomorphism.
Theorem~\ref{thm:main-explicit}.(1)--(2) implies the following.

\begin{proposition} \label{prop:qI}
In the setting of Theorem~\ref{thm:main-explicit}, for any maximal ideal $\mathcal{I}$ of $Z(\kk_{\C})$, the map $q_{\mathcal{I}}$ induces algebra isomorphisms
$$\D_{\tilG}(X) \overset{\sim}{\longrightarrow} \D_G(X)_{\mathcal{I}}$$
and
$$Z(\g_{\C})/\Ker(q_{\mathcal{I}}\circ\dd\ell) \overset{\sim}{\longrightarrow} \D_G(X)_{\mathcal{I}}.$$
\end{proposition}

These isomorphisms combine into an algebra isomorphism
\begin{equation} \label{eqn:phi-ideal-I}
\varphi_{\mathcal{I}} : Z(\g_{\C})/\Ker(q_{\mathcal{I}}\circ\dd\ell) \overset{\sim}{\longrightarrow} \D_{\tilG}(X),
\end{equation}
which induces a natural map
\begin{eqnarray*}
\varphi_{\mathcal{I}}^* : \Hom_{\C\text{-}\mathrm{alg}}(\D_{\tilG}(X),\C) & \overset{\sim}{\longrightarrow} & \Hom_{\C\text{-}\mathrm{alg}}\big(Z(\g_{\C})/\Ker(q_{\mathcal{I}}\circ\dd\ell),\C\big)\\
& & \hspace{0.5cm} \subset\ \Hom_{\C\text{-}\mathrm{alg}}(Z(\g_{\C}),\C).
\end{eqnarray*}
We note that there is no a priori homomorphism between the two algebras $Z(\g_{\C})$ and $\D_{\tilG}(X)$.

\subsubsection{The case of the annihilator of an irreducible representation of~$K$} \label{subsubsec:intro-I-tau}

When $\mathcal{I}$ is the annihilator of an irreducible representation of~$K$, the map $\varphi_{\mathcal{I}}^*$ has a geometric meaning, which we formulate below as a ``transfer map''.

For each irreducible $K$-module $(\tau,V_{\tau})$ with nonzero $H$-fixed vectors, we consider the isotypic $K$-module $W_{\tau} := (V_{\tau}^{\vee})^H\otimes V_{\tau}$ and form the $G$-equivariant vector bundle $\W_{\tau} := G\times_K W_{\tau}$ over $Y:=G/K$.
The group $G$ acts by translations on the space $C^{\infty}(Y,\W_{\tau})$ of smooth sections of this bundle, and we may view $C^{\infty}(Y,\W_{\tau})$ as a subrepresentation via the natural injective $G$-homomorphism
$$\ii_{\tau} : C^{\infty}(Y,\W_{\tau}) \longhookrightarrow C^{\infty}(X)$$
(see Section~\ref{subsec:L2decomp}).
In our setting, $W_{\tau}$ is isomorphic to~$V_{\tau}$ because the subspace of $H$-fixed vectors in~$\tau$ is one-dimensional (see Lemma~\ref{lem:spherical2}.(4) and Fact~\ref{fact:spherical-cpt}.(iv)).
The center $Z(\g_{\C})$ of the enveloping algebra $U(\g_{\C})$ acts on the space $C^{\infty}(Y,\W_{\tau})$ of smooth sections as differential operators which are $G$-invariant, and thus we have an algebra homomorphism into the ring $\D_G(Y,\W_{\tau})$ of $G$-invariant differential operators acting on $C^{\infty}(Y,\W_{\tau})$:
$$\dd\ell^{\tau} : Z(\g_{\C}) \longrightarrow \D_G(Y,\W_{\tau}).$$
We relate joint eigenfunctions for $\D_{\tilG}(X)$ on $C^{\infty}(X)$ to joint eigenfunctions for $Z(\g_{\C})$ on $C^{\infty}(Y,\W_{\tau})$ as follows.

Let $\mathcal{I}_{\tau}$ be the annihilator in $Z(\kk_{\C})$ of the contragredient representation $\tau^{\vee}$ of~$\tau$.
We have $\Ker(q_{\mathcal{I}_{\tau}}\circ\dd\ell) \subset \Ker(\dd\ell^{\tau})$, and the action of $Z(\g_{\C})$ on $C^{\infty}(Y,\W_{\tau})$ factors through $Z(\g_{\C})/\Ker(q_{\mathcal{I}_{\tau}}\circ\dd\ell)$.
The algebra isomorphism $\varphi_{\mathcal{I}_{\tau}}$ of \eqref{eqn:phi-ideal-I} implies that $\ii_{\tau}$ transfers joint eigenfunctions for $Z(\g_{\C})$ on the subrepresentation $C^{\infty}(Y,\W_{\tau})$ to joint eigenfunctions for $\D_G(X)$ on $C^{\infty}(X)$ via $\varphi_{\mathcal{I}_{\tau}}^*$.
Such an algebra isomorphism $\varphi_{\mathcal{I}_{\tau}}$ also exists when $(\tilG,\tilH,G)$ is the triple \eqref{eqn:nonsimple-case}, see Proposition~\ref{prop:nu-ex(*)}.
To describe the relation between joint eigenvalues for $\D_{\tilG}(X)$ and $Z(\g_{\C})$, we introduce ``transfer maps''
\begin{equation} \label{eqn:nu-lambda-tau}
\left\{ \begin{array}{l}
\nnu(\cdot,\tau) : \Hom_{\C\text{-}\mathrm{alg}}(\D_{\tilG}(X),\C) \longrightarrow \Hom_{\C\text{-}\mathrm{alg}}(Z(\g_{\C}),\C),\\
\llambda(\cdot,\tau) : \Hom_{\C\text{-}\mathrm{alg}}(Z(\g_{\C})/\Ker(\dd\ell^{\tau}),\C) \longrightarrow \Hom_{\C\text{-}\mathrm{alg}}(\D_{\tilG}(X),\C)
\end{array} \right.
\end{equation}
for every $\tau\in\Disc(K/H)$ by using the bijection $\varphi_{\mathcal{I}_{\tau}}^*$ as follows:
$$\xymatrixcolsep{4pc}
\xymatrix{
& \Hom_{\C\text{-}\mathrm{alg}}(Z(\g_{\C}),\C)\\
\Hom_{\C\text{-}\mathrm{alg}}(\D_{\tilG}(X),\C) \ar[ur]^{\nnu(\cdot,\tau)} \ar[r]^-{\sim}_-{\varphi_{\mathcal{I}_{\tau}}^*} & \ar@{}[u]|{\textstyle\cup} \Hom_{\C\text{-}\mathrm{alg}}(Z(\g_{\C})/\Ker(q_{\mathcal{I}_{\tau}}\circ\dd\ell),\C)\\
\mathrm{Spec}(X)_{\tau} \ar@{}[u]|{\textstyle\cup} \ar[r] & \ar[ul]^{\llambda(\cdot,\tau)} \ar@{}[u]|{\textstyle\cup} \Hom_{\C\text{-}\mathrm{alg}}(Z(\g_{\C})/\Ker(\dd\ell^{\tau}),\C)}$$
Here we set
$$C^{\infty}(X;\M_{\lambda})_{\tau} := \big\{ F\in\ii_{\tau}(C^{\infty}(Y,\W_{\tau})) ~:~ PF=\lambda(P)F \quad\forall P\in\D_{\tilG}(X)\big\}$$
and
\begin{equation} \label{eqn:Spec-tau-X}
\mathrm{Spec}(X)_{\tau} := \big\{\lambda\in\Hom_{\C\text{-}\mathrm{alg}}(\D_{\tilG}(X),\C) \,:\, C^{\infty}(X;\M_{\lambda})_{\tau}\neq\{ 0\}\big\}.
\end{equation}
We shall see (Proposition~\ref{prop:phiItau-vanish}) that $\varphi_{\mathcal{I}_{\tau}}^*(\lambda)$ vanishes on $\Ker(\dd\ell^{\tau})$ if $\lambda\in\Spec(X)_{\tau}$, hence
$\nnu(\llambda(\nu,\tau)) = \nu$ for all $\nu\in\Hom_{\C\text{-}\mathrm{alg}}(Z(\g_{\C})/\Ker(\dd\ell^{\tau}),\C)$ and $\llambda(\nnu(\lambda,\tau)) = \lambda$ for all $\lambda\in\mathrm{Spec}(X)_{\tau}$.

By using the closed formulas in Theorem~\ref{thm:main-explicit}.(1)--(2), we find an explicit formula for the transfer map
$$\nnu(\cdot,\tau) := \varphi_{\mathcal{I}_{\tau}}^* : \Hom_{\C\text{-}\mathrm{alg}}(\D_{\tilG}(X),\C)\longrightarrow\Hom_{\C\text{-}\mathrm{alg}}(Z(\g_{\C}),\C)$$
in terms of the highest weight of~$\tau$ and the Harish-Chandra isomorphisms
\begin{eqnarray*}
\Hom_{\C\text{-}\mathrm{alg}}(Z(\g_{\C}),\C) & \overset{\sim}{\longrightarrow} & \jj_{\C}^*/W(\g_{\C}),\\
\Hom_{\C\text{-}\mathrm{alg}}(\D_{\tilG}(X),\C) & \overset{\sim}{\longrightarrow} & \tila_{\C}^*/\widetilde{W},
\end{eqnarray*}
where $\jj_{\C}$ and~$\tila_{\C}$ are certain abelian subspaces of $\g_{\C}$ and~$\tilg_{\C}$, respectively, and $W(\g_{\C})$ and~$\widetilde{W}$ are finite reflection groups (see Section~\ref{subsec:nu-tau} for details).
We note that there is no a priori homomorphism between $\tila_{\C}$ and~$\jj_{\C}$.

\begin{theorem} \label{thm:transfer}
In the setting of Theorem~\ref{thm:main-explicit}, for any irreducible $K$-module $\tau$ with nonzero $H$-fixed vectors, there is an affine map $S_{\tau} : \tila_{\C}^* \to \jj_{\C}^*$ such that the transfer map
$$\nnu(\cdot,\tau) : \Hom_{\C\text{-}\mathrm{alg}}(\D_{\tilG}(X),\C)\longrightarrow\Hom_{\C\text{-}\mathrm{alg}}(Z(\g_{\C}),\C)$$
is given by $S_{\tau} : \tila_{\C}^*/\widetilde{W} \to \jj_{\C}^*/W(\g_{\C})$ via the Harish-Chandra isomorphisms.
This means that for any $\lambda\in\tila_{\C}^*$ and $\nu\in\jj_{\C}^*$ with $\nu=S_{\tau}(\lambda)$ mod $W(\g_{\C})$, the following two conditions on $f\in C^{\infty}(Y,\W_{\tau})$ are equivalent:
\begin{eqnarray*}
\dd\ell^{\tau}(R) f & = & \nu(R)\,f \quad\quad\quad\quad\forall R\in Z(\g_{\C}),\\
D(\ii_{\tau}f) & = & \lambda(D)\,\ii_{\tau}f \quad\quad\quad\forall D\in\D_{\tilG}(X).
\end{eqnarray*}
\end{theorem}

Theorem~\ref{thm:transfer} also holds in the case \eqref{eqn:nonsimple-case}.
We refer to Theorem~\ref{thm:nu-tau} for a more precise statement.
An explicit formula for the affine map $S_{\tau}$ for all~$\tau$ is obtained in Section~\ref{sec:computations} for each triple $(\tilG,\tilH,G)$ of Table~\ref{table1}, and in Section~\ref{sec:ex(*)} for the triple \eqref{eqn:nonsimple-case}.

\subsection{Application to noncompact real forms} \label{subsec:intro-applic}

We may reformulate Theorem~\ref{thm:main} in terms of complex Lie algebras, as follows.
Suppose $\tilg_{\C}\supset\g_{\C},\tilh_{\C},\kk_{\C}$ are reductive Lie algebras over~$\C$ such that
$$\left\{\begin{array}{l}
\tilg_{\C} = \g_{\C} + \tilh_{\C},\\
\h_{\C} := \g_{\C}\cap\tilh_{\C} \subset \kk_{\C} \subset \g_{\C}.
\end{array}
\right.$$
The ring $\D_{\tilG_{\C}}(\tilG_{\C}/\tilH_{\C})$ of $\tilG_{\C}$-invariant holomorphic differential operators on the complex homogeneous space $\tilG_{\C}/\tilH_{\C}$ is isomorphic to the ring $\D_G(G/H)$ and does not depend on the choice of connected complex Lie groups $\tilG_{\C}\supset\tilH_{\C}$ with Lie algebras $\tilg_{\C}\supset\tilh_{\C}$ in our setting, see Theorem~\ref{thm:HoloInvDiff} below.
Consider the following two conditions on the quadruple $(\tilg_{\C},\g_{\C},\tilh_{\C},\kk_{\C})$:
\begin{enumerate}
  \item[($\widetilde{\mathrm{A}}$)] $\D_{\tilG_{\C}}(\tilG_{\C}/\tilH_{\C})$ is contained in the $\C$-algebra generated by $\dd\ell(Z(\g_{\C}))$ and $\dd r(Z(\kk_{\C}))$;
  \item[($\widetilde{\mathrm{B}}$)] $\dd\ell(Z(\g_{\C}))$ is contained in the $\C$-algebra generated by $\D_{\tilG_{\C}}(\tilG_{\C}/\tilH_{\C})$ and $\dd r(Z(\kk_{\C}))$.
\end{enumerate}
Here is an immediate consequence of Theorem~\ref{thm:main}.(1)--(2).

\begin{corollary}
In the setting \ref{setting}, suppose $\tilG$ is simple.
Let $\kk$ be a maximal proper Lie subalgebra of~$\g$ containing $\tilh\cap\g$.
Then conditions ($\widetilde{\mathrm{A}}$) and~($\widetilde{\mathrm{B}}$) both hold for the quadruple $(\tilg_{\C},\g_{\C},\tilh_{\C},\kk_{\C})$.
\end{corollary}

\begin{remark}
In the case \eqref{eqn:nonsimple-case} where $\tilG$ is semisimple but not simple, condition ($\widetilde{\mathrm{B}}$) still holds (Proposition~\ref{prop:ex(*)}), but condition ($\widetilde{\mathrm{A}}$) fails (Proposition~\ref{prop:ex(*)-RDGX}).
\end{remark}

Since the rings of invariant differential operators depend only on the complexification (see Theorem~\ref{thm:HoloInvDiff}), our results hold for any real forms having the same complexification.
In particular, the generators $P_k$, $\iota(Q_k)$, and $\dd\ell(R_k)$ are defined on real forms of~$X_{\C}$ by the restriction of their holomorphic continuation, satisfying the same relations \eqref{eqn:rel-diff-op}.
Similarly, the relation \eqref{eqn:rel-Lapl} for the Casimir operators in Corollary~\ref{cor:rel-Lapl} holds on any real forms of $X_{\C}=\tilG_{\C}/\tilH_{\C}$.

Let $\tau\in\Disc(K/H)$.
The transfer map $\nnu(\cdot,\tau)$ of \eqref{eqn:nu-lambda-tau} gives certain constraints on $Z(\g_{\C})$-infinitesimal characters of irreducible $G$-modules realized in $C^{\infty}(Y,\W_{\tau})$.
We now formulate this more explicitly by using the argument of holomorphic continuation and the affine map $S_{\tau}$.
In the setting of Theorem~\ref{thm:main-explicit}, let $\tilG_{\C}\supset\tilH_{\C},G_{\C},K_{\C}$ be the complexifications of the compact Lie groups $\tilG\supset\tilH,G,K$, and let $\tilG_{\R}\supset\tilH_{\R},G_{\R},K_{\R}$ be other real forms.
We set $H_{\R}:=\tilH_{\R}\cap G_{\R}$ and $X_{\R}:=G_{\R}/H_{\R}$.
For simplicity, we assume that $K_{\R}=K$ and $H_{\R}=H$, hence $G_{\R}$ acts properly on~$X_{\R}$.
We use the same letter $\W_{\tau}$ to denote the $G_{\R}$-equivariant vector bundle over $Y_{\R}:=G_{\R}/K_{\R}$, which is given by the restriction of the holomorphic $G_{\C}$-equivariant vector bundle $\W_{\tau}^{\C}$ over $Y_{\C}:=G_{\C}/K_{\C}$ to the totally real submanifold~$Y_{\R}$.
For $\lambda\in\jj_{\C}^*/W(\g_{\C})$, we define the space of joint eigensections for $Z(\g_{\C})$ by
$$C^{\infty}(Y_{\R},\W_{\tau};\M_{\lambda}) := \{ f\in C^{\infty}(Y_{\R},\W_{\tau}) ~:~ \dd\ell^{\tau}(z)f = \chi_{\lambda}^G(z)f \quad\forall z\in Z(\g_{\C})\},$$
see \eqref{eqn:HCchi} for the notation of Harish-Chandra homomorphisms.
The set of possible infinitesimal characters for subrepresentations of the regular representation $C^{\infty}(Y_{\R},\W_{\tau})$ is defined by
$$\Supp_{Z(\g_{\C})}(C^{\infty}(Y_{\R},\W_{\tau})) := \{ \lambda\in\jj_{\C}^*/W(\g_{\C}) ~:~ C^{\infty}(Y_{\R},\W_{\tau};\M_{\lambda})\neq\{0\}\}.$$
Theorem~\ref{thm:transfer} implies the following.

\begin{corollary}
In the setting of Theorem~\ref{thm:main-explicit}, for any $\tau\in\Disc(K/H)$,
$$\Supp_{Z(\g_{\C})}(C^{\infty}(Y_{\R},\W_{\tau})) \subset S_{\tau}(\tila_{\C}^*)\ \mod W(\g_{\C}),$$
where $S_{\tau} : \tila_{\C}^*\to\jj_{\C}^*$ is the affine map of Theorem~\ref{thm:transfer}.
\end{corollary}

In \cite{kkII}, we shall prove that under condition~($\widetilde{\mathrm{B}}$), any irreducible unitary representation $\pi$ of~$\tilG_{\R}$ realized in the space $\DD'(X_{\R})$ of distributions on~$X_{\R}$ is discretely decomposable when restricted to~$G_{\R}$, even when $G_{\R}$ is noncompact.
Then the relations \eqref{eqn:rel-diff-op} give crucial information for the bran\-ching law of irreducible representations $\pi$ of $\tilG_{\R}$ restricted to~$G_{\R}$, using the analysis on the fiber
\begin{equation} \label{eqn:fiber}
F := K/H \longrightarrow X_{\R} \longrightarrow Y_{\R}=G_{\R}/K_{\R}.
\end{equation}
In subsequent papers, we use the present results to find:
\begin{enumerate}[(a)]
  \item relationships between spectrum for Riemannian locally symmetric spaces $\Gamma\backslash G_{\R}/K_{\R}$ and spectrum for pseudo-Riemannian manifolds $\Gamma\backslash\tilG_{\R}/\tilH_{\R}$, using Theorems \ref{thm:main}.(2) and~\ref{thm:transfer}, see \cite{kkII};
  \item explicit branching laws of irreducible unitary representations of~$\tilG_{\R}$ (\eg Zuckerman's derived functor modules $A_{\mathfrak{q}}(\lambda)$) when restricted to the subgroup~$G$, using Theorem~\ref{thm:main}.(1), see \cite{kkI}.
\end{enumerate}
Thus in both (a) and~(b) we obtain results on infinite-dimensional representations of noncompact groups by reducing to finite-dimensional representations of compact groups and using Theorem~\ref{thm:main}.

\subsection{Remarks}

The idea of studying the interaction between harmonic analysis on homogeneous spaces with overgroups and branching laws of infinite-dimensional representations goes back to the papers \cite{kob93,kob94,kob09}, where computations were carried out in some situations where $\tilG/\tilH$ is a symmetric space of rank one.
The work of the current paper was started in the spring of 2011, as an attempt to generalize the machinery of \cite{kob93,kob94,kob09} to cases where $\tilG/\tilH$ has higher rank, and to find the right general framework in which such results hold.
Our results were announced in \cite{kob17}.

One important motivation for this paper has been the application to the analysis on locally pseudo-Riemannian symmetric spaces, as described in Section~\ref{subsec:intro-applic} and in \cite{kkII}.
Relations between Casimir operators as in Corollary~\ref{cor:rel-Lapl} were also announced by Mehdi--Olbrich at a talk at the Max Planck Institute in Bonn in August 2011.

Recently Schlichtkrull--Trapa--Vogan put on the arXiv the preprint \cite{stv18}, investigating the rank-one cases (i), (iv), (vi), (x) of Table~\ref{table1}, and proving the irreducibility of the representations of the exceptional group $G_{2(2)}$ in \cite[Th.\,6.4]{kob94} for the last singular parameters.

\subsection{Organization of the paper}

Sections \ref{sec:reminders} to~\ref{sec:disconnected-H} are of a theoretical nature.
Our analysis is centered around the $G$-equivariant fiber bundle $X=\tilG/\tilH\overset{F}{\longrightarrow}G/K$.
In Section~\ref{sec:reminders} we collect some basic facts on invariant differential operators and explain the diagram \eqref{eqn:diagram}.
In Section~\ref{sec:anal-fiber-bund} we discuss geometric approaches to the restriction of representations of $\tilG$ to the subgroup~$G$ in the space of square integrable or holomorphic sections.
The assumption that $X_{\C}$ is $G_{\C}$-spherical implies several multiplicity-freeness results for representations, not only of~$G$, but also of $\tilG$ and~$K$.
Using this, in Section~\ref{sec:strategy} we explain a precise strategy for proving Theorems \ref{thm:main} and~\ref{thm:transfer}.
In Section~\ref{sec:disconnected-H} we examine the connected components of $H=\tilH\cap G$, and prove that the subalgebras $\D_{\tilG}(X)$, $\dd r(Z(\kk_{\C}))$, and $\dd\ell(Z(\tilg_{\C}))$ are completely determined by the triple of Lie algebras $(\tilg_{\C},\tilh_{\C},\g_{\C})$.

Sections \ref{sec:computations} and~\ref{sec:ex(*)} are the technical heart of the paper: we complete the proofs of the main theorems through a case-by-case analysis.
In particular, we find the closed formula for the ``transfer map'' for simple~$\tilG$ in Section~\ref{sec:computations} in each case of Table~\ref{table1}, by carrying out computations of finite-dimensional representations.
Section~\ref{sec:ex(*)} focuses on the case of the triple \eqref{eqn:nonsimple-case}.

\addtocontents{toc}{\SkipTocEntry}
\subsection*{Notation}\label{subsec:notation-computations}

In the whole paper, we use the notation $\N=\Z\cap [0,+\infty)$ and $\N_+=\Z\cap (0,+\infty)$.
For $n\in\N_+$ we set
$$(\Z^n)_{\geq} := \{ (a_1,\dots,a_n)\in\Z^n : a_1\geq\dots\geq a_n\} $$
and $(\N^n)_{\geq}:=(\Z^n)_{\geq}\cap\N^n$.

\addtocontents{toc}{\SkipTocEntry}
\subsection*{Acknowledgements}

We would like to thank the referee for a careful reading of the paper and for very helpful comments and suggestions.
We are grateful to the University of Tokyo for its support through the GCOE program, and to the University of Chicago, the Max Planck Institut f\"ur Mathematik (Bonn), the Mathematical Sciences Research Institute (Berkeley), and the Institut des Hautes \'Etudes Scientifiques (Bures-sur-Yvette) for giving us opportunities to work together in very good conditions.

\section{Reminders and basic facts}
\label{sec:reminders}

In this section we set up some notation and review some known facts on spherical homogeneous spaces, in particular about invariant differential operators and regular representations.

Let $X=G/H$ be a reductive homogeneous space, by which we mean that $G$ is a connected real reductive linear Lie group and $H$ a reductive subgroup of~$G$.
We shall always assume that $H$ is algebraic.
The group $G$ naturally acts on the ring of differential operators on~$X$ by
$$g\cdot D = \ell_g^{\ast}\circ D\circ (\ell_g^{\ast})^{-1},$$
where $\ell_g^{\ast}$ is the pull-back by the left translation $\ell_g : x\mapsto g\cdot x$.
We denote by $\D_G(X)$ the ring of $G$-invariant differential operators on~$X$.

\subsection{General structure of $\D_G(X)$} \label{subsec:DGH}

We first recall some classical results on the structure of the $\C$-algebra $\D_G(X)$; see \cite[Ch.\,II]{hel00} for proofs and more details.
Let $U(\g_{\C})$ be the enveloping algebra of the complexified Lie algebra $\g_{\C}:=\g\otimes_{\R}\C$.
It acts on $C^{\infty}(X)$ by differentiation on the left:
$$\big((Y_1\cdots Y_m)\cdot f\big)(g) = \frac{\partial}{\partial t_1}\Big|_{t_1=0}\,\cdots\ \frac{\partial}{\partial t_m}\Big|_{t_m=0}\ f\big(\exp(-t_mY_m)\cdots\exp(-t_1Y_1) x\big)$$
for all $Y_1,\dots,Y_m\in\g$, all $f\in C^{\infty}(X)$, and all $x\in X$.
This gives a $\C$-algebra homomorphism
\begin{equation}\label{eqn:dl}
\dd\ell : U(\g_{\C}) \longrightarrow \D(X),
\end{equation}
where $\D(X)$ is the full $\C$-algebra of differential operators on~$X$.
On the other hand, $U(\g_{\C})$ acts on $C^{\infty}(G)$ by differentiation on the right:
$$\big((Y_1\cdots Y_m)\cdot f\big)(g) = \frac{\partial}{\partial t_1}\Big|_{t_1=0}\,\cdots\ \frac{\partial}{\partial t_m}\Big|_{t_m=0}\ f\big(g\exp(t_1Y_1)\cdots\exp(t_mY_m)\big)$$
for all $Y_1,\dots,Y_m\in\g$, all $f\in C^{\infty}(G)$, and all $g\in G$.
By identifying $C^{\infty}(X)$ with the set of right-$H$-invariant elements in~$C^{\infty}(G)$, we obtain a $\C$-algebra homomorphism
\begin{equation}\label{eqn:dr}
\dd r :\ U(\g_{\C})^H \,\longrightarrow\, \D_G(X),
\end{equation}
where $U(\g_{\C})^H$ is the subalgebra of $\Ad_G(H)$-invariant elements in $U(\g_{\C})$.
It is surjective and induces an algebra isomorphism
\begin{equation}\label{eqn:DX}
U(\g_{\C})^H/U(\g_{\C})\h_{\C} \cap U(\g_{\C})^H \,\overset{\sim}\longrightarrow\, \D_G(X)
\end{equation}
(see \cite[Ch.\,II, Th.\,4.6]{hel00}).

Since the center $Z(\g_{\C})$ is contained in $U(\g_{\C})^H$, the homomorphisms $\dd\ell$ and $\dd r$ of \eqref{eqn:dl} and \eqref{eqn:dr} restrict to homomorphisms from $Z(\g_{\C})$ to $\D_G(X)$.
To see the relationship between them, consider the inversion $g\mapsto g^{-1}$ of~$G$.
Its differential gives rise to an antiautomorphism $\eta$ of the enveloping algebra $U(\g_{\C})$, given by $Y_1\cdots Y_m\mapsto (-Y_m)\cdots (-Y_1)$ for all $Y_1,\dots,Y_m\in\g_{\C}$.
This antiautomorphism induces an automorphism of the commutative subalgebra $Z(\g_{\C})$.
The following is an immediate consequence of the definitions.

\begin{lemma}\label{lem:lreta}
We have $\dd\ell\circ\eta=\dd r$ on $Z(\g_{\C})$.
\end{lemma}

\subsection{Spherical homogeneous spaces} \label{subsec:DGH-spherical}

Recall the following two characterizations of spherical homogeneous spaces, in terms of the ring of invariant differential operators (condition~(ii)) and in terms of representation theory (condition~(iii)).
For a continuous representation $\pi$ of~$G$, we denote by $\Hom_G(\pi,C^{\infty}(X))$ the set of $G$-intertwining continuous operators from $\pi$ to $C^{\infty}(X)$.

\begin{fact}\label{fact:spherical}
Suppose $X=G/H$ is a reductive homogeneous space.
Then the following conditions are equivalent:
\begin{enumerate}[(i)]
  \item $X_{\C}=G_{\C}/H_{\C}$ is $G_{\C}$-spherical;
  \item the $\C$-algebra $\D_G(X)$ is commutative;
  \item $\dim\Hom_G(\pi,C^{\infty}(X))$ is uniformly bounded for any irreducible representation $\pi$ of~$G$.
\end{enumerate}
\end{fact}

For (i)\,$\Leftrightarrow$\,(ii), see \eg \cite{vin01}; for (i)\,$\Leftrightarrow$\,(iii), see \cite{ko13}.

If $X_{\C}=G_{\C}/H_{\C}$ is $G_{\C}$-spherical, then by work of Knop \cite{kno94} the ring $\D_G(X)$ is finitely generated as a $Z(\g_{\C})$-module, and there is a $\C$-algebra isomorphism
\begin{equation}\label{eqn:Psi}
\Psi : \D_G(X) \overset{\scriptscriptstyle\sim\,}{\longrightarrow} S(\aaa_{\C})^W,
\end{equation}
where $S(\aaa_{\C})^W$ is the ring of $W$-invariant elements in the symmetric algebra $S(\aaa_{\C})$ for some subspace $\aaa_{\C}$ of a Cartan subalgebra of~$\g_{\C}$ and some finite reflection group $W$ acting on~$\aaa_{\C}$.
In particular, $\D_G(X)$ is a polynomial algebra in $r$ generators, by a theorem of Chevalley (see \eg \cite[Th.\,2.1.3.1]{war72}), where
$$r := \dim_{\C} \aaa_{\C}$$
is called the \emph{rank} of $G/H$, denoted by $\rank G/H$.
A typical example of a spherical homogeneous space is a complex reductive symmetric space; in this case the isomorphism $\D_G(X)\simeq S(\aaa_{\C})^W$ is explicit, as we shall recall in Section~\ref{subsec:symmsp}.

\subsection{A geometric interpretation of the subalgebra $\dd r(Z(\kk_{\C}))$}\label{subsec:dliota}

Let $K$ be a connected reductive subgroup of~$G$ containing~$H$.
The reductive homogeneous space $X:=G/H$ fibers over $G/K$ with fiber $F:=K/H$.
There is a natural injective homomorphism
\begin{equation}\label{eqn:iotaF}
\iota : \D_K(F) \longhookrightarrow \D_G(X)
\end{equation}
defined as follows: for any $D\in\D_K(F)$, any $f\in C^{\infty}(X)$, and any $g\in G$,
\begin{equation}\label{eqn:defiota}
\big(\iota(D)f\big)|_{gF} = \big((\ell_g^{\ast})^{-1} \circ D \circ \ell_g^{\ast}\big)(f|_{gF}),
\end{equation}
where $\ell_g : X\rightarrow X$ is the translation by~$g$ and $\ell_g^{\ast} : C^{\infty}(X)\rightarrow C^{\infty}(X)$ the pull-back by~$\ell_g$.
Note that in \eqref{eqn:defiota} the right-hand side does not depend on the representative $g$ in $gF$ since $D$ is $K$-invariant.
Thus $\iota(D)$ is defined ``along the fibers $gF$ of the bundle $X=G/H\rightarrow G/K$'', and makes the following diagram commute for any $g\in G$ (where the unlabeled horizontal arrows denote restriction).
$$\xymatrixcolsep{3pc}
\xymatrix{
C^{\infty}(X) \ar[d]^{\iota(D)} \ar[r] & C^{\infty}(gF) \ar[r]^{\ell_g^{\ast}} & C^{\infty}(F) \ar[d]^D\\
C^{\infty}(X) \ar[r] & C^{\infty}(gF) \ar[r]^{\ell_g^{\ast}} & C^{\infty}(F)
}$$
Similarly to \eqref{eqn:dr}, we can define a map
\begin{equation}\label{eqn:dr-F}
\dd r_{\scriptscriptstyle F} : U(\kk_{\C})^H \longrightarrow \D_K(F).
\end{equation}
In particular, $\dd r_{\scriptscriptstyle F}$ is defined on the center $Z(\kk_{\C})$ of the enveloping algebra $U(\kk_{\C})$.
The following diagram commutes.
$$\xymatrixcolsep{3pc}
\xymatrix{
Z(\kk_{\C}) \ar[d]^{\dd r_{\scriptscriptstyle F}} \ar@{^{(}->}[r] & U(\g_{\C})^H \ar@{->>}[d]^{\dd r}\\
\D_K(F) \ar@{^{(}->}[r]^{\iota} & \D_G(X)
}$$

\subsection{The case of reductive symmetric spaces} \label{subsec:symmsp}

Reductive symmetric spaces are a special case of spherical homogeneous spaces, and the results of Section~\ref{subsec:DGH-spherical} are known in a more explicit form in this case, as we now explain.
We also collect a few other useful facts on symmetric spaces.

Note that in most cases of Table~\ref{table1}, both $\tilG/\tilH$ and $F=K/H$ are symmetric spaces; in Section~\ref{sec:strategy}, we shall apply the present results to $\tilG/\tilH$ and $F$ instead of $X=G/H$, replacing
$(\aaa\subset\jj, W, W(\g_{\C}), \rho=\rho_{\aaa}+\rho_{\m})$
with
$(\tila\subset\widetilde{\jj}, \widetilde{W}, W(\tilg_{\C}), \widetilde{\rho}=\rho_{\tila}+\rho_{\widetilde{m}})$
and
$(\aaa_F\subset\jj_K, W_F, W(\kk_{\C}), \rho_{\kk}=\rho_{\aaa_F}+\rho_{\m_F})$.

Suppose that $X=G/H$ is a reductive symmetric space, \ie $H$ is an open subgroup of the group of fixed points of~$G$ under some involutive automorphism~$\sigma$.
Let $\g=\h+\q$ be the decomposition of~$\g$ into eigenspaces of~$\dd\sigma$, with respective eigenvalues $+1$ and~$-1$.
Fix a maximal semisimple abelian subspace~$\aaa$ of~$\q$; we shall call such a subspace a \emph{Cartan subspace for the symmetric space $G/H$}.
Let $W$ be the Weyl group of the restricted root system $\Sigma(\g_{\C},\aaa_{\C})$ of~$\aaa_{\C}$ in~$\g_{\C}$.
There is a natural $\C$-algebra isomorphism
\begin{equation}\label{eqn:HC-isom}
\Psi : \D_G(X) \overset{\scriptscriptstyle\sim\,}{\longrightarrow} S(\aaa_{\C})^W
\end{equation}
as in Section~\ref{subsec:DGH-spherical}, known as the Harish-Chandra isomorphism.
Any $\nu\in\aaa_{\C}^{\ast}/W$ gives rise to a $\C$-algebra homomorphism
\begin{eqnarray*}
\chi_{\nu}^X :\ \D_G(X) & \longrightarrow & \quad\;\; \C\\
D\quad\;\; & \longmapsto & \langle\Psi(D),\nu\rangle.
\end{eqnarray*}
We extend $\aaa_{\C}$ to a Cartan subalgebra $\jj_{\C}$ of~$\g_{\C}$ and write $W(\g_{\C})$ for the Weyl group of the root system $\Delta(\g_{\C},\jj_{\C})$.

Harish-Chandra's original isomorphism concerned a special case of reductive symmetric spaces, namely \emph{group manifolds} $(G\times G)/\Diag(G)\simeq G$.
In this case the isomorphism amounts to
\begin{equation}\label{eqn:HC-isom-group}
\Phi : Z(\g_{\C}) \,\simeq\, \D_{G\times G}(G) \,\overset{\sim}{\longrightarrow}\, S(\jj_{\C})^{W(\g_{\C})}.
\end{equation}
Any $\lambda\in\jj_{\C}^{\ast}/W(\g_{\C})$ induces a $\C$-algebra homomorphism $\chi_{\lambda}^G : Z(\g_{\C})\rightarrow\C$, and we have a natural description of the set of maximal ideals of $Z(\g_{\C})$ as follows:
\begin{eqnarray}
\jj_{\C}^*/W(\g_{\C}) & \overset{\sim}{\longrightarrow} & \Hom_{\C\text{-}\mathrm{alg}}(Z(\g_{\C}),\C)\label{eqn:HCchi}\\
\lambda & \longmapsto & \chi_{\lambda}^G.\nonumber
\end{eqnarray}
We now discuss the relationship between $\chi_{\nu}^X$ and~$\chi_{\lambda}^G$.

Fix a positive system $\Delta^+(\g_{\C},\jj_{\C})$ of roots of~$\jj_{\C}$ in~$\g_{\C}$ and let $\Sigma^+(\g_{\C},\aaa_{\C})$ be a positive system of restricted roots of $\aaa_{\C}$ in~$\g_{\C}$ such that the restriction map $\alpha\mapsto\alpha|_{\jj_{\C}}$ sends $\Delta^+(\g_{\C},\jj_{\C})$ to $\Sigma^+(\g_{\C},\aaa_{\C})\cup\nolinebreak\{0\}$.
We set $\ttt_{\C}:=\jj_{\C}\cap\h_{\C}$.
Then we have a direct sum decomposition $\jj_{\C}=\ttt_{\C}+\aaa_{\C}$.
Let $\rho_{\aaa}$ (\resp $\rho$) be half the sum of the elements of $\Sigma^+(\g_{\C},\aaa_{\C})$ (\resp $\Delta^+(\g_{\C},\jj_{\C})$), counted with multiplicities, and let $\rho_{\m}:=\rho-\rho_{\aaa}$.
Then $\rho=\rho_{\m}+\rho_{\aaa}\in\jj_{\C}^{\ast}=\ttt_{\C}^{\ast}+\aaa_{\C}^{\ast}$.
The $\rho_{\m}$-shift map $\nu\mapsto\nu+\rho_{\m}$ from $\aaa_{\C}^{\ast}$ to~$\jj_{\C}^{\ast}$ induces a map
\begin{equation}\label{eqn:aj}
T : \aaa_{\C}^{\ast}/W \longrightarrow \jj_{\C}^{\ast}/W(\g_{\C}),
\end{equation}
which is independent of the choice of the positive systems.
The relationship between $\chi_{\nu}^X$ and~$\chi_{\lambda}^G$ is then given as follows.

\begin{lemma}\label{lem:chiZgDGX}
For any $\nu\in\aaa_{\C}^{\ast}/W$, the following diagrams commute.
$$\xymatrixcolsep{1.5pc}
\xymatrix{
Z(\g_{\C}) \ar[dd]_{\dd\ell} \ar[r]^-{\sim}_-{\Psi} & S(\jj_{\C})^{W(\g_{\C})} \ar[dr]^{\chi_{-T(\nu)}^G} & & Z(\g_{\C}) \ar[dd]^{\dd r} \ar[r]^-{\sim}_-{\Psi} & S(\jj_{\C})^{W(\g_{\C})} \ar[dr]^{\chi_{T(\nu)}^G} & \\
& & \C & & & \C\\
\D_G(X) \ar[r]^-{\sim} & S(\aaa_{\C})^W \ar[ur]^{\chi_{\nu}^X} & & \D_G(X) \ar[r]^-{\sim} & S(\aaa_{\C})^W \ar[ur]^{\chi_{\nu}^X} &
}$$
\end{lemma}

\begin{proof}
For the left diagram, see \cite{hel92} or \cite[Ch.\,2, \S\,1.5]{war72}.
The commutativity of the right diagram follows from that of the left and from Lemma~\ref{lem:lreta}.
\end{proof}

The following fact is due to Helgason \cite{hel92}.

\begin{fact}\label{fact:ZgDGX}
If $G$ is a classical group, then $T$ is injective and the $\C$-algebra homomorphisms $\dd\ell : Z(\g_{\C})\rightarrow\D_G(X)$ and $\dd r : Z(\g_{\C})\rightarrow\D_G(X)$ are surjective.
\end{fact}

The Cartan--Weyl highest weight theory establishes a bijection between irreducible finite-dimensional representations of~$\g_{\C}$ and dominant integral weights with respect to the positive system $\Delta^+(\g_{\C},\jj_{\C})$:
$$\Rep(\g_{\C},\lambda) \longleftrightarrow \lambda.$$
When it exists, we denote by $\Rep(G,\lambda)$ the lift of $\Rep(\g_{\C},\lambda)$ to the connected compact group~$G$.
Among such representation, the irreducible finite-dimensional representations with nonzero $H_{\C}$-fixed vectors are characterized by the following theorem of Cartan--Helgason (see \eg \cite[Th.\,3.3.1.1]{war72}):

\begin{fact}[Cartan--Helgason theorem]\label{fact:CartanHelgason}
Suppose $X=G/H$ is a compact reductive symmetric space, and let $\lambda\in\jj_{\C}^{\ast}$ be a dominant integral weight with respect to $\Delta^+(\g_{\C},\jj_{\C})$.
\begin{enumerate}
  \item The representation $\Rep(\g_{\C},\lambda)$ has a nonzero $\h_{\C}$-fixed vector if and only if
  \begin{equation}\label{eqn:CarHel}
  \lambda|_{\ttt_{\C}}=0 \quad\quad \mathrm{and} \quad\quad \frac{\langle\lambda,\alpha\rangle}{\langle\alpha,\alpha\rangle} \in \N \quad \forall\alpha\in\Sigma^+(\g_{\C},\aaa_{\C}).
  \end{equation}
  In this case, the space of $\h_{\C}$-fixed vectors in $\Rep(\g_{\C},\lambda)$ is one-dimen\-sional, and we shall regard $\lambda$ as an element of~$\aaa_{\C}^{\ast}$ since $\lambda|_{\ttt_{\C}}=0$.
  \item Suppose $\lambda$ satisfies \eqref{eqn:CarHel}.
  Then $\Rep(\g_{\C},\lambda)$ lifts to a representation $\Rep(G,\lambda)$ of~$G$ if and only if $\lambda\in\aaa_{\C}^{\ast}$ lifts to the compact torus\linebreak $\exp\aaa\,\,(\subset\!\!G)$.
  In this case, if $H$ is connected, then the $G$-module $\Rep(\g_{\C},\lambda)$ is realized uniquely in the regular representation $C^{\infty}(X)$.
  \item The algebra $\D_G(X)$ acts on $\Rep(\g_{\C},\lambda)$ by the scalars $\chi_{\lambda+\rho_{\aaa}}^X$.
\end{enumerate}
\end{fact}

\subsection{A surjectivity result}

In the general setting of Theorem~\ref{thm:main-explicit}, we observe the following.

\begin{lemma}\label{lem:surj-dl-drF}
In the setting \ref{setting}, suppose that $\tilG$ is simple or that the triple $(\tilG,\tilH,G)$ is \eqref{eqn:nonsimple-case}, and let $K$ and $F=K/H$ be as in Theorem~\ref{thm:main-explicit} or Remark~\ref{rem:nonsimple-case}.
Then the homomorphisms
\begin{eqnarray*}
\dd\ell & : & Z(\tilg_{\C}) \longrightarrow \D_{\tilG}(X),\\
\dd r_{\scriptscriptstyle F} & : & Z(\kk_{\C}) \longrightarrow \D_K(F)
\end{eqnarray*}
of \eqref{eqn:dl} and \eqref{eqn:dr-F} are surjective.
\end{lemma}

\begin{proof}
Suppose $\tilG$ is simple.
It follows from the classification of Table~\ref{table1} that $\tilG/\tilH$ is always a classical symmetric space, except in cases (viii) and~(ix), where $\tilG/\tilH=\SO(7)/G_{2(-14)}$, and in cases (xi) and~(xiii), where $\tilG/\tilH=\SO(8)/\Spin(7)$.
Similarly, $F=K/H$ is always a classical symmetric space or a singleton, except in case (v)$'$, where
$$F \!=\! ((\Sp(n)\times\Sp(1))\cdot\U(1))/(\Sp(n)\cdot\Diag(\U(1))) \simeq (\Sp(1)\times\U(1))/\Diag(\U(1)),$$
in case (viii), where $F=(\SO(4)\times\SO(2))/\iota_8(\U(2))$ (see Section~\ref{subsec:ex(viii)} for the definition of~$\iota_8$), and in the example of Section~\ref{sec:ex(*)}, where $F=\Spin(7)/G_{2(-14)}$.
Thus, by Fact~\ref{fact:ZgDGX}, we only need to prove that $\dd\ell : Z(\g_{\C})\rightarrow\D_G(X)$ is surjective in the following four cases:
\begin{enumerate}
  \item $X=G/H=\SO(7)/G_{2(-14)}$;
  \item $X=G/H=\SO(8)/\Spin(7)$;
  \item $X=G/H=(\Sp(1)\times\U(1))/\Diag(\U(1))$;
  \item $X=G/H=(\SO(4)\times\SO(2))/\iota_8(\U(2))$.
\end{enumerate}
For~(1) we see from Lemma~\ref{lem:struct-DGX} and Lemma~\ref{lem:ex(xi)}.(3) below that $\D_G(X)$ is generated by the Casimir operator.
For~(2) we reduce to the classical symmetric space $\SO(8)/\SO(7)$ by taking a double covering and using the triality of~$D_4$ (see Section~\ref{subsec:ex(vii)}).
For~(3) we note that $\D_G(X)$ is generated by the Casimir operators of $\Sp(1)$ and the Euler operator of $\U(1)$.
For~(4) we see from Lemmas \ref{lem:struct-DGX} and~\ref{lem:ex(viii)}.(5) below that $\D_G(X)$ is generated by the Casimir operator of $\SO(4)$ and the Euler operator of $\SO(2)$.

Suppose $(\tilG,\tilH,G)$ is the triple \eqref{eqn:nonsimple-case}.
Then $X=\tilG/\tilH$ is a direct product of two copies of $\Spin(8)/\Spin(7)$, and $F=K/H=\Spin(7)/G_{2(-14)}$, hence both $\dd\ell$ and $\dd r_{\scriptscriptstyle F}$ are surjective.
\end{proof}

\section{Analysis on fiber bundles and branching laws} \label{sec:anal-fiber-bund}

In this section, we collect some useful results on finite-dimensional representations of compact groups.
A similar approach will be used in \cite{kkII} to deal with infinite-dimensional representations of noncompact groups; this is why we use the terminology of discrete series representations here.

\subsection{Discrete series representations}\label{subsec:discreteseries}

Let $G$ be a unimodular Lie group and $H$ a closed unimodular subgroup.
The homogeneous space $G/H$ carries a $G$-invariant Radon measure.
Recall that an irreducible unitary representation~$\pi$ of~$G$ is called a \emph{discrete series representation} for $X=G/H$ if there exists a nonzero continuous $G$-intertwining operator from $\pi$ to the regular representation of~$G$ on~$L^2(X)$ or, equivalently, if $\pi$ can be realized as a closed $G$-invariant subspace of $L^2(X)$.
Let $\widehat{G}$ be the unitary dual of~$G$, \ie the set of equivalence classes of irreducible unitary representations of~$G$.
We shall denote by $\Disc(G/H)$ the subset of~$\widehat{G}$ consisting of unitary equivalence classes of discrete series representations for $G/H$.

We now assume that $G$ is compact.
Then any $\pi\in\widehat{G}$ is finite-dimensional.
By the Frobenius reciprocity theorem, $\Disc(G/H)$ is the set of equivalence classes of irreducible finite-dimensional representations $\pi$ of~$G$ with nonzero $H$-fixed vectors.
Furthermore,
$$\dim\Hom_G(\pi,L^2(X)) = [\pi|_H:\mathbf{1}] := \dim V_{\pi}^{H},$$
where $V_{\pi}^{H}$ is the subspace of $H$-invariant vectors in the representation space $V_{\pi}$ of~$\pi$.
Here is a version of Fact~\ref{fact:spherical} for compact~$G$.

\begin{fact}\label{fact:spherical-cpt}
Let $G$ be a connected compact Lie group.
Then the following conditions on $(G,H)$ are equivalent:
\begin{enumerate}[(i)]
  \item $X_{\C}=G_{\C}/H_{\C}$ is $G_{\C}$-spherical;
  \item the $\C$-algebra $\D_G(X)$ is commutative;
  \item the discrete series for $G/H$ have uniformly bounded multiplicities;
  \item $G/H$ is multiplicity-free (\ie all discrete series for $G/H$ have multiplicity~$1$).
\end{enumerate}
\end{fact}

For (i)\,$\Leftrightarrow$\,(iv), see \cite{vk78}; for (iii)\,$\Leftrightarrow$\,(iv), see \cite{kra76}.

When $X=G/H$ is a reductive symmetric space, the set $\Disc(G/H)$ is described by the Cartan--Helgason theorem (Fact~\ref{fact:CartanHelgason}.(2)).
For nonsymmetric spherical $X=G/H$ with $G$ simple, the set $\Disc(G/H)$ was determined by Kr\"amer \cite{kra79}.
We shall consider nonsymmetric spherical $X=G/H$ with an overgroup~$\tilG$ as in Table~\ref{table1} (where $G$ is not necessarily simple); in this case, the description of $\Disc(G/H)$ is enriched in Section~\ref{sec:computations} by a description of the branching laws of representations for the restriction $\tilG\downarrow G$.

\subsection{A decomposition of $L^2(X)$ using discrete series for a fiber}\label{subsec:L2decomp}

Let $\tilG$ be a compact connected Lie group and $\tilH,G$ two connected subgroups of~$\tilG$ such that $\tilG=\tilH G$.
Let $H:=\tilH\cap G$ and let $K$ be a connected subgroup of~$G$ containing~$H$ (see Proposition~\ref{prop:H-subset-K} for later applications).
The space $X:=G/H$ fibers over $Y:=G/K$ with fiber $F:=K/H$.
For any finite-dimensional (complex)  irreducible representation $(\tau,V_{\tau})$ of~$K$, we set
$$W_{\tau} := V_{\tau} \otimes (V_{\tau}^{\vee})^H \simeq V_{\tau} \otimes_{\C} \C^{\ell_{\tau}},$$
where $(\tau^{\vee},V_{\tau}^{\vee})$ is the contragredient representation and $\ell_{\tau}:=[\tau|_H:\mathbf{1}]\in\N$; by definition, $\ell_{\tau}\neq 0$ if and only if $\tau\in\Disc(K/H)$.
The matrix coefficient
\begin{equation} \label{eqn:matcoeff}
W_{\tau} \ni u\otimes v' \longmapsto \langle\tau(\cdot)^{-1}u,v'\rangle = \langle u,\tau^{\vee}(\cdot)v'\rangle \in C^{\infty}(K)
\end{equation}
induces an injective $K$-homomorphism $W_{\tau}\rightarrow C^{\infty}(K/H)$, yielding the isotypic decomposition
$$L^2(K/H) \simeq\ \ \sumplus{\tau\in\Disc(K/H)}\ W_{\tau}$$
of the regular representation of $K$ on $L^2(K/H)$.
(Here $\sum^{\oplus}$ denotes the Hilbert completion of the algebraic direct sum.)
For any~$\tau$, let $L^2(Y,\W_{\tau})$ be the Hilbert space of square-integrable sections of the Hermitian vector bundle
$$\W_{\tau} := G \times_K W_{\tau} \ \longrightarrow\ Y.$$
The group $G$ naturally acts on $L^2(Y,\W_{\tau})$ as a unitary representation, the \emph{regular representation}.
The Hilbert space $L^2(Y,\W_{\tau})$ identifies with the space of square-integrable, $K$-equivariant maps $G\rightarrow\nolinebreak W_{\tau}$.
(Here the action of $K$ on~$G$ is by right translation.)
The $K$-homomorphism $W_{\tau}\hookrightarrow C^{\infty}(K/H)$ induces a $(G\times K)$-homomorphism $C^{\infty}(G,W_{\tau})\hookrightarrow C^{\infty}(G,C^{\infty}(K/H))$, where $G\times K$ acts on the domain $C^{\infty}(G,\W_{\tau})$ via $\mathrm{id}\times\mathrm{diag} : G\times K\hookrightarrow G\times G\times K$.
Taking $K$-invariant elements yields a $G$-homomorphism
$$\begin{array}{cccc}
\ii_{\tau} : & C^{\infty}(Y,\W_{\tau})\quad & \longhookrightarrow & C^{\infty}(X)\\
& \text{\rotatebox{-90}{$\simeq$}} & & \text{\rotatebox{-90}{$\simeq$}}\\
& & & \\
& C^{\infty}(G,W_{\tau})^K & \longhookrightarrow & C^{\infty}(G,C^{\infty}(K/H))^K.
\end{array}$$
Since the map $C^{\infty}(G,\W_{\tau})\to C^{\infty}(G,C^{\infty}(K/H))$ is a $(K\times H)$-homomorphism, it commutes with the infinitesimal action of $U(\g_{\C})\otimes U(\kk_{\C})$, hence in particular of $Z(\g_{\C})\otimes Z(\kk_{\C})$.
This action preserves $K$-invariant elements.
Thus for any $Q'\in Z(\kk_{\C})$, any $R\in Z(\g_{\C})$, and any $\varphi\in L^2(Y,\W_{\tau})\cap C^{\infty}(Y,\W_{\tau})$,
\begin{eqnarray*}
\dd r({Q'}^{\vee})\big(\ii_{\tau}(\varphi)\big) & = & \ii_{\tau}(\dd\tau(Q')\,\varphi),\\
\dd\ell(R)\big(\ii_{\tau}(\varphi)\big) & = & \ii_{\tau}(\dd\ell(R)\,\varphi).
\end{eqnarray*}
Here ${}^{\vee} : U(\kk_{\C})\to U(\kk_{\C})$ denotes the anti-automorphism of the enveloping algebra induced by $\kk_{\C}\to\kk_{\C}$, $z\mapsto -z$.
The restriction to $Z(\kk_{\C})$ is actually an automorphism because $Z(\kk_{\C})$ is commutative.

With appropriate normalizations of the $G$-invariant measures on $Y=G/K$ and $X=G/H$, this defines an isometric embedding
\begin{equation}\label{eqn:itau}
\ii_{\tau} :\ L^2(Y,\W_{\tau}) \longhookrightarrow L^2(X)
\end{equation}
of Hilbert spaces.
The embeddings~$\ii_{\tau}$ induce a unitary operator
\begin{equation}\label{eqn:operator-i}
\ii\ :\quad \sumplus{\tau\in\Disc(K/H)}\ L^2(Y,\W_{\tau}) \ \overset{\sim}{\longrightarrow}\ L^2(X).
\end{equation}

\subsection{Application of the Borel--Weil theorem to branching laws}

In this section we give an upper estimate for possible irreducible summands in branching laws by using a geometric realization of representations via the Borel--Weil theorem and the analysis of the conormal bundle for orbits of the subgroup.
The results here will be used in the proofs of Lemmas \ref{lem:ex(vi)} and~\ref{lem:ex(*)}.

Let $G$ be a connected compact Lie group with Lie algebra~$\g$.
There exists a unique complex reductive Lie group $G_{\C}$ with Lie algebra $\g_{\C}:=\g\otimes_{\R}\C$ such that $G$ is a maximal compact subgroup of~$G_{\C}$.

Given an element $A\in\sqrt{-1}\g$, we define the subalgebras $\n_{\C}\equiv\n_{\C}(A)$, $\llll_{\C}\equiv\llll_{\C}(A)$, and $\n_{\C}^-\equiv\n_{\C}^-(A)$ as the sum of the eigenspaces of $\ad(A)$ with positive, zero, and negative eigenvalues, respectively.
We say that $A$ is the \emph{characteristic element} of the parabolic subalgebra $\p_{\C}:=\llll_{\C}+\n_{\C}$.
The opposite parabolic subalgebra is denoted by $\p_{\C}^-:=\llll_{\C}+\n_{\C}^-$.
We write $P_{\C}=L_{\C}N_{\C}$ and $P_{\C}^-=L_{\C}N_{\C}^-$ for the parabolic subgroups of~$G_{\C}$ with Lie algebras $\p_{\C}$ and $\p_{\C}^-$, respectively.

We take a Cartan subalgebra $\jj_{\C}$ of~$\g_{\C}$, and fix a positive system $\Delta^+(\g_{\C},\jj_{\C})$.
The parabolic subagebra $\p_{\C}$ is called \emph{standard} if the characteristic element $A\in\jj_{\C}$ is dominant with respect to $\Delta^+(\g_{\C},\jj_{\C})$.

For a holomorphic finite-dimensional representation $(\sigma,V)$ of~$P_{\C}^-$, we form a $G_{\C}$-equivariant holomorphic vector bundle
$$\V := G_{\C} \times_{P_{\C}^-} V$$
over the (partial) flag variety $G_{\C}/P_{\C}^-$.
We shall write $\mathcal{L}_{\lambda}$ for~$\V$ if $(\sigma,V)$ is a one-dimensional representation whose differential restricted to~$\jj_{\C}$ is given by $\lambda\in\jj_{\C}^*$.
There is a natural representation of $G_{\C}$ on the space $\mathcal{O}(G_{\C}/P_{\C}^-,\V)$ of holomorphic sections of the bundle $\V\to G_{\C}/P_{\C}^-$, which is irreducible or zero whenever $(\sigma,V)$ is irreducible as a $P_{\C}^-$-module.
More precisely, if $(\sigma,V)$ is an irreducible representation of~$L_{\C}$ with highest weight $\mu\in\jj_{\C}^*$ for $\Delta^+(\llll_{\C},\jj_{\C}) := \Delta(\llll_{\C},\jj_{\C}) \cap \Delta^+(\g_{\C},\jj_{\C})$ extended to $P_{\C}=L_{\C}N_{\C}^-$ with trivial $N_{\C}^-$-action, then the Borel--Weil theorem gives the following isomorphism of $G_{\C}$-modules:
$$\mathcal{O}(G_{\C}/P_{\C}^-,\V) \simeq \left\{ \begin{array}{ll}
\Rep(G_{\C},\mu) & \text{if $\mu$ is $\Delta^+(\g_{\C},\jj_{\C})$-dominant},\\
\{0\} & \text{otherwise}.
\end{array}\right.$$

We now apply this geometric realization of finite-dimensional representations to obtain an upper bound for possible irreducible representations that may occur in the restriction of representations.
From now, we consider a pair of complex reductive Lie groups $G_{\C}\subset\tilG_{\C}$.
We use a parabolic subgroup of~$\tilG_{\C}$ that has the following compatibility property with~$G_{\C}$.

\begin{definition}[{\cite[Def.\,3.7]{kob12}}]
Let $\g_{\C}\subset\tilg_{\C}$ be a pair of reductive Lie algebras.
A parabolic subalgebra $\widetilde{\p}_{\C}$ of~$\tilg_{\C}$ is \emph{$\g_{\C}$-compatible} if $\widetilde{\p}_{\C}$ is given by a characteristic element $A$ in~$\g_{\C}$.
\end{definition}

We shall also say that a parabolic subgroup $\widetilde{P}_{\C}$ of~$\tilG_{\C}$ is \emph{$G_{\C}$-compatible} if its Lie algebra $\widetilde{\p}_{\C}$ is $\g_{\C}$-compatible, where $G_{\C}$ is a reductive subgroup of~$\tilG_{\C}$ with Lie algebra~$\g_{\C}$.
If $\widetilde{\p}_{\C}=\llll_{\C}+\n_{\C}$ is the Levi decomposition given by a characteristic element $A$ in~$\g_{\C}$, then $\p_{\C}:=\widetilde{\p}_{\C}\cap\g_{\C}$ is a parabolic subalgebra of~$\g_{\C}$ with Levi decomposition
$$\p_{\C} = \llll_{\C} + \n_{\C} := (\widetilde{\llll}_{\C}\cap\g_{\C}) + (\widetilde{\n}_{\C}\cap\g_{\C}).$$

Since the holomorphic cotangent bundle of the flag variety $\tilG_{\C}/\widetilde{P}_{\C}^-$ is given as the homogeneous vector bundle
$$\tilG_{\C} \times_{\widetilde{P}_{\C}^-} \widetilde{\n}_{\C}^- \longrightarrow G_{\C}/P_{\C}^-,$$
the holomorphic conormal bundle for the submanifold $G_{\C}/P_{\C}^-\hookrightarrow\tilG_{\C}/\widetilde{P}_{\C}^-$ is given by
\begin{align*}
T^*_{G_{\C}/P_{\C}^-}(\tilG_{\C}/\widetilde{P}_{\C}^-) & = \mathrm{Ker}\big(T^*(\tilG_{\C}/\widetilde{P}_{\C}^-)\big|_{G_{\C}/P_{\C}^-} \longrightarrow \quad T^*(G_{\C}/P_{\C}^-)\big)\\
& \simeq G_{\C} \times_{P_{\C}^-} (\widetilde{\n}_{\C}^-/\n_{\C}^-).
\end{align*}
Since a holomorphic section is determined by its restriction to a submanifold with all normal derivatives, we obtain the following upper estimate for possible irreducible representations of the subgroup~$G_{\C}$ occurring in the branching law of the restriction of representations.

\begin{proposition} \label{prop:branchnormal}
Let $\tilG_{\C}\supset G_{\C}$ be a pair of connected complex reductive Lie groups, and let $\widetilde{P}_{\C}$ be a $G_{\C}$-compatible parabolic subgroup of~$\tilG_{\C}$.
For any $\tilG_{\C}$-equivariant holomorphic vector bundle $\widetilde{\V}$ over $\tilG_{\C}/\widetilde{P}_{\C}^-$, we have an injective $G_{\C}$-homomorphism
$$\mathcal{O}\big(\tilG_{\C}/\widetilde{P}_{\C}^-,\widetilde{\V}\big)\big|_{G_{\C}} \longhookrightarrow \bigoplus_{\ell=0}^{+\infty} \, \mathcal{O}\Big(G_{\C}/P_{\C}^-,\V\big|_{G_{\C}/P_{\C}^-} \otimes \mathcal{S}^{\ell}(\widetilde{\n}_{\C}^-/\n_{\C}^-)\Big),$$
where $\mathcal{S}^{\ell}(\widetilde{\n}_{\C}^-/\n_{\C}^-) \simeq G_{\C} \times_{P_{\C}^-} S^{\ell}(\widetilde{\n}_{\C}^-/\n_{\C}^-)$ is the $\ell$-th symmetric tensor bundle of the holomorphic conormal bundle.
\end{proposition}

Applying Proposition~\ref{prop:branchnormal} to the pair $G_{\C} \subset G_{\C}\times G_{\C}$, we obtain the following upper estimate for possible irreducible representations occurring in the tensor product representations.

\begin{proposition} \label{prop:tensornormal}
Let $P_{\C}$ and $Q_{\C}$ be standard parabolic subgroups of a connected complex reductive Lie group~$G_{\C}$.
Suppose that $\lambda,\nu\in\jj_{\C}^*$ are dominant with respect to $\Delta^+(\g_{\C},\jj_{\C})$ and that they lift to one-dimensional holomorphic characters of the opposite parabolic subgroups $P_{\C}^-$ and $Q_{\C}^-$, respectively.
Then we have an injective $G_{\C}$-homomorphism
$$\mathcal{O}(G_{\C}/P_{\C}^-,\mathcal{L}_{\lambda}) \otimes \mathcal{O}(G_{\C}/Q_{\C}^-,\mathcal{L}_{\nu}) \subset \bigoplus_{\ell=0}^{+\infty} \mathcal{O}\Big(G_{\C}/(P_{\C}^-\cap Q_{\C}^-),\mathcal{L}_{\lambda+\nu} \otimes \mathcal{S}^{\ell}(\n_{\C}^-\cap\mathfrak{u}_{\C}^-)\Big),$$
where $\n_{\C}^-$ and~$\mathfrak{u}_{\C}^-$ are the nilpotent radicals of the parabolic subalgebras $\p_{\C}^-$ and~$\q_{\C}^-$, respectively, and $\mathcal{S}^{\ell}(\n_{\C}^-\cap\mathfrak{u}_{\C}^-)$ is the $G_{\C}$-equivariant holomorphic vector bundle $G_{\C} \times_{P_{\C}^-\cap Q_{\C}^-} S^{\ell}(\n_{\C}^-\cap\mathfrak{u}_{\C}^-)$ over the flag variety $G_{\C}/(P_{\C}^-\cap Q_{\C}^-)$.
\end{proposition}

\section{General strategy for the proof of Theorems \ref{thm:main} and~\ref{thm:main-explicit}}\label{sec:strategy}

In this section we give a method for finding explicit relations among three subalgebras of $\D_G(X)$.
The basic tools are finite-dimensional representations and their branching laws, looking at the function space $L^2(X)$ in two different ways.
The key point, under the assumption that $X_{\C}$ is $G_{\C}$-spherical, is the existence of a map $\vartheta\mapsto (\pi(\vartheta),\tau(\vartheta))$ relating discrete series representations for $G/H$, $\tilG/\tilH$, and $K/H$ via branching laws, see Proposition~\ref{prop:pi-tau-theta} below.
We summarize the precise steps of the proof of Theorems \ref{thm:main} and~\ref{thm:main-explicit} in Section~\ref{subsec:descript-comput}, and that of Theorem~\ref{thm:nu-tau} (hence of Theorem~\ref{thm:transfer}) in Section~\ref{subsec:strategy-transfer}.
The explicit computations will be carried out case by case in Section~\ref{sec:computations} for $\tilG$ simple, and in Section~\ref{sec:ex(*)} in the case \eqref{eqn:nonsimple-case} where $\tilG$ is not simple.

\subsection{A double decomposition for $L^2(X)$}

We use branching laws for the restriction $\tilG\downarrow G$ to derive explicit relations among the generators of $\D_{\tilG}(X)$, $\iota(\D_K(F))$, and $\dd\ell(Z(\g_{\C}))$.
Let us explain this idea in more detail.

We decompose the regular representation on $L^2(X)$ into irreducible $G$-modules in two different ways.
The first way is to begin by decomposing the regular representation $L^2(X)\simeq L^2(\tilG/\tilH)$ into irreducible $\tilG$-modules, then use branching laws $\tilG\downarrow G$ as in \cite{kob93,kob94}:
\begin{eqnarray}\label{eqn:decomp1}
L^2(X) & \simeq & \sumplus{\pi\in\Disc(\tilG/\tilH)}\ [\pi|_{\tilH}:\mathbf{1}] \ \pi\nonumber\\
& \simeq & \sumplus{\pi\in\Disc(\tilG/\tilH)} \Bigg(\bigoplus_{\vartheta\in\widehat{G}}\ [\pi|_{\tilH}:\mathbf{1}] \ [\pi|_G:\vartheta] \ \vartheta\Bigg),
\end{eqnarray}
where $[\pi|_G:\vartheta]:=\dim\Hom_G(\vartheta,\pi|_G)\geq 0$ is the dimension of the space of $G$-intertwining operators from $\vartheta$ to the restriction of~$\pi$ to~$G$.
The second way is to expand functions on~$X$ along the fiber, and decompose $L^2(X)=L^2(G/H)$ using the unitary operator \eqref{eqn:operator-i}, and then to further decompose each summand into irreducible $G$-modules:
\begin{eqnarray}\label{eqn:decomp2}
L^2(X) & \simeq & \sumplus{\tau\in\Disc(K/H)}\ L^2(Y,\W_{\tau}) \nonumber\\
& \simeq & \sumplus{\tau\in\Disc(K/H)} \Bigg(\,\sumplus{\vartheta\in\widehat{G}}\ [\tau|_H:\mathbf{1}]\ [\vartheta|_K\!:\tau]\ \vartheta\Bigg),
\end{eqnarray}
where $[\vartheta|_K\!:\tau]=\dim\Hom_K(\tau,\vartheta|_K)\in\N$.

We compute the action of $\dd\ell(Z(\tilg_{\C}))$ and~$\dd\ell(Z(\g_{\C}))$ on each summand $\vartheta$ of \eqref{eqn:decomp1}, and the action of $\dd r(Z(\kk_{\C}))$ and~$\dd\ell(Z(\g_{\C}))$ on each summand $\vartheta$ of \eqref{eqn:decomp2}.
These actions can be compared explicitly (see Proposition~\ref{prop:PQR} below) if each $\vartheta$ appears only once in $L^2(X)$, which is the case if $X_{\C}=G_{\C}/H_{\C}$ is $G_{\C}$-spherical (Fact~\ref{fact:spherical-cpt}).
Using this method and applying Lemma~\ref{lem:chiZgDGX} to $\tilG$ and~$K$, we find explicit linear relations among the generators of $\D_{\tilG}(X)$, $\dd r(Z(\kk_{\C}))$, and $\dd\ell(Z(\g_{\C}))$, in particular among the Casimir operators $\dd\ell(C_{\tilG})$, $\dd r(C_K)$, and~$\dd\ell(C_G)$.

\subsection{Sphericity and strong multiplicity-freeness}\label{subsec:strong-mult-free}

We now give a method to find relations among generators of the three algebras $\dd\ell(Z(\tilg_{\C}))$, $\dd r(Z(\kk_{\C}))$, and $\dd\ell(Z(\g_{\C}))$, using finite-dimensional representations.

A key tool is the following canonical map.

\begin{proposition} \label{prop:pi-tau-theta}
In the setting \ref{setting}, let $K$ be any connected subgroup of~$G$ containing~$H$.
If $X_{\C}=G_{\C}/H_{\C}$ is $G_{\C}$-spherical, then there exists a map
\begin{eqnarray} \label{eqn:DtoDD}
\Disc(G/H) & \longrightarrow & \Disc(\tilG/\tilH) \times \Disc(K/H)\\
\vartheta\hspace{0.9cm} & \longmapsto & \hspace{1.1cm}(\pi(\vartheta),\tau(\vartheta))\nonumber
\end{eqnarray}
such that $[\pi(\vartheta)|_G:\vartheta] = [\vartheta|_K:\tau(\vartheta)]=1$ for all $\vartheta\in\Disc(G/H)$.
\end{proposition}

We note that in our setting, $\Disc(G/H)$, $\Disc(\tilG/\tilH)$, and $\Disc(K/H)$ are free abelian semigroups, and their numbers of generators satisfy
$$\rank G/H = \rank \tilG/\tilH + \rank G/H$$
by Corollary~\ref{cor:rank}.

Proposition~\ref{prop:pi-tau-theta} is an immediate consequence of points (3) and~(6) of the following lemma, which summarizes some consequences of the $G_{\C}$-sphericity of~$X_{\C}$ in the presence of an overgroup~$\tilG_{\C}$.

\begin{lemma} \label{lem:spherical2}
In the setting of Proposition~\ref{prop:pi-tau-theta},
\begin{enumerate}
  \item $X_{\C}$ is $\tilG_{\C}$-spherical;
  \item for any $\pi\in\Disc(\tilG/\tilH)$, the restriction $\pi|_G$ is multiplicity-free (\ie $[\pi|_G:\vartheta]=1$ for all $\vartheta\in\widehat{G}$);
  \item for any $\vartheta\in\Disc(G/H)$ there is a unique element $\pi(\vartheta)\in\Disc(\tilG/\tilH)$ such that $[\pi(\vartheta)|_G:\vartheta]=1$;
  \item $F_{\C}=K_{\C}/H_{\C}$ is $K_{\C}$-spherical;
  \item $[\vartheta|_K\!:\tau]\leq 1$ for all $\vartheta\in\widehat{G}$ and $\tau\in\Disc(K/H)$;
  \item for any $\vartheta\in\Disc(G/H)$ there is a unique element $\tau(\vartheta)\in\Disc(K/H)$ such that $[\vartheta|_K:\tau(\vartheta)]=1$.
\end{enumerate}
\end{lemma}

\begin{proof}[Proof of Lemma~\ref{lem:spherical2}]
Decompose $L^2(X)$ into irreducible $G$-modules as in \eqref{eqn:decomp1} and \eqref{eqn:decomp2}.
Since $X_{\C}$ is $G_{\C}$-spherical, Fact~\ref{fact:spherical-cpt} implies that these decompositions are multiplicity-free.
In particular, $[\pi|_{\tilH}:\mathbf{1}]=1$ for all $\pi\in\Disc(\tilG/\tilH)$ and $[\tau|_H:\mathbf{1}]=1$ for all $\tau\in\Disc(K/H)$, and so (1) and~(4) hold by Fact~\ref{fact:spherical-cpt}.
Moreover, for any $\vartheta\in\widehat{G}$, by considering the multiplicities of $\vartheta$ in the regular representation on $L^2(X)$ in \eqref{eqn:decomp1} and \eqref{eqn:decomp2}, we see that
$$\sum_{\pi\in\Disc(\tilG/\tilH)} [\pi|_G:\vartheta] = \sum_{\tau\in\Disc(K/H)} [\vartheta|_K:\tau] \leq 1,$$
and the inequality is an equality if and only if $\vartheta\in\Disc(G/H)$.
This implies (2), (3), (5), and~(6).
\end{proof}

\begin{remark}\label{rem:strongmf}
Lemma~\ref{lem:spherical2} implies that if $X_{\C}=G_{\C}/H_{\C}$ is $G_{\C}$-spherical, then the double summation \eqref{eqn:decomp1} may be thought of as a \emph{strong multiplicity-free branching law}, in the sense that the restriction $\pi|_G$ is multiplicity-free and that the irreducible summands make up a disjoint union as $\pi$ ranges over $\Disc(\tilG/\tilH)$.
A similar interpretation holds for \eqref{eqn:decomp2}.
\end{remark}

Via the multiplicity-free decomposition
\begin{equation}
L^2(X) \ \simeq\ \sumplus{\vartheta\in\Disc(G/H)} \vartheta
\end{equation}
given by Lemma~\ref{lem:spherical2}, we can diagonalize any $G$-endomorphism of $L^2(X)$ by Schur's lemma.
This idea may also be applied to $G$-invariant differential operators on~$X$, and the map $\vartheta\mapsto (\pi(\vartheta),\tau(\vartheta))$ of Proposition~\ref{prop:pi-tau-theta} may then be interpreted in terms of spectral data, which provide useful information in analyzing the three subalgebras $\D_{\tilG}(X)$, $\dd r(Z(\kk_{\C}))$, and $\dd\ell(Z(\g_{\C}))$ of $\D_G(X)$.
To be more precise, we recall that the center $Z(\g_{\C})$ acts on the representation space of any $\vartheta\in\widehat{G}$ as scalars by Schur's lemma, yielding a $\C$-algebra homomorphism
$$\Psi_{\vartheta} : Z(\g_{\C}) \longrightarrow \C$$
(\emph{$Z(\g_{\C})$-infinitesimal character}).
Similarly, to any $\pi\in\widehat{\tilG}$ corresponds a $Z(\tilg_{\C})$-infinitesimal character $\Psi_{\pi} : Z(\tilg_{\C})\rightarrow\C$, and to any $\tau\in\widehat{K}$ a $Z(\kk_{\C})$-infinitesimal character $\Psi_{\tau} : Z(\kk_{\C})\rightarrow\C$.
We denote by ${}^{\vee} : U(\kk_{\C})\to U(\kk_{\C})$\linebreak the antiautomorphism of the enveloping algebra induced by $\kk_{\C}\to\kk_{\C}$,\linebreak $Z\mapsto -Z$.
Its restriction to the center $Z(\kk_{\C})$ of $U(\kk_{\C})$ is an automorphism since $Z(\kk_{\C})$ is commutative.
We have
$$\Psi_{\tau^{\vee}}(Q') = \Psi_{\tau}({Q'}^{\vee})$$
for all $Q'\in Z(\kk_{\C})$, where $\tau^{\vee}$ is the contragredient representation of~$\tau$.

Using the canonical map $\vartheta\mapsto (\pi(\vartheta),\tau(\vartheta))$ of Proposition~\ref{prop:pi-tau-theta}, we can reduce the question of finding explicit relations among $G$-invariant differential operators on~$X$ to the simpler question of finding identities among polynomials via the evaluation at $\vartheta\in\Disc(G/H)$, by the following proposition.

\begin{proposition} \label{prop:D-spec-Hom}
Let $G$ be a connected compact Lie group and $X=G/H$ where is $H$ a closed subgroup of~$G$, such that $X_{\C}=G_{\C}/H_{\C}$ is $G_{\C}$-spherical.
\begin{enumerate}
  \item There is a map
  $$\psi : \Disc(G/H) \times \D_G(X) \longrightarrow \C$$
  such that any $D\in\D_G(X)$ acts on the $G$-isotypic subspace $U_{\vartheta}$ of $\vartheta$ in $C^{\infty}(X)$ by the scalar $\psi(\vartheta,D)$.
  Moreover, $\psi$ induces an injective algebra homomorphism
  \begin{equation} \label{eqn:psi-Map}
  \tilde{\psi} : \D_G(X) \longhookrightarrow \mathrm{Map}(\Disc(G/H),\C).
  \end{equation}
  \item Suppose that $X\simeq\tilG/\tilH$ for some connected compact overgroup $\tilG$ of~$G$.
  Let $K$ be a connected subgroup of~$G$ containing~$H$.
  Then
  \begin{align*}
  \psi(\vartheta,\dd\ell(P')) & = \Psi_{\pi(\vartheta)}(P') & \mathrm{for\ all\ }P'\in Z(\tilg_{\C}),\\
  \psi(\vartheta,\dd r(Q')) & = \Psi_{\tau(\vartheta)}({Q'}^{\vee}) & \mathrm{for\ all\ }Q'\in Z(\kk_{\C}),\\
  \psi(\vartheta,\dd\ell(R)) & = \Psi_{\vartheta}(R) & \mathrm{for\ all\ }R\in Z(\g_{\C}).
  \end{align*}
\end{enumerate}
\end{proposition}

\begin{proof}
(1) All differential operators $D\in\D_G(X)$ preserve each $G$-isotypic subspace $U_{\vartheta}$.
Since $X_{\C}$ is $G_{\C}$-spherical, $U_{\vartheta}$ is an irreducible $G$-module.
By Schur's lemma, $D$ acts on~$U_{\vartheta}$ by a scalar, which we denote by $\psi(D,\vartheta)\in\C$.
This gives the desired map~$\psi$.
Since the action of $\D_G(X)$ on $C^{\infty}(X)$ is faithful, and since $\bigoplus_{\vartheta\in\Disc(G/H)} U_{\vartheta}$ is dense in $C^{\infty}(X)$, the induced map $\tilde{\psi}$ is injective.

(2) For any $R\in Z(\g_{\C})$ the operator $\dd\ell(R)\in\D_G(X)$ acts on~$U_{\vartheta}$ by the scalar $\Psi_{\vartheta}(R)$.
By definition \eqref{eqn:DtoDD} of $\pi(\vartheta)$, the $G$-module $U_{\vartheta}$ occurs in the $\tilG$-irreducible module $\pi(\vartheta)$, and so for any $P'\in Z(\tilg_{\C})$ the operator $\dd\ell(P')\in\D_G(X)$ acts on~$U_{\vartheta}$ by the scalar $\Psi_{\pi(\vartheta)}(P')$.
By definition \eqref{eqn:DtoDD} of $\tau(\vartheta)$, we have $U_{\vartheta}\subset\ii_{\tau(\vartheta)}(C^{\infty}(Y,\W_{\tau(\vartheta)}))$, and so for any $Q'\in Z(\kk_{\C})$ the operator $\dd r(Q')\in\D_G(X)$ acts on~$\vartheta$ by the scalar $\Psi_{\tau(\vartheta)^{\vee}}(Q')=\Psi_{\tau(\vartheta)}({Q'}^{\vee})$.
\end{proof}

\begin{remark}
In the setting of Theorems \ref{thm:main} and~\ref{thm:main-explicit}, by a natural parametri\-zation of $\Disc(G/H)$ by a certain semilattice in~$\aaa^*$, we may regard \eqref{eqn:psi-Map} as an algebra homomorphism from $\D_G(X)$ into the algebra of polynomials on~$\aaa_{\C}^*$.
See Lemma~\ref{lem:ex(*)-DGX-polyn} below for an example.
\end{remark}

The next proposition follows immediately from Proposition~\ref{prop:D-spec-Hom}.

\begin{proposition}\label{prop:PQR}
Suppose $X_{\C}=G_{\C}/H_{\C}$ is $G_{\C}$-spherical.
If $P'\in Z(\tilg_{\C})$, $Q'\in\nolinebreak Z(\kk_{\C})$, and $R\in Z(\g_{\C})$ satisfy
\begin{equation}\label{eqn:PsiPQR}
\Psi_{\pi(\vartheta)}(P') + \Psi_{\tau(\vartheta)^{\vee}}(Q') + \Psi_{\vartheta}(R) = 0
\end{equation}
for all $\vartheta\in\Disc(G/H)$, then $\dd\ell(P')+\dd r(Q')+\dd\ell(R)=0$ in $\D_G(X)$.
\end{proposition}

In most cases of Table~\ref{table1}, we will be in the following situation: $X_{\C}=G_{\C}/H_{\C}$ is $G_{\C}$-spherical and both $\tilG/\tilH$ and $K/H$ are symmetric spaces.
Then we can reformulate Proposition~\ref{prop:PQR} in terms of $\D_{\tilG}(X)$ and $\D_K(F)$ instead of $Z(\tilg_{\C})$ and $Z(\kk_{\C})$, as follows.
Let $\tila$ (\resp $\aaa_F$) be a Cartan subspace for the symmetric space $\tilG/\tilH$ (\resp $K/H$) (see Section~\ref{subsec:symmsp}).
By the Cartan--Helgason theorem (Fact~\ref{fact:CartanHelgason}), for any $\vartheta\in\Disc(G/H)$ there exist $\lambda(\vartheta)\in\tila_{\C}^{\ast}$ and $\mu(\vartheta)\in (\aaa_F^{\ast})_{\C}$ such that $\pi(\vartheta)=\Rep(\tilG,\lambda(\vartheta))$ and $\tau(\vartheta)=\Rep(K,\mu(\vartheta))$.
By Lemma~\ref{lem:chiZgDGX}, for any $P'\in Z(\tilg_{\C})$ and $Q'\in Z(\kk_{\C})$ we have
\begin{eqnarray}
\Psi_{\pi(\vartheta)}(P') & = & \chi_{\lambda(\vartheta)+\rho_{\tila}}^X \circ \dd\ell(P'),\label{eqn:HCX-Psi-chi}\\
\Psi_{\tau(\vartheta)^{\vee}}(Q') & = & \chi_{\mu(\vartheta)+\rho_{\aaa_F}}^F \circ \dd r(Q').\nonumber
\end{eqnarray}
In Section~\ref{sec:computations}, we extend the formula \eqref{eqn:HCX-Psi-chi} to the cases where $\tilG/\tilH$ is nonsymmetric, see \eqref{eqn:HCX-vii} and \eqref{eqn:rhoa-nonsymm}.
Thus Proposition~\ref{prop:PQR} yields the following.

\begin{proposition}\label{prop:chiPQR}
Suppose that $X_{\C}=G_{\C}/H_{\C}$ is $G_{\C}$-spherical and that $K/H$ is a symmetric space.
If $P\in\D_{\tilG}(X)$, $Q\in\D_K(F)$, and $R\in Z(\g_{\C})$ satisfy
$$\chi_{\lambda(\vartheta)+\rho_{\tila}}^X(P) + \chi_{\mu(\vartheta)+\rho_{\aaa_F}}^F(Q) + \Psi_{\vartheta}(R)=0$$
for all $\vartheta\in\Disc(G/H)$, then $P+\iota(Q)+\dd\ell(R)=0$ in $\D_G(X)$.
\end{proposition}

\subsection{The transfer map $\nnu(\cdot,\tau)$} \label{subsec:nu-tau}

Let $\tau\in\Disc(K/H)$.
Recall from Section~\ref{subsec:intro-transfer} that the transfer maps
$$\left\{ \begin{array}{l}
\nnu(\cdot,\tau) : \Hom_{\C\text{-}\mathrm{alg}}(\D_{\tilG}(X),\C) \longrightarrow \Hom_{\C\text{-}\mathrm{alg}}(Z(\g_{\C}),\C),\\
\llambda(\cdot,\tau) : \Hom_{\C\text{-}\mathrm{alg}}(Z(\g_{\C})/\Ker(\dd\ell^{\tau}),\C) \longrightarrow \Hom_{\C\text{-}\mathrm{alg}}(\D_{\tilG}(X),\C)
\end{array} \right.$$
of \eqref{eqn:nu-lambda-tau} are induced from a bijection
\begin{equation} \label{eqn:abstract-transfer}
\varphi_{\mathcal{I}_{\tau}}^* : \Hom_{\C\text{-}\mathrm{alg}}(\D_{\tilG}(X),\C) \overset{\sim}{\longrightarrow} \Hom_{\C\text{-}\mathrm{alg}}\big(Z(\g_{\C})/\Ker(q_{\mathcal{I}_{\tau}}\circ\dd\ell),\C\big),
\end{equation}
where $\mathcal{I}_{\tau}$ is the annihilator of $\tau^{\vee}$ in $Z(\kk_{\C})$; see the commutative diagram in Section~\ref{subsubsec:intro-I-tau}.
Such a bijection $\varphi_{\mathcal{I}_{\tau}}^*$ exists in the setting \ref{setting} when $\tilG$ is simple by Proposition~\ref{prop:qI}, and also in the case \eqref{eqn:nonsimple-case} where $\tilG$ is a direct product of simple Lie groups by Proposition~\ref{prop:qItau-*}.

Theorem~\ref{thm:main-explicit}.(1)--(2) for $\tilG$ simple and Proposition~\ref{prop:ex(*)} for $\tilG$ a product imply that the transfer maps $\nnu(\cdot,\tau)$ and $\llambda(\cdot,\tau)$ are inverse to each other, in the following sense.

\begin{proposition} \label{prop:phiItau-vanish}
In the setting \ref{setting}, suppose that $\tilG$ is simple or $(\tilG,\tilH,G)$ is the triple \eqref{eqn:nonsimple-case}.
Let $K$ be a maximal connected proper subgroup of~$G$ containing~$H$ if $\h$ is not a maximal proper subalgebra of~$\g$, and $K=H$ otherwise.
Let $\tau\in\Disc(K/H)$.
\begin{enumerate}
  \item If $\lambda\in\Spec(X)_{\tau}$ (see \eqref{eqn:Spec-tau-X}), then $\varphi_{\mathcal{I}_{\tau}}^*(\lambda)$ vanishes on $\Ker(\dd\ell^{\tau})$.
  \item We have $\nnu(\llambda(\nu,\tau)) = \nu$ for all $\nu\in\Hom_{\C\text{-}\mathrm{alg}}(Z(\g_{\C})/\Ker(\dd\ell^{\tau}),\C)$ and $\llambda(\nnu(\lambda,\tau)) = \lambda$ for all $\lambda\in\mathrm{Spec}(X)_{\tau}$.
\end{enumerate}
\end{proposition}

\begin{proof}
(1) Let $\lambda\in\Spec(X)_{\tau}$.
Consider a nonzero $F\in C^{\infty}(X;\M_{\lambda})_{\tau}$, and write $F=\ii_{\tau}(f)$ where $f\in C^{\infty}(Y,\W_{\tau})$.
By Theorem~\ref{thm:main}.(1) for $\tilG$ simple and Proposition~\ref{prop:ex(*)} for the triple \eqref{eqn:nonsimple-case}, for any $R\in Z(\g_{\C})$ there exist $P_j\in\D_{\tilG}(X)$ and $Q_j\in Z(\kk_{\C})$, $1\leq j\leq m$, such that
$$\dd\ell(R) = \sum_j \dd r(Q_j)\,P_j$$
in $\D_G(X)$.
By definition \eqref{eqn:phi-ideal-I} of~$\varphi_{\mathcal{I}_{\tau}}$, we have $\varphi_{\mathcal{I}_{\tau}}^*(\lambda)(R) \!=\! \sum_j \Psi_{\tau^{\vee}}^K(Q_j)\,\lambda(P_j)$, because
$$\ii_{\tau}\big(\dd\ell^{\tau}(R)f\big) = \dd\ell^{\tau}(R)F = \sum_j \Psi_{\tau^{\vee}}^K(Q_j)\,\lambda(P_j)\,F = \varphi_{\mathcal{I}_{\tau}}^*(\lambda)(R)\,F.$$
Therefore, if $\dd\ell^{\tau}(R)=0$, then $\varphi_{\mathcal{I}_{\tau}}^*(\lambda)(R)=0$ because $F$ is nonzero.

(2) This follows readily from (1) and from the definition of $\nnu(\cdot,\tau)$ and $\llambda(\cdot,\tau)$ in Section~\ref{subsubsec:intro-I-tau}.
\end{proof}

In the rest of this section, we give a description of the map \eqref{eqn:abstract-transfer} that relates joint eigenvalues for $\D_{\tilG}(X)$ and for $Z(\g_{\C})$, by introducing an affine map $S_{\tau} : \tila_{\C}^*\to\jj_{\C}^*$; in this way, we give a more precise version of Theorem~\ref{thm:transfer}.
This description is given via the Harish-Chandra isomorphism, which we recall now.

For a symmetric space $X=\tilG/\tilH$, the Harish-Chandra isomorphism $\Psi$ of \eqref{eqn:Psi} gives an identification
\begin{equation}\label{eqn:Psi-dual}
\Psi^* : \tila_{\C}^*/\widetilde{W} \overset{\sim}{\longrightarrow} \Hom_{\C\text{-}\mathrm{alg}}(\D_{\tilG}(X),\C),
\end{equation}
where $\widetilde{W}$ is the Weyl group of the restricted root system $\Sigma(\tilg_{\C},\tila_{\C})$.
There are a few cases where $X_{\C}=\tilG_{\C}/\tilH_{\C}$ is a nonsymmetric spherical homogeneous space such as $X_{\C}=\SO(7,\C)/G_2(\C)$, and in Section~\ref{sec:computations} we give an explicit normalization of the identification \eqref{eqn:Psi-dual} in each case of Table~\ref{table1}, see \eqref{eqn:HCX-vii} and \eqref{eqn:HCX-viii} below.

The Harish-Chandra isomorphism $\Phi$ of \eqref{eqn:HC-isom-group} for the group manifold~$G_{\C}$ gives an identification
$$\Phi^* : \jj_{\C}^*/W(\g_{\C}) \overset{\sim}{\longrightarrow} \Hom_{\C\text{-}\mathrm{alg}}(Z(\g_{\C}),\C),$$
where $\jj_{\C}$ is a Cartan subalgebra of~$\g_{\C}$ and $W(\g_{\C})$ the Weyl group of the root system $\Delta(\g_{\C},\jj_{\C})$.

Let $\dd\ell : Z(\g_{\C})\to\D_G(X)$ be the natural $\C$-algebra homomorphism (see \eqref{eqn:dl-dr}).
Recall from Section~\ref{subsec:intro-transfer} that $\dd\ell^{\tau} : Z(\g_{\C})\to\D_G(Y,\W_{\tau})$ is a $\C$-algebra homomorphism into the ring of matrix-valued $G$-invariant differential operators on $C^{\infty}(Y,\W_{\tau})$, for $\tau\in\Disc(K/H)$.

\begin{theorem} \label{thm:nu-tau}
In the setting \ref{setting}, suppose that either $\tilG$ is simple, or $\tilG=\!^{\backprime}G\times\!^{\backprime}G$ and $\tilH=H_1\times H_2$ and $G=\Diag(^{\backprime}G)=\{(g,g)\,:\,g\in\,^{\backprime}G\}$ for some simple Lie group $^{\backprime}G$ and some subgroups $H_1$ and~$H_2$.
Let $K$ be a maximal connected proper subgroup of $G$ containing~$H$ if $\h$ is not a maximal proper subalgebra of~$\g$, and $K=H$ otherwise.
We set $Y:=G/K$, and let $\tau\in\Disc(K/H)$.
Then
\begin{enumerate}
  \item the ring $\D_G(X)$ preserves the subspace $\ii_{\tau}(C^{\infty}(Y,\W_{\tau}))$ of $C^{\infty}(X)$;
  \item for $f\in C^{\infty}(Y,\W_{\tau})$, the function $\ii_{\tau}(f)\in C^{\infty}(X)$ is a joint eigenfunction for $\D_{\tilG}(X)$ if and only if $f$ is a joint eigenfunction for $Z(\g_{\C})$ via $\dd\ell^{\tau}$;
  \item the joint eigenvalues for $\D_{\tilG}(X)$ and $Z(\g_{\C})$ on $\ii_{\tau}(C^{\infty}(Y,\W_{\tau}))$ in~(2) are related via the transfer map
$$\nnu(\cdot,\tau) : \Hom_{\C\text{-}\mathrm{alg}}(\D_{\tilG}(X),\C) \longrightarrow \Hom_{\C\text{-}\mathrm{alg}}(Z(\g_{\C}),\C)$$
in the sense that for any $\lambda\in\Hom_{\C\text{-}\mathrm{alg}}(\D_{\tilG}(X),\C)$, the following two conditions on $f\in C^{\infty}(Y,\W_{\tau})$ are equivalent:
\begin{align*}
\dd\ell^{\tau}(R) f & = \nnu(\lambda,\tau)(R)\,f && \hspace{-1.8cm}\forall R\in Z(\g_{\C}),\\
D(\ii_{\tau}f) & = \lambda(D)\,\ii_{\tau}f && \hspace{-1.8cm}\forall D\in\D_{\tilG}(X);
\end{align*}
  \item there exists an affine map $S_{\tau} : \tila_{\C}^*\to\jj_{\C}^*$ such that the following diagram commutes.
\begin{equation} \label{eqn:comm-diag}
\xymatrix{\tila_{\C}^* \ar[d] \ar[r]^{S_{\tau}} & \jj_{\C}^* \ar[d]\\
\tila_{\C}^*/\widetilde{W} \ar[d]_{\Psi^*} & \jj_{\C}^*/W(\g_{\C}) \ar[d]^{\Phi^*}\\
\Hom_{\C\text{-}\mathrm{alg}}(\D_{\tilG}(X),\C) \ar[r]_{\nnu(\cdot,\tau)} & \Hom_{\C\text{-}\mathrm{alg}}(Z(\g_{\C}),\C)}
\end{equation}
\end{enumerate}
\end{theorem}

We give a proof of Theorem~\ref{thm:nu-tau}.(1)--(3) in Section~\ref{subsec:strategy-transfer}, postponing the proof of Proposition~\ref{prop:qI} and its counterpart for the product case $(\tilG,\tilH,G)=(^{\backprime}G\times\!^{\backprime}G,H_1\times H_2,\Diag(^{\backprime}G))$ (Proposition~\ref{prop:qItau-*}) until Sections \ref{sec:computations} and~\ref{sec:ex(*)}.
We note that by the classification of Proposition~\ref{prop:class-triple-prod} below, the product case essentially reduces to the triple \eqref{eqn:nonsimple-case}.

An explicit formula for the affine map~$S_{\tau}$ is given in Section~\ref{sec:computations} for simple~$\tilG$ in each case, and in Section~\ref{sec:ex(*)} for the case \eqref{eqn:nonsimple-case}.

Statement~(3) provides useful information on possible $Z(\g_{\C})$-infinitesimal characters for irreducible $G$-modules in $C^{\infty}(Y,\W_{\tau})$, by means of the affine map~$S_{\tau}$.

\begin{remark} \label{rem:extend-nu-tau}
Theorem~\ref{thm:nu-tau}.(1) is not true if we do not assume $X_{\C}=G_{\C}/H_{\C}$ to be $G_{\C}$-spherical.
For instance, it is not true for
$$X = G/H = (\SO(2n-\nolinebreak 1)\times\nolinebreak\U(n))/\Diag(\U(n-1))$$
and $\tilG = \SO(2n)\times\SO(2n)$, where $X_{\C}$ is $\tilG_{\C}$-spherical but not $G_{\C}$-spherical: see \cite[Ex.\,8.8]{kkII}.
\end{remark}

\begin{remark}
The standard homomorphism $T : \aaa_{\C}^*/W\to\jj_{\C}^*/W(\g_{\C})$ of \eqref{eqn:aj} is induced by the inclusion $\aaa_{\C}\subset\jj_{\C}$ and the ``$\rho$-shift''.
In contrast, the map $S_{\tau} : \tila_{\C}^*\to\jj_{\C}^*$ of Theorem~\ref{thm:nu-tau}.(4) is defined even though there is a priori no inclusion relation between $\tila_{\C}$ (which is contained in $\tilg_{\C}$) and $\jj_{\C}$ (which is contained in $\g_{\C}$).
\end{remark}

\subsection{Graded algebras $\operatorname{gr}(\D_G(X))$} \label{subsec:graded-alg}

In order to prove that two of the three algebras $\dd\ell(Z(\tilg_{\C}))$, $\dd r(Z(\kk_{\C}))$, $\dd\ell(Z(\g_{\C}))$ above generate the $\C$-algebra\linebreak $\D_G(X)$ as in Theorems \ref{thm:main} and~\ref{thm:main-explicit}, we use the filtered algebra structure of $\D_G(X)$.
In this section, we give preliminary results on the graded algebra $\operatorname{gr}(\D_G(X))$ which will be used in Sections \ref{sec:computations} and~\ref{sec:ex(*)}.

The $\C$-algebra $\D_G(X)$ has a natural filtration $\{ \D_G(X)_N\}_{N\in\N}$ by the order of differential operators, with $\D_G(X)_M\D_G(X)_N\subset\D_G(X)_{M+N}$ for all $M,N\in\N$.
Therefore, the graded module
$$\operatorname{gr}(\D_G(X)) := \bigoplus_{N\in\N} \operatorname{gr}_N(\D_G(X)),$$
where $\operatorname{gr}_N(\D_G(X)) := \D_G(X)_N/\D_G(X)_{N+1}$, becomes a $\C$-algebra, which is isomorphic, as graded $\C$-algebras, to the subalgebra $S(\g_{\C}/\h_{\C})^H=\linebreak\bigoplus_{N\in\N} S^N(\g_{\C}/\h_{\C})^H$ of the symmetric algebra $S(\g_{\C}/\h_{\C})$.
We relate the two algebras $\D_G(X)$ and $S(\g_{\C}/\h_{\C})^H$ using the following lemma.

\begin{lemma}\label{lem:struct-DGX}
For any $\underline{m}=(m_1,\dots,m_k)\in\N^k$ and $N\in\N$, let
$$v_{\underline{m}}(N) := \# \bigg\{ (a_1,\dots,a_k)\in\N^k \ :\ \sum_{i=1}^k a_i m_i = N\bigg\} .$$
\begin{enumerate}
  \item The sequence $(v_{\underline{m}}(N))_{N\in\N}$ determines $k$ and~$\underline{m}$ up to permutation.
  \item Suppose $S(\g_{\C}/\h_{\C})^H$ is a polynomial ring generated by algebraically independent homogeneous elements $P_1,\dots,P_k$ of respective degrees $m_1,\dots,m_k$.
  Then
  $v_{\underline{m}}(N)=\dim S^N(\g_{\C}/\h_{\C})^H$ for all $N\in\N$.
  \item Suppose $X_{\C}=G_{\C}/H_{\C}$ is $G_{\C}$-spherical, and let $P_1,\dots,P_k$ be as in~(2).
  For $1\leq j\leq k$, let $D_j\in\D_G(X)_{m_j}$ be the preimage of $P_j\in S^{m_j}(\g_{\C}/\h_{\C})^H$.
  Then $D_1,\dots,D_k$ are algebraically independent, and $\D_G(X)$ is the polynomial ring generated by them.
\end{enumerate}
\end{lemma}

\begin{proof}
Statements (1) and~(2) are obvious.
For~(3), let $R$ be the $\C$-subalgebra of $\D_G(X)$ generated by $D_1,\dots,D_k$.
Since $P_1,\dots,P_k$ are algebraically independent in $\operatorname{gr}(\D_G(X))\simeq S(\g_{\C}/\h_{\C})^H$, so are $D_1,\dots,D_k$ in $\D_G(X)$.
Furthermore,
\begin{eqnarray*}
\dim\big(\D_G(X)_N\cap R\big) & = & \sum_{j=0}^N \dim S^j(\g_{\C}/\h_{\C})^H\\
& = & \sum_{j=0}^N \dim \operatorname{gr}_j(\D_G(X)) \;=\; \dim \D_G(X)_N
\end{eqnarray*}
for any~$N$, hence $R=\D_G(X)$.
\end{proof}

\subsection{Strategy for the proof of Theorems \ref{thm:main} and~\ref{thm:main-explicit}} \label{subsec:descript-comput}

We now explain how this machinery is used to find generators and relations for $\D_G(X)$ in Section~\ref{sec:computations}.
There are four steps.

The first step is to describe the map
\begin{eqnarray*}
\Disc(G/H) & \longrightarrow & \Disc(\tilG/\tilH) \times \Disc(K/H)\\
\vartheta & \longmapsto & (\pi(\vartheta),\tau(\vartheta))
\end{eqnarray*}
of Proposition~\ref{prop:pi-tau-theta}, which exists by $G_{\C}$-sphericity of~$X_{\C}$.
We note that an explicit description of the sets $\Disc(G/H)$, $\Disc(\tilG/\tilH)$, and $\Disc(K/H)$ was previously known in most cases, and is easily obtained in the remaining cases.
In fact, both $\tilG/\tilH$ and $K/H$ are symmetric spaces in most cases, hence $\Disc(\tilG/\tilH)$ and $\Disc(K/H)$ are described by the Cartan--Helgason theorem (Fact~\ref{fact:CartanHelgason}).
On the other hand, $G/H$ is never symmetric, but $\Disc(G/H)$ for $G_{\C}$-spherical $G_{\C}/H_{\C}$ was classified in \cite{gg13,kra79} under the assumption that $G$ is simple.
There are a few remaining cases where $G/H$ or $K/H$ is nonsymmetric and $G$ or~$K$ is not simple.
All of them are homogeneous spaces of classical groups of low dimension, and the classification of $\Disc(G/H)$ or $\Disc(K/H)$ can then be carried out easily.
To find an explicit formula for the map $\vartheta\mapsto\pi(\vartheta)$ or $\tau(\vartheta)$, we use the branching laws for the restriction $\tilG\downarrow G$ or $G\downarrow K$, respectively.
Some of them are obtained as special cases of the classical branching laws, whereas Proposition~\ref{prop:branchnormal} and an a priori knowledge of $\Disc(\tilG/\tilH)$ or $\Disc(K/H)$ help us find the branching laws when the subgroups are embedded in a nontrivial way (\eg for $\SO(16)\downarrow\Spin(7)$).

The second step consists in taking generators $P_k,Q_k,R_k$ for the three algebras $\D_{\tilG}(X)$, $\D_K(F)$, and $Z(\g_{\C})$, respectively.
In most cases, $\tilG/\tilH$ and $K/H$ are symmetric spaces, hence we can use the Harish-Chandra isomorphism (see \eqref{eqn:HC-isom} and \eqref{eqn:HC-isom-group}).
The choices of $P_k,Q_k,R_k$ are not unique; we make them carefully so that $P_k,Q_k,R_k$ have linear relations in the next step.

The third step consists in finding explicit linear relations among the differential operators $P_k,\iota(Q_k),\dd\ell(R_k)\in\D_G(X)$.
For this we use the map $\vartheta\mapsto (\pi(\vartheta),\tau(\vartheta))$ of Proposition~\ref{prop:pi-tau-theta} and compute the scalars by which these operators act on $\pi(\vartheta)$, $\tau(\vartheta)$, and~$\vartheta$, respectively.
For appropriate choices of $P_k,Q_k,R_k$, we find linear relations among these scalars which hold for all~$\vartheta$.
We then conclude using Proposition~\ref{prop:chiPQR}.

The last step is to prove that any two of the three subalgebras $\D_{\tilG}(X)$, $\iota(\D_K(F))$, and $\dd\ell(Z(\g_{\C}))$ generate $\D_G(X)$ (with one exception in case (ix) of Table~\ref{table1}).
For this, we exhibit algebraically independent subsets of the $P_k$ and $\iota(Q_k)$, of the $\iota(Q_k)$ and $\dd\ell(R_k)$, and of the $P_k$ and $\dd\ell(R_k)$, that generate $\D_G(X)$.
The proof is reduced to some estimates in the graded algebra $S(\g_{\C}/\h_{\C})^H$ by Lemma~\ref{lem:struct-DGX}.

These four steps complete the proof of Theorems \ref{thm:main} and~\ref{thm:main-explicit}, with explicit linear relations \eqref{eqn:rel-diff-op}.

\subsection{Strategy for the proof of Theorem~\ref{thm:nu-tau} (hence of Theorem~\ref{thm:transfer})} \label{subsec:strategy-transfer}

Postponing the proof of Proposition~\ref{prop:qI} (consequence of Theorem~\ref{thm:main}) until Section~\ref{sec:computations}, and the proof of its counterpart for the product case (Proposition~\ref{prop:qItau-*}) until Section~\ref{sec:ex(*)}, we now give a proof of Theorem~\ref{thm:nu-tau}.(1)--(3).

\begin{proof}[Proof of Theorem~\ref{thm:nu-tau}.(1)--(3)]
For $\tau\in\Disc(K/H)$, let $\mathcal{I}_{\tau}$ be the annihilator of the irreducible contragredient representation $\tau^{\vee}$ in $Z(\kk_{\C})$, and\linebreak $q_{\mathcal{I}_{\tau}} : \D_G(X) \to \D_G(X)_{\mathcal{I}_{\tau}} := \D_G(X)/\langle\mathcal{I}_{\tau}\rangle$ the quotient map \eqref{eqn:qI} as in Sections \ref{subsec:intro-transfer} and~\ref{subsec:nu-tau}.
 By Proposition~\ref{prop:qI} for $\tilG$ simple and Proposition~\ref{prop:qItau-*} for the product case (see Proposition~\ref{prop:class-triple-prod}), the map $q_{\mathcal{I}_{\tau}}$ induces an algebra isomorphism
$$q_{\mathcal{I}_{\tau}}\circ\dd\ell : Z(\g_{\C}) \longrightarrow \D_G(X)_{\mathcal{I}_{\tau}} = \D_G(X)/\mathcal{I}_{\tau},$$
which itself induces a bijection
$$\varphi_{\mathcal{I}_{\tau}}^* : \Hom_{\C\text{-}\mathrm{alg}}(\D_{\tilG}(X),\C) \overset{\sim}{\longrightarrow} \Hom_{\C\text{-}\mathrm{alg}}\big(Z(\g_{\C})/\Ker(q_{\mathcal{I}_{\tau}}\circ\dd\ell),\C\big).$$

(1) By Schur's lemma, the algebra $Z(\kk_{\C})$ acts on $\ii_{\tau}(C^{\infty}(Y,\W_{\tau}))$ via $\dd r$ as scalars, given by the algebra homomorphism $Z(\kk_{\C})\to Z(\kk_{\C})/\mathcal{I}_{\tau}\simeq\C$.
On the other hand, $\dd\ell(Z(\g_{\C}))$ preserves the subspace $\ii_{\tau}(C^{\infty}(Y,\W_{\tau}))$ of $C^{\infty}(X)$ because $\ii_{\tau}\circ\dd\ell^{\tau}(R)=\dd\ell(R)\circ\ii_{\tau}$ for all $R\in Z(\g_{\C})$.
Since $q_{\mathcal{I}_{\tau}}\circ\dd\ell : Z(\g_{\C})\to\D_G(X)_{\mathcal{I}_{\tau}}$ is surjective, any element of $\D_G(X)$ preserves $\ii_{\tau}(C^{\infty}(Y,\W_{\tau}))$.

(2) We again use the fact that the algebra $\dd r(Z(\kk_{\C}))$ acts on $C^{\infty}(Y,\W_{\tau})$ via $\dd r$ as scalars, given by the algebra homomorphism $Z(\kk_{\C})\to Z(\kk_{\C})/\mathcal{I}_{\tau}\simeq\C$.
Since the map $\varphi_{\mathcal{I}_{\tau}}^*$ above is surjective, $f$ is a joint eigenfunction for $Z(\g_{\C})$ via $\dd\ell^{\tau}$ if and only if $\ii_{\tau}(f)$ is a joint eigenfunction for $\D_G(X)$, if and only if $\ii_{\tau}(f)$ is a joint eigenfunction for $\D_{\tilG}(X)$.

(3) This follows from the definition of the transfer map $\nnu(\cdot,\tau)$ in Section~\ref{subsubsec:intro-I-tau}.
\end{proof}

The following proposition reduces the proof of Theorem~\ref{thm:nu-tau}.(4) to the question of finding an explicit formula for the map $\vartheta\mapsto (\pi(\vartheta),\tau(\vartheta))$ of Proposition~\ref{prop:pi-tau-theta}.

\begin{proposition} \label{prop:Stau-transfer}
In the setting of Proposition~\ref{prop:pi-tau-theta}, write
\begin{align*}
\pi(\vartheta) & = \Rep(\tilG,\lambda(\vartheta)) && \hspace{-1.5cm}\text{for }\lambda(\vartheta)\in\tila_{\C}^*,\\
\tau(\vartheta) & = \Rep(K,\nu(\vartheta)) && \hspace{-1.5cm}\text{for }\nu(\vartheta)\in\jj_{\C}^*.
\end{align*}
Let $\tau\in\Disc(K/H)$.
Suppose there is an affine map $S_{\tau} : \tila_{\C}^*\to\jj_{\C}^*$ such that
$$S_{\tau}\big(\lambda(\vartheta) + \rho_{\tila}\big) = \nu(\vartheta) + \rho \quad\quad \mod W(\g_{\C})$$
for all $\vartheta\in\Disc(G/H)$ with $\tau(\vartheta)=\tau$.
Then the transfer map $\nnu(\cdot,\tau)$ is given by the commutative diagram \eqref{eqn:comm-diag} for this~$S_{\tau}$.
\end{proposition}

\begin{proof}
For every $\vartheta\in\Disc(G/H)$, let $U_{\vartheta}$ be the $\vartheta$-isotypic component of the regular representation of~$G$ on $C^{\infty}(Y,\W_{\tau(\vartheta)})$, and for the irreducible representation $\pi(\vartheta)$ of~$\tilG$, let $\widetilde{U}_{\pi(\vartheta)}$ be the $\pi$-isotypic component of the regular representation of~$\tilG$ on $C^{\infty}(X)$.
Then $\ii_{\tau(\vartheta)}(U_{\vartheta}) \subset \widetilde{U}_{\pi(\vartheta)}$, and the algebras $Z(\g_{\C})$ and $\D_{\tilG}(X)$ act on $\ii_{\tau(\vartheta)}(U_{\vartheta})$ and $\widetilde{U}_{\pi(\vartheta)}$ as scalars, given by $\chi_{\lambda(\vartheta)+\rho_{\tila}}^X$ and $\chi_{\nu(\vartheta)+\rho}^G$ via the Harish-Chandra homomorphisms (see \eqref{eqn:HCX-Psi-chi} and \eqref{eqn:HCchi}), respectively.
Since the algebraic direct sum
$$\bigoplus_{\substack{\vartheta\in\Disc(G/H)\\ \tau(\vartheta)=\tau}} \ii_{\tau}(U_{\vartheta}) \quad\quad (\subset C^{\infty}(X))$$
of the eigenspaces of the algebras $Z(\g_{\C})$ and $\D_{\tilG}(X)$ is dense in $\ii_{\tau}(C^{\infty}(Y,\W_{\tau}))$, the transfer map $\nnu(\cdot,\tau)$ is given by the commutative diagram \eqref{eqn:comm-diag} for this~$S_{\tau}$.
\end{proof}

We prove that we can define an affine map $S_{\tau} : \tila_{\C}^*\to\jj_{\C}$ as in Proposition~\ref{prop:Stau-transfer} by determining, in each case in Sections \ref{sec:computations} and~\ref{sec:ex(*)}, an explicit description of the map $\vartheta\mapsto (\pi(\vartheta),\tau(\vartheta))$.

\section{Disconnected isotropy subgroups~$H$} \label{sec:disconnected-H}

In this section we prove that, in the setting of Theorem~\ref{thm:main-explicit}, the algebra $\D_G(X)$ and its subalgebras $\D_{\tilG}(X)$, $\dd r(Z(\kk_{\C}))$, and $\dd\ell(Z(\g_{\C}))$ are completely determined by the triple of Lie algebras $(\tilg_{\C},\h_{\C},\g_{\C})$.

For reductive symmetric spaces $G/H$, it is easy to check that the ring $\D_G(G/H)$ is isomorphic to $\D_G(G/H_0)$ where $H_0$ is the identity component of~$H$ (see \cite[Rem.\,3.1]{kk16} for instance).
However, the homogeneous spaces $X=G/H$ in Table~\ref{table1} or their coverings are never symmetric spaces, and in general, when a subgroup $H$ is disconnected, it may happen that $\D_G(G/H)$ is a proper subalgebra of $\D_G(G/H_0)$.
In the setting of Theorem~\ref{thm:main-explicit}, the group $H:=\tilH\cap G$ is not always connected: its number of connected components may vary under taking a covering of~$\tilG$.
However we prove the following.

\begin{theorem} \label{thm:InvDiff-covering}
Let $\tilG$ be a connected compact simple Lie group, and $\tilH$ and~$G$ two connected subgroups of~$\tilG$ such that $\tilG_{\C}/\tilH_{\C}$ is $G_{\C}$-spherical.
Let $H:=\tilH\cap G$.
\begin{enumerate}
  \item The algebra $\D_G(G/H)$ is completely determined by the pair of Lie algebras $(\g_{\C},\h_{\C})$, and does not vary under coverings of~$\tilG$.
  \item Let $\kk$ be a maximal proper subalgebra of~$\g$ containing $\tilh\cap\g$.
  Then the adjoint action of $H$ on $Z(\kk_{\C})$ is trivial, and so the homomorphism $\dd r : Z(\kk_{\C})\to\D_G(G/H)$ of \eqref{eqn:diagram} is well defined.
  \item The subalgebras $\D_{\tilG}(\tilG/\tilH)$, $\dd r(Z(\kk_{\C}))$, and $\dd\ell(Z(\g_{\C}))$ are completely determined by the triple of Lie algebras $(\tilg_{\C},\h_{\C},\g_{\C})$.
\end{enumerate}
\end{theorem}

We may reformulate Theorem~\ref{thm:InvDiff-covering} in terms of the ring of invariant \emph{holomorphic} differential operators (Section~\ref{subsec:intro-applic}) on the complex manifold $X_{\C}=\tilG_{\C}/H_{\C}$, as follows.

\begin{theorem} \label{thm:HoloInvDiff}
Let $\tilG_{\C}$ be a connected complex simple Lie group, and $\tilH_{\C}$ and~$G_{\C}$ two connected complex reductive subgroups of~$\tilG_{\C}$ such that $\tilG_{\C}/\tilH_{\C}$ is $G_{\C}$-spherical.
Let $H_{\C}:=\tilH_{\C}\cap G_{\C}$.
\begin{enumerate}
  \item The algebra $\D_{G_{\C}}(G_{\C}/H_{\C})$ is completely determined by the pair of Lie algebras $(\g_{\C},\h_{\C})$, and does not vary under coverings of~$\tilG_{\C}$.
  \item Let $\kk_{\C}$ be a maximal proper complex reductive subalgebra of~$\g_{\C}$ containing $\tilh_{\C}\cap\g_{\C}$.
  Then the homomorphism $\dd r : Z(\kk_{\C})\to\D_{G_{\C}}(G_{\C}/H_{\C})$ is well defined.
  \item The subalgebras $\D_{\tilG_{\C}}(\tilG_{\C}/\tilH_{\C})$, $\dd r(Z(\kk_{\C}))$, and $\dd\ell(Z(\g_{\C}))$ are completely determined by the triple of Lie algebras $(\tilg_{\C},\h_{\C},\g_{\C})$.
\end{enumerate}
\end{theorem}

Theorem~\ref{thm:HoloInvDiff} is derived from Theorem~\ref{thm:InvDiff-covering} in Section~\ref{subsec:inv-diff-op-real-form}, by using the natural isomorphism (Lemma~\ref{lem:bij-iota*})
$$\D_{G_{\C}}(G_{\C}/H_{\C}) \overset{\scriptscriptstyle\sim}{\longrightarrow} \D_G(G/H).$$
In Section~\ref{subsec:disconnectedH} we reduce the proof of Theorem~\ref{thm:InvDiff-covering} to two inclusions of Lie groups described in Proposition~\ref{prop:H-subset-K}.
These inclusions are established in Sections \ref{subsec:proof-H-connected} and~\ref{subsec:proof-HG-ZG} for most cases, with a separate treatment for coverings of cases (v), (vi), (vii) of Table~\ref{table1} in Section~\ref{subsec:proof-H-subset-K-rest}.

By Theorem~\ref{thm:InvDiff-covering}, it is sufficient to prove Theorems \ref{thm:main}, \ref{thm:main-explicit}, \ref{thm:transfer}, and~\ref{thm:nu-tau} for the triples $(\tilG,\tilH,G)$ of Table~\ref{table1}, and they are then automatically true for all other triples obtained by a covering of~$\tilG$.

\subsection{Invariant differential operators and real forms} \label{subsec:inv-diff-op-real-form}

We begin with some basic observations on invariant differential operators in the setting where the groups $G$ and~$H$ are not necessarily compact.
A holomorphic continuation argument will be used to apply our main results on compact groups to the analysis of locally homogeneous spaces of other real forms, see Section~\ref{subsec:intro-applic} and \cite{kkII}.
Recall that a subgroup $G_2$ of a complex Lie group~$G_1$ is said to be a \emph{real form} of~$G_1$ if the Lie algebra $\g_2$ of~$G_2$ is a real form of the complex Lie algebra $\g_1$ of~$G_1$, namely $\g_1=\g_2+\sqrt{-1}\,\g_2$ (direct sum).

\begin{lemma} \label{lem:inj-iota*}
Let $G_{\C}\supset H_{\C}$ be a pair of complex Lie groups, and $G\supset H$ respective real forms.
Suppose $H\subset G\cap H_{\C}$.
Then the natural $G$-equivariant smooth map
$$\iota : X = G/H \longrightarrow X_{\C} = G_{\C}/H_{\C}$$
induces an injective $\C$-algebra homomorphism
$$\iota^* : \D_{G_{\C}}(X_{\C}) \longhookrightarrow \D_G(X).$$
\end{lemma}

\begin{proof}
The map $\iota$ is not necessarily injective, but it factors as follows:
$$X\ \underset{\text{covering}}{\ensuremath{\relbar\joinrel\relbar\joinrel\relbar\joinrel\relbar\joinrel\twoheadrightarrow}}\ G/G\cap H_{\C}\ \underset{\text{real form}}{\ensuremath{\lhook\joinrel\relbar\joinrel\relbar\joinrel\relbar\joinrel\relbar\joinrel\rightarrow}}\ X_{\C}.$$
This induces two homomorphisms whose composition is the desired map~$\iota^*$:
$$\D_{G_{\C}}(X_{\C})\ \underset{\text{restriction}}{\ensuremath{\relbar\joinrel\relbar\joinrel\relbar\joinrel\relbar\joinrel\rightarrow}}\ \D_G(G/G\cap H_{\C})\ \ensuremath{\lhook\joinrel\relbar\joinrel\relbar\joinrel\relbar\joinrel\relbar\joinrel\rightarrow}\ \D_G(X).$$
The second homomorphism is injective because $X\to G/(G\cap H_{\C})$ is a covering.
The first homomorphism $\D_{G_{\C}}(X_{\C})\to\D_G(G/(G\cap H_{\C}))$ is also injective because, locally, we can find coordinates $(z_1,\dots,z_n)$ on~$X_{\C}$, with $z_j=x_j+\sqrt{-1}\,y_j$, such that the totally real submanifold $G/(G\cap H_{\C})$ is given by $y_1=\dots=y_n=0$ and any differential operator $P\in\D_{G_{\C}}(X_{\C})$ is represented as $P = \sum_{\alpha} c_{\alpha}(z)\,\frac{\partial^{|\alpha|}}{\partial z^{\alpha}}$ with holomorphic coefficients $c_{\alpha}(z)$, and therefore the restriction map $P\mapsto \sum_{\alpha} c_{\alpha}(x)\,\frac{\partial^{|\alpha|}}{\partial x^{\alpha}}$ is injective.
\end{proof}

Lemma~\ref{lem:inj-iota*} implies the following.

\begin{lemma} \label{lem:bij-iota*}
In the setting of Lemma~\ref{lem:inj-iota*}, if $H$ meets every connected component of~$H_{\C}$, then $\iota^*$ is a ring isomorphism
$$\D_{G_{\C}}(G_{\C}/H_{\C}) \overset{\scriptscriptstyle\sim}{\longrightarrow} \D_G(G/H).$$
In particular, if $H_{\C}$ is connected, then the ring $\D_G(G/H)$ is completely determined by the pair of complex Lie algebras $(\g_{\C},\h_{\C})$, and does not depend on the real form $H$ of~$H_{\C}$.
\end{lemma}

\begin{proof}
To see that the injective algebra homomorphism $\iota^*$ from Lemma~\ref{lem:inj-iota*} is surjective, it suffices to show that the induced map on the graded modules
$$\operatorname{gr}(\iota^*) : S(\g_{\C}/\h_{\C})^{H_{\C}} \longrightarrow S(\g_{\C}/\h_{\C})^H$$
is surjective, see Section~\ref{subsec:graded-alg}.
If $v\in S(\g_{\C}/\h_{\C})$ is $H$-invariant, then $v$ is invariant under the infinitesimal action of the Lie algebra~$\h$, hence of its complexification~$\h_{\C}$.
If $H$ meets every connected component of~$H_{\C}$, then any $v\in S(\g_{\C}/\h_{\C})$ is $H_{\C}$-invariant, hence $\operatorname{gr}(\iota^*)$ is surjective.
\end{proof}

\begin{proof}[Proof of Theorem~\ref{thm:HoloInvDiff} assuming Theorem~\ref{thm:InvDiff-covering}]
Up to replacing $\tilH_{\C}$ by some conjugate, we may and do assume that there is a Cartan involution $\theta$ of~$\tilG_{\C}$ which leaves both $\tilH_{\C}$ and~$G_{\C}$ invariant.
Since $G_{\C}$ acts transitively on $\tilG_{\C}/\tilH_{\C}$, the intersection $H_{\C}=\tilH_{\C}\cap G_{\C}$ is conjugate to the original one if we take a conjugation of~$G_{\C}$.
Then the subgroups $\tilG$, $\tilH$, $G$, and~$H$ of fixed points by~$\theta$ in $\tilG_{\C}$, $\tilH_{\C}$, $G_{\C}$, and~$H_{\C}$, respectively, are maximal compact subgroups of these complex groups.
In particular, $H$ meets every connected component of~$H_{\C}$.
Now Lemma~\ref{lem:bij-iota*} implies that Theorem~\ref{thm:HoloInvDiff} follows from Theorem~\ref{thm:InvDiff-covering}.
\end{proof}

\subsection{Proof of Theorem~\ref{thm:InvDiff-covering}} \label{subsec:disconnectedH}

Theorem~\ref{thm:InvDiff-covering} reduces to the following.

\begin{proposition} \label{prop:H-subset-K}
Let $\tilG$ be a connected compact simple Lie group, and $\tilH$ and~$G$ two connected closed subgroups of~$\tilG$ such that $\tilG_{\C}/\tilH_{\C}$ is $G_{\C}$-spherical.
Let $H:=\tilH\cap G$.
\begin{enumerate}
  \item We have
  \begin{equation} \label{eqn:HG-ZG}
  H  \subset H_0 \, Z(\tilG),
  \end{equation}
  where $H_0$ is the identity component of~$H$ and $Z(\tilG)$ the center of~$\tilG$.
  \item Suppose $\tilh\cap\g$ is not a maximal proper subalgebra of~$\g$.
Let $\kk$ be a maximal proper subalgebra of~$\g$ containing $\tilh\cap\g$, and $K$ the corresponding analytic subgroup of~$G$.
Then
\begin{equation} \label{eqn:H-subset-K}
(H =)\,\tilH\cap G\,\subset K.
\end{equation}
\end{enumerate}
\end{proposition}

Proposition~\ref{prop:H-subset-K}.(2) implies that $F:=K/H$ and the algebra homomorphism $\dd r : Z(\kk_{\C})\to\D(F)$ in Section~\ref{subsec:intro-three-subalg} are well defined, and that $\dd r(Z(\kk_{\C}))$ is contained in the subalgebra $\D_K(F)$ of $K$-invariant differential operators on~$F$.

\begin{proof}[Proof of Theorem~\ref{thm:InvDiff-covering}]
(1) The injective algebra homomorphism $\D_G(G/H)\to\D_G(G/H_0)$ induces an injective homomorphism of graded algebras
$$S(\g_{\C}/\h_{\C})^H \longrightarrow S(\g_{\C}/\h_{\C})^{H_0},$$
which is surjective by \eqref{eqn:HG-ZG}.
Thus the homomorphism $\D_G(G/H)\!\to\!\D_G(G/H_0)$ is surjective by Lemma~\ref{lem:struct-DGX}, hence is an isomorphism.

(2) The fact that the adjoint action of $H$ on $Z(\kk_{\C})$ is trivial is clear from \eqref{eqn:HG-ZG}.
In particular, $Z(\kk_{\C})\subset U(\g_{\C})^H$, hence $\dd r : Z(\kk_{\C})\to\D_G(G/H)$ is well defined by the restriction of the $\C$-algebra homomorphism \eqref{eqn:dr} to $Z(\kk_{\C})$.

(3) By~(1), the algebra $\D_G(G/H)$ is completely determined by the triple of Lie algebras $(\tilg_{\C},\h_{\C},\g_{\C})$, hence so are the subalgebras $\dd r(Z(\kk_{\C}))$ and $\dd\ell(Z(\g_{\C}))$.
Since $\tilG$ is connected, $\D_{\tilG}(\tilG/\tilH)$ is the subalgebra of $\D_G(G/H)$ consisting of $\tilg$-invariant elements, and so it is also completely determined by the triple of Lie algebras $(\tilg_{\C},\h_{\C},\g_{\C})$.
\end{proof}

\subsection{Proof of Proposition~\ref{prop:H-subset-K}.(2) in most cases of Table~\ref{table1}} \label{subsec:proof-H-connected}

Here are two basic tools.

\begin{lemma} \label{lem:simply-connected-implies-H-subset-K}
In the setting of Proposition~\ref{prop:H-subset-K}, if $\tilG/\tilH$ or $\tilG/G$ is simply connected, then $\tilH\cap G$ is connected.
Moreover, $\tilH_1\cap G_1$ is connected for any triple $(\tilG_1,\tilH_1,G_1)$ of connected Lie groups such that $\tilG_1$ is connected and a covering of~$\tilG$ and $\tilH_1$ and~$G_1$ are analytic subgroups of~$\tilG_1$ with respective Lie algebras $\tilh$ and~$\g$.
\end{lemma}

\begin{proof}
Let $H:=\tilH\cap G$.
For $(L,L')=(\tilH,G)$ or $(G,\tilH)$, we have an exact sequence of homotopy groups
$$\pi_1(\tilG/L) \longrightarrow \pi_0(H) \longrightarrow \pi_0(L')$$
for the fibration $H\to L'\to L'/H\simeq\tilG/L$.
Thus if $\tilG/L$ is simply connected and $L'$ connected, then $H$ is connected.
Since the assumption is not changed under taking a covering of~$\tilG$, the last statement also holds.
\end{proof}

\begin{lemma} \label{lem:move-up-down-H-subset-K}
In the setting of Proposition~\ref{prop:H-subset-K}, let $Z$ be a central subgroup of~$\tilG$.
\begin{enumerate}
  \item If $Z\subset G$ or $Z\subset\tilH$, then \eqref{eqn:H-subset-K} for~$\tilG$ implies \eqref{eqn:H-subset-K} for $\tilG/Z$ (\ie $\varpi(\tilH)\cap\varpi(G)\subset\varpi(K)$ where $\varpi : \tilG\to\tilG/Z$ is the quotient map).
  \item If $Z\subset K$, then \eqref{eqn:H-subset-K} for $\tilG/Z$ implies \eqref{eqn:H-subset-K} for~$\tilG$.
\end{enumerate}
\end{lemma}

\begin{proof}
(1) If $Z\subset G$ or $Z\subset\tilH$, then $\tilH Z\cap GZ = (\tilH\cap G)Z$, and so $\varpi(\tilH) \cap \varpi(G) = \varpi(\tilH\cap G)$.
In particular, $\varpi(\tilH)\cap\varpi(G)\subset\varpi(K)$ as soon as $\tilH\cap G\subset K$.

(2) If $Z\subset K$, then $\varpi^{-1}(\varpi(K))=K$.
In particular,
$$\tilH\cap G \subset \varpi^{-1}\big(\varpi(\tilH)\cap\varpi(G)\big) \subset \varpi^{-1}(\varpi(K)) = K$$
as soon as $\varpi(\tilH)\cap\varpi(G)\subset\varpi(K)$.
\end{proof}

\begin{proposition} \label{prop:H-subset-K-for-i-x}
If $(\tilG,\tilH,G)$ is any triple of connected groups locally isomorphic to the triples in case (i), (ii), (iii), (iv), (viii), (ix), (x), (xiii), or (xiv) of Table~\ref{table1}, then $\tilH\cap G$ is connected.
\end{proposition}

\begin{proof}
Let $(\tilG,\tilH,G)$ be any triple of connected groups locally isomorphic to the triples in case (ii), (iii), or (xiv) of Table~\ref{table1}.
Then $\tilH$ is the centralizer of a toral subgroup of~$\tilG$, and so $\tilG/\tilH$ is a (generalized) flag manifold, hence simply connected.
We conclude using Lemma~\ref{lem:simply-connected-implies-H-subset-K}.

Similarly, let $(\tilG,\tilH,G)$ be any triple of connected groups locally isomorphic to the triples in case (i), (iv), (viii), or (xiii) of Table~\ref{table1}.
Then $G$ is the centralizer of a toral subgroup of~$\tilG$, and so $\tilG/G$ is a (generalized) flag manifold, hence simply connected.
We conclude using Lemma~\ref{lem:simply-connected-implies-H-subset-K}.

For cases (ix) and~(x) of Table~\ref{table1}, we consider the triple $(\tilG,\tilH,G)$ given in the table.
For this triple the group $\tilG=\SO(7)$ is adjoint, hence any other triple of connected groups locally isomorphic to $(\tilG,\tilH,G)$ is obtained by a covering of~$\tilG$.
For the triple of the table we note that either $\tilG/\tilH$ (case~(ix)) or $\tilG/G$ (case~(x)) is diffeomorphic to~$\mathbb{S}^6$, which is simply connected.
We conclude using Lemma~\ref{lem:simply-connected-implies-H-subset-K}.
\end{proof}

\subsection{Proof of Proposition~\ref{prop:H-subset-K}.(2) in the remaining cases (v), (vi), (vii) of Table~\ref{table1}} \label{subsec:proof-H-subset-K-rest}

For the proof of Proposition~\ref{prop:H-subset-K}.(2), we do not need to consider cases (xi) and~(xii) of Table~\ref{table1}, because $\kk=\tilh\cap\g$ in this case.
Therefore, by Proposition~\ref{prop:H-subset-K-for-i-x}, it is sufficient to treat the remaining cases (v), (vi), and~(vii) of Table~\ref{table1}, as follows.

\begin{proposition} \label{prop:H-subset-K-for-v-vi-vii}
If $(\tilG,\tilH,G)$ is any triple of connected groups locally isomorphic to the triples in case (v), (vi), or~(vii) of Table~\ref{table1}, and if $K$ is a maximal connected proper subgroup of~$G$ such that $\h:=\tilh\cap\g\subset\kk$, then the inclusion \eqref{eqn:H-subset-K} holds.
\end{proposition}

For this we consider the coverings
$$\Spin(4N) \overset{\varpi}{\longrightarrow} \SO(4N) \overset{p}{\longrightarrow} \PSO(4N).$$
The center $\{\pm I_{4N}\}$ of $\SO(4N)$ is isomorphic to $\Z/2\Z$, while that of $\Spin(4N)$ is isomorphic to $\Z/2\Z\times\Z/2\Z$ (see \cite[Chap.\,X, Th.\,3.32]{hel01}).
We write $\{1,-1,E,-E\}$ for the center of $\Spin(4N)$, where $\varpi(\pm 1)=I_{4N}\in\SO(4N)$ and $\varpi(\pm E)=-I_{4N}\in\SO(4N)$.
Therefore there are five Lie groups with Lie algebras $\so(4N)$ and they are related by the following double covering maps.
$$\xymatrix{
& \Spin(4N) \ar[dl]_{\varpi_-} \ar[d]^{\varpi} \ar[dr]^{\varpi_+} & \\
\Spin(4N)/\{1,-E\} \ar[dr] & \SO(4N) \ar[d]^p & \Spin(4N)/\{1,E\} \ar[dl]\\
& \PSO(4N) &}$$
Let $L$ be a connected Lie subgroup of $\SO(4N)$.
(In the sequel, we shall take $L$ to be $\tilG$, $\tilH$, $G$, or~$K$.)
We consider connected subgroups with the same Lie algebra $\llll$ in the above five Lie groups.
Among them, we denote by $L^{\bullet}:=\varpi^{-1}(L)_0$ the identity component of $\varpi^{-1}(L)$ in $\Spin(4N)$.
The following diagram summarizes the situation.
\begin{equation} \label{eqn:diagram-L}
\xymatrix{
& L^{\bullet} \ar[dl]_{\varpi_-} \ar[d]^{\varpi} \ar[dr]^{\varpi_+} & \\
\varpi_-(L^{\bullet}) \ar[dr] & L \ar[d]^p & \varpi_+(L^{\bullet}) \ar[dl]\\
& p(L) &}
\end{equation}
Each arrow in this diagram is either a double covering or an isomorphism.
We now refine the diagram \eqref{eqn:diagram-L} in cases (v), (vi), (vii) of Table~\ref{table1} by writing a double arrow in the case of a double covering and a single arrow in the case of an isomorphism.
The following three patterns appear, up to switching $\varpi_+$ and~$\varpi_-$.

$$\xymatrixcolsep{-1.2pc}
\xymatrix{
\text{Pattern }(a) & L^{\bullet} \ar[dl] \ar@{=>}[d] \ar[dr] & \hspace{2cm} & & \text{Pattern }(b) & L^{\bullet} \ar[dl] \ar[d] \ar@{=>}[dr] & \hspace{2cm} & & \text{Pattern }(c) & L^{\bullet} \ar@{=>}[dl] \ar@{=>}[d] \ar@{=>}[dr] & \hspace{2cm}\\
\varpi_-(L^{\bullet}) \ar@{=>}[dr] & L \ar[d] & \varpi_+(L^{\bullet}) \ar@{=>}[dl] & & \varpi_-(L^{\bullet}) \ar@{=>}[dr] & L \ar@{=>}[d] & \varpi_+(L^{\bullet}) \ar[dl] & & \varpi_-(L^{\bullet}) \ar@{=>}[dr] & L \ar@{=>}[d] & \varpi_+(L^{\bullet}) \ar@{=>}[dl]\\
& p(L) & & & & p(L) & & & & p(L) &}$$

For instance, $\tilG=\SO(4N)$ has pattern~(c).

We claim the following.

\begin{lemma} \label{lem:patterns-cases-v-vi-vii}
\begin{enumerate}
  \item Inside $\tilG=\SO(4n+4)$ (case (v) of Table~\ref{table1}), either
  \begin{itemize}
    \item $\tilH:=\SO(4n+3)$ has pattern~(a);
    \item $G:=\Sp(n+1)\cdot\Sp(1)$ and $K:=\Sp(n+1)\cdot\Sp(1)$ have pattern~(b) up to switching $\varpi_+$ and~$\varpi_-$;
  \end{itemize}
  or $G:=\Sp(n+1)\cdot\Sp(1)$ has pattern~(c).
  \item Inside $\tilG=\SO(16)$ (case (vi) of Table~\ref{table1}),
  \begin{itemize}
    \item $\tilH:=\SO(15)$ has pattern~(a);
    \item $G:=\Spin(9)$ and $K:=\Spin(8)$ have pattern~(b) up to switching $\varpi_+$ and~$\varpi_-$.
  \end{itemize}
  \item Inside $\tilG=\SO(8)$ (case (vii) of Table~\ref{table1}),
  \begin{itemize}
    \item $\tilH:=\Spin(7)$ has pattern~(b) up to switching $\varpi_+$ and~$\varpi_-$;
    \item $G:=\SO(5)\times\SO(3)$ and $K:=\SO(4)\times\SO(3)$ have pattern~(a).
  \end{itemize}
\end{enumerate}
\end{lemma}

To check this, we make the following observations.

\begin{lemma} \label{lem:patterns}
\begin{enumerate}
  \item If $-I_{4N}\notin L$ and $\SO(4N)/L$ is simply connected, then $L$ has pattern~(a).
  \item If $-I_{4N}\in L$ and $-1\notin L^{\bullet}$ (\eg $L$ is simply connected), then $L$ has pattern~(b) up to switching $\varpi_+$ and~$\varpi_-$.
  \item If $-I_{4N}\in L$ and $-1\in L^{\bullet}$, then $L$ has pattern~(c).
\end{enumerate}
\end{lemma}

\begin{proof}[Proof of Lemma~\ref{lem:patterns}]
(1) If $-I_{4N}\notin L$, then $p|_L$ is an isomorphism.
If $\SO(4N)/L$ is simply connected, then $\varpi|_{L^{\bullet}}$ is a double covering; in particular, $-1\in L^{\bullet}$ and so the two unlabeled arrows are double coverings.
We deduce that $\varpi_-|_{L^{\bullet}}$ and $\varpi_+|_{L^{\bullet}}$ are isomorphisms, since any of the three maps from $L^{\bullet}$ to $p(L)$ is a double covering.

(2) If $-I_{4N}\in L$, then $p|_L$ is a double covering.
If $-1\notin L^{\bullet}$, then $\varpi|_{L^{\bullet}}$ is an isomorphism.
The fact that $-I_{4N}\in L$ means that $E\in L^{\bullet}$ or $-E\in L^{\bullet}$ (possibly both).
If $-1\notin L^{\bullet}$, only one of $E$ or~$-E$ can belong to~$L^{\bullet}$, hence exactly one of the two unlabeled arrows is a double covering.
We conclude for $\varpi_-|_{L^{\bullet}}$ and $\varpi_+|_{L^{\bullet}}$ using the fact that any of the three maps from $L^{\bullet}$ to $p(L)$ is a double covering.

(3) If $-I_{4N}\in L$, then $p|_L$ is a double covering.
If $-1\in L^{\bullet}$, then $\varpi|_{L^{\bullet}}$ is a double covering.
The fact that $-I_{4N}\in L$ means that $E\in L^{\bullet}$ or $-E\in L^{\bullet}$ (possibly both).
If $-1\notin L^{\bullet}$, only one of $E$ or~$-E$, then both $E$ and~$-E$ belong to~$L^{\bullet}$, hence both unlabeled arrow are double coverings.
We conclude for $\varpi_-|_{L^{\bullet}}$ and $\varpi_+|_{L^{\bullet}}$ using the fact that any of the three maps from $L^{\bullet}$ to $p(L)$ is covering of degree~$4$.
\end{proof}

\begin{proof}[Proof of Lemma~\ref{lem:patterns-cases-v-vi-vii}]
(a) Since $-I_{4n+4}\notin\tilH$ and $\tilG/\tilH\simeq\mathbb{S}^{4n+3}$ is simply connected, we can apply Lemma~\ref{lem:patterns}.(1) to $L=\tilH$.
Since $-I_{4n+4}\in K\subset G$, we can apply Lemma~\ref{lem:patterns}.(2) to $L=K$ and $L=G$ or Lemma~\ref{lem:patterns}.(3) to~$G$ depending on whether $-1\notin G^{\bullet}$ or $-1\in G^{\bullet}$.

(b) Since $-I_{16}\notin\tilH$ and $\tilG/\tilH\simeq\mathbb{S}^{15}$ is simply connected, we can apply Lemma~\ref{lem:patterns}.(1) to $L=\tilH$.
Since $-I_{4n+4}\in K\subset G$ and $K$ and~$G$ are simply connected, we can apply Lemma~\ref{lem:patterns}.(2) to $L=K$ and $L=G$.

(c) Since $-I_8\!\in\!\tilH$ and $\tilH$ is simply connected, we can apply Lemma~\ref{lem:patterns}.(2) to $L=\tilH$.
Since $\SO(8)/\SO(3)$ is simply connected, its quotients $\SO(8)/K$ and $\SO(8)/G$ by connected groups are also simply connected.
Since $-I_8\notin G$, we can apply Lemma~\ref{lem:patterns}.(1) to $L=K$ and $L=G$.
\end{proof}

\begin{proof}[Proof of Proposition~\ref{prop:H-subset-K-for-v-vi-vii}]
In cases (v), (vi), (vii) of Table~\ref{table1}, the patterns for the groups $\tilH$, $G$, and~$K$ are given by Lemma~\ref{lem:patterns-cases-v-vi-vii}.
On the other hand, Lemma~\ref{lem:simply-connected-implies-H-subset-K} implies that \eqref{eqn:H-subset-K} is satisfied for the form of~$\tilG$ which is simply connected, and Lemma~\ref{lem:move-up-down-H-subset-K} implies that we can transfer \eqref{eqn:H-subset-K} successively between locally isomorphic Lie groups in the diagram:
\begin{itemize}
  \item property \eqref{eqn:H-subset-K} is transferred downwards in case of a double covering for $H$ or~$\tilG$;
  \item property \eqref{eqn:H-subset-K} is transferred upwards in case of a double covering for~$K$.
\end{itemize}
It is then an easy verification to check that in our cases property \eqref{eqn:H-subset-K} holds for all five locally isomorphic quadruples $(\tilG,\tilH,G,K)$.
\end{proof}

Thus the proof of Proposition~\ref{prop:H-subset-K}.(2) is completed.

\begin{remark} \label{rem:H-subset-K}
We cannot drop the assumption that $\tilh\cap\g$ is not a maximal proper subalgebra of~$\g$ in Proposition~\ref{prop:H-subset-K}.(2).
In fact, as we have already seen, there are two cases where $\tilh\cap\g$ is a maximal proper subalgebra of~$\g$, namely cases (xi) and (xii) of Table~\ref{table1}.
In each case, there are five locally isomorphic triples $(\tilG,\tilH,G)$ of connected groups, and we can show by using a similar argument as above that the intersection $\tilH\cap G$ is connected in four cases among the five, but has two connected components in the remaining case.
This shows that \eqref{eqn:HG-ZG} does not always hold if we take $K$ to be the analytic subgroup of~$G$ with Lie algebra~$\kk$, a maximal proper subalgebra of~$\g$ containing $\tilh\cap\g$, when $\kk$ coincides with $\tilh\cap\g$.
\end{remark}

\subsection{Proof of Proposition~\ref{prop:H-subset-K}.(1)} \label{subsec:proof-HG-ZG}

We now complete the proof of Proposition~\ref{prop:H-subset-K}.(1).
By Proposition~\ref{prop:H-subset-K-for-i-x}, we only need to treat cases (v), (vi), (vii), (xi), and~(xii) of Table~\ref{table1}.
Furthermore, the proof is reduced to adjoint groups, as in the second statement of the following lemma.

\begin{lemma} \label{lem:H-connected-adjoint}
Let $(\tilG,\tilH,G)$ be a triple of connected groups as in cases (v), (vi), (vii), (xi), or~(xii) of Table~\ref{table1}.
We note that $\tilG=\SO(4N)$ for some $N\geq 1$ in all cases.
\begin{enumerate}
  \item Inside $\tilG=\SO(4N)$, the group $\tilH\cap G$ is connected.
  \item Inside $\Ad(\tilG)=\SO(4N)/\{\pm I_{4N}\}$ ($\simeq p(\tilG)$), the group $p(\tilH)\cap p(G)$ is also connected.
\end{enumerate}
\end{lemma}

\begin{proof}[Proof of Lemma~\ref{lem:H-connected-adjoint}]
(1) The homogeneous space $\tilG/\tilH$ is simply connected in cases (v), (vi), and (vii), and $\tilG/G$ is simply connected in cases (vii) and (xi).
In either case, $\tilH\cap G$ is connected by Lemma~\ref{lem:simply-connected-implies-H-subset-K}.

(2) By Lemma~\ref{lem:patterns-cases-v-vi-vii}, one of $\tilH$ or~$G$ has pattern (b) or~(c) in case (v), (vi), or (vii), and so does in case (xi) or (xii) as special cases.
Thus $-I_{4N}\in\tilH\cap G$.
Therefore $p|_{\tilH\cap G} : \tilH\cap G\to p(\tilH)\cap p(G)$ is surjective, and so $p(\tilH)\cap p(G)$ is connected.
\end{proof}

\numberwithin{theorem}{subsection}
\numberwithin{equation}{subsection}
\section{Explicit generators and relations when $\tilG$ is simple} \label{sec:computations}

In this section we complete the proof of Theorems \ref{thm:main} and~\ref{thm:main-explicit}, Corollary~\ref{cor:rel-Lapl}, Theorem~\ref{thm:transfer}, and Theorem~\ref{thm:nu-tau} for simple~$\tilG$.

By Theorem~\ref{thm:InvDiff-covering} on coverings of~$\tilG$, it suffices to prove these results for the triples $(\tilG,\tilH,G)$ of Table~\ref{table1}.
We find, for each such triple, some explicit generators and relations for the ring $\D_G(X)$ of $G$-invariant differential operators on~$X$, in terms of the three subalgebras $\D_{\tilG}(X)=\dd\ell(Z(\tilg_{\C}))$, $\dd r(Z(\kk_{\C}))$, and $\dd\ell(Z(\g_{\C}))$ of $\D_G(X)$, and determine the affine map $S_{\tau}$ in Theorems \ref{thm:transfer} and~\ref{thm:nu-tau} which induces the transfer map
$$\nnu(\cdot,\tau) : \Hom_{\C\text{-}\mathrm{alg}}(\D_{\tilG}(X),\C)\to\Hom_{\C\text{-}\mathrm{alg}}(Z(\g_{\C}),\C)$$
of \eqref{eqn:nu-lambda-tau}.
The key step in the proof is to find explicitly the map $\vartheta\mapsto (\pi(\vartheta),\tau(\vartheta))$ of Proposition~\ref{prop:pi-tau-theta} between discrete series representations, via branching laws of compact Lie groups.

\addtocontents{toc}{\SkipTocEntry}
\subsection*{Notation and conventions}

We first specify some conventions that will be used throughout the section.
Any nondegenerate, $\Ad(G)$-invariant bilinear form $B$ on the Lie algebra $\g$ of a connected reductive Lie group~$G$ defines an inner product $\langle\cdot,\cdot\rangle$ on the dual of the Cartan subalgebra, and also the Casimir element $C_G\in Z(\g)$.
We fix a positive system, and use the notation $\Rep(G,\lambda)$ to denote the irreducible representation of~$G$ with highest weight~$\lambda$.
If it is one-dimensional, we write $\C_{\lambda}$ for $\Rep(G,\lambda)$.
The trivial one-dimensional representation is denoted by~$\mathbf{1}$.
For $G=\Sp(1)(\simeq\SU(2))$, we sometimes write $\C^{\lambda+1}$ for $\Rep(G,\lambda)$, which is the unique $(\lambda+1)$-dimensional irreducible representation of~$G$.
For representations $\tau_j$ of~$G_j$ ($j=1,2$), the outer tensor product representation of the direct product group $G_1\times G_2$ is denoted by $\tau_1\boxtimes\tau_2$.
The Casimir element $C_G$ acts on $\Rep(G,\lambda)$ as the scalar
$$|\lambda + \rho|^2 - |\rho|^2 = \langle\lambda,\lambda+2\rho\rangle,$$
where $\rho$ is half the sum of the positive roots.
When $G$ is simple, $B$ is a scalar multiple of the Killing form.
For classical groups $G=\U(n)$, $\SO(n)$, or $\Sp(n)$, we shall normalize $B$ in such a way that in the standard basis $\{ e_1,\dots,e_n\}$ of the dual of a Cartan subalgebra we have $B(e_i,e_i)=1$.
With this normalization, the Casimir element $C_G$ acts on the natural representation $V$ of~$G$ as the following scalars:

\begin{center}
\begin{tabular}{|p{1.6cm}|p{1.6cm}|p{2.9cm}|}
\hline
\centering $G$ & \centering $V$ & \centering Eigenvalue of $C_G$ \tabularnewline
\hline
\centering $\U(n)$ & \centering $\C^n$ & \centering $n$\tabularnewline
\centering $\SO(n)$ & \centering $\C^n$ & \centering $n-1$\tabularnewline
\centering $\Sp(n)$ & \centering $\C^{2n}$ & \centering $2n+1$\tabularnewline
\hline
\end{tabular}
\end{center}

We use similar normalizations for $\tilG$ and for~$K$.
By Fact~\ref{fact:spherical-cpt} and Lemma \ref{lem:spherical2}.(4), the multiplicity $\ell_{\tau}=[\tau|_H:\mathbf{1}]$ is equal to~$1$ for all $\tau\in\Disc(K/H)$, and so
$$\W_{\tau} = G \times_K (V_{\tau}\otimes\C^{\ell_{\tau}}) \simeq G \times_K V_{\tau}.$$
Inside the computations, we sometimes use the notation $L^2(G/K,\tau)$ instead of $L^2(G/K,\W_{\tau})$ for $\tau\in\Disc(K/H)$.

In describing the polynomial ring $\D(G/H)$ below, we use the notation $\D_G(X)=\C[A,B,\dots]$ to mean that $\D_G(X)$ is generated by elements $A,B,\dots$ which are algebraically independent.

\subsection{The case $(\tilG,\tilH,G)=(\SO(2n+2),\SO(2n+1),\U(n+1))$}\label{subsec:ex(i)}

Here $H=\tilH\cap G=\U(n)$, and the only maximal connected proper subgroup of $G$ containing~$H$ is $K=\U(n)\times\U(1)$.
Note that the fibration
$$F =  K/H \simeq \mathbb{S}^1 \longrightarrow X=\mathbb{S}^{2n+1} \longrightarrow Y=\PP^n\C$$
of \eqref{eqn:fiber} is the Hopf fibration.
Let $E_K$ be a generator of the complexified Lie algebra $\C$ of the second factor $\mathfrak{u}(1)$ of $\kk=\mathfrak{u}(n)\oplus\mathfrak{u}(1)$, such that the eigenvalues of $\ad(E_K)$ in~$\g_{\C}$ are $0,\pm 1$.

\begin{proposition}[Generators and relations] \label{prop:ex(i)}
For
$$X = \tilG/\tilH = \SO(2n+2)/\SO(2n+1) \simeq \U(n+1)/\U(n) = G/H$$
and $K=\U(n)\times\U(1)$, we have
\begin{enumerate}
  \item $\dd\ell(C_{\tilG}) = 2\,\dd\ell(C_G) - \dd r(C_K)$;
  \item $\left\{ \begin{array}{ccl}
  \D_{\tilG}(X) & \!\!\!=\!\!\! & \C[\dd\ell(C_{\tilG})];\\
  \D_K(F) & \!\!\!=\!\!\! & \C[\dd r(E_K)];\\
  \D_G(X) & \!\!\!=\!\!\! & \C[\dd\ell(C_{\tilG}),\dd r(E_K)] = \C[\dd\ell(C_G),\dd r(E_K)].
  \end{array}\right.$
\end{enumerate}
\end{proposition}

We identify
\begin{align}
\Hom_{\C\text{-}\mathrm{alg}}(Z(\g_{\C}),\C)\ & \simeq & \hspace{-1cm}\jj_{\C}^*/W(\g_{\C})\ & \simeq & \hspace{-1.1cm}\C^{n+1}/\mathfrak{S}_{n+1},\label{eqn:HCZg-i}\\
\Hom_{\C\text{-}\mathrm{alg}}(\D_{\tilG}(X),\C)\ & \simeq & \hspace{-1cm}\tila_{\C}^*/\widetilde{W}\hspace{0.5cm} & \simeq & \hspace{-1.1cm}\C/(\Z/2\Z)\hspace{0.3cm}\label{eqn:HCX-i}
\end{align}
by the standard bases.
The set $\Disc(K/H)$ consists of the representations of $K=\U(n)\times\U(1)$ of the form $\tau = \mathbf{1} \boxtimes \C_a$ for $a\in\Z$, where $\C_a$ is the one-dimensional representation of~$\U(1)$ given by $z\mapsto z^a$.
The element $E_K$ acts on~$\C_a$ by $\sqrt{-1}\,a$.

\begin{proposition}[Transfer map] \label{prop:nu-ex(i)}
Let
$$X = \tilG/\tilH = \SO(2n+2)/\SO(2n+1) \simeq \U(n+1)/\U(n) = G/H$$
and $K=\U(n)\times\U(1)$.
For $\tau=\mathbf{1}\boxtimes\C_a\in\Disc(K/H)$ with $a\in\Z$, the affine map
\begin{eqnarray*}
S_{\tau} :\ \tila_{\C}^* \simeq \C & \longrightarrow & \hspace{3.2cm} \C^{n+1} \hspace{3cm} \simeq \jj_{\C}^*\\
\lambda & \longmapsto & \frac{1}{2} \, \big(\lambda+a, n-2, n-4, \dots, -n+2, -\lambda+a\big)
\end{eqnarray*}
induces a transfer map
$$\nnu(\cdot,\tau) : \Hom_{\C\text{-}\mathrm{alg}}(\D_{\tilG}(X),\C) \longrightarrow \Hom_{\C\text{-}\mathrm{alg}}(Z(\g_{\C}),\C)$$
as in Theorem~\ref{thm:nu-tau}.
\end{proposition}

In order to prove Propositions \ref{prop:ex(i)} and~\ref{prop:nu-ex(i)}, we use the following results on finite-dimensional representations.

\begin{lemma}\label{lem:ex(i)}
\begin{enumerate}
  \item Discrete series for $\tilG/\tilH$, $G/H$, and $F=K/H$:
  \begin{eqnarray*}
    \Disc(\SO(2n+2)/\SO(2n+1)) & = & \{ \mathcal{H}^j(\R^{2n+2}) \,:\, j\in\N\} ;\\
    \Disc(\U(n+1)/\U(n)) & = & \{ \mathcal{H}^{k,\ell}(\C^{n+1}) \,:\, k,\ell\in\N\} ;\\
    \Disc((\U(n)\times\U(1))/\U(n)) & = & \{ \mathbf{1} \boxtimes \C_a \,:\, a\in\Z\} .
  \end{eqnarray*}
  \item Branching laws for $\SO(2n+2)\downarrow\U(n+1)$: For $j\in\N$,
  $$\mathcal{H}^j(\R^{2n+2})\ \simeq\ \bigoplus_{k=0}^j \mathcal{H}^{k,j-k}(\C^{n+1}).$$
  \item Irreducible decomposition of the regular representation of~$G$: For~$a\in\nolinebreak\Z$,
  $$L^2\big(\U(n+1)/(\U(n)\times\U(1)),\mathbf{1}\boxtimes\C_a\big) \,\simeq\ \sumplus{\substack{j\in\N\\ j-|a|\in 2\N}}\ \mathcal{H}^{\frac{j+a}{2},\frac{j-a}{2}}(\C^{n+1}).$$
  \item The ring $S(\g_{\C}/\h_{\C})^H=S(\gl(n+1,\C)/\gl(n,\C))^{\U(n)}$ is generated by two algebraically independent homogeneous elements of respective degrees $1$ and~$2$.
\end{enumerate}
\end{lemma}

Here we denote by $\mathcal{H}^j(\R^m)$ the space of spherical harmonics in~$\R^m$, \ie of complex-valued homogeneous polynomials $f(x_1,\dots,x_m)$ of degree $j\in\N$ such that
$$\sum_{i=1}^m \frac{\partial^2 f}{\partial x_i^2} = 0.$$
For $m>2$, the special orthogonal group $\SO(m)$ acts irreducibly on $\mathcal{H}^j(\R^m)$; the highest weight is $(j,0,\dots,0)$ in the standard coordinates.

For $k,\ell\in\N$, we denote by $\mathcal{H}^{k,\ell}(\C^m)$ the space of homogeneous polynomials $f(z_1,\dots,z_m,\overline{z_1},\dots,\overline{z_m})$ of degree~$k$ in $z_1,\dots,z_m$ and degree~$\ell$ in $\overline{z_1},\dots,\overline{z_m}$, such that
$$\sum_{i=1}^m \frac{\partial^2 f}{\partial z_i\partial\overline{z_i}} = 0.$$
For $m>1$ or for $m=1$ and $k\ell=0$, the unitary group $\U(m)$ acts irreducibly on $\mathcal{H}^{k,\ell}(\C^m)$; the highest weight is $(k,0,\dots,0,-\ell)$ in the standard coordinates.

\begin{proof}[Proof of Lemma~\ref{lem:ex(i)}]
For statements (1)--(3), see \eg \cite[\S\,2.1 \& 4.2]{ht93} in the context of the see-saw dual pair
$$\xymatrix{\OO(2n+2) \ar@{-}[d] \ar@{-}[dr] & \U(1) \ar@{-}[d] \ar@{-}[dl]\\ \U(n+1) & \OO(1).}$$

Statement~(4) follows from \cite{sch78}; we give a proof for the sake of completeness.
Via the decomposition
$$\gl(n+1,\C)/\gl(n,\C) \simeq \C^n \oplus (\C^n)^{\vee} \oplus \C$$
of $\gl(n+1,\C)/\gl(n,\C)$ into irreducible $\GL(n,\C)$-modules, the symmetric tensor space decomposes as
$$S\big(\gl(n+1,\C)/\gl(n,\C)\big) \simeq \bigoplus_{a,b,c\in\N} S^a(\C^n) \otimes S^b\big((\C^n)^{\vee}\big) \otimes S^c(\C).$$
Since the $\GL(n,\C)$-modules $S^a(\C^n)$, for $a\in\N$, are irreducible and mutually inequivalent, we have
$$S\big(\gl(n+1,\C)/\gl(n,\C)\big)^{\U(n)} \simeq \bigoplus_{a,c\in\N} \Big( S^a(\C^n) \otimes S^a\big((\C^n)^{\vee}\big) \Big)^{\GL(n,\C)} \otimes S^c(\C).$$
Therefore,
$$\dim S^N\big(\gl(n+1,\C)/\gl(n,\C)\big)^{\U(n)} = \#\big\{ (a,c)\in\N^2 \ :\ 2a+c=N\big\} ,$$
which is the dimension of the space of homogeneous polynomials of degree~$N$ in $\C[x,y^2]$, and so we may apply Lemma~\ref{lem:struct-DGX}.
\end{proof}

\begin{proof}[Proof of Proposition~\ref{prop:ex(i)}]
(1) By Lemma~\ref{lem:ex(i)}, the map $\vartheta\mapsto (\pi(\vartheta),\tau(\vartheta))$ of Proposition~\ref{prop:pi-tau-theta} is given by
\begin{equation}\label{eqn:pitau-i}
\mathcal{H}^{k,\ell}(\C^{n+1}) \longmapsto (\mathcal{H}^{k+\ell}(\R^{2n+2}), \mathbf{1}\boxtimes\C_{k-\ell}).
\end{equation}
The Casimir operators for $\tilG$, $G$, and~$K$ act on these representations as the following scalars.
\begin{center}
\begin{tabular}{|c|c|c|}
\hline
Operator & Representation & Scalar\\
\hline\hline
$C_{\tilG}$ & $\mathcal{H}^{k+\ell}(\R^{2n+2})$ & $(k+\ell) (k+\ell+2n)$\\
\hline
$C_G$ & $\mathcal{H}^{k,\ell}(\C^{n+1})$ & $k^2+\ell^2+kn+\ell n$\\
\hline
$C_K$ & $\mathbf{1}\boxtimes\C_{k-\ell}$ & $(k-\ell)^2$\\
\hline
$E_K$ & $\mathbf{1}\boxtimes\C_{k-\ell}$ & $\sqrt{-1}\,(k-\ell)$\\
\hline
\end{tabular}
\end{center}
This, together with the identity
$$(k+\ell) (k+\ell+2n) = 2\,(k^2+\ell^2+kn+\ell n) - (k-\ell)^2,$$
implies $\dd\ell(C_{\tilG}) \!=\! 2\,\dd\ell(C_G) - \dd r(C_K)$ on the $G$-isotypic component $\mathcal{H}^{k,\ell}(\C^{n+1})$ in $C^{\infty}(X)$ for all $k,\ell\in\N$, hence on the whole of $C^{\infty}(X)$.

(2) Since $\tilG/\tilH$ and $F=K/H$ are symmetric spaces, we obtain $\D_{\tilG}(X)$ and $\D_K(F)$ using the Harish-Chandra isomorphism \eqref{eqn:HC-isom}.
We now focus on $\D_G(X)$.
We only need to prove the first equality, since the other one follows from the relations between the generators.
For this, using Lemmas \ref{lem:struct-DGX}.(3) and~\ref{lem:ex(i)}.(4), it suffices to show that the two differential operators $\dd\ell(C_{\tilG})$ and $\dd r(E_K)$ on~$X$ are algebraically independent.
Let $f$ be a polynomial in two variables such that $f(\dd\ell(C_{\tilG}),\dd r(E_K))=0$ in $\D_G(X)$.
By letting this differential operator act on the $G$-isotypic component $\vartheta = \mathcal{H}^{k,\ell}(\C^{n+1})$ in $C^{\infty}(X)$, we obtain $f((k+\ell)(k+\ell+2n),-\sqrt{-1}\,(k-\nolinebreak\ell))=0$ for all $k,\ell\in\N$, hence $f$ is the zero polynomial.
\end{proof}

\begin{proof}[Proof of Proposition~\ref{prop:nu-ex(i)}]
We use Proposition~\ref{prop:Stau-transfer} and the formula \eqref{eqn:pitau-i} for the map $\vartheta\mapsto (\pi(\vartheta),\tau(\vartheta))$ of Proposition~\ref{prop:pi-tau-theta}.
Let $\tau=\mathbf{1}\boxtimes\C_a\in\Disc(K/H)$ with $a\in\Z$.
If $\vartheta\in\Disc(G/H)$ satisfies $\tau(\vartheta)=\tau$, then $\vartheta$ is of the form $\vartheta=\mathcal{H}^{k,\ell}(\C^{n+1})$ for some $k,\ell\in\N$ with $k-\ell=a$, by \eqref{eqn:pitau-i}.
The algebra $\D_{\tilG}(X)$ acts on the irreducible $\tilG$-submodule $\pi(\vartheta)=\mathcal{H}^{k+\ell}(\R^{2n+2})$ by the scalars
$$\lambda(\vartheta) + \rho_{\tila} = (k+\ell) + n \in \C/(\Z/2\Z)$$
via the Harish-Chandra isomorphism \eqref{eqn:HCX-i}, whereas the algebra $Z(\g_{\C})$ acts on the irreducible $G$-module $\vartheta = \Rep(\SO(5),(j,k))\boxtimes\C_a$ by the scalars
$$\nu(\vartheta) + \rho = \Big(k+\frac{n}{2}, \frac{n}{2}-1, \frac{n}{2}-2, \dots, 1-\frac{n}{2},-\ell-\frac{n}{2}\Big) \in \C^{n+1}/\mathfrak{S}_{n+1}$$
via \eqref{eqn:HCZg-i}.
Thus the affine map $S_{\tau}$ in Proposition~\ref{prop:nu-ex(i)} sends $\lambda(\vartheta)+\rho_{\tila}$ to $\nu(\vartheta)+\rho$ for any $\vartheta\in\Disc(G/H)$ such that $\tau(\vartheta)=\tau$, and we conclude using Proposition~\ref{prop:Stau-transfer}.
\end{proof}

This completes the proof of Theorems \ref{thm:main}, \ref{thm:main-explicit}, \ref{thm:transfer}, and~\ref{thm:nu-tau}, as well as Corollary~\ref{cor:rel-Lapl}, in case (i) of Table~\ref{table1}.
For case (i)$'$ of Table~\ref{table1}, the homogeneous spaces $G/H$ and $K/H$ change as follows:\begin{align*}
G/H\,: && \hspace{-0.6cm}\U(n+1)/\U(n) & \ \leadsto\ \SU(n+1)/\SU(n),\\
K/H\,: && \hspace{-0.6cm}(\U(n)\times\U(1))/\U(n) & \ \leadsto\ \U(n)/\SU(n).
\end{align*}
The proof works similarly, and so we omit it.

\subsection{The case $(\tilG,\tilH,G)=(\SO(2n+2),\U(n+1),\SO(2n+1))$}\label{subsec:ex(ii)}

There are two inequivalent embeddings of $\tilH=\U(n+1)$ into $\tilG=\SO(2n+2)$, which are conjugate by an outer automorphism of~$\tilG$.
In either case, we have $H=\linebreak\tilH\cap G=\U(n)$ and the only maximal connected proper subgroup of $G$ containing~$H$ is $K=\SO(2n)$.
Note that $X_{\C}=G_{\C}/H_{\C}=\SO(2n+1,\C)/\GL(n,\C)$ is $G_{\C}$-spherical but is not a symmetric space.
We shall give explicit generators of the ring $\D_G(X)$ by using the fibration $X\overset{F}{\longrightarrow} G/K$.

Since $X=\tilG/\tilH$ and $F=K/H$ are symmetric spaces, the structure of the rings $\D_{\tilG}(X)$ and $\D_K(F)$ is well understood by the Harish-Chandra isomorphism \eqref{eqn:HC-isom}.
Further, $\D_{\tilG}(X)=\dd\ell(Z(\tilg_{\C}))$ and $\D_K(F)=\dd r_F(Z(\kk_{\C}))$ by Fact~\ref{fact:ZgDGX}.
We now recall generators of the $\C$-algebras $\D_{\tilG}(X)$ and $\D_K(F)$, to be used in the description of $\D_G(X)$ (Propositions \ref{prop:ex(ii)} and~\ref{prop:ex(ii)bis}).
We refer to Section~\ref{subsec:strong-mult-free} for the notation $\chi_{\lambda}^X,\chi_{\mu}^F,\chi_{\nu}^G$ for the Harish-Chandra isomorphisms.

The rank of the symmetric space $X=\tilG/\tilH$ is $m:=\lfloor\frac{n+1}{2}\rfloor$, and the restricted root system $\Sigma(\tilg_{\C},\tila_{\C})$ is of type $BC_m$ if $n=2m$ is even and of type $C_m$ if $n=2m-1$ is odd.
In either case, the Weyl group $\widetilde{W}$ of the restricted root system $\Sigma(\tilg_{\C},\tila_{\C})$ is isomorphic to $\mathfrak{S}_m\ltimes (\Z/2\Z)^m$.
The ring $S(\tila_{\C})^{\widetilde{W}}$ is a polynomial ring generated by algebraically independent homogeneous elements of respective degrees $2,4,\dots,2m$.
Consider the standard basis $\{ h_1,\dots,h_m\}$ of~$\tila_{\C}^{\ast}$ and choose a positive system such that
$$\Sigma^+(\tilg_{\C},\tila_{\C}) \!=\! \left\{ \begin{array}{ll}
\!\{ 2h_i, h_j\pm h_k : 1\leq i\leq m,\ 1\leq j<k\leq m\} & \!\!\!\text{if }n=2m-1,\\
\!\{ h_i, 2h_i, h_j\pm h_k : 1\leq i\leq m,\ 1\leq j<k\leq m\} & \!\!\!\text{if }n=2m.
\end{array}\right.$$
Using these coordinates, for $k\in\N$ we define $P_k\in\D_{\tilG}(X)$ by
$$\chi_{\lambda}^X(P_k) = \sum_{j=1}^m \lambda_j^{2k}$$
for $\lambda=(\lambda_1,\dots,\lambda_m)\in\tila_{\C}^{\ast}/\widetilde{W}$.
Then $\D_{\tilG}(X)$ is a polynomial algebra generated by $P_1,\dots,P_m$.
The rank of the symmetric space $F=K/H$ is $\lfloor\frac{n}{2}\rfloor$, and the restricted root system $\Sigma(\kk_{\C},(\aaa_F)_{\C})$ is of type $BC_{\lfloor n/2\rfloor}$.
We define similarly $Q_k\in\D_K(F)$ for $k\in\N_+$ by
$$\chi_{\mu}^F(Q_k) = \sum_{j=1}^{\lfloor n/2\rfloor} \mu_j^{2k}$$
for $\mu=(\mu_1,\dots,\mu_{\lfloor n/2\rfloor})\in(\jj_F)_{\C}^{\ast}/\widetilde{W}$.
Then $\D_K(F)$ is a polynomial algebra generated by $Q_1,\dots,Q_{\lfloor\frac{n}{2}\rfloor}$.
Finally, take a Cartan subalgebra $\jj_{\C}$ of~$\g_{\C}$ and the standard basis $\{ f_1,\dots,f_n\}$ of $\jj_{\C}^{\ast}$ such that the root system $\Delta(\g_{\C},\jj_{\C})$ is given as
$$\big\{\!\pm f_i, \pm f_j\pm f_k \,:\, 1\leq i\leq n,\ 1\leq j<k\leq n\big\} .$$
For $k\in\N_+$, we define $R_k\in Z(\g_{\C})$ by
$$\chi_{\nu}^G(R_k) = \sum_{j=1}^n \nu_j^{2k}$$
for $\nu=(\nu_1,\dots,\nu_n)\in\jj_{\C}^{\ast}/W(B_n)$.
Note that $P_k$, $\iota(Q_k)$, and $\dd r(R_k)$ are all differential operator of order $2k$ on~$X$.

With this notation, here is our description of $\D_G(X)$.

\subsubsection{The case that $n=2m-1$ is odd}

\begin{proposition}[Generators and relations] \label{prop:ex(ii)}
For
$$X = \tilG/\tilH = \SO(4m)/\U(2m) \simeq \SO(4m-1)/\U(2m-1) = G/H$$
and $K=\SO(4m-2)$, we have
\begin{enumerate}
  \item $\left\{ \begin{array}{l}
  P_k + \iota(Q_k) = 2^{2k}\,\dd\ell(R_k) \quad\mathrm{for\ all\ }k\in\N_+;\\
  \dd\ell(C_{\tilG}) = 2\,\dd\ell(C_G) - \dd r(C_K);
  \end{array}\right.$
  \item $\left\{ \begin{array}{ccl}
\D_{\tilG}(X) & \!\!\!=\!\!\! & \C[P_1,\dots,P_m];\\
\D_K(F) & \!\!\!=\!\!\! & \C[Q_1,\dots,Q_{m-1}];\\
\D_G(X) & \!\!\!=\!\!\! & \C[P_1,\dots,P_m,\iota(Q_1),\dots,\iota(Q_{m-1})]\\
& & \hspace{0.3cm} =\, \C[P_1,\dots,P_m,\dd r(R_1),\dots,\dd r(R_{m-1})]\\
& & \hspace{0.3cm} =\,\C[\iota(Q_1),\dots,\iota(Q_{m-1}),\dd r(R_1),\dots,\dd r(R_m)].
\end{array}\right.$
\end{enumerate}
\end{proposition}

In all the equalities of Proposition~\ref{prop:ex(ii)}.(2), the right-hand side denotes the polynomial ring generated by algebraically independent elements.

In order to give an explicit description of the affine map $S_{\tau}$ inducing the transfer map, we write the Harish-Chandra isomorphism as
\begin{align}
\Hom_{\C\text{-}\mathrm{alg}}(Z(\g_{\C}),\C) & \ \simeq\ \jj_{\C}^*/W(\g_{\C}) && \hspace{-0.8cm} \simeq\ \C^{2m-1}/W(B_{2m-1}),\label{eqn:HCZg-iia}\\
\Hom_{\C\text{-}\mathrm{alg}}(\D_{\tilG}(X),\C) & \ \simeq\ \hspace{0.3cm} \tila_{\C}^*/\widetilde{W} && \hspace{-0.8cm} \simeq\ \C^m/W(C_m)\label{eqn:HCX-iia}
\end{align}
using the standard bases.
The set $\Disc(K/H)$ is the set of representations of $K=\SO(4m-2)$ of the form $\tau=\Rep(\SO(4m-2),t_{m,m-1}((k,0),k)$ for $k\in (\N^{m-1})_{\geq}$.
We set
$$b(k) := \Big(k_i + 2(m-i) - \frac{1}{2}\Big)_{1\leq i\leq m-1} \in \C^{m-1}.$$

\begin{proposition}[Transfer map] \label{prop:nu-ex(ii)}
Let
$$X = \tilG/\tilH = \SO(4m)/\U(2m) \simeq \SO(4m-1)/\U(2m-1) = G/H$$
and $K\!=\!\SO(4m-2)$.
For $\tau\!=\!\Rep(\SO(4m-2),t_{m,m-1}((k,0),k)\in\nolinebreak\Disc(K/H)$ with $k\in (\N^{m-1})_{\geq}$, the affine map
\begin{eqnarray*}
S_{\tau} :\ \tila_{\C}^* \simeq \C^m & \longrightarrow & \hspace{0.8cm} \C^{2m-1} \hspace{0.8cm} \simeq \jj_{\C}^*\\
\lambda \hspace{0.3cm} & \longmapsto & t_{m,m-1}\Big(\frac{\lambda}{2}, b(k)\Big)\nonumber
\end{eqnarray*}
induces a transfer map
$$\nnu(\cdot,\tau) : \Hom_{\C\text{-}\mathrm{alg}}(\D_{\tilG}(X),\C) \longrightarrow \Hom_{\C\text{-}\mathrm{alg}}(Z(\g_{\C}),\C)$$
as in Theorem~\ref{thm:nu-tau}.
\end{proposition}

For the proof of Propositions \ref{prop:ex(ii)} and~\ref{prop:nu-ex(ii)}, we use the following lemma on finite-dimensional representations.

\begin{lemma}\label{lem:ex(ii)}
\begin{enumerate}
  \item Discrete series for $\tilG/\tilH$, $G/H$, and $F=K/H$:
  \begin{align*}
    & \Disc(\SO(4m)/\U(2m)) = \{ \Rep(\SO(4m),t_{m,m}(j,j)) \,:\, j\in (\N^m)_{\geq}\} ;\\
    & \Disc(\SO(4m-1)/\U(2m-1)) = \{ \Rep(\SO(4m-1),\omega) \,:\, \omega\in (\N^{2m-1})_{\geq}\} ;\\
    & \Disc(\SO(4m-2)/\U(2m-1))\\
    & \hspace{1cm} = \big\{ \Rep\big(\SO(4m-2),t_{m,m-1}((k,0),k)\big) \,:\, k\in (\N^{m-1})_{\geq}\big\} .
  \end{align*}
  \item Branching laws for $\SO(4m)\downarrow\SO(4m-1)$: For $j\in (\N^m)_{\geq}$,
  \begin{align*}
  & \Rep\big(\SO(4m),t_{m,m}(j,j)\big)|_{\SO(4m-1)}\\
  & \simeq \bigoplus_{\substack{k\in(\N^{m-1})_{\geq}\\ t_{m,m-1}(j,k)\in(\N^{2m-1})_{\geq}}} \Rep\big(\SO(4m-1),t_{m,m-1}(j,k)\big).
  \end{align*}
  \item Irreducible decomposition of the regular representation of~$G$:\\ For $k\in\nolinebreak (\N^{m-1})_{\geq}$,
  \begin{align*}
  & L^2\Big(\SO(4m-1)/\SO(4m-2),\Rep\big(\SO(4m-2),t_{m,m-1}((k,0),k)\big)\Big)\\
  & \simeq\ \ \sumplus{\substack{j\in(\N^m)_{\geq}\\ t_{m,m-1}(j,k)\in(\N^{2m-1})_{\geq}}}\, \Rep\big(\SO(4m-1),t_{m,m-1}(j,k)\big).
  \end{align*}
  \item The ring $S(\g_{\C}/\h_{\C})^H=S(\so(4m-1,\C)/\gl(2m-2,\C))^{\U(2m-2)}$ is generated by algebraically independent homogeneous elements of respective degrees $2,2,4,4,\dots,m-1,m-1,m$.
\end{enumerate}
\end{lemma}

Here we use the following notation: for $m',m''\in\N_+$ with $m''\in\{ m',m'-\nolinebreak 1\}$, we define an ``alternating concatenation'' map $t_{m',m''} : \Z^{m'}\times\Z^{m''}\!\rightarrow\Z^{m'+m''}$~by
\begin{equation}\label{eqn:t-m'-m''}
t_{m',m''}(j,k) = (j_1,k_1,j_2,k_2,\dots)
\end{equation}
for $j=(j_1,\ldots,j_{m'})\in\Z^{m'}$ and $k=(k_1,\ldots,k_{m''})\in\Z^{m''}$.

\begin{proof}[Proof of Lemma~\ref{lem:ex(ii)}]
Since $\tilG/\tilH$ and $K/H$ are symmetric spaces, the first and third formulas of~(1) follow from the Cartan--Helgason theorem (Fact~\ref{fact:CartanHelgason}).
For the second formula of~(1) (description of the discrete series of the nonsymmetric spherical homogeneous space $G/H$), see \cite{kra79}; the argument below using branching laws and \eqref{eqn:decomp2} gives an alternative proof.
One immediately deduces (2) and~(3) from the classical branching laws for $\SO(\ell)\downarrow\SO(\ell-1)$, see \eg \cite[Th.\,8.1.3 \& 8.1.4]{gw09}, and the Frobenius reciprocity.
For~(4), see the tables in \cite{sch78}, or \cite[\S\,10]{kno94}.
\end{proof}

\begin{proof}[Proof of Proposition~\ref{prop:ex(ii)}]
(1) We first prove that $P_k+\iota(Q_k)=2^{2k}\,\dd\ell(R_k)$.
By Lemma~\ref{lem:ex(ii)}, the map $\vartheta\mapsto (\pi(\vartheta),\tau(\vartheta))$ of Proposition~\ref{prop:pi-tau-theta} is given by
\begin{equation} \label{eqn:pitau-iia}
\begin{array}{l}
\Rep(\SO(4m-1),t_{m,m-1}(j,k))\\
\longmapsto \big(\Rep(\SO(4m),t_{m,m}(j,j)),\Rep(\SO(4m-2),t_{m,m-1}((k,0),k))\big)\hspace{-0.5cm}
\end{array}
\end{equation}
for $j\in\N^m$ and $k\in\N^{m-1}$ with $j_1\geq k_1\geq\dots\geq j_{m-1}\geq k_{m-1}\geq j_m$.
For $1\leq i\leq m$ we set
$$a_i := j_i + 2(m-i) + \frac{1}{2}.$$
Since $\rho_{\aaa} = \sum_{i=1}^m (4m-4i+1)\,h_i$, we obtain
$$2 \sum_{i=1}^m j_i h_i + \rho_{\tila} \,=\, 2 \sum_{i=1}^m a_i h_i \,\in \tila_{\C}^*.$$
Since the embedding $\tila_{\C}^{\ast}\hookrightarrow\widetilde{\jj}_{\C}^{\ast}$ is given by $2j\mapsto t_{m,m}(j,j)$ via the standard bases of $\tila_{\C}^{\ast}$ and~$\widetilde{\jj}_{\C}^{\ast}$, the map $T : \tila_{\C}^{\ast}/\widetilde{W}\rightarrow\widetilde{\jj}_{\C}^{\ast}/W(\tilg_{\C})$ (see \eqref{eqn:aj}) satisfies
$$T\bigg(2 \sum_{i=1}^m a_i h_i\bigg) = \Big(a_1+\frac{1}{2}, a_1-\frac{1}{2}, \dots, a_m+\frac{1}{2}, a_m-\frac{1}{2}\Big) \in \widetilde{\jj}_{\C}^{\ast}/W(\tilg_{\C}),$$
which is the $Z(\tilg_{\C})$-infinitesimal character $\Psi_{\pi(\vartheta)}$ of $\pi(\vartheta)$, and by Lemma~\ref{lem:chiZgDGX}.(1) the operator $P_k$ acts on the representation space of $\pi(\vartheta)$ as the scalar
\begin{equation} \label{eqn:piP-iia}
\chi_{(2a_1,\dots,2a_m)}^X(P_k) = 2^{2k} \sum_{i=1}^m a_i^{2k}.
\end{equation}
For $1\leq i\leq m-1$ we set
$$b_i := k_i + 2(m-i) - \frac{1}{2}.$$
Then $\tau(\vartheta)^{\vee}$ has $Z(\kk_{\C})$-infinitesimal character
$$\Psi_{\tau(\vartheta)^{\vee}} = - \Big(b_1+\frac{1}{2}, b_1-\frac{1}{2}, \dots, b_{m-1}+\frac{1}{2}, b_{m-1}-\frac{1}{2}, 0\Big) \in (\jj_{\kk})_{\C}^{\ast}/W(\kk_{\C}),$$
and by Lemma~\ref{lem:chiZgDGX}.(2) the operator $\iota(Q_k)$ acts on the representation space of $\tau(\vartheta)$ as the scalar
$$\chi_{(2b_1,\dots,2b_{m-1})}^F(Q_k) = 2^{2k} \sum_{i=1}^{m-1} b_i^{2k}.$$
On the other hand, $\vartheta$ itself has $Z(\g_{\C})$-infinitesimal character
$$\Psi_{\vartheta} = (a_1, b_1, \dots, a_{m-1}, b_{m-1}, a_m) \in \jj_{\C}^{\ast}/W(B_{2m-1}),$$
and so $\dd\ell(R_k)$ acts on the representation space of~$\vartheta$ as the scalar
$$\Psi_{\vartheta}(R_k) = \sum_{i=1}^m a_i^{2k} + \sum_{i=1}^{m-1} b_i^{2k}.$$
Thus $P_k+\iota(Q_k)=2^{2k}\,\dd\ell(R_k)$ by Proposition~\ref{prop:PQR}.

We now check the relation among Casimir operators.
For any $\ell\in\N_+$ and $j\in\Z^{\ell}$, we set
\begin{align*}
h_{\ell}(j) & = \sum_{i=1}^{\ell} j_i^2 + \sum_{i=1}^{\ell} (4\ell-4i+1)\,j_i\\
\text{and}\quad\quad h'_{\ell}(j) & = \sum_{i=1}^{\ell} j_i^2 + \sum_{i=1}^{\ell} (4\ell-4i+3)\,j_i.
\end{align*}
The Casimir operators for $\tilG$, $G$, and~$K$ act on the following irreducible representations as the following scalars.
\begin{center}
\begin{tabular}{|c|c|c|}
\hline
Operator & Representation & Scalar\\
\hline\hline
$C_{\tilG}$ & $\Rep(\SO(4m),t_{m,m}(j,j))$ & $2\,h_m(j)$\\
\hline
$C_G$ & $\Rep(\SO(4m-1),t_{m,m-1}(j,k))$ & $h_m(j)+h'_{m-1}(k)$\\
\hline
$C_K$ & $\Rep(\SO(4m-2),t_{m,m-1}((k,0),k))$ & $2\,h'_{m-1}(k)$\\
\hline
\end{tabular}
\end{center}
\noindent
This implies $\dd\ell(C_{\tilG})=2\,\dd\ell(C_G)-\dd r(C_K)$.

(2) We have already given descriptions of $\D_{\tilG}(X)$ and $\D_K(F)$, so we now focus on $\D_G(X)$.
We only need to prove the first equality for $\D_G(X)$, since the other ones follow from the relations between the generators.
For this, using Lemmas \ref{lem:struct-DGX}.(3) and~\ref{lem:ex(ii)}.(4), it suffices to show that the differential operators $P_1,\dots,P_m,\iota(Q_1),\dots,\iota(Q_{m-1})$ are algebraically independent.
Let $f$ be a po\-lynomial in $(2m-1)$ variables such that $f(P_1,\dots,P_m,\iota(Q_1),\dots,\iota(Q_{m-1}))\linebreak =\nolinebreak 0$ in $\D_G(X)$.
By letting this differential operator act on the $G$-isotypic component $\vartheta = \Rep(\SO(4m-1),t_{m,m-1}(j,k))$ in $C^{\infty}(X)$, we obtain
$$f\bigg(\sum_{i=1}^m A_i^2, \sum_{i=1}^m A_i^4, \dots, \sum_{i=1}^m A_i^{2m}, \sum_{i=1}^{m-1} B_i^2, \dots, \sum_{i=1}^{m-1} B_i^{2m-2}\bigg) = 0,$$
where we set
$$\left\{\begin{array}{l}
A_i := 2a_i = 2j_i+4(m-i)+1,\\
B_i := 2b_i = 2k_i+4(m-i)+1.
\end{array}\right.$$
Since the set of elements $(A_1,\dots,A_m,B_1,\dots,B_{m-1})\in\C^{2m-1}$ for $\vartheta$ ranging over $\Disc(G/H)$ is Zariski-dense in~$\C^{2m-1}$, we conclude that $f$ is the zero polynomial.
\end{proof}

\begin{proof}[Proof of Proposition~\ref{prop:nu-ex(ii)}]
We retain the notation of the first half of the proof of Proposition~\ref{prop:ex(ii)}.(1), and use Proposition~\ref{prop:Stau-transfer} and the formula \eqref{eqn:pitau-iia} for the map $\vartheta\mapsto (\pi(\vartheta),\tau(\vartheta))$ of Proposition~\ref{prop:pi-tau-theta}.
Let
$$\tau = \Rep\big(\SO(4m-2),t_{m,m-1}((k,0),k)\big) \in \Disc(K/H)$$
with $k\in (\N^{m-1})_{\geq}$.
If $\vartheta\in\Disc(G/H)$ satisfies $\tau(\vartheta)=\tau$, then $\vartheta$ is of the form $\vartheta = \Rep(\SO(4m-1),t_{m,m-1}(j,k))$ for some $j\in\N$ with $j_1\geq k_1\geq\dots\geq j_{m-1}\geq k_{m-1}\geq j_m$, by \eqref{eqn:pitau-iia}.
The algebra $\D_{\tilG}(X)$ acts on the irreducible $\tilG$-submodule $\pi(\vartheta) = \Rep(\SO(4m),t_{m,m}(j,j))$ by the scalars
$$\lambda(\vartheta) + \rho_{\tila} = (2a_1,\dots,2a_m) \in \C^m/W(C_m)$$
via the Harish-Chandra isomorphism \eqref{eqn:HCX-iia}, where $a_i = j_i + 2(m-i) + 1/2$ ($1\leq i\leq m$), whereas the algebra $Z(\g_{\C})$ acts on the irreducible $G$-module $\vartheta = \Rep(\SO(4m-1),t_{m,m-1}(j,k))$ by the scalars
$$\nu(\vartheta) + \rho = (a_1,\dots,a_m,b_1,\dots,b_{m-1}) \in \C^{2m-1}/W(B_{2m-1})$$
via \eqref{eqn:HCZg-iia}.
Thus the affine map $S_{\tau}$ in Proposition~\ref{prop:nu-ex(ii)} sends $\lambda(\vartheta)+\rho_{\tila}$ to $\nu(\vartheta)+\rho$ for any $\vartheta\in\Disc(G/H)$ such that $\tau(\vartheta)=\tau$, and we conclude using Proposition~\ref{prop:Stau-transfer}.
\end{proof}

\subsubsection{The case that $n=2m$ is even}

\begin{proposition}[Generators and relations] \label{prop:ex(ii)bis}
For
$$X = \tilG/\tilH = \SO(4m+2)/\U(2m+1) \simeq \SO(4m+1)/\U(2m) = G/H$$
and $K=\SO(4m)$, we have
\begin{enumerate}
  \item $\left\{ \begin{array}{l}
  P_k + \iota(Q_k) = 2^{2k}\,\dd\ell(R_k) \quad\mathrm{for\ all\ }k\in\N_+;\\
  \dd\ell(C_{\tilG}) = 2\,\dd\ell(C_G) - \dd r(C_K);
  \end{array}\right.$
  \item $\left\{ \begin{array}{ccl}
\D_{\tilG}(X) & \!\!\!=\!\!\! & \C[P_1,\dots,P_m];\\
\D_K(F) & \!\!\!=\!\!\! & \C[Q_1,\dots,Q_m];\\
\D_G(X) & \!\!\!=\!\!\! & \C[P_1,\dots,P_m,\iota(Q_1),\dots,\iota(Q_m)]\\
& & \hspace{0.3cm} =\,\C[P_1,\dots,P_m,\dd r(R_1),\dots,\dd r(R_m)]\\
& & \hspace{0.3cm} =\,\C[\iota(Q_1),\dots,\iota(Q_m),\dd r(R_1),\dots,\dd r(R_m)].
\end{array}\right.$
\end{enumerate}
\end{proposition}

In all the equalities of Proposition~\ref{prop:ex(ii)bis}.(2), the right-hand side denotes the polynomial ring generated by algebraically independent elements.

We identify
\begin{align*}
\Hom_{\C\text{-}\mathrm{alg}}(Z(\g_{\C}),\C) & \ \simeq\ \jj_{\C}^*/W(\g_{\C}) && \hspace{-1.7cm} \simeq\ \C^{2m}/W(B_{2m}),\\
\Hom_{\C\text{-}\mathrm{alg}}(\D_{\tilG}(X),\C) & \ \simeq\ \hspace{0.3cm} \tila_{\C}^*/\widetilde{W} && \hspace{-1.7cm} \simeq \hspace{0.2cm} \C^m/W(BC_m)
\end{align*}
by the standard bases.
The set $\Disc(K/H)$ consists of the representations of $K=\SO(4m)$ of the form $\tau=\Rep(\SO(4m),t_{m,m}(k,k))$ for $k=(k_1,\dots,k_m)\in (\N^m)_{\geq}$.
We set
$$b(k) := \Big(k_i + 2(m-i) + \frac{1}{2}\Big)_{1\leq i\leq m} \in \C^m.$$

\begin{proposition}[Transfer map] \label{prop:nu-ex(ii)bis}
Let
$$X = \tilG/\tilH = \SO(4m+2)/\U(2m+1) \simeq \SO(4m+1)/\U(2m) = G/H$$
and $K=\SO(4m)$.
For $\tau=\Rep(\SO(4m),t_{m,m}(k,k))\in\Disc(K/H)$ with $k=(k_1,\dots,k_m)\in (\N^m)_{\geq}$, the affine map
\begin{eqnarray*}
S_{\tau} :\ \tila_{\C}^* \simeq \C^m & \longrightarrow & \hspace{1cm} \C^{2m} \hspace{0.9cm} \simeq \jj_{\C}^*\\
\lambda \hspace{0.3cm} & \longmapsto & t_{m,m-1}\Big(\frac{\lambda}{2},b(k)\Big)
\end{eqnarray*}
induces a transfer map
$$\nnu(\cdot,\tau) : \Hom_{\C\text{-}\mathrm{alg}}(\D_{\tilG}(X),\C) \longrightarrow \Hom_{\C\text{-}\mathrm{alg}}(Z(\g_{\C}),\C)$$
as in Theorem~\ref{thm:nu-tau}.
\end{proposition}

For the proof of Propositions \ref{prop:ex(ii)bis} and~\ref{prop:nu-ex(ii)bis}, we use the following lemma on finite-dimensional representations, again with the notation \eqref{eqn:t-m'-m''}.

\begin{lemma}\label{lem:ex(ii)bis}
\begin{enumerate}
  \item Discrete series for $\tilG/\tilH$, $G/H$, and $F=K/H$:
  \begin{align*}
    & \Disc(\SO(4m+2)/\U(2m+1))\\
    & \hspace{1cm} = \big\{ \Rep\big(\SO(4m+2),t_{m+1,m}((j,0),j)\big) \,:\, j\in (\N^m)_{\geq}\big\} ;\\
    & \Disc(\SO(4m+1)/\U(2m)) = \{ \Rep(\SO(4m+1),\omega) \,:\, \omega\in (\N^{2m})_{\geq}\} ;\\
    & \Disc(\SO(4m)/\U(2m)) = \{ \Rep(\SO(4m),t_{m,m}(k,k)) \,:\, k\in (\N^m)_{\geq}\} .
  \end{align*}
  \item Branching laws for $\SO(4m+2)\downarrow\SO(4m+1)$: For $j\in\nolinebreak (\N^m)_{\geq}$,
  \begin{align*}
  & \Rep\big(\SO(4m+2),t_{m+1,m}((j,0),j)\big)|_{\SO(4m+1)}\\
  & \simeq \bigoplus_{\substack{k\in(\N^m)_{\geq}\\ t_{m,m}(j,k)\in(\N^{2m})_{\geq}}} \Rep\big(\SO(4m+1),t_{m,m}(j,k)\big).
  \end{align*}
  \item Irreducible decomposition of the regular representation of~$G$:\\ For $k\in\nolinebreak (\N^m)_{\geq}$,
  \begin{align*}
  & L^2\Big(\SO(4m+1)/\SO(4m),\Rep\big(\SO(4m),t_{m,m}(k,k)\big)\Big)\\
  & \simeq\ \ \sumplus{\substack{j\in(\N^m)_{\geq}\\ t_{m,m}(j,k)\in(\N^{2m})_{\geq}}}\, \Rep\big(\SO(4m+1),t_{m,m}(j,k)\big).
  \end{align*}
  \item The ring $S(\g_{\C}/\h_{\C})^H=S(\so(4m+1,\C)/\gl(2m,\C))^{\U(2m)}$ is generated\linebreak by algebraically independent homogeneous elements of respective degrees $2,2,4,4,\dots,m,m$.
\end{enumerate}
\end{lemma}

The proof of Lemma~\ref{lem:ex(ii)bis} is similar to that of Lemma~\ref{lem:ex(ii)}, and the proof of Proposition~\ref{prop:nu-ex(ii)bis} to that of Proposition~\ref{prop:nu-ex(ii)}, using the formula \eqref{eqn:pitau-iib}.
We omit these proofs.
The proof of Proposition~\ref{prop:ex(ii)bis} is also similar to that of Proposition~\ref{prop:ex(ii)}; we now briefly indicate some minor changes.

\begin{proof}[Proof of Proposition~\ref{prop:ex(ii)bis}]
(1) We first prove that $P_k+\iota(Q_k)=2^{2k}\,\dd\ell(R_k)$.
Let $\vartheta=\Rep(\SO(4m+1),t_{m,m}(j,k))$, where $j,k\in\N^m$ satisfy $j_1\geq k_1\geq\dots\geq j_m\geq k_m$.
By Lemma~\ref{lem:ex(ii)bis}, the map $\vartheta\mapsto (\pi(\vartheta),\tau(\vartheta))$ of Proposition~\ref{prop:pi-tau-theta} is given by
\begin{equation} \label{eqn:pitau-iib}
\left\{ \begin{array}{ccl}
\pi(\vartheta) & = & \Rep(\SO(4m+2),t_{m+1,m}((j,0),j)),\\
\tau(\vartheta) & = & \Rep(\SO(4m),t_{m,m}(k,k)).
\end{array} \right.
\end{equation}
We conclude as in the proof of Proposition~\ref{prop:ex(ii)}.(1), with $a_i=\linebreak j_i+2(m-i)+3/2$ and $b_i=k_i+2(m-i)+1/2$ for $1\leq i\leq m$.

We now check the relation among Casimir operators.
For any $\ell\in\N_+$ and $j\in\Z^{\ell}$, we set
\begin{align*}
h_{\ell}(j) & = \sum_{i=1}^{\ell} j_i^2 + \sum_{i=1}^{\ell} (4\ell-4i+1)\,j_i\\
\text{and}\quad\quad h'_{\ell}(j) & = \sum_{i=1}^{\ell} j_i^2 + \sum_{i=1}^{\ell} (4\ell-4i+3)\,j_i.
\end{align*}
The Casimir operators for $\tilG$, $G$, and~$K$ act on the following irreducible representations as the following scalars.
\begin{center}
\begin{tabular}{|c|c|c|}
\hline
Operator & Representation & Scalar\\
\hline\hline
$C_{\tilG}$ & $\Rep(\SO(4m+2),t_{m+1,m}((j,0),j))$ & $2\,h'_m(j)$\\
\hline
$C_G$ & $\Rep(\SO(4m+1),t_{m,m}(j,k))$ & $h'_m(j)+h_m(k)$\\
\hline
$C_K$ & $\Rep(\SO(4m),t_{m,m}(k,k))$ & $2\,h_m(k)$\\
\hline
\end{tabular}
\end{center}
\noindent
This implies $\dd\ell(C_{\tilG})=2\,\dd\ell(C_G)-\dd r(C_K)$.

(2) This is similar to the proof of Proposition~\ref{prop:ex(ii)}.(2).
\end{proof}

\subsection{The case $(\tilG,\tilH,G)=(\SU(2n+2),\U(2n+1),\Sp(n+1))$}\label{subsec:ex(iii)}

Here $H=\Sp(n)\times\U(1)$, and the only maximal connected proper subgroup of $G$ containing~$H$ is $K=\Sp(n)\times\Sp(1)$; we have $F=K/H=\mathbb{S}^2$.

\begin{proposition}[Generators and relations] \label{prop:ex(iii)}
For
$$X = \tilG/\tilH = \SU(2n+2)/\U(2n+1) \simeq \Sp(n+1)/(\Sp(n)\times\U(1)) = G/H$$
and $K=\Sp(n)\times\Sp(1)$, we have
\begin{enumerate}
  \item $2\,\dd\ell(C_{\tilG}) = 2\,\dd\ell(C_G) - \dd r(C_K)$;
  \item $\left\{ \begin{array}{ccl}
  \D_{\tilG}(X) & \!\!\!\!=\!\!\!\! & \C[\dd\ell(C_{\tilG})];\\
  \D_K(F) & \!\!\!\!=\!\!\!\! & \C[\dd r(C_K)];\\
  \D_G(X) & \!\!\!\!=\!\!\!\! & \C[\dd\ell(C_{\tilG}),\dd r(C_K)] = \C[\dd\ell(C_{\tilG}),\dd\ell(C_G)] = \C[\dd\ell(C_G),\dd r(C_K)].
  \end{array}\right.$
\end{enumerate}
\end{proposition}

We identify
\begin{align}
\Hom_{\C\text{-}\mathrm{alg}}(Z(\g_{\C}),\C) & \ \simeq\ \jj_{\C}^*/W(\g_{\C}) && \hspace{-1cm} \simeq\ \C^{n+1}/W(C_{n+1}),\label{eqn:HCZg-iii}\\
\Hom_{\C\text{-}\mathrm{alg}}(\D_{\tilG}(X),\C) & \ \simeq\ \hspace{0.3cm} \tila_{\C}^*/\widetilde{W} && \hspace{-1cm} \simeq\ \C/(\Z/2\Z)\label{eqn:HCX-iii}
\end{align}
by the standard bases.
The set $\Disc(K/H)$ consists of the representations of $K=\Sp(n)\times\Sp(1)$ of the form $\tau=\mathbf{1}\boxtimes\C^{2a+1}$ for $a\in\N$.
Here, for $b\in\N_+$ we denote by $\C^b$ the (unique) irreducible $b$-dimensional representation of $\Sp(1)\simeq\SU(2)$.

\begin{proposition}[Transfer map] \label{prop:nu-ex(iii)}
Let
$$X = \tilG/\tilH = \SU(2n+2)/\U(2n+1) \simeq \Sp(n+1)/(\Sp(n)\times\U(1)) = G/H$$
and $K=\Sp(n)\times\Sp(1)$.
For $\tau=\mathbf{1}\boxtimes\C^{2a+1}\in\Disc(K/H)$ with $a\in\N$, the affine map
\begin{eqnarray*}
S_{\tau} :\ \tila_{\C}^* \simeq \C & \longrightarrow & \hspace{3.3cm} \C^{n+1} \hspace{3.2cm} \simeq \jj_{\C}^*\\
\lambda & \longmapsto & \Big(\frac{\lambda}{2}+a+\frac{1}{2}, \frac{\lambda}{2}-\Big(a+\frac{1}{2}\Big), n-1, n-2, \dots, 1\Big)
\end{eqnarray*}
induces a transfer map
$$\nnu(\cdot,\tau) : \Hom_{\C\text{-}\mathrm{alg}}(\D_{\tilG}(X),\C) \longrightarrow \Hom_{\C\text{-}\mathrm{alg}}(Z(\g_{\C}),\C)$$
as in Theorem~\ref{thm:nu-tau}.
\end{proposition}

In order to prove Proposition~\ref{prop:ex(iii)}, we use the following results on finite-dimensional representations.

\begin{lemma}\label{lem:ex(iii)}
\begin{enumerate}
  \item Discrete series for $\tilG/\tilH$, $G/H$, and $F=K/H$:
  \begin{eqnarray*}
    \Disc(\SU(2n+2)/\U(2n+1)) & \!\!=\!\! & \{ \mathcal{H}^{j,j}(\C^{2n+2}) \,:\, j\in\N\} ;\\
    \Disc(\Sp(n+1)/\Sp(n)\times\U(1)) & \!\!=\!\! & \{ \mathcal{H}^{k,\ell}(\HH^{n+1}) \,:\, k,\ell\in\N,\ k-\ell\in 2\N\} ;\\
    \Disc(\Sp(n)\times\Sp(1)/\Sp(n)\times\U(1)) & \!\!=\!\! & \{ \mathbf{1} \boxtimes \C^{2a+1} \,:\, a\in\N\} .
  \end{eqnarray*}
  \item Branching laws for $\SU(2n+2)\downarrow\Sp(n+1)$: For $j\in\N$,
  $$\mathcal{H}^{j,j}(\C^{2n+2})|_{\Sp(n+1)} \ \simeq\ \bigoplus_{k=j}^{2j} \mathcal{H}^{k,2j-k}(\HH^{n+1}).$$
  \item Irreducible decomposition of the regular representation of~$G$:~For~$a\!\in\nolinebreak\!\N$,
  $$L^2\big(\Sp(n+1)/\Sp(n)\times\Sp(1),\mathbf{1}\boxtimes\C^{2a+1}\big) \ \simeq\ \sumplus{\substack{j\in\N\\ j\geq a}}\ \mathcal{H}^{j+a,j-a}(\HH^{n+1}).$$
  \item The ring $S(\g_{\C}/\h_{\C})^H=S(\ssp(n+1,\C)/(\C\oplus\ssp(n,\C))^{\Sp(n)\times\U(1)}$ is generated by two algebraically independent homogeneous elements of degree~$2$.
\end{enumerate}
\end{lemma}

Here, for $k\geq\ell\geq 0$ we denote by $\mathcal{H}^{k,\ell}(\HH^{n+1})$ the irreducible finite-dimensional representation of $\Sp(n+1)$ with highest weight $(k,\ell,0,\dots,0)$.

\begin{proof}[Proof of Lemma~\ref{lem:ex(iii)}]
Since $\tilG/\tilH$ and $K/H$ are symmetric spaces, the first and third formulas of~(1) follow from the Cartan--Helgason theorem (Fact~\ref{fact:CartanHelgason}), see also \cite{ht93} for the spherical harmonics on~$\C^N$.
For the second formula of~(1), see \cite{kra79}.
The branching law in~(2) and the decomposition in~(3) are classical.
They can be derived from \cite[Prop.\,5.1]{ht93} and the Frobenius reciprocity; they are also a special case of the general results of \cite{kob94} on the branching laws of unitary representations.

We now prove~(4).
Recall that $\C_m$, for $m\in\Z$, denotes the one-dimensional representation $\U(1)\rightarrow\GL(1,\C)$ given by $z\mapsto z^m$.
For $m\neq 0$, the graded ring $S(\C_m\oplus\C_{-m})^{\U(1)}$ is generated by one homogeneous element of degree~$2$.
As $H$-modules, we have
$$\g_{\C}/\h_{\C} \simeq (\mathbf{1}\boxtimes\C_2) \oplus (\mathbf{1}\boxtimes\C_{-2}) \oplus (\C^{2n}\boxtimes\C_1) \oplus (\C^{2n}\boxtimes\C_{-1}),$$
hence the isomorphism of $\U(1)$-modules
$$S(\g_{\C}/\h_{\C})^{\Sp(n)\times\{ 1\}} \simeq S(\C_2\oplus\C_{-2}) \otimes S(\C^{2n}\oplus\C^{2n})^{\Sp(n)\times\{ 1\}}.$$
Since $S(\C^{2n})$ decomposes into a \emph{multiplicity-free} sum
$$S(\C^{2n}) \simeq \bigoplus_{j\in\N} S^j(\C^{2n})$$
of \emph{self-dual} irreducible representations of $\Sp(n)$, we see that the dimension of $S^N(\C^{2n}\oplus\C^{2n})^{\Sp(n)}$ is $0$ if $N$ is odd and $1$ if $N$ is even.
In particular, $\Sp(1)$ acts trivially on $S(\C^2\boxtimes\C^{2n})^{\Sp(n)\times\{ 1\}}=S(\C^{2n}\oplus\C^{2n})^{\Sp(n)}$, and so does its subgroup $\U(1)$.
Therefore,
$$S(\g_{\C}/\h_{\C})^{\Sp(n)\times\U(1)} \simeq S(\C_2\oplus\C_{-2})^{\U(1)} \otimes S(\C^{2n}\oplus\C^{2n})^{\Sp(n)}.$$
We conclude using the fact that both factors in the tensor product are polynomial rings generated by a single homogeneous element of degree~$2$.
\end{proof}

\begin{proof}[Proof of Proposition~\ref{prop:ex(iii)}]
(1) By Lemma~\ref{lem:ex(iii)}, the map $\vartheta\mapsto (\pi(\vartheta),\tau(\vartheta))$ of Proposition~\ref{prop:pi-tau-theta} is given by
\begin{equation} \label{eqn:pitau-iii}
\mathcal{H}^{k,\ell}(\HH^{n+1}) \longmapsto \big(\mathcal{H}^{\frac{k+\ell}{2},\frac{k+\ell}{2}}(\C^{2n+2}), \mathbf{1} \boxtimes \C^{k-\ell+1}\big).
\end{equation}
The Casimir operators for $\tilG$, $G$, and~$K$ act on these representations as the following scalars.
\begin{center}
\begin{tabular}{|c|c|c|}
\hline
Operator & Representation & Scalar\\
\hline\hline
$C_{\tilG}$ & $\mathcal{H}^{\frac{k+\ell}{2},\frac{k+\ell}{2}}(\C^{2n+2})$ & $\frac{1}{2} (k+\ell) (k+\ell+4n+2)$\\
\hline
$C_G$ & $\mathcal{H}^{k,\ell}(\HH^{n+1})$ & $k^2+\ell^2+2(k+\ell)n+2k$\\
\hline
$C_K$ & $\mathbf{1}\boxtimes\C^{k-\ell+1}$ & $(k-\ell)(k-\ell+2)$\\
\hline
\end{tabular}
\end{center}
This, together with the identity
$$(k+\ell) (k+\ell+4n+2) = 2 \big(k^2+\ell^2+2(k+\ell)n+2k\big) - (k-\ell)(k-\ell+2),$$
implies $2\,\dd\ell(C_{\tilG}) = 2\,\dd\ell(C_G) - \dd r(C_K)$.

(2) Since $\tilG/\tilH$ and $F=K/H$ are symmetric spaces, we obtain $\D_{\tilG}(X)$ and $\D_K(F)$ using the Harish-Chandra isomorphism.
We now focus on $\D_G(X)$.
We only need to prove the first equality, since the other ones follow from the relations between the generators.
For this, using Lemmas \ref{lem:struct-DGX}.(3) and \ref{lem:ex(iii)}.(4), it suffices to show that the two differential operators $\dd\ell(C_{\tilG})$ and $\dd r(C_K)$ on~$X$ are algebraically independent.
Let $f$ be a polynomial in two variables such that $f(\dd\ell(C_{\tilG}),\dd r(C_K))=0$ in $\D_G(X)$.
By letting this differential operator act on the $G$-isotypic component $\vartheta = \mathcal{H}^{k,\ell}(\HH^{n+1})$ in $C^{\infty}(X)$, we obtain
$$f\Big( \frac{1}{2} (k+\ell) (k+\ell+4n+2), (k-\ell)(k-\ell+2) \Big) = 0$$
for all $k,\ell\in\N$ with $k-\ell\in 2\N$, hence $f$ is the zero polynomial.
\end{proof}

\begin{proof}[Proof of Proposition~\ref{prop:nu-ex(iii)}]
We use Proposition~\ref{prop:Stau-transfer} and the formula \eqref{eqn:pitau-iii} for the map $\vartheta\mapsto (\pi(\vartheta),\tau(\vartheta))$ of Proposition~\ref{prop:pi-tau-theta}.
Let $\tau=\mathbf{1}\boxtimes\C^{2a+1}\in\Disc(K/H)$ with $a\in\N$.
If $\vartheta\in\Disc(G/H)$ satisfies $\tau(\vartheta)=\tau$, then $\vartheta$ is of the form
$$\vartheta = \mathcal{H}^{j+a,j-a}(\HH^{n+1}) \simeq \Rep\big(\Sp(n+1),(j+a,j-a,0,\dots,0)\big)$$
for some $j\in\N$ with $j\geq a$, by \eqref{eqn:pitau-iii}.
The algebra $\D_{\tilG}(X)$ acts on the irreducible $\tilG$-submodule
$$\pi(\vartheta) = \mathcal{H}^{j,j}(\C^{2n+2}) \simeq \Rep\big(\U(2n+2),(j,0,\dots,0,-j)\big)$$
by the scalars
$$\lambda(\vartheta) + \rho_{\tila} = 2(j + n + 1/2) \in \C/(\Z/2\Z)$$
via the Harish-Chandra isomorphism \eqref{eqn:HCX-iii}, whereas the algebra $Z(\g_{\C})$ acts on the irreducible $G$-module $\vartheta = \mathcal{H}^{j+a,j-a}(\HH^{n+1})$ by the scalars
$$\nu(\vartheta) + \rho = (n+1+j+a,n+j-a,n-1,\dots,1) \in \C^{n+1}/W(C_{n+1})$$
via \eqref{eqn:HCZg-iii}.
Thus the affine map $S_{\tau}$ in Proposition~\ref{prop:nu-ex(iii)} sends $\lambda(\vartheta)+\rho_{\tila}$ to $\nu(\vartheta)+\rho$ for any $\vartheta\in\Disc(G/H)$ such that $\tau(\vartheta)=\tau$, and we conclude using Proposition~\ref{prop:Stau-transfer}.
\end{proof}

\subsection{The case $(\tilG,\tilH,G)=(\SU(2n+2),\Sp(n+1),\U(2n+1))$}\label{subsec:ex(iv)}

To simplify the computations, we use a central extension of $\tilG$ by $\U(1)$ and work with
$$(\tilG,\tilH,G) = \big(\U(2n+2), \Sp(n+1), \U(2n+1)\big).$$
This increases the dimension of $X=\tilG/\tilH$ by one.
By using block matrices of sizes $2n+1$ and~$1$, the group $G=\U(2n+1)$ embeds into $\SU(2n+2)$ as $g\mapsto (g,\det g^{-1})$ and into $\U(2n+2)$ as $g\mapsto (g,1)$.
Consequently, $H=\tilH\cap G$ is isomorphic to $\Sp(n)\times\U(1)$ in the original setting, and to $\Sp(n)$ in the present setting where $\tilG=\U(2n+2)$.
The only maximal connected proper subgroup of $G$ containing~$H$ is $K=\U(2n)\times\U(1)$.
The space $X_{\C}=G_{\C}/H_{\C}=\GL(2n+1,\C)/\Sp(n,\C)$ is $G_{\C}$-spherical but is not a symmetric space.

Since $X=\tilG/\tilH$ and $F=K/H$ are classical symmetric spaces, the structure of the rings $\D_{\tilG}(X)$ and $\D_K(F)$ is well understood by the Harish-Chandra isomorphism \eqref{eqn:HC-isom}.
Further, $\D_{\tilG}(X)=\dd\ell(Z(\tilg_{\C}))$ and $\D_K(F)=\dd r_F(Z(\kk_{\C}))$ by Fact~\ref{fact:ZgDGX}.
On the other hand, $G/H\simeq X$ is not a symmetric space.
We now give explicit generators of the ring $\D_G(X)$ by using the fibration $X\overset{F}{\longrightarrow} G/K$.
We refer to Section~\ref{subsec:symmsp} for the notation $\chi_{\lambda}^X$, $\chi_{\mu}^F$, and~$\chi_{\nu}^G$ for the Harish-Chandra isomorphisms.

The rank of the symmetric space $X=\tilG/\tilH$ is $n+1$, and the restricted root system $\Sigma(\tilg_{\C},\tila_{\C})$ is of type~$A_n$.
We take the standard basis $\{ h_1,\dots,h_{n+1}\}$ of~$\tila_{\C}^{\ast}$ and choose a positive system such that
$$\Sigma^+(\tilg_{\C},\tila_{\C}) = \{ h_j - h_k : 1\leq j<k\leq n+1\} .$$
Then the Harish-Chandra isomorphism \eqref{eqn:Psi-dual} amounts to
\begin{equation} \label{eqn:HCZg-iv}
\Hom_{\C\text{-}\mathrm{alg}}(\D_{\tilG}(X),\C)\ \simeq\ \tila_{\C}^*/\widetilde{W}\ \simeq\ \C^{n+1}/\mathfrak{S}_{n+1}.
\end{equation}
Using these coordinates, for $k\in\N_+$ we define $P_k\in\D_{\tilG}(X)$ by
$$\chi_{\lambda}^X(P_k) = \sum_{j=1}^{n+1} \lambda_j^k$$
for $\lambda=(\lambda_1,\dots,\lambda_{n+1})\in\tila_{\C}^{\ast}/W(A_n)\simeq\C^{n+1}/\mathfrak{S}_{n+1}$.
Then $\D_{\tilG}(X)$ is a polynomial algebra generated by $P_1,\dots,P_{n+1}$.

We observe that the fiber $F=K/H$ is isomorphic to the direct product $\U(2n)/\Sp(n) \times \U(1)$; the first component is the same as $\tilG/\tilH$ with $n+1$ replaced by~$n$.
Thus the restricted root system $\Sigma(\kk_{\C},(\aaa_F)_{\C})$ is of type $A_{n-1}$, and we define similarly $Q\in\D_K(F)$ and $Q_k\in\D_K(F)$ for $k\in\N_+$ by
$$\chi_{\mu}^F(Q) = \mu_0 \quad\quad\mathrm{and}\quad\quad \chi_{\mu}^F(Q_k) = \sum_{j=1}^n \mu_j^k$$
for $\mu=(\mu_1,\dots,\mu_n,\mu_0)\in(\aaa_F)_{\C}^{\ast}/(W(A_{n-1})\times\{ 1\})$.
Then
\begin{eqnarray*}
\D_K(F) & \simeq & \D_{\U(2n)}(\U(2n)/\Sp(1)) \otimes \D_{\U(1)}(\U(1))\\
& \simeq & \C[Q_1,\dots,Q_n] \otimes \C[Q].
\end{eqnarray*}

Finally, take a Cartan subalgebra $\jj_{\C}$ of $\g_{\C}=\gl(2n+1,\C)$ and the standard basis $\{ f_1,\dots,f_{2n+1}\}$ of~$\jj_{\C}^{\ast}$ such that the root system $\Delta(\g_{\C},\jj_{\C})$ is given as
$$\{ \pm (f_j - f_k) : 1\leq j<k\leq 2n+1\} .$$
The Harish-Chandra isomorphism \eqref{eqn:HCchi} amounts to
\begin{equation} \label{eqn:HCX-iv}
\Hom_{\C\text{-}\mathrm{alg}}(Z(\g_{\C}),\C)\ \simeq\ \jj_{\C}^*/W(\g_{\C})\ \simeq\ \C^{2n+1}/\mathfrak{S}_{2n+1}.
\end{equation}
For $k\in\N_+$, we define $R_k\in Z(\g_{\C})$ by
$$\chi_{\nu}^G(R_k) = \sum_{j=1}^{2n+1} \nu_j^k$$
for $\nu=(\nu_1,\dots,\nu_{2n+1})\in\jj_{\C}^{\ast}/W(A_{2n})\simeq\C^{2n+1}/\mathfrak{S}_{2n+1}$.
Note that $P_k$, $\iota(Q_k)$, and $\dd r(R_k)$ are all differential operator of order $2k$ on~$X$.

The group $K=\U(2n)\times\nolinebreak\U(1)$ is not simple; for $i\in\{ 1,2\}$, we denote by $C_K^{(i)}\in Z(\kk_{\C})$ the Casimir element of the $i$-th factor of~$K$.
With this notation, here is our description of $\D_G(X)$.

\begin{proposition}[Generators and relations] \label{prop:ex(iv)}
For
$$X = \tilG/\tilH = \U(2n+2)/\Sp(n+1) \simeq \U(2n+1)/\Sp(n) = G/H$$
and $K=\U(2n)\times\U(1)$, we have
\begin{enumerate}
  \item $\left\{ \begin{array}{l}
  P_k + \iota(Q_k) = 2^k\,\dd\ell(R_k) \quad\mathrm{for\ all\ }k\in\N_+;\\
  \dd\ell(C_{\tilG}) = 2\,\dd\ell(C_G) - \dd r(C_K^{(1)});
  \end{array}\right.$
  \item $\left\{ \begin{array}{ccl}
\D_{\tilG}(X) & \!\!\!=\!\!\! & \C[P_1,\dots,P_{n+1}];\\
\D_K(F) & \!\!\!=\!\!\! & \C[Q,Q_1,\dots,Q_n];\\
\D_G(X) & \!\!\!=\!\!\! & \C[P_1,\dots,P_{n+1},\iota(Q_1),\dots,\iota(Q_n)]\\
& & \hspace{0.3cm} =\, \C[\iota(Q_1),\dots,\iota(Q_n),\dd\ell(R_1),\dots,\dd\ell(R_{n+1})]\\
& & \hspace{0.3cm} =\, \C[P_1,\dots,P_{n+1},\dd\ell(R_1),\dots,\dd\ell(R_n)]\\
& & \hspace{0.3cm} =\, \C[P_1,\dots,P_n,\dd\ell(R_1),\dots,\dd\ell(R_{n+1})].
\end{array}\right.$
\end{enumerate}
\end{proposition}

In all the equalities of Proposition~\ref{prop:ex(iv)}.(2), the right-hand side denotes the polynomial ring generated by algebraically independent elements.

The set $\Disc(K/H)$ consists of the representations of $K=\U(2n)\times\U(1)$ of the form $\tau=\Rep(\U(2n),t_n(k,k))\boxtimes\Rep(\U(1),a)$ for $k\in (\Z^n)_{\geq}$ and $a\in\Z$ (see Lemma~\ref{lem:ex(iv)}.(1) below).
For $k=(k_1,\dots,k_n)\in\Z^n$, we set
$$b(k) := \big(k_i + n - 2i + 1\big)_{1\leq i\leq n} \in \C^n.$$

\begin{proposition}[Transfer map] \label{prop:nu-ex(iv)}
Let
$$X = \tilG/\tilH = \U(2n+2)/\Sp(n+1) \simeq \U(2n+1)/\Sp(n) = G/H$$
and $K=\U(2n)\times\U(1)$.
For $\tau=\Rep(\U(2n),t_n(k,k))\boxtimes\Rep(\U(1),a)\in\Disc(K/H)$ with $k\in (\Z^n)_{\geq}$ and $a\in\Z$, the affine map
\begin{eqnarray*}
S_{\tau} :\ \tila_{\C}^* \simeq \C^n & \longrightarrow & \hspace{0.7cm} \C^{2n+1} \hspace{0.6cm} \simeq \jj_{\C}^*\\
\lambda \hspace{0.15cm} & \longmapsto & t_{n+1,n}\Big(\frac{\lambda}{2},b(k)\Big)
\end{eqnarray*}
induces a transfer map
$$\nnu(\cdot,\tau) : \Hom_{\C\text{-}\mathrm{alg}}(\D_{\tilG}(X),\C) \longrightarrow \Hom_{\C\text{-}\mathrm{alg}}(Z(\g_{\C}),\C)$$
as in Theorem~\ref{thm:nu-tau}.
\end{proposition}

Here we use the notation $t_{m',m''}(j,k)$ from \eqref{eqn:t-m'-m''}.
To avoid confusion, we write $\Rep(\U(1),a)$ for the one-dimensional representation of~$\U(1)$ given by $z\mapsto z^a$, and not $\C_a$ as in Sections \ref{subsec:ex(i)} and~\ref{subsec:ex(iii)}.

Propositions \ref{prop:ex(iv)} and~\ref{prop:nu-ex(iv)} are consequences of the following results on finite-dimensional representations.

\begin{lemma}\label{lem:ex(iv)}
\begin{enumerate}
  \item Discrete series for $\tilG/\tilH$, $G/H$, and $F=K/H$:
  \begin{align*}
    & \Disc(\U(2n+2)/\Sp(n+1))\\
    & \hspace{1cm} = \big\{ \Rep\big(\U(2n+2),t_{n+1,n+1}(j,j)\big) \,:\, j\in (\Z^{n+1})_{\geq}\big\} ;\\
    & \Disc\big(\U(2n+1)/\Sp(n)\big) = \{ \Rep(\U(2n+1),\omega) \,:\, \omega\in (\Z^{2n+1})_{\geq}\} ;\\
    & \Disc\big(\U(2n)\times\U(1)/\Sp(n)\big)\\
    & \hspace{1cm} = \big\{ \Rep\big(\U(2n),t_{n,n}(k,k)\big) \boxtimes \Rep(\U(1),a) \,:\, k\in (\Z^n)_{\geq},\,a\in\Z\big\} .
  \end{align*}
  \item Branching laws for $\U(2n+2)\downarrow\U(2n+1)$: For $j\in\nolinebreak (\Z^{n+1})_{\geq}$,
  \begin{align*}
  & \Rep\big(\U(2n+2),t_{n+1,n+1}(j,j)\big)|_{\U(2n+1)}\\
  & \simeq \bigoplus_{\substack{k\in\Z^n\\ t_{n+1,n}(j,k)\in (\Z^{2n+1})_{\geq}}} \Rep\big(\U(2n+1),t_{n+1,n}(j,k)\big).
  \end{align*}
  \item Irreducible decomposition of the regular representation of~$G$: For $a\in\nolinebreak\Z$ and $k\in (\Z^n)_{\geq}$,
  \begin{align*}
  \hspace{1.1cm} & L^2\big(\U(2n+1)/(\U(2n)\times\U(1)),\Rep(\U(2n),t_{n,n}(k,k)) \boxtimes\,\Rep(\U(1),a)\big)\\
  \hspace{1.1cm} & \hspace{1.1cm} \simeq\ \ \sumplus{\substack{j\in\Z^{n+1}\\ t_{n+1,n}(j,k)\in (\Z^{2n+1})_{\geq}\\ \sum_{i=1}^{n+1} j_i - \sum_{i=1}^n k_i = a}} \Rep\big(\U(2n+1),t_{n+1,n}(j,k)\big).
  \end{align*}
  \item The ring $S(\g_{\C}/\h_{\C})^H=S(\gl(2n+1,\C)/\ssp(n,\C))^{\Sp(n,\C)}$ is generated by algebraically independent homogeneous elements of respective degrees $1,1,2,2,\dots,n,n,n+\nolinebreak 1$.
\end{enumerate}
\end{lemma}

\begin{proof}[Proof of Lemma~\ref{lem:ex(iv)}]
(1) Since $\tilG/\tilH$ and $K/H$ are symmetric spaces, the descriptions of $\Disc(\tilG/\tilH)$ and $\Disc(K/H)$ follow from the Cartan--Helgason theorem (Fact~\ref{fact:CartanHelgason}).
For the description of $\Disc(G/H)$ for the nonsymmetric spherical homogeneous space $G/H$, see \cite{kra79}; the argument below using branching laws and \eqref{eqn:decomp2} gives an alternative proof.

One immediately deduces (2) and~(3) from the classical branching laws for $\U(\ell+1)\downarrow\U(\ell)\times\U(1)$, see \eg \cite[Th.\,8.1.1]{gw09}.

We now prove~(4).
The quotient module $\gl(2n+1,\C)/\ssp(n,\C)$ is isomorphic to $\C\oplus\ssl(2n+1,\C)/\ssp(n,\C)$, and the second summand splits into a direct sum of four irreducible representations of $\ssp(n,\C)$:
$$2\,\Rep\big(\ssp(n),(1,0,\dots,0)\big) \oplus \Rep\big(\ssp(n),(1,1,0,\dots,0)\big) \oplus \C.$$
We use the observation that this $\ssp(n,\C)$-module is isomorphic to the restriction of the following irreducible $\gl(2n,\C)$-module:
$$\C^{2n} \oplus (\C^{2n})^{\vee} \oplus \Lambda^2(\C^{2n}).$$
By the Cartan--Helgason theorem (Fact~\ref{fact:CartanHelgason}), the highest weights of irreducible representations of $\gl(2n,\C)$ having nonzero $\ssp(n,\C)$-fixed vectors are of the form
\begin{equation}\label{eqn:CH-gl-sp}
(\lambda_1,\lambda_1,\lambda_2,\lambda_2,\dots,\lambda_n,\lambda_n) \quad\text{with}\ (\lambda_1,\dots,\lambda_n)\in (\Z^n)_{\geq}.
\end{equation}
Then, for any $N\in\N$, the dimension of $S^N(\ssl(2n+1,\C)/\ssp(n,\C))^{\Sp(n,\C)}$ coin\-cides with the multiplicity of such irreducible $\gl(2n,\C)$-modules occurring~in
$$S^N\big( \C^{2n} \oplus (\C^{2n})^{\vee} \oplus \Lambda^2(\C^{2n}) \big) \simeq \bigoplus_{\substack{i,j,k\in\N\\ i+j+k=N}} S^i(\C^{2n}) \otimes S^j((\C^{2n})^{\vee}) \otimes S^k(\Lambda^2(\C^{2n})),$$
because $\ssp(n,\C)$-fixed vectors in irreducible $\gl(2n,\C)$-modules are unique up to a multiplicative scalar (Fact~\ref{fact:CartanHelgason}).
The $\gl(2n,\C)$-module $S^k(\Lambda^2(\C^{2n}))$ has a multiplicity-free decomposition
$$\bigoplus_{\substack{b_1\geq\dots\geq b_n\geq 0\\ b_1+\dots+b_n=k}} \Rep\big(\gl(2n,\C),(b_1,b_1,b_2,b_2,\dots,b_n,b_n)\big).$$
By Pieri's law \cite[Cor.\,9.2.4]{gw09}, the tensor product representation $S^i(\C^{2n})\otimes S^k(\Lambda^2(\C^{2n}))$ is decomposed as
$$\bigoplus \Rep\big(\gl(2n,\C),(a_1,b_1,a_2,b_2,\dots,a_n,b_n)\big),$$
where the sum is taken over $(a_1,\dots,a_n),(b_1,\dots,b_n)\in\Z^n$ satisfying
$$\left \{
\begin{array}{l}
  a_1 \geq b_1 \geq a_2 \geq b_2 \geq \dots \geq a_n \geq b_n \geq 0,\\
  b_1 + \dots + b_n = k,\\
  (a_1 - b_1) + (a_2 - b_2) + \dots + (a_n - b_n) = i.
\end{array}
\right.$$
By using Pieri's law again, we see that irreducible representations of $\gl(2n,\C)$ with highest weights of the form \eqref{eqn:CH-gl-sp} occur in
$$S^j((\C^{2n})^{\vee}) \otimes \Rep\big(\gl(2n,\C),(a_1,b_1,a_2,b_2,\dots,a_n,b_n)\big)$$
if and only if
$$\left \{
\begin{array}{l}
  \lambda_{\ell} = b_{\ell} \quad\quad \forall 1\leq\ell\leq n,\\
  (a_1 - \lambda_1) + (a_2 - \lambda_2) + \dots + (a_n - \lambda_n) = j.
\end{array}
\right.$$
Therefore, $\dim S^N(\ssl(2n+1,\C)/\ssp(n,\C))^{\Sp(n,\C)}$ is equal to
\begin{equation}\label{eqn:combin}
\# \Big\{ (a_1,b_1,\dots,a_n,b_n)\in (\N^{2n})_{\geq} \ :\ 2\,\sum_{\ell=1}^n a_{\ell} - \sum_{\ell=1}^n b_{\ell} = N \Big\} .
\end{equation}
For any $1\leq j\leq n$, we set
$$\left \{
\begin{array}{ccl}
  c_{2j-1} & := & a_j - b_j,\\
  c_{2j} & := & b_j - a_{j+1},
\end{array}
\right.$$
with the convention that $a_{2n+1}=0$.
Then \eqref{eqn:combin} amounts to
$$\# \Big\{ (c_1,\dots,c_{2n})\in\N^{2n} \ :\ \sum_{j=1}^n (j+1) \, c_{2j-1} + \sum_{j=1}^n j \, c_{2j} = N \Big\} ,$$
which is the dimension of the space of homogeneous polynomials of degree~$N$ in $\C[x_1,x_2^2,\dots,x_n^n,y_1^2,y_2^3,\dots,y_n^{n+1}]$ for algebraically independent elements $x_1,\dots,x_n,y_1,\dots,y_n$.
We conclude using Lemma~\ref{lem:struct-DGX}.
\end{proof}

\begin{proof}[Proof of Proposition~\ref{prop:ex(iv)}]
(1) We first prove that $P_k + \iota(Q_k) = 2^k\,\dd\ell(R_k)$.
By Lemma~\ref{lem:ex(iv)}, the map $\vartheta\mapsto (\pi(\vartheta),\tau(\vartheta))$ of Proposition~\ref{prop:pi-tau-theta} is given by
\begin{equation} \label{eqn:pitau-iv}
\begin{array}{l}
\Rep(\U(2n+1),t_{n+1,n}(j,k))\\
\longmapsto \Big(\Rep(\U(2n+2),t_{n+1,n+1}(j,j)),\\
\hspace{1.6cm}\Rep(\U(2n),t_{n,n}(k,k)) \boxtimes \Rep\big(\U(1), \sum_{i=1}^{n+1} j_i - \sum_{i=1}^n k_i\big)\Big)\hspace{-0.5cm}
\end{array}
\end{equation}
for $j\in\N^{n+1}$ and $k\in\N^n$ with $j_1\geq k_1\geq\dots\geq j_n\geq k_n\geq j_{n+1}$.
For $1\leq i\leq n+1$ we set
$$a_i := j_i + n - 2i + 2.$$
Then
$$2 \sum_{i=1}^{n+1} j_i h_i + \rho_{\tila} = 2 \sum_{i=1}^{n+1} a_i h_i.$$
Since the embedding $\tila_{\C}^{\ast}\hookrightarrow\widetilde{\jj}_{\C}^{\ast}$ is given by $2j\mapsto t_{n+1,n+1}(j,j)$ via the standard bases of $\tila_{\C}^{\ast}$ and~$\widetilde{\jj}_{\C}^{\ast}$, the map $T : \tila_{\C}^{\ast}/\widetilde{W}\rightarrow\widetilde{\jj}_{\C}^{\ast}/W(\tilg_{\C})$ (see \eqref{eqn:aj}) satisfies
$$T\bigg(2 \sum_{i=1}^{n+1} a_i h_i\bigg) = \Big(a_1+\frac{1}{2}, a_1-\frac{1}{2}, \dots, a_{n+1}+\frac{1}{2}, a_{n+1}-\frac{1}{2}\Big) \in \widetilde{\jj}_{\C}^{\ast}/W(\tilg_{\C}),$$
which is the $Z(\tilg_{\C})$-infinitesimal character $\Psi_{\pi(\vartheta)}$ of $\pi(\vartheta)$, and by Lemma~\ref{lem:chiZgDGX}.(1) the operator $P_k$ acts on the representation space of $\pi(\vartheta)$ as the scalar
\begin{equation} \label{eqn:piP-iv}
\chi_{(2a_1,\dots,2a_{n+1})}^X(P_k) = 2^k \sum_{i=1}^{n+1} a_i^k.
\end{equation}
For $1\leq i\leq n$ we set
\begin{eqnarray*}
b_i & := & k_i + n - 2i + 1,\\
\nu_0 & := & \sum_{i=1}^{n+1} j_i - \sum_{i=1}^n k_i.
\end{eqnarray*}
Then $\tau(\vartheta)^{\vee}$ has $Z(\kk_{\C})$-infinitesimal character
$$\Psi_{\tau(\vartheta)^{\vee}} = - \Big(\nu_0, b_1+\frac{1}{2}, b_1-\frac{1}{2}, \dots, b_n+\frac{1}{2}, b_n-\frac{1}{2}\Big) \in (\jj_{\kk})_{\C}^{\ast}/W(\kk_{\C}),$$
and by Lemma~\ref{lem:chiZgDGX}.(2) the operator $\iota(Q_k)$ acts on the representation space of $\tau(\vartheta)$ as the scalar
$$\chi_{(2b_1,\dots,2b_n,\nu_0)}^F(Q_k) = 2^k \sum_{i=1}^n b_i^k.$$
On the other hand, $\vartheta$ itself has $Z(\g_{\C})$-infinitesimal character
$$\Psi_{\vartheta} = (a_1, b_1, \dots, a_n, b_n, a_{n+1}) \in \jj_{\C}^{\ast}/W(A_{2n}) \simeq \C^{2n+1}/\mathfrak{S}_{2n+1},$$
and so $\dd\ell(R_k)$ acts on the representation space of~$\vartheta$ as the scalar
$$\Psi_{\vartheta}(R_k) = \sum_{i=1}^{n+1} a_i^k + \sum_{i=1}^n b_i^k.$$
Thus $P_k+\iota(Q_k)=2^k\,\dd\ell(R_k)$ by Proposition~\ref{prop:PQR}.

We now check the relation among Casimir operators for $\tilG$, $G$, and~$K$.
These act on the following irreducible representations as the following scalars.
\begin{center}
\begin{tabular}{|c|c|c|}
\hline
Operator & Representation & Scalar\\
\hline\hline
$C_{\tilG}$ & $\Rep\big(\U(2n+2),t_{n+1,n+1}(j,j)\big)$ & $2 \sum_{i=1}^{n+1} (j_i^2 + 2 (n+2-2i) j_i)$\\
\hline
$C_G$ & $\Rep\big(\U(2n+1),t_{n+1,n}(j,k)\big)$ & $\sum_{i=1}^{n+1} (j_i^2 + 2 (n+2-2i) j_i)$\\
& & $+ \sum_{i=1}^n (k_i^2 + 2 (n+1-2i) k_i)$\\
\hline
$C_K^{(1)}$ & $\Rep(\U(2n),t_{n,n}(k,k))$ & $2 \sum_{i=1}^n (k_i^2 + 2 (n+1-2i) k_i)$\\
\hline
$C_K^{(2)}$ & $\Rep\big(\U(1),\sum_{i=1}^{n+1} j_i - \sum_{i=1}^n k_i\big)$ & $\big(\sum_{i=1}^{n+1} j_i - \sum_{i=1}^n k_i\big)^2$\\
\hline
\end{tabular}
\end{center}
This implies $\dd\ell(C_{\tilG})=2\,\dd\ell(C_G)-\dd r(C_K^{(1)})$.

(2) We have already given descriptions of $\D_{\tilG}(X)$ and $\D_K(F)$.
The description of $\D_G(X)$ can be deduced from Proposition~\ref{prop:ex(iv)}.(1) and Lemma~\ref{lem:ex(iv)}.(4), similarly to the proof of Proposition~\ref{prop:ex(ii)}.(2).
\end{proof}

\begin{proof}[Proof of Proposition~\ref{prop:nu-ex(iv)}]
We use Proposition~\ref{prop:Stau-transfer} and the formula \eqref{eqn:pitau-iv} for the map $\vartheta\mapsto (\pi(\vartheta),\tau(\vartheta))$ of Proposition~\ref{prop:pi-tau-theta}.
Let
$$\tau=\Rep(\U(2n),t_n(k,k))\boxtimes\Rep(\U(1),a)\in\Disc(K/H)$$
with $k\in (\Z^n)_{\geq}$ and $a\in\Z$.
If $\vartheta\!\in\!\Disc(G/H)$ satisfies $\tau(\vartheta)\!=\!\tau$, then~$\vartheta$~is~of~the\linebreak form $\vartheta = \Rep(\U(2n+1),t_{n+1,n}(j,k))$ for some $j\in\N^{n+1}$ with $j_1\geq k_1\geq\dots\linebreak\geq j_n\geq k_n\geq j_{n+1}$ and $\sum_{i=1}^{n+1} j_i = a + \sum_{i=1}^n k_i$, by \eqref{eqn:pitau-iv}.
The algebra $\D_{\tilG}(X)$ acts on the irreducible $\tilG$-submodule $\pi(\vartheta) = \Rep(\U(2n+\nolinebreak 2),t_{n+1,n+1}(j,j))$ by the scalars
$$\lambda(\vartheta) + \rho_{\tila} = (2a_1,\dots,2a_{n+1}) \in \C^{n+1}/\mathfrak{S}_{n+1}$$
via the Harish-Chandra isomorphism \eqref{eqn:HCX-iv}, where $a_i = j_i + n - 2i + 2$ ($1\leq i\leq n+1$), whereas the algebra $Z(\g_{\C})$ acts on the irreducible $G$-module $\vartheta = \Rep(\U(2n+1),t_{n+1,n}(j,k))$ by the scalars
$$\nu(\vartheta) + \rho = (a_1,\dots,a_{n+1},b_1,\dots,b_n) \in \C^{2n+1}/\mathfrak{S}_{2n+1}$$
via \eqref{eqn:HCZg-iv}, where $b_i = k_i + n - 2i + 1$ as in the proof of Proposition~\ref{prop:ex(iv)}.(1).
Thus the affine map $S_{\tau}$ in Proposition~\ref{prop:nu-ex(iv)} sends $\lambda(\vartheta)+\rho_{\tila}$ to $\nu(\vartheta)+\rho$ for any $\vartheta\in\Disc(G/H)$ such that $\tau(\vartheta)=\tau$, and we conclude using Proposition~\ref{prop:Stau-transfer}.
\end{proof}

\begin{remark}\label{rem:K-not-max}
(1) One can deduce analogous results in the original setting where $\tilG=\SU(2n+2)$ from the corresponding ones for $\tilG=\U(2n+2)$ which we have just discussed, such as Propositions \ref{prop:ex(iv)} and~\ref{prop:nu-ex(iv)}.
If $\tilG=\SU(2n+2)$, then $\tila_{\C}^*$ is isomorphic to $\C^n$, rather than $\tila_{\C}^*$ for $\tilG=\U(2n+2)$.

(2) Let $K':=\Sp(n)\times\U(1)\times\U(1)$, so that $H\subsetneq K'\subsetneq K\subsetneq G=\SU(2n+2)$.
Then $K'/H\simeq\mathbb{S}^1$, hence the $\C$-algebra $\dd r(Z(\kk'_{\C}))$ is generated by a single vector field (the Euler homogeneity differential operator).
It follows from Proposition~\ref{prop:ex(iv)} that neither condition~($\widetilde{\mathrm{A}}$) nor condition~($\widetilde{\mathrm{B}}$) of Section~\ref{subsec:intro-applic} holds if we replace $K$ with~$K'$.
In particular, Theorem~\ref{thm:main}.(1) and~(2) fail if we replace $K$ with the nonmaximal subgroup~$K'$.
\end{remark}

\subsection{The case $(\tilG,\tilH,G)=(\SO(4n+4),\SO(4n+3),\Sp(n+1)\cdot\Sp(1))$}\label{subsec:ex(v)}

Here $H=\Sp(n)\cdot\Diag(\Sp(1))$, and the only maximal connected proper subgroup of $G$ containing~$H$ is $K = (\Sp(n) \times \Sp(1)) \cdot \Sp(1)$.
The groups $G$ and~$K$ are not simple.
For $i\in\{ 1,2\}$ we denote by $C_G^{(i)}\in\nolinebreak Z(\g_{\C})$ the Casimir element of the $i$-th factor of~$G$, and for $j\in\{ 1,2,3\}$ by $C_K^{(j)}\in\nolinebreak Z(\kk_{\C})$ the Casimir element of the $j$-th factor of~$K$.
Then $\dd r(C_K^{(1)})=0$, and $\dd r(C_K^{(2)})=\dd r(C_K^{(3)})\in\D_G(X)$; we denote this last element by $\dd r(C_K)$.

\begin{proposition}[Generators and relations] \label{prop:ex(v)}
For
\begin{eqnarray*}
X = \tilG/\tilH & = & \SO(4n+4)/\SO(4n+3)\\
& \simeq & (\Sp(n+1)\cdot\Sp(1))/(\Sp(n)\cdot\Diag(\Sp(1))) = G/H
\end{eqnarray*}
and $K = (\Sp(n) \times \Sp(1)) \cdot \Sp(1)$, we have
\begin{enumerate}
  \item $\left \{
\begin{array}{l}
  \dd\ell(C_{\tilG}) = 2\,\dd\ell(C_G^{(1)}) - \dd r(C_K);\\
  \dd\ell(C_G^{(2)}) = \dd r(C_K);
\end{array}
\right.$
  \item $\left \{ \begin{array}{cll}
  \D_{\tilG}(X) & \!\!\!=\!\!\! & \C[\dd\ell(C_{\tilG})];\\
  \D_K(F) & \!\!\!=\!\!\! & \C[\dd r(C_K)];\\
  \D_G(X) & \!\!\!=\!\!\! & \C[\dd\ell(C_{\tilG}),\dd r(C_K)] = \C[\dd\ell(C_{\tilG}),\dd\ell(C_G^{(1)})] = \C[\dd\ell(C_G^{(1)}),\dd r(C_K)].
  \end{array}\right.$
\end{enumerate}
\end{proposition}

\begin{remark}\label{rem:Spnonsph}
Let $G':=\Sp(n+1)\subset G$ and $H':=\Sp(n)\subset H$.
Then $X$ is isomorphic to $G'/H'$ as a $G'$-space, and $K':=\Sp(n)\times\Sp(1)\subset K$ is a maximal connected proper subgroup of $G'$ containing~$H'$.
However, $X_{\C}$ is \emph{not} $G'_{\C}$-spherical, and none of the assertions in Theorem~\ref{thm:main} holds if we replace $(G,H,K)$ with $(G',H',K')$.
In fact, the subalgebra of $\D_{G'}(X)$ generated by $\D_{\tilG}(X)$, $\dd r(Z(\kk'_{\C}))$, and $\dd\ell(Z(\g'_{\C}))$ is contained in the commutative algebra $\D_G(X)$, whereas $\D_{G'}(X)$ is noncommutative by Fact~\ref{fact:spherical}.
\end{remark}

We identify
\begin{align}
\Hom_{\C\text{-}\mathrm{alg}}(Z(\g_{\C}),\C) & \ \simeq\ \jj_{\C}^*/W(\g_{\C}) & & \hspace{-0.25cm} \simeq\, (\C^{n+1}\oplus\C)/W(C_{n+1}\times C_1),\hspace{-0.2cm}\label{eqn:HCZg-v}\\
\Hom_{\C\text{-}\mathrm{alg}}(\D_{\tilG}(X),\C) & \ \simeq \hspace{0.4cm} \tila_{\C}^*/\widetilde{W} & & \hspace{-0.25cm} \simeq\ \,\C/(\Z/2\Z)\hspace{-0.2cm}\label{eqn:HCX-v}
\end{align}
by the standard bases.
The set $\Disc(K/H)$ consists of the representations of $K = (\Sp(n) \times \Sp(1)) \cdot \Sp(1)$ of the form $\tau=\mathbf{1}\boxtimes\C^a\boxtimes\C^a$ for $a\in\N_+$.

\begin{proposition}[Transfer map] \label{prop:nu-ex(v)}
Let
\begin{eqnarray*}
X = \tilG/\tilH & = & \SO(4n+4)/\SO(4n+3)\\
& \simeq & (\Sp(n+1)\cdot\Sp(1))/(\Sp(n)\cdot\Diag(\Sp(1))) = G/H
\end{eqnarray*}
and $K = (\Sp(n) \times \Sp(1)) \cdot \Sp(1)$.
For $\tau=\mathbf{1}\boxtimes\C^a\boxtimes\C^a\in\Disc(K/H)$ with $a\in\N_+$, the affine map
\begin{eqnarray*}
S_{\tau} :\ \tila_{\C}^* \simeq \C & \longrightarrow & \hspace{2.4cm} \C^{n+1}\oplus\C \hspace{2.4cm} \simeq \jj_{\C}^*\\
\lambda & \longmapsto & \Big(\Big(\frac{\lambda+a}{2}, \frac{\lambda-a}{2}, n-1, n-2, \dots, 1\Big), a\Big)
\end{eqnarray*}
induces a transfer map
$$\nnu(\cdot,\tau) : \Hom_{\C\text{-}\mathrm{alg}}(\D_{\tilG}(X),\C) \longrightarrow \Hom_{\C\text{-}\mathrm{alg}}(Z(\g_{\C}),\C)$$
as in Theorem~\ref{thm:nu-tau}.
\end{proposition}

In order to prove Propositions \ref{prop:ex(v)} and~\ref{prop:nu-ex(v)}, we use the following results on finite-dimensional representations.

\begin{lemma}\label{lem:ex(v)}
\begin{enumerate}
  \item Discrete series for $\tilG/\tilH$, $G/H$, and $F=K/H$:
  \begin{align*}
    & \Disc(\SO(4n+4)/\SO(4n+3)) = \{ \mathcal{H}^j(\R^{4n+4}) \,:\, j\in\N\} ;\\
    & \Disc\big((\Sp(n+1)\cdot\Sp(1))/(\Sp(n)\cdot\Diag(\Sp(1)))\big)\\
    & \hspace{1cm} = \{ \mathcal{H}^{k,\ell}(\HH^{n+1}) \boxtimes \C^{k-\ell+1} \,:\, k\geq\ell\geq 0\} ;\\
    & \Disc\big(((\Sp(n)\times\Sp(1))\cdot\Sp(1))/(\Sp(n)\cdot\Diag(\Sp(1)))\big)\\
    & \hspace{1cm} = \{ \mathbf{1} \boxtimes \C^a \boxtimes \C^a \,:\, a\in\N_+\} .
  \end{align*}
  \item Branching laws for $\SO(4n+4)\downarrow\Sp(n+1)\cdot\Sp(1)$: For $j\in\N$,
  $$\mathcal{H}^j(\R^{4n+4})|_{\Sp(n+1)\cdot\Sp(1)}\ \simeq \bigoplus_{\substack{k\in\N\\ \frac{j}{2}\leq k\leq j}} \mathcal{H}^{k,j-k}(\HH^{n+1}) \boxtimes \C^{2k-j+1}.$$
  \item Irreducible decomposition of the regular representation of~$G$: For~$a\in\nolinebreak\N_+$,
  \begin{align*}
  & L^2\big((\Sp(n+1)\cdot\Sp(1)) / ((\Sp(n)\times\Sp(1))\cdot\Sp(1)), \mathbf{1} \boxtimes \C^a \boxtimes \C^a\big)\\
  & \simeq\ \sumplus{\substack{k\in\N\\ k\geq a-1}}\ \mathcal{H}^{k,k+1-a}(\HH^{n+1}) \boxtimes \C^a.
  \end{align*}
  \item The ring $S(\g_{\C}/\h_{\C})^H$ is generated by two algebraically independent homogeneous elements of degree~$2$.
\end{enumerate}
\end{lemma}

Here, for $m>1$ we denote by $\mathcal{H}^{k,k'}(\HH^m)$ the irreducible representation of~$\Sp(m)$ whose highest weight is $(k,k',0,\dots,0)$ in the standard coordinates.
For $a\in\N_+$, we denote by $\C^a$ the unique $a$-dimensional $\Sp(1)$-module.

\begin{proof}[Proof of Lemma~\ref{lem:ex(v)}]
Statements (1), (2), (3) follow from \cite[Prop.\,5.1]{ht93} and the Frobenius reciprocity.
To see~(4), we note that there is an isomorphism of $H$-modules
$$\g_{\C}/\h_{\C} \,\simeq\, \g_{\C}/\kk_{\C} \oplus \kk_{\C}/\h_{\C} \,\simeq\, (\C^{2n}\boxtimes\C^2) \oplus (\C\boxtimes\C^3).$$
The action of  $H = \Sp(n)\cdot\Diag(\Sp(1))$ on $\g_{\C}/\h_{\C}$ thus comes from an action of $\Sp(n)\times\Sp(1)$, and we have $S(\g_{\C}/\h_{\C})^H = S(\g_{\C}/\h_{\C})^{\Sp(n)\times\Sp(1)}$.
Note that
$$S(\g_{\C}/\h_{\C})^{\Sp(n)\times\{ 1\}} \,\simeq\, S(\C^{2n}\oplus\C^{2n})^{\Sp(n)\times\{ 1\}} \otimes S(\C^3)$$
as $\Sp(1)$-modules.
As in the proof of Lemma~\ref{lem:ex(iii)}.(4), the group $\Sp(1)$ acts trivially on $S(\C^{2n}\oplus\C^{2n})^{\Sp(n)\times\{ 1\}}$.
Therefore,
$$S(\g_{\C}/\h_{\C})^{\Sp(n)\times\Sp(1)} \,\simeq\, S(\C^{2n}\oplus\C^{2n})^{\Sp(n)} \otimes S(\C^3)^{\Sp(1)}$$
is a polynomial ring generated by two homogeneous elements of degree~$2$ coming from $S(\C^{2n}\oplus\C^{2n})^{\Sp(n)}$ and $S(\C^3)^{\Sp(1)}$.
\end{proof}

\begin{proof}[Proof of Proposition~\ref{prop:ex(v)}]
(1) By Lemma~\ref{lem:ex(v)}, the map $\vartheta\mapsto (\pi(\vartheta),\tau(\vartheta))$ of Proposition~\ref{prop:pi-tau-theta} is given by
\begin{equation} \label{eqn:pitau-v}
\mathcal{H}^{k,\ell}(\HH^{n+1}) \boxtimes \C^{k-\ell+1} \longmapsto \big(\mathcal{H}^{k+\ell}(\R^{4n+4}), \mathbf{1} \boxtimes \C^{k-\ell+1} \boxtimes \C^{k-\ell+1}\big).
\end{equation}
The Casimir operators for $\tilG$ and for the factors of $G$ and~$K$ act on these representations as the following scalars.
\begin{center}
\begin{tabular}{|c|c|c|}
\hline
Operator & Representation & Scalar\\
\hline\hline
$C_{\tilG}$ & $\mathcal{H}^{k+\ell}(\R^{4n+4})$ & $(k+\ell)(k+\ell+4n+2)$\\
\hline
$C_G^{(1)}$ & $\mathcal{H}^{k,\ell}(\HH^{n+1})\boxtimes\C^{k-\ell+1}$ & $k^2+\ell^2+2(k+\ell)n+2k$\\
\hline
$C_G^{(2)}$ & $\mathcal{H}^{k,\ell}(\HH^{n+1})\boxtimes\C^{k-\ell+1}$ & $(k-\ell)(k-\ell+2)$\\
\hline
$C_K$ & $\mathbf{1}\boxtimes\C^{k-\ell+1}\boxtimes\C^{k-\ell+1}$ & $(k-\ell)(k-\ell+2)$\\
\hline
\end{tabular}
\end{center}
This, together with the identity
$$(k+\ell)(k+\ell+4n+2) = 2 \big(k^2+\ell^2+2(k+\ell)n+2k\big) - (k-\ell)(k-\ell+2),$$
implies $\dd\ell(C_{\tilG}) = 2\,\dd\ell(C_G^{(1)}) - \dd r(C_K)$ and $\dd\ell(C_G^{(2)}) = \dd r(C_K)$.

(2) Since $\tilG/\tilH$ and $F=K/H$ are symmetric spaces, we obtain $\D_{\tilG}(X)$ and $\D_K(F)$ using the Harish-Chandra isomorphism.
We now focus on $\D_G(X)$.
We only need to prove the first equality, since the other ones follow from the relations between the generators.
For this, using Lemmas \ref{lem:struct-DGX}.(3) and \ref{lem:ex(v)}.(4), it suffices to show that the two differential operators $\dd\ell(C_{\tilG})$ and $\dd r(C_K)$ on~$X$ are algebraically independent.
Let $f$ be a polynomial in two variables such that $f(\dd\ell(C_{\tilG}),\dd r(C_K))=0$ in $\D_G(X)$.
By letting this differential operator act on the $G$-isotypic component $\vartheta = \mathcal{H}^{k,\ell}(\HH^{n+1}) \boxtimes \C^{k-\ell+1}$ in $C^{\infty}(X)$, we obtain
$$f\big( (k+\ell)(k+\ell+4n+2), (k-\ell)(k-\ell+2) \big) = 0$$
for all $k,\ell\in\N$ with $k\geq\ell$, hence $f$ is the zero polynomial.
\end{proof}

\begin{proof}[Proof of Proposition~\ref{prop:nu-ex(v)}]
We use Proposition~\ref{prop:Stau-transfer} and the formula \eqref{eqn:pitau-v} for the map $\vartheta\mapsto (\pi(\vartheta),\tau(\vartheta))$ of Proposition~\ref{prop:pi-tau-theta}.
Let $\tau=\mathbf{1}\boxtimes\C^a\boxtimes\C^a\in\Disc(K/H)$ with $a\in\N_+$.
If $\vartheta\in\Disc(G/H)$ satisfies $\tau(\vartheta)=\tau$, then $\vartheta$ is of the form
\begin{eqnarray*}
\vartheta & = & \mathcal{H}^{k,\ell}(\HH^{n+1})\boxtimes\C^a\\
& \simeq & \Rep(\Sp(n+1),(k,\ell,0,\dots,0))\boxtimes\Rep(\U(1),a)
\end{eqnarray*}
for some $k,\ell\in\N$ with $k-\ell+1=a$, by \eqref{eqn:pitau-v}.
The algebra $\D_{\tilG}(X)$ acts on the irreducible $\tilG$-submodule $\pi(\vartheta) = \mathcal{H}^{k+\ell}(\R^{4n+4})$ by the scalars
$$\lambda(\vartheta) + \rho_{\tila} = k+\ell+2n+1 \in \C/(\Z/2\Z)$$
via the Harish-Chandra isomorphism \eqref{eqn:HCX-v}, whereas the algebra $Z(\g_{\C})$ acts on the irreducible $G$-module $\vartheta = \mathcal{H}^{k,\ell}(\HH^{n+1})\boxtimes\C^a\mathcal{H}^{j+a,j-a}(\HH^{n+1})$ by the scalars
$$\nu(\vartheta) + \rho = (k+n+1, \ell+n; n-1,\dots,-1;k-\ell+1) \in (\C^{n+1}\oplus\C)/W(C_{n+1}\times C_1)$$
via \eqref{eqn:HCZg-v}.
Thus the affine map $S_{\tau}$ in Proposition~\ref{prop:nu-ex(v)} sends $\lambda(\vartheta)+\rho_{\tila}$ to $\nu(\vartheta)+\rho$ for any $\vartheta\in\Disc(G/H)$ such that $\tau(\vartheta)=\tau$, and we conclude using Proposition~\ref{prop:Stau-transfer}.
\end{proof}

This completes the proof of Theorems \ref{thm:main}, \ref{thm:main-explicit}, \ref{thm:transfer}, and~\ref{thm:nu-tau}, as well as Corollary~\ref{cor:rel-Lapl}, in the case (v) of Table~\ref{table1}.
For the case (v)$'$ of Table~\ref{table1}, the homogeneous spaces $G/H$ and $K/H$ change as follows:
\begin{eqnarray*}
G/H\,: & \big(\Sp(n+1)\cdot\Sp(1))/(\Sp(n)\cdot\Diag(\Sp(1))\big)\\
& \hspace{0.5cm} \leadsto\ \big(\Sp(n+1)\cdot\U(1))/(\Sp(n)\cdot\Diag(\U(1))\big),\\
K/H\,: & \big((\Sp(n)\times\Sp(1))\cdot\Sp(1))/(\Sp(n)\cdot\Diag(\Sp(1))\big)\\
& \hspace{0.5cm} \leadsto\ \big((\Sp(n)\times\Sp(1))\cdot\U(1))/(\Sp(n)\cdot\Diag(\U(1))\big).
\end{eqnarray*}
The proof works similarly, and so we omit it.

\subsection{The case $(\tilG,\tilH,G)=(\SO(16),\SO(15),\Spin(9))$}\label{subsec:ex(vi)}

Here $H=\Spin(7)$, and the only maximal connected proper subgroup of~$G$ containing~$H$ is~$K=\Spin(8)$.
The fibration $F\to X\to Y$ is the fibration of spheres $\mathbb{S}^7\to\nolinebreak\mathbb{S}^{15}\to\nolinebreak\mathbb{S}^8$.

\begin{proposition}[Generators and relations] \label{prop:ex(vi)}
For
$$X = \tilG/\tilH = \SO(16)/\SO(15) \simeq \Spin(9)/\Spin(7) = G/H$$
and $K=\Spin(8)$, we have
\begin{enumerate}
  \item $\dd\ell(C_{\tilG}) = 4\,\dd\ell(C_G) - 3\,\dd r(C_K)$;
  \item $\left \{ \begin{array}{cll}
  \D_{\tilG}(X) & \!\!\!=\!\!\! & \C[\dd\ell(C_{\tilG})];\\
  \D_K(F) & \!\!\!=\!\!\! & \C[\dd r(C_K)];\\
  \D_G(X) & \!\!\!=\!\!\! & \C[\dd\ell(C_{\tilG}),\dd r(C_K)] = \C[\dd\ell(C_{\tilG}),\dd\ell(C_G)] = \C[\dd\ell(C_G),\dd r(C_K)].
  \end{array}\right.$
\end{enumerate}
\end{proposition}

We identify
\begin{eqnarray}
\Hom_{\C\text{-}\mathrm{alg}}(Z(\g_{\C}),\C) & \simeq & \jj_{\C}^*/W(\g_{\C}) \ \simeq\ \C^4/W(B_4),\label{eqn:HCZg-vi}\\
\Hom_{\C\text{-}\mathrm{alg}}(\D_{\tilG}(X),\C) & \simeq & \hspace{0.3cm} \tila_{\C}^*/\widetilde{W} \hspace{0.45cm} \simeq \hspace{0.25cm} \C/(\Z/2\Z)\label{eqn:HCX-vi}
\end{eqnarray}
by the standard bases.
The set $\Disc(K/H)$ consists of the representations of $K=\Spin(8)$ of the form $\tau=\mathcal{H}^k(\R^8)$ for $k\in\N$.

\begin{proposition}[Transfer map] \label{prop:nu-ex(vi)}
Let
$$X = \tilG/\tilH = \SO(16)/\SO(15) \simeq \Spin(9)/\Spin(7) = G/H$$
and $K=\Spin(8)$.
For $\tau=\mathcal{H}^k(\R^8)\in\Disc(K/H)$ with $k\in\N$, the affine map
\begin{eqnarray*}
S_{\tau} :\ \tila_{\C}^* \simeq \C & \longrightarrow & \hspace{1.8cm} \C^4 \hspace{1.8cm} \simeq \jj_{\C}^*\\
\lambda & \longmapsto & \frac{1}{2}\,(\lambda, k+5, k+3, k+1)
\end{eqnarray*}
induces a transfer map
$$\nnu(\cdot,\tau) : \Hom_{\C\text{-}\mathrm{alg}}(\D_{\tilG}(X),\C) \longrightarrow \Hom_{\C\text{-}\mathrm{alg}}(Z(\g_{\C}),\C)$$
as in Theorem~\ref{thm:nu-tau}.
\end{proposition}

In order to prove Propositions \ref{prop:ex(vi)} and~\ref{prop:nu-ex(vi)}, we use the following results on finite-dimensional representations.

\begin{lemma}\label{lem:ex(vi)}
\begin{enumerate}
  \item Discrete series for $\tilG/\tilH$, $G/H$, and $F=K/H$:
  \begin{eqnarray*}
    \Disc(\SO(16)/\SO(15)) & = & \{ \mathcal{H}^j(\R^{16}) \,:\, j\in\N\} ;\\
    \Disc(\Spin(9)/\Spin(7)) & = & \Big\{ \Rep\Big(\Spin(9),\frac{1}{2}(j,k,k,k)\Big) \,:\, j,k\in\N,\ j-k\in 2\N\Big\} ;\\
    \Disc(\Spin(8)/\Spin(7)) & = & \Big\{ \Rep\Big(\Spin(8),\frac{1}{2}(k,k,k,k)\Big) \,:\, k\in\N\Big\} .
  \end{eqnarray*}
  \item Branching laws for $\SO(16)\downarrow\Spin(9)$: For $j\in\N$,
  $$\mathcal{H}^j(\R^{16})|_{\Spin(9)}\ \simeq \bigoplus_{\substack{k\in\N\\ j-k\in 2\N}} \Rep\Big(\Spin(9),\frac{1}{2}(j,k,k,k)\Big).$$
  \item Irreducible decomposition of the regular representation of~$G$: For~$k\in\nolinebreak\N$,
  \begin{align*}
  & L^2\Big(\Spin(9)/\Spin(8),\Rep\Big(\Spin(8),\frac{1}{2}(k,k,k,k)\Big)\Big)\\
  & \simeq\ \sumplus{j-k\in 2\N}\ \Rep\Big(\Spin(9),\frac{1}{2}(j,k,k,k)\Big).
  \end{align*}
  \item The ring $S\big(\g_{\C}/\h_{\C}\big)^H=S(\C^8 \oplus \C^7)^{\Spin(7)}$ is generated by two algebraically independent homogeneous elements of degree~$2$.
\end{enumerate}
\end{lemma}

\begin{proof}[Proof of Lemma~\ref{lem:ex(vi)}]
(1) The description of $\Disc(\SO(16)/\SO(15))$ and of $\Disc(\Spin(8)/\Spin(7))$ follows from the classical theory of spherical harmonics (see \eg \cite[Intro. Th.\,3.1]{hel00}) or from the Cartan--Helgason theorem (Fact~\ref{fact:CartanHelgason}), and the description of $\Disc(\Spin(9)/\Spin(7))$ from \cite{kra79}.

(2) Let $\widetilde{\jj}_{\C}$ be a Cartan subalgebra of $\tilg_{\C}=\so(16,\C)$ and let $\{f_1,\dots,f_8\}$ be the standard basis of~$\widetilde{\jj}_{\C}^*$.
Fix a positive system
$$\Delta^+(\tilg_{\C},\widetilde{\jj}_{\C}) = \{ f_i\pm f_j \,:\, 1\leq i<j\leq 8\}.$$
Let $\widetilde{\p}_{\C} = \widetilde{\ell}_{\C} + \widetilde{\n}_{\C}$ be the standard parabolic subalgebra of~$\tilg_{\C}$ with
\begin{equation} \label{eqn:SO16-nweight}
\Delta(\widetilde{\n}_{\C},\widetilde{\jj}_{\C}) = \{ f_i\pm f_j \,:\, 2\leq i<j\leq 8\},
\end{equation}
and $\widetilde{P}_{\C}$ the corresponding parabolic subgroup.
By the Borel--Weil theorem, we can realize the irreducible representation $\mathcal{H}^j(\R^{16})=\Rep(\SO(16),jf_1)$ of~$\tilG$ on the space $\mathcal{O}(\tilG_{\C}/\widetilde{P}_{\C}^-,\mathcal{L}_{jf_1})$ of holomorphic sections of the $\tilG_{\C}$-equivariant holomorphic line bundle $\mathcal{L}_{jf_1} = \tilG_{\C} \times_{\widetilde{P}_{\C}^-} \C_{jf_1}$, where $\widetilde{P}_{\C}^-$ is the opposite parabolic subgroup to~$\widetilde{P}_{\C}$.

Let $\widetilde{\GL}(4,\C)$ be the double covering group of $\GL(4,\C)$, given by
$$\{ (c,A)\in\C^*\times\GL(4,\C) \,:\, \det(A)=c^2\}.$$
The natural embeddings $\GL(4,\C)\hookrightarrow\SO(8,\C)\hookrightarrow\SO(9,\C)$ lift to $\widetilde{\GL}(4,\C)\hookrightarrow\Spin(8,\C)\hookrightarrow\Spin(9,\C)=G_{\C}$.

We take the standard basis $\{ e_1,e_2,e_3,e_4\}$ of~$\jj_{\C}^*$, and a positive system $\Delta^+(\g_{\C},\jj_{\C}) = \{ e_i\pm e_j \,:\, 1\leq i<j\leq 4\} \cup \{ e_i \,:\, 1\leq i\leq 4\}$.
We set
$$\omega_+ := \frac{1}{2} (e_1+e_2+e_3+e_4)$$
and define a maximal parabolic subgroup $P_{\C}=L_{\C}N_{\C}$ by the characteristic element $(1,1,1,1)\in\C^4\simeq\jj_{\C}$.
Then the Levi subgroup $L_{\C}$ is isomorphic to $\widetilde{\GL}(4,\C)$ and
\begin{equation} \label{eqn:Spin9-nweight}
\Delta(\n_{\C},\jj_{\C}) = \{ e_i+e_j \,:\, 1\leq i<j\leq 4\} \cup \{ e_i \,:\, 1\leq i\leq 4\}.
\end{equation}
Given that the spin representation $\Rep(\Spin(9),\omega_+)\simeq\C^{16}$ has $\jj_{\C}$-weights
$$\Big\{ \frac{1}{2} (x_1,x_2,x_3,x_4) ~:~ x_j\in\{ \pm 1\}\Big\} \subset \C^4 \simeq \jj_{\C}^*,$$
we may and do assume that $\jj_{\C}$ is contained in~$\widetilde{\jj}_{\C}$ and that
\begin{equation} \label{eqn:Spin9-1111}
\iota(1,1,1,1) = (2,1,1,1,1,0,0,0) \in \widetilde{\jj}_{\C} \quad\mathrm{and}\quad \iota^* f_1 = \omega_+ \in \jj_{\C}^*,
\end{equation}
where $\iota$ denotes the inclusion map $\jj_{\C}\hookrightarrow\widetilde{\jj}_{\C}$.

Let $\widetilde{Q}_{\C}$ be the standard parabolic subgroup of~$\tilG_{\C}$ given by the characteristic element $(2,1,1,1,1,0,0,0)\in\widetilde{\jj}_{\C}$.
Then \eqref{eqn:Spin9-1111} shows that $\widetilde{Q}_{\C}$ is compatible with the reductive subgroup~$G_{\C}$ and $P_{\C}=\widetilde{Q}_{\C}\cap G_{\C}$.
We now verify that $P_{\C}=\widetilde{P}_{\C}\cap G_{\C}$.
Clearly, $\widetilde{Q}_{\C}$ is a subgroup of~$\widetilde{P}_{\C}$, hence $P_{\C}\subset\widetilde{P}_{\C}\cap G_{\C}$.
We note that $\widetilde{P}_{\C}\cap G_{\C}$ is a proper subgroup of~$G_{\C}$, because $G_{\C}=\Spin(9,\C)$ cannot be a subgroup of the Levi subgroup $\SO(2,\C)\times\SO(14,\C)$ of $\widetilde{P}_{\C}$.
Since $P_{\C}$ is a maximal parabolic subgroup of~$G_{\C}$, we conclude that $P_{\C}=\widetilde{P}_{\C}\cap G_{\C}$ with $L_{\C}\subset\widetilde{L}_{\C}$ and $N_{\C}\subset\widetilde{N}_{\C}$, and so $\widetilde{P}_{\C}$ is also compatible with the reductive subgroup~$G_{\C}$.

We claim that the Levi subgroup $L_{\C}\simeq\widetilde{\GL}(4,\C)$ acts on $\widetilde{\n}_{\C}/\n_{\C}\simeq\C^4$ by $\Rep(\widetilde{\GL}(4,\C),(1,1,1,0))$.
To see this, we observe from \eqref{eqn:SO16-nweight} that
\begin{eqnarray*}
\Delta(\widetilde{\n}_{\C},\jj_{\C}) & \!\!\!=\!\!\! & \iota^*(\Delta(\widetilde{\n}_{\C},\widetilde{\jj}_{\C}))\\
& \!\!\!=\!\!\! & \{ (x_1,x_2,x_3,x_4) \,:\, x_j\in\{ 0,1\}\big\} \smallsetminus \{ (0,0,0,0),(-1,-1,-1,-1)\}.
\end{eqnarray*}
Comparing this with \eqref{eqn:Spin9-nweight}, we find
$$\Delta(\widetilde{\n}_{\C}/\n_{\C},\jj_{\C}) = \{ (1,1,1,0), (1,1,0,1), (1,0,1,1), (0,1,1,1)\}.$$

Applying Proposition~\ref{prop:branchnormal} to the embedding $G_{\C}/P_{\C}^-\hookrightarrow\tilG_{\C}/\widetilde{P}_{\C}^-$, we obtain the following upper estimate for possible irreducible $\Spin(9)$-modules occurring in the restriction of the $\SO(16)$-module $\mathcal{H}^j(\R^{16}) \simeq \mathcal{O}(\tilG_{\C}/\widetilde{P}_{\C}^-,\mathcal{L}_{jf_1})$:
$$\mathcal{O}\big(\tilG_{\C}/\widetilde{P}_{\C}^-,\mathcal{L}_{jf_1}\big)\big|_{\Spin(9)} \,\subset\, \bigoplus_{\ell=0}^{+\infty} \mathcal{O}\big(G_{\C}/P_{\C}^-,\mathcal{L}_{\omega_+}\otimes\mathcal{S}^{\ell}(\widetilde{\n}_{\C}^-/\n_{\C}^-)\big),$$
because $\iota^*f_1=\omega_+$.
The right-hand side is actually a finite sum because the Borel--Weil theorem states that
\begin{align*}
& \mathcal{O}\big(G_{\C}/P_{\C}^-,\mathcal{L}_{\omega_+}\otimes\mathcal{S}^{\ell}(\widetilde{\n}_{\C}^-/\n_{\C}^-)\big)\\
& \simeq \left\{ \begin{array}{ll}
\Rep\big(\Spin(9),j\omega_+ + (0,-\ell,-\ell,-\ell)\big) & \!\text{if $j\geq 2\ell$},\\
\{0\} & \!\text{otherwise}.
\end{array}\right.
\end{align*}
Thus we have shown
\begin{equation} \label{eqn:HjR16-Spin9}
\mathcal{H}^j(\R^{16})\big|_{\Spin(9)} \subset \bigoplus_{\substack{k\in\N\\ j-k\in 2\N}} \Rep\Big(\Spin(9),\frac{1}{2}(j,k,k,k)\Big).
\end{equation}
Since the union of all irreducible $\Spin(9)$-modules occurring in $\mathcal{H}^j(\R^{16})$ for some~$j$ coincides with $\Disc(\Spin(9)/\Spin(7))$ by the comparison of \eqref{eqn:decomp1} with \eqref{eqn:decomp2}, the description of $\Disc(\Spin(9)/\Spin(7))$ in~(1) forces \eqref{eqn:HjR16-Spin9} to be an equality.
This completes the proof of~(2).

(3) By the Frobenius reciprocity, this follows from the classical branching law for $\mathfrak{o}(N)\downarrow\mathfrak{o}(N-1)$, see \eg \cite[Th.\,8.1.3]{gw09}, with $N=9$.

(4) This follows from the fact that there is no common nontrivial $\Spin(7)$-module in $S(\C^8)$ and $S(\C^7)$, as is seen from the following irreducible decompositions:
\begin{eqnarray}
S(\C^8) & \simeq & \bigoplus_{j\in\N} \Rep\Big(\Spin(7),\frac{j}{2}(1,1,1)\Big) \otimes \C[r^2],\nonumber\\
S(\C^7) & \simeq & \bigoplus_{j\in\N} \Rep\big(\SO(7),(k,0,0)\big) \otimes \C[s^2],\label{eqn:decomp-S(C7)}
\end{eqnarray}
where $r^2$ denotes a $\Spin(7)$-invariant quadratic form on~$\C^8$ and $s^2$ an $\SO(7)$-invariant quadratic form on~$\C^7$.
\end{proof}

\begin{proof}[Proof of Proposition~\ref{prop:ex(vi)}]
(1) By Lemma~\ref{lem:ex(vi)}, the map $\vartheta\mapsto (\pi(\vartheta),\tau(\vartheta))$ of Proposition~\ref{prop:pi-tau-theta} is given by
\begin{equation} \label{eqn:pitau-vi}
\Rep\Big(\Spin(9),\frac{1}{2}(j,k,k,k)\Big) \longmapsto \Big(\mathcal{H}^j(\R^{16}), \Rep\Big(\Spin(8),\frac{1}{2}(k,k,k,k)\Big)\Big).
\end{equation}
The Casimir operators for $\tilG$, $G$, and~$K$ act on these representations as the following scalars.
\begin{center}
\begin{tabular}{|c|c|c|}
\hline
Operator & Representation & Scalar\\
\hline\hline
$C_{\tilG}$ & $\mathcal{H}^j(\R^{16})$ & $j^2+14j$\\
\hline
$C_G$ & $\Rep(\Spin(9),\frac{1}{2}(j,k,k,k))$ & $\frac{1}{4}\,(j^2+14j+3k^2+18k)$\\
\hline
$C_K$ & $\Rep(\Spin(8),\frac{1}{2}(k,k,k,k))$ & $k^2+6k$\\
\hline
\end{tabular}
\end{center}
This, together with the identity
$$j^2 + 14j = (j^2 + 14j +3k^2 + 18k) - 3 \, (k^2 + 6k),$$
implies $\dd\ell(C_{\tilG}) = 4\,\dd\ell(C_G) - 3\,\dd r(C_K)$.

(2) Since $\tilG/\tilH$ and $F=K/H$ are symmetric spaces, we obtain $\D_{\tilG}(X)$ and $\D_K(F)$ using the Harish-Chandra isomorphism.
We now focus on $\D_G(X)$.
We only need to prove the first equality, since the others follow from the relations between the generators.
For this, using Lemmas \ref{lem:struct-DGX}.(3) and~\ref{lem:ex(vi)}.(4), it suffices to show that the two differential operators $\dd\ell(C_{\tilG})$ and $\dd r(C_K)$ on~$X$ are algebraically independent.
Let $f$ be a polynomial in two variables such that $f(\dd\ell(C_{\tilG}),\dd r(C_K))=0$ in $\D_G(X)$.
By letting this differential operator act on the $G$-isotypic component $\vartheta = \Rep\big(\Spin(9),\frac{1}{2}(j,k,k,k)\big)$ in $C^{\infty}(X)$, we obtain
$$f\big( j^2+14j, k^2+6k \big) = 0$$
for all $j,k\in\N$ with $j-k\in 2\N$, hence $f$ is the zero polynomial.
\end{proof}

\begin{proof}[Proof of Proposition~\ref{prop:nu-ex(vi)}]
We use Proposition~\ref{prop:Stau-transfer} and the formula \eqref{eqn:pitau-vi} for the map $\vartheta\mapsto (\pi(\vartheta),\tau(\vartheta))$ of Proposition~\ref{prop:pi-tau-theta}.
Let $\tau=\mathcal{H}^k(\R^8)\in\Disc(K/H)$ with $k\in\N$.
If $\vartheta\in\Disc(G/H)$ satisfies $\tau(\vartheta)=\tau$, then $\vartheta$ is of the form $\vartheta = \Rep(\Spin(9),\frac{1}{2}(j,k,k,k))$ for some $j\in\N$ with $j-k\in 2\N$, by \eqref{eqn:pitau-vi}.
The algebra $\D_{\tilG}(X)$ acts on the irreducible $\tilG$-submodule $\pi(\vartheta) = \mathcal{H}^j(\R^{16})$ by the scalars
$$\lambda(\vartheta) + \rho_{\tila} = j+7 \in \C/(\Z/2\Z)$$
via the Harish-Chandra isomorphism \eqref{eqn:HCX-vi}, whereas the algebra $Z(\g_{\C})$ acts on the irreducible $G$-module $\vartheta = \Rep(\Spin(9),\frac{1}{2}(j,k,k,k))$ by the scalars
$$\nu(\vartheta) + \rho = \frac{1}{2}(j+7, k+5, k+3, k+1) \in \C^4/W(B_4)$$
via \eqref{eqn:HCZg-vi}.
Thus the affine map $S_{\tau}$ in Proposition~\ref{prop:nu-ex(vi)} sends $\lambda(\vartheta)+\rho_{\tila}$ to $\nu(\vartheta)+\rho$ for any $\vartheta\in\Disc(G/H)$ such that $\tau(\vartheta)=\tau$, and we conclude using Proposition~\ref{prop:Stau-transfer}.
\end{proof}

\begin{remark}
The homogeneous space $\C^{\ast}\times\Spin(9,\C)/\Spin(7,\C)$ arises\linebreak as the unique open orbit of the prehomogeneous vector space\linebreak $(\C^{\ast}\times\Spin(9,\C),\C^{16})$ (see \eg \cite{igu70,hu91}).
Howe--Umeda \cite[\S\,11.11]{hu91} proved that the $\C$-algebra homomorphism $\dd\ell : Z(\g_{\C})\to\D_G(X)$ is not surjective and that the ``abstract Capelli problem'' has a negative answer for this prehomogeneous space.
Proposition~\ref{prop:ex(vi)}.(2) gives a refinement of their assertion in this case.
The novelty here is to introduce the operator $\dd r(C_K)\in\D_G(X)\smallsetminus\dd\ell(Z(\g_{\C}))$ to describe the ring $\D_G(X)$.
\end{remark}

\subsection{The case $(\tilG,\tilH,G)=(\SO(8),\Spin(7),\SO(5)\times\SO(3))$}\label{subsec:ex(vii)}

Here $H=\tilH\cap\nolinebreak G$ is isomorphic to $\SO(4) \simeq \big(\SU(2)\times\SU(2)\big)/\{\pm(I_2,I_2)\}$.
The only maximal connected proper subgroup of~$G$ containing~$H$ is $K=\SO(4)\times\nolinebreak\SO(3)$, which is realized in~$G$ in the standard manner.
The group $H$ is the image of the embedding $\iota_7 :\nolinebreak\SO(4)\simeq\big(\SU(2)\times\SU(2)\big)/\{\pm(I_2,I_2)\} \to K$ induced from the following diagram:
$$\begin{array}{ccccccccl}
1 & \!\!\!\longrightarrow\!\!\! & \{ \pm\Diag(I_2)\} \times \{ \pm I_2\} & \!\!\!\longrightarrow\!\!\! & \SU(2) \times \SU(2) \times \SU(2) & \!\!\!\longrightarrow\!\!\! & K & \!\!\!\longrightarrow\!\!\! & 1\\
& & \cup & & \cup & & \cup & & \\
1 & \!\!\!\longrightarrow\!\!\! & \{ \pm (I_2,I_2,I_2)\} & \!\!\!\longrightarrow\!\!\! & \SU(2) \times \Diag(\SU(2)) & \!\!\!\longrightarrow\!\!\! & H & \!\!\!\longrightarrow\!\!\! & 1.
\end{array}$$
Thus $F=K/H$ is diffeomorphic to the $3$-dimensional projective space $\PP^3(\R)$, and the fibration $F\to X\to Y$ identifies with a variant of the quaternionic Hopf fibration $\PP^3(\R)\to\PP^7(\R)\to\mathbb{S}^4$.
The groups $G$ and~$K$ are not simple.
We denote by $C_G^{(i)}\in Z(\g_{\C})$ the Casimir element of the $i$-th factor of~$G$, for $i\in\{ 1,2\}$, and by $C'_K\in Z(\kk_{\C})$ the Casimir element of the $\SO(4)$ factor of~$K$.

\begin{proposition}[Generators and relations] \label{prop:ex(vii)}
For
$$X = \tilG/\tilH = \SO(8)/\Spin(7) \simeq (\SO(5)\times\SO(3))/\iota_7(\SO(4)) = G/H$$
and $K=\SO(4)\times\SO(3)$, we have
\begin{enumerate}
  \item $\left \{
\begin{array}{l}
  \dd\ell(C_{\tilG}) = 4\,\dd\ell(C_G^{(1)}) - 4\,\dd\ell(C_G^{(2)});\\
  2\,\dd\ell(C_G^{(2)}) = \dd r(C'_K);
\end{array}
\right.$
  \item $\left \{
\begin{array}{l}
  \D_{\tilG}(X) = \C[\dd\ell(C_{\tilG})];\\
  \D_K(F) = \C[\dd r(C'_K)];\\
  \D_G(X) = \C[\dd\ell(C_{\tilG}),\dd r(C'_K)] = \C[\dd\ell(C_G^{(1)}),\dd\ell(C_G^{(2)})].
\end{array}
\right.$
\end{enumerate}
\end{proposition}

We identify
\begin{align}
\Hom_{\C\text{-}\mathrm{alg}}(Z(\g_{\C}),\C) & & \hspace{-0.2cm} \simeq\ \jj_{\C}^*/W(\g_{\C}) & & \hspace{-0.3cm} \simeq\ (\C^2\oplus\C)/W(B_2\times B_1),\label{eqn:HCZg-vii}\\
\Hom_{\C\text{-}\mathrm{alg}}(\D_{\tilG}(X),\C) & & \hspace{-1.2cm} \simeq \hspace{0.35cm} \tila_{\C}^*/\widetilde{W} \hspace{0.4cm}  & & \hspace{-1.2cm} \simeq\ \C/(\Z/2\Z)\hspace{2.22cm}\label{eqn:HCX-vii}
\end{align}
by the standard bases.
We note that $\tilG/\tilH$ is not a symmetric space, but the Lie algebras $(\tilg,\tilh)=(\mathfrak{so}(8),\mathfrak{spin}(7))$ form a symmetric pair and the description of the algebra $\D_{\tilG}(X)$ is the same as that for symmetric spaces.
In order to be more precise, let us recall the triality of the Dynkin diagram~$D_4$.

\subsubsection*{Triality}

For $\tilg_{\C}=\so(8,\C)$, the outer automorphism group $\mathrm{Out}(\tilg_{\C})=\mathrm{Aut}(\tilg_{\C})/\mathrm{Int}(\tilg_{\C})$ is isomorphic to the symmetric group~$\mathfrak{S}_3$, corresponding to the automorphism group of the Dynkin diagram~$D_4$.
More precisely, let $\widetilde{\jj}_{\C}$ be a Cartan subalgebra of~$\tilg_{\C}$ and $\{ e_1,e_2,e_3,e_4\}$ the standard basis of~$\widetilde{\jj}_{\C}^{\ast}$.
Fix a positive system $\Delta^+(\tilg_{\C},\widetilde{\jj}_{\C})=\{ e_i\pm e_j\,:\,1\leq i<j\leq 4\}$ and set
$$\omega_{\pm} := \frac{1}{2} (e_1 + e_2 + e_3 \pm e_4).$$
Then the automorphism group of the Dynkin diagram~$D_4$ is the permutation group of the set $\{ e_1,\omega_+,\omega_-\}$.
It gives rise to triality in $\mathfrak{so}(8)$.
We denote by $\varsigma$ the outer automorphism of $\mathfrak{so}(8)$ of order three corresponding to the outer automorphism of~$\mathrm{D}_4$ as described in the figure below.
With this choice,
\begin{equation} \label{eqn:triality-basis}
\varsigma^*(e_1)=\omega_+, \quad \varsigma^*(\omega_+)=\omega_-, \quad \varsigma^*(\omega_-)=e_1.
\end{equation}

\begin{figure}
\centering
\includegraphics{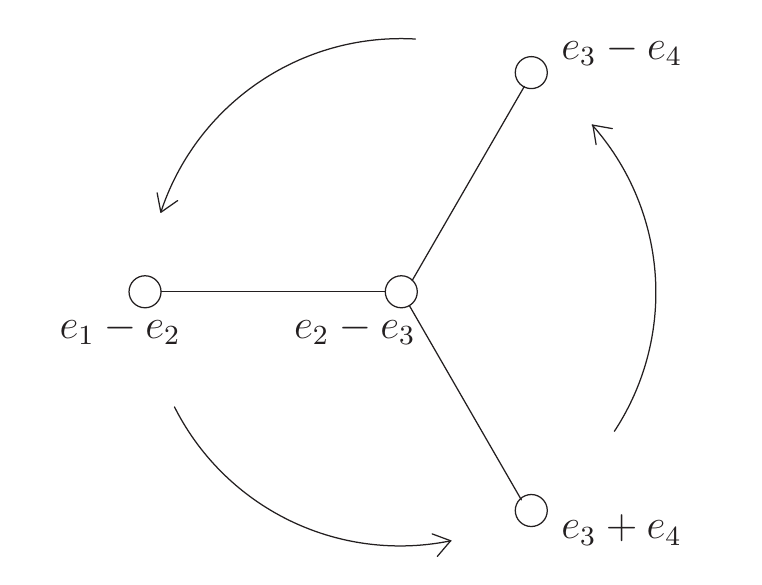}
\end{figure}

The set $\Disc(K/H)$ consists of the representations of $K=\SO(4)\times\SO(3)$ of the form $\tau=\Rep(\SO(4),(k,k))\boxtimes\Rep(\SO(3),k)$ for $k\in\N$.

\begin{proposition}[Transfer map] \label{prop:nu-ex(vii)}
Let
$$X = \tilG/\tilH = \SO(8)/\Spin(7) \simeq (\SO(5)\times\SO(3))/\iota_7(\SO(4)) = G/H$$
and $K=\SO(4)\times\SO(3)$.
For $\tau=\Rep(\SO(4),(k,k))\boxtimes\Rep(\SO(3),k)\in\Disc(K/H)$ with $k\in\N$, the affine map
\begin{eqnarray*}
S_{\tau} :\ \tila_{\C}^* \simeq \C & \longrightarrow & \hspace{1cm} \C^2\oplus\C \hspace{1cm} \simeq \jj_{\C}^*\\
\lambda & \longmapsto & \frac{1}{2}\,(\lambda, 2k+3, 2k+1)
\end{eqnarray*}
induces a transfer map
$$\nnu(\cdot,\tau) : \Hom_{\C\text{-}\mathrm{alg}}(\D_{\tilG}(X),\C) \longrightarrow \Hom_{\C\text{-}\mathrm{alg}}(Z(\g_{\C}),\C)$$
as in Theorem~\ref{thm:nu-tau}.
\end{proposition}

In order to prove Propositions \ref{prop:ex(vii)} and~\ref{prop:nu-ex(vii)}, we use the following results on finite-dimensional representations.

\begin{lemma}\label{lem:ex(vii)}
\begin{enumerate}
  \item Discrete series for $\tilG/\tilH$, $G/H$, and $F=K/H$:
  \begin{align*}
  & \Disc(\SO(8)/\Spin(7)) = \{ \Rep(\SO(8),(j,j,j,j)) \ :\ j\in\N\} ;\\
  & \Disc\big((\SO(5)\times\SO(3))/\iota_7(\SO(4))\big)\\
  & \hspace{1cm} = \{ \Rep(\SO(5),(j,k))\boxtimes\Rep(\SO(3),k) \ :\ k,j\in\N,\ k\leq j\} ;\\
  & \Disc\big((\SO(4)\times\SO(3))/\iota_7(\SO(4))\big)\\
  & \hspace{1cm} = \{ \Rep(\SO(4),(k,k)) \boxtimes \Rep(\SO(3),k) \ :\ k\in\N\} .
  \end{align*}
  \item Branching laws for $\SO(8)\downarrow\SO(5)\times\SO(3)$: For $j\in\N$,
  $$\Rep\big(\SO(8),(j,j,j,j)\big)|_{\SO(5)\times\SO(3)} \,\simeq\, \bigoplus_{k=0}^j \Rep\big(\SO(5),(j,k)\big) \boxtimes \Rep(\SO(3),k).$$
  \item Irreducible decomposition of the regular representation of~$G$: For~$k\in\nolinebreak\N$,
  $$L^2\big((\SO(5)\times\SO(3))/(\SO(4)\times\SO(3)),\Rep\big(\SO(4),(k,k)\big) \boxtimes \Rep(\SO(3),k)\big)$$
  $$\simeq\, \sumplus{\substack{j\in\N\\ j\geq k}}\, \Rep\big(\SO(5),(j,k)\big) \boxtimes \Rep(\SO(3),k).$$
  \item The ring $S(\g_{\C}/\h_{\C})^H=S(\C^7)^{\iota_7(\SO(4))}$ is generated by two algebraically independent homogeneous elements of degree~$2$.
\end{enumerate}
\end{lemma}

\begin{proof}[Proof of Lemma~\ref{lem:ex(vii)}]
(1) We use the triality of $\so(8)$.
The automorphism $\varsigma$ of $\so(8)$ sends $\mathfrak{spin}(7)$ to $\so(7)$, and induces a double covering $\mathbb{S}^7\to\PP^7(\R)$~by
$$\SO(8)^{\sim}/\SO(7)^{\sim} \underset{\varsigma}{\simeq} \SO(8)^{\sim}/\Spin(7) \longrightarrow \SO(8)/\Spin(7) = \tilG/\tilH,$$
where $\SO(N)^{\sim}$ denotes the double covering of $\SO(N)$ for $N=7$ or~$8$.
Thus $\Disc(\tilG/\tilH)$ is obtained by taking the even part of $\Disc(\SO(8)/\SO(7))$ (see Lemma~\ref{lem:ex(i)}.(1) with $n=3$) via~$\varsigma$, and the computation boils down to the isomorphism
\begin{equation} \label{eqn:Hk-tri}
\Rep\big(\SO(8),\varsigma\cdot (2j,0,0,0)\big) = \Rep\big(\SO(8),(j,j,j,j)\big).
\end{equation}

(2) We use another expression of the double covering $\mathbb{S}^7\to\PP^7(\R)$, namely
$$(\Sp(2)\times\Sp(1)) / (\Sp(1)\times\Diag(\Sp(1))) \longrightarrow (\SO(5)\times\SO(3)) / \iota_7(\SO(4)).$$
With the notation of Lemma~\ref{lem:ex(iii)}, we have
\begin{align*}
& \Disc\big((\Sp(2)\times\Sp(1))/\Diag(\Sp(1))\big)\\
& = \{ \mathcal{H}^{a,b}(\HH^2) \boxtimes \C^{a-b+1} \,:\, a\geq b\geq 0,\ a,b\in\Z\}.
\end{align*}
We conclude using the fact that the $\Sp(2)$-module $\mathcal{H}^{a,b}(\HH^2)$ descends to $\SO(5)$ if $a\equiv b\mod 2$ and is isomorphic to $\Rep\big(\SO(5),(\frac{a+b}{2},\frac{a-b}{2})\big)$ as an $\SO(5)$-module.

(3) The branching law is a special case of Lemma~\ref{lem:ex(v)} with $n=1$ via the triality automorphism $\varsigma$ and the covering $\Sp(2)\times\Sp(1)\to\SO(5)\times\SO(3)$.

(4) We have an irreducible decomposition as $\SO(4)$-modules via~$\iota_7$:
$$\g_{\C}/\h_{\C} \,\simeq\, \g_{\C}/\kk_{\C} \oplus \kk_{\C}/\h_{\C} \,\simeq\, (\C^2\boxtimes\C^2) \oplus (\mathbf{1}\boxtimes\C^3).$$
Then the ring $S(\C^2\boxtimes\C^2)^{\SU(2)\times\{1\}}$ is a polynomial ring generated by a single homogeneous element of degree~$2$, on which $\{1\}\times\SU(2)$ acts trivially.
Therefore, $S(\g_{\C}/\h_{\C})^H$ is isomorphic to
$$S(\C^2\boxtimes\C^2)^{\SU(2)\times\{1\}} \otimes S(\mathbf{1}\boxtimes\C^3)^{\{1\}\times\SO(3)},$$
and the statement follows.
\end{proof}

\begin{proof}[Proof of Proposition~\ref{prop:ex(vii)}]
(1) By Lemma~\ref{lem:ex(vii)}, the map $\vartheta\mapsto (\pi(\vartheta),\tau(\vartheta))$ of Proposition~\ref{prop:pi-tau-theta} is given by
\begin{equation} \label{eqn:pitau-vii}
\begin{array}{l}
\Rep(\SO(5),(j,k)) \boxtimes \Rep(\SO(3),k)\\
\longmapsto \big(\Rep(\SO(8),(j,j,j,j)),\Rep(\SO(4),(k,k)) \boxtimes \Rep(\SO(3),k)\big).\hspace{-0.2cm}
\end{array}
\end{equation}
The Casimir operators for $\tilG$ and for the factors of $G$ and~$K$ act on these representations as the following scalars.
\begin{center}
\begin{tabular}{|c|c|c|}
\hline
Operator & Representation & Scalar\\
\hline\hline
$C_{\tilG}$ & $\Rep(\SO(8),(j,j,j,j))$ & $4 (j^2 + 3j)$\\
\hline
$C_G^{(1)}$ & $\Rep(\SO(5),(j,k))$ & $j^2 + 3j + k^2 + k$\\
\hline
$C_G^{(2)}$ & $\Rep(\SO(3),k)$ & $k^2 + k$\\
\hline
$C'_K$ & $\Rep(\SO(4),(k,k))$ & $2 (k^2 + k)$\\
\hline
\end{tabular}
\end{center}
This implies $\dd\ell(C_{\tilG}) = 4\,\dd\ell(C_G^{(1)}) - 4\,\dd\ell(C_G^{(2)})$ and $2\,\dd\ell(C_G^{(2)}) = \dd r(C'_K)$.

(2) The description of $\D_{\tilG}(X)$ from the classical result for the symmetric space $\SO(8)/\SO(7) \simeq \Spin(8)/\Spin(7)$ (see Proposition~\ref{prop:ex(i)}.(2) with $n=3$) by the triality of~$D_4$.
The description of $\D_K(F)$ is reduced to the group manifold case $(^{\backprime}G\times\!^{\backprime}G)/\Diag(^{\backprime}G)$ with $^{\backprime}G=\SU(2)$ using the diagram just before Proposition~\ref{prop:ex(vii)}.
We now focus on $\D_G(X)$.
We only need to prove the first equality, since the other one follows from the relations between the generators.
For this, using Lemmas \ref{lem:struct-DGX}.(3) and~\ref{lem:ex(vii)}.(4), it suffices to show that the two differential operators $\dd\ell(C_{\tilG})$ and $\dd r(C'_K)$ on~$X$ are algebraically independent.
Let $f$ be a polynomial in two variables such that $f(\dd\ell(C_{\tilG}),\dd r(C'_K))=0$ in $\D_G(X)$.
By letting this differential operator act on the $G$-isotypic component $\vartheta = \Rep\big(\SO(5),(j,k)\big) \boxtimes \Rep(\SO(3),k)$ in $C^{\infty}(X)$, we obtain
$$f\big( 4(j^2+3j), 2(k^2+k) \big) = 0$$
for all $j,k\in\N$ with $j\geq k$, hence $f$ is the zero polynomial.
\end{proof}

\begin{proof}[Proof of Proposition~\ref{prop:nu-ex(vii)}]
We use Proposition~\ref{prop:Stau-transfer} and the formula \eqref{eqn:pitau-vii} for the map $\vartheta\mapsto (\pi(\vartheta),\tau(\vartheta))$ of Proposition~\ref{prop:pi-tau-theta}.
Let $\tau=\Rep(\SO(4),(k,k))\boxtimes\Rep(\SO(3),k)\in\Disc(K/H)$ with $k\in\N$.
If $\vartheta\in\Disc(G/H)$ satisfies $\tau(\vartheta)=\tau$, then $\vartheta$ is of the form $\vartheta = \Rep(\SO(5),(j,k))\boxtimes\Rep(\SO(3),k)$ for some $j\in\N$ with $j\geq k$, by \eqref{eqn:pitau-vii}.
The algebra $\D_{\tilG}(X)$ acts on the irreducible $\tilG$-submodule $\pi(\vartheta) = \Rep(\SO(8),(j,j,j,j))$ by the scalars
$$\lambda(\vartheta) + \rho_{\tila} = 2j + 3 \in \C/(\Z/2\Z)$$
via the Harish-Chandra isomorphism \eqref{eqn:HCX-vii}, whereas the algebra $Z(\g_{\C})$ acts on the irreducible $G$-module $\vartheta = \Rep(\SO(5),(j,k))\boxtimes\Rep(\SO(3),k)$ by the scalars
$$\nu(\vartheta) + \rho = \frac{1}{2}(2j+3, 2k+1; 2k+1) \in (\C^2\oplus\C)/W(B_2\times B_1)$$
via \eqref{eqn:HCZg-vii}.
Thus the affine map $S_{\tau}$ in Proposition~\ref{prop:nu-ex(vii)} sends $\lambda(\vartheta)+\rho_{\tila}$ to $\nu(\vartheta)+\rho$ for any $\vartheta\in\Disc(G/H)$ such that $\tau(\vartheta)=\tau$, and we conclude using Proposition~\ref{prop:Stau-transfer}.
\end{proof}

\subsection{The case $(\tilG,\tilH,G)=(\SO(7),G_{2(-14)},\SO(5)\times\SO(2))$}\label{subsec:ex(viii)}

Here $H=\tilH\cap\nolinebreak G$ is isomorphic to $\U(2) \simeq (\SU(2)\times\SO(2))/\{\pm(I_2,I_2)\}$.
The only maxi\-mal connected proper subgroup of~$G$ containing~$H$ is $K=\SO(4)\times\SO(2)$, which is realized in~$G$ in the standard manner.
The group $H$ is the image of the embedding $\iota_8 :\nolinebreak\U(2)\simeq\big(\SU(2)\times\SO(2)\big)/\{\pm(I_2,I_2)\} \to K$ induced from the following diagram:
$$\begin{array}{ccccccccl}
1 & \!\!\!\longrightarrow\!\!\! & \{ \pm\Diag(I_2)\} \times \{ \pm I_2\} & \!\!\!\longrightarrow\!\!\! & \SU(2) \times \SU(2) \times \SO(2) & \!\!\!\longrightarrow\!\!\! & K & \!\!\!\longrightarrow\!\!\! & 1\\
& & \cup & & \cup & & \cup & & \\
1 & \!\!\!\longrightarrow\!\!\! & \{ \pm (I_2,I_2,I_2)\} & \!\!\!\longrightarrow\!\!\! & \SU(2) \times \Diag(\SO(2)) & \!\!\!\longrightarrow\!\!\! & H & \!\!\!\longrightarrow\!\!\! & 1.
\end{array}$$
This case and case~(ix) of Table~\ref{table1} (see Section~\ref{subsec:ex(ix)}) are different from the other cases in the sense that neither $\tilG/\tilH$ nor $K/H$ is a symmetric space.

The groups $G$ and~$K$ are not simple.
We denote by $C_G^{(1)}$ (\resp $C_K^{(1)}$) the Casimir element of the first factor of $G=\SO(5)\times\SO(2)$ (\resp $K=\SO(4)\times\SO(2)$), and by $E_G$ (\resp $E_K$) a generator of the abelian ideal $\so(2)$ of $\g$ (\resp $\kk$) such that the eigenvalues of $\ad(E_G)$ (\resp $\ad(E_K)$) in~$\tilg_{\C}$ are $0,\pm 1$.

\begin{proposition}[Generators and relations] \label{prop:ex(viii)}
For
$$X = \tilG/\tilH = \SO(7)/G_{2(-14)} \simeq (\SO(5)\times\SO(2))/\iota_8(\U(2)) = G/H$$
and $K=\SO(4)\times\SO(2)$, we have
\begin{enumerate}
  \item $\left \{
\begin{array}{l}
  \dd\ell(E_G) = \dd r(E_K);\\
  2\,\dd\ell(C_{\tilG}) = 6\,\dd\ell(C_G^{(1)}) - 3\,\dd r(C_K^{(1)});
\end{array}
\right.$
  \item $\left \{
\begin{array}{ccl}
  \D_{\tilG}(X) & \!\!\!=\!\!\! & \C[\dd\ell(C_{\tilG})];\\
  \D_K(F) & \!\!\!=\!\!\! & \C[\dd r(C_K^{(1)}),\dd r(E_K)];\\
  \D_G(X) & \!\!\!=\!\!\! & \C[\dd\ell(C_{\tilG}),\dd r(C_K^{(1)}),\dd r(E_K)]\\
  & & \hspace{0.3cm} =\, \C[\dd\ell(C_G^{(1)}),\dd r(C_K^{(1)}),\dd r(E_K)]\\
  & & \hspace{0.3cm} =\, \C[\dd\ell(C_{\tilG}),\dd\ell(C_G^{(1)}),\dd\ell(E_G)].
\end{array}
\right.$
\end{enumerate}
\end{proposition}

Identifying $\jj_{\C}^*$ with $\C^2\oplus\C$ via the standard basis, the Harish-Chandra homomorphism amounts to
\begin{equation} \label{eqn:HCZg-viii}
\Hom_{\C\text{-}\mathrm{alg}}(Z(\g_{\C}),\C) \ \simeq\ \jj_{\C}^*/W(\g_{\C}) \ \simeq\ (\C^2\oplus\C)/(W(B_2)\!\times\!\{1\}).
\end{equation}
On the other hand, $X_{\C}=\tilG_{\C}/\tilH_{\C}=\SO(7,\C)/G_2(\C)$ is a nonsymmetric spherical homogeneous space of rank one.
We take $\aaa_{\C}^*:=\C(1,1,1)$ viewed as a subspace of~$\jj_{\C}^*$, and normalize the generalized Harish-Chandra isomorphism of \eqref{eqn:Psi} as
\begin{eqnarray} \label{eqn:HCX-viii}
\Hom_{\C\text{-}\mathrm{alg}}(\D_{\tilG}(X),\C)\ \simeq\ \tila_{\C}^*/\widetilde{W} & \simeq & \C/(\Z/2\Z)\\
\chi_{\lambda}^X \hspace{3cm} & \leftlongmapsto & \hspace{0.5cm}\lambda\nonumber
\end{eqnarray}
so that $\chi_{\lambda}^X(\dd\ell(C_{\tilG}))=3(\lambda^2-9/4)$.
Then $\vartheta=\Rep(\SO(7),\lambda)$ belongs to $\Disc(\tilG/\tilH)$ if and only if $\lambda$ is of the form $\lambda=j(1,1,1)$ for some $j\in\N$ (see Lemma~\ref{lem:ex(viii)}.(1) below), and $P\in\D_{\tilG}(X)$ acts on~$\vartheta$ by the scalar $\chi_{\lambda+\rho_{\tila}}^X(P)$ where we set
\begin{equation} \label{eqn:rhoa-nonsymm}
\rho_{\tila} := \frac{3}{2} \, (1,1,1).
\end{equation}
With this normalization, the set $\Disc(K/H)$ consists of the representations of $K=\SO(4)\times\SO(2)$ of the form $\tau=\Rep(\SO(4),(k,k))\boxtimes\C_a$ for $a,k\in\Z$ and $|a|\leq k$, and the following holds.

\begin{proposition}[Transfer map] \label{prop:nu-ex(viii)}
Let
$$X = \tilG/\tilH = \SO(7)/G_{2(-14)} \simeq (\SO(5)\times\SO(2))/\iota_8(\U(2)) = G/H$$
and $K=\SO(4)\times\SO(2)$.
For $\tau=\Rep(\SO(4),(k,k))\boxtimes\C_a\in\Disc(K/H)$ with $a,k\in\Z$ and $|a|\leq k$, the affine map
\begin{eqnarray*}
S_{\tau} :\ \tila_{\C}^* \simeq \C & \longrightarrow & \hspace{0.7cm} \C^2\oplus\C \hspace{0.6cm} \simeq \jj_{\C}^*\\
\lambda & \longmapsto & \Big(\Big(\lambda,k+\frac{1}{2}\Big),a\Big)
\end{eqnarray*}
induces a transfer map
$$\nnu(\cdot,\tau) : \Hom_{\C\text{-}\mathrm{alg}}(\D_{\tilG}(X),\C) \longrightarrow \Hom_{\C\text{-}\mathrm{alg}}(Z(\g_{\C}),\C)$$
as in Theorem~\ref{thm:nu-tau}.
\end{proposition}

In order to prove Propositions \ref{prop:ex(viii)} and~\ref{prop:nu-ex(viii)}, we use the following results on finite-dimensional representations.

\begin{lemma}\label{lem:ex(viii)}
\begin{enumerate}
  \item Discrete series for $\tilG/\tilH$, $G/H$, and $F=K/H$:
  \begin{align*}
  & \Disc(\SO(7)/G_{2(-14)}) = \{ \Rep(\SO(7),(j,j,j)) \,:\, j\in\N\} ;\\
  & \Disc\big((\SO(5)\times\SO(2))/\iota_8(\U(2))\big)\\
  & \hspace{1cm} = \{ \Rep(\SO(5),(j,k)) \boxtimes \C_a \,:\, |a|\leq k\leq j,\ a,j,k\in\Z\} ;\\
  & \Disc\big((\SO(4)\times\SO(2))/\iota_8(\U(2))\big)\\
  & \hspace{1cm} = \{ \Rep(\SO(4),(k,k)) \boxtimes \C_a \,:\, |a|\leq k,\ a,k\in\Z\} .
  \end{align*}
  \item Branching laws for $\SO(7)\downarrow\SO(5)\times\SO(2)$: For $j\in\N$,
  $$\Rep\big(\SO(7),(j,j,j)\big)|_{\SO(5)\times\SO(2)}\ \simeq \bigoplus_{\substack{a,k\in\Z\\ |a|\leq k\leq j}} \Rep\big(\SO(5),(j,k)\big) \boxtimes \C_a.$$
  \item Irreducible decomposition of the regular representation of~$G$:\\ For $a\in\Z$ and $k\in\N$,
  \begin{align*}
  & L^2\big((\SO(5)\times\SO(2))/(\SO(4)\times\SO(2)),\Rep\big(\SO(4),(k,k)\big) \boxtimes \C_a\big)\\
  & \hspace{1.3cm} \simeq\ \sumplus{\substack{j\in\N\\ j\geq k}}\ \Rep\big(\SO(5),(j,k)\big) \boxtimes \C_a.
  \end{align*}
  \item The ring $S(\tilg_{\C}/\tilh_{\C})^{\tilH}=S(\C^7)^{G_2}$ is generated by a single homogeneous element of degree~$2$.
  \item The ring $S(\g_{\C}/\h_{\C})^H=S(\C^4\oplus\C^3)^{\SU(2)\times\SO(2)}$ is generated by three algebraically independent homogeneous elements of respective degrees $1,2,2$.
  \item The ring $S(\kk_{\C}/\h_{\C})^H=S(\C^3)^{\SO(2)}$ is generated by two algebraically independent homogeneous elements of respective degrees $1,2$.
\end{enumerate}
\end{lemma}

\begin{proof}[Proof of Lemma~\ref{lem:ex(viii)}]
(1) The description of $\Disc(\SO(7)/G_{2(-14)})$ is given by Kr\"amer \cite{kra79}.
The description of $\Disc\big((\SO(4)\times\SO(2))/\iota_8(\U(2))\big)$ readily follows from a computation for $\SU(2)$ using the diagram just before Proposition~\ref{prop:ex(viii)}.
The description of $\Disc\big((\SO(5)\times\SO(2))/\iota_8(\U(2))\big)$ follows from (2) via \eqref{eqn:decomp1} or from (3) via \eqref{eqn:decomp2}.

(2) See \cite{tsu81}.

(3) By the classical branching law for $\SO(N)\downarrow\SO(N-1)$, see \eg \cite[Th.\,8.1.3]{gw09}, with $N=5$, the assertion follows by the Frobenius reciprocity.

(4) See \cite{sch78} or the proof of Lemma~\ref{lem:ex(*)}.(4) below.

(5) We have an isomorphism of $(\SU(2)\times\SO(2))$-modules
$$\g_{\C}/\h_{\C} \simeq \C^4 \oplus \C^3 \simeq \big(\C^2 \boxtimes (\C_1 \oplus \C_{-1})\big) \oplus \big(\mathbf{1} \boxtimes (\C_2 \oplus \C_0 \oplus \C_{-2})\big).$$
The ring $S(\C^2\otimes\C^2)^{\SU(2)\times\{ 1\}}$ is a polynomial ring generated by one homogeneous element of degree~$2$, on which $\{ 1\}\times\SO(2)$ acts trivially.
Therefore, $S(\C^4\oplus\C^3)^{\SU(2)\times\SO(2)}$ is isomorphic to
$$S\big(\C^2 \boxtimes (\C_1 \oplus \C_{-1})\big)^{\SU(2)\times\{ 1\}} \otimes S(\C_2 \oplus \C_{-2})^{\SO(2)} \otimes S(\C),$$
and statement~(5) follows.

(6) Via the double covering $\SU(2)\times\SO(2)\overset{\sim}{\longrightarrow} H\simeq\U(2)$, the group $\SU(2)\times\SO(2)$ acts on $\g_{\C}/\h_{\C}\simeq\C^3$ as $\mathbf{1} \boxtimes (\C_2\oplus\C_0\oplus\C_{-2})$, and then the ring $S(\kk_{\C}/\h_{\C})^H$ is isomorphic to
$$S(\C_0) \otimes S(\C_2\oplus\C_{-2})^{\SO(2)},$$
and statement~(6) follows.
\end{proof}

\begin{proof}[Proof of Proposition~\ref{prop:ex(viii)}]
(1) By Lemma~\ref{lem:ex(viii)}, the map $\vartheta\mapsto (\pi(\vartheta),\tau(\vartheta))$ of Proposition~\ref{prop:pi-tau-theta} is given by
\begin{equation} \label{eqn:pitau-viii}
\begin{array}{l}
\Rep(\SO(5),(j,k)) \boxtimes \C_a\\
\longmapsto \big(\Rep(\SO(7),(j,j,j)),\Rep(\SO(4),(k,k)) \boxtimes \C_a
\end{array}
\end{equation}
for $a\in\Z$ and $j,k\in\N$ with $|a|\leq k\leq j$.
The Casimir operators for $\tilG$ and for the factors of $G$ and~$K$ act on these irreducible representations as the following scalars.
\begin{center}
\begin{tabular}{|c|c|c|}
\hline
Operator & Representation & Scalar\\
\hline\hline
$C_{\tilG}$ & $\Rep(\SO(7),(j,j,j))$ & $3 (j^2 + 3j)$\\
\hline
$C_G^{(1)}$ & \multirow{2}{*}{$\Rep(\SO(5),(j,k))\boxtimes\C_a$} & $j^2 + 3j + k^2 + k$\\
\cline{1-1} \cline{3-3}
$E_G$ & & $\sqrt{-1}\,a$\\
\hline
$C_K^{(1)}$ & \multirow{2}{*}{$\Rep(\SO(4),(k,k))\boxtimes\C_a$} & $2 (k^2 + k)$\\
\cline{1-1} \cline{3-3}
$E_K$ & & $\sqrt{-1}\,a$\\
\hline
\end{tabular}
\end{center}
This implies $\dd\ell(E_G) = \dd r(E_K)$ and $2\,\dd\ell(C_{\tilG}) = 6\,\dd\ell(C_G^{(1)}) - 3\,\dd r(C_K^{(1)})$.

(2) The description of $\D_{\tilG}(X)$ follows from the fact that it is generated by a single differential operator of degree~$2$, by Lemma~\ref{lem:ex(viii)}.(4).
Using the diagram just before Proposition~\ref{prop:ex(viii)}, we see that
$$F = K/H \simeq (\SU(2) \times \SO(2)/\{\pm I_2\})/\Diag(\SO(2)),$$
hence $\D_K(F)$ is isomorphic to $\D_{\SU(2)\times\SO(2)}(\SU(2))$, which contains $\dd\ell (C_K^{(1)})\!= \dd r(C_K^{(1)})$ and $\dd r(E_K)$; we conclude using Lemma~\ref{lem:ex(viii)}.(6).
We now focus on $\D_G(X)$.
We only need to prove the first equality, since the other ones follow from the relations between the generators.
For this, using Lemmas \ref{lem:struct-DGX}.(3) and~\ref{lem:ex(viii)}.(5), it suffices to show that the three differential operators $\dd\ell(C_{\tilG})$, $\dd r(C_K^{(1)})$, and $\dd r(E_K)$ on~$X$ are algebraically independent.
Let $f$ be a polynomial in three variables such that $f(\dd\ell(C_{\tilG}),\dd r(C_K^{(1)}),\dd r(E_K))=0$ in $\D_G(X)$.
By letting this differential operator act on the $G$-isotypic component $\vartheta = \Rep(\SO(5),(j,k)) \boxtimes \C_a $ in $C^{\infty}(X)$, we obtain
$$f\big( 3(j^2 + 3j), 2(k^2 + k), - \sqrt{-1}\,a \big) = 0$$
for all $a,j,k\in\Z$ with $|a|\leq k\leq j$, hence $f$ is the zero polynomial.
\end{proof}

\begin{proof}[Proof of Proposition~\ref{prop:nu-ex(viii)}]
We use Proposition~\ref{prop:Stau-transfer} and the formula \eqref{eqn:pitau-viii} for the map $\vartheta\mapsto (\pi(\vartheta),\tau(\vartheta))$ of Proposition~\ref{prop:pi-tau-theta}.
Let $\tau = \Rep(\SO(4),(k,k))\boxtimes\C_a \in \Disc(K/H)$ with $k\geq |a|$.
If $\vartheta\in\Disc(G/H)$ satisfies $\tau(\vartheta)=\tau$, then $\vartheta$ is of the form $\vartheta = \Rep(\SO(5),(j,k))\boxtimes\C_a$ for some $j\geq k$, by \eqref{eqn:pitau-viii}.
The algebra $\D_{\tilG}(X)$ acts on the irreducible $\tilG$-submodule $\pi(\vartheta) = \Rep(\SO(7),(j,j,j))$ of $C^{\infty}(X)$ by the scalars
$$\lambda(\vartheta) + \rho_{\tila} = \frac{1}{2} (2j+3) \in \C/(\Z/2\Z)$$
via the Harish-Chandra isomorphism \eqref{eqn:HCX-viii}, whereas the algebra $Z(\g_{\C})$ acts on the irreducible $G$-module $\vartheta = \Rep(\SO(5),(j,k))\boxtimes\C_a$ by the scalars
$$\nu(\vartheta) + \rho = \Big(j+\frac{3}{2},k+\frac{1}{2};a\Big) \in (\C^2\oplus\C)/W(B_2)\times\{1\}$$
via \eqref{eqn:HCZg-viii}.
Thus the affine map $S_{\tau}$ in Proposition~\ref{prop:nu-ex(viii)} sends $\lambda(\vartheta)+\rho_{\tila}$ to $\nu(\vartheta)+\rho$ for any $\vartheta\in\Disc(G/H)$ such that $\tau(\vartheta)=\tau$, and we conclude using Proposition~\ref{prop:Stau-transfer}.
\end{proof}

\begin{remark}
The group $\SO(5)$ already acts transitively on $X=\tilG/\tilH$.
If, instead of $(\SO(5)\times\SO(2),\iota_8(\U(2)))$, we take $(G,H)=(\SO(5),\SU(2))$, then $X=G/H$ is the same as in Proposition~\ref{prop:ex(viii)} and Lemma~\ref{lem:ex(viii)}.
However, $X_{\C}$ is not $G_{\C}$-spherical anymore and Theorem~\ref{thm:main}.(1)--(2) fail, as one can see from Proposition~\ref{prop:ex(viii)}.
\end{remark}

\subsection{The case $(\tilG,\tilH,G)=(\SO(7),G_{2(-14)},\SO(6))$}\label{subsec:ex(ix)}

Here $H=\SU(3)$, and the only maximal connected proper subgroup of~$G$ containing~$H$ is $K=\U(3)$.
Neither $\tilG/\tilH$ nor $G/H$ is a symmetric space.
Let $E_K$ be a generator of the center of $\kk=\mathfrak{u}(3)$ such that the eigenvalues of $\ad(E_K)$ in~$\g_{\C}$ are $0,\pm 1,\pm 2$.

\begin{proposition}[Generators and relations] \label{prop:ex(ix)}
For
$$X = \tilG/\tilH = \SO(7)/G_{2(-14)} \simeq \SO(6)/\SU(3) = G/H$$
and $K=\U(3)$, we have
\begin{enumerate}
  \item $2\,\dd\ell(C_{\tilG}) = 3\,\dd\ell(C_G) - 3\,\dd r(C_K)$;
  \item $\left \{ \begin{array}{cll}
  \D_{\tilG}(X) & \!\!\!=\!\!\! & \C[\dd\ell(C_{\tilG})];\\
  \D_K(F) & \!\!\!=\!\!\! & \C[\dd r(E_K)];\\
  \D_G(X) & \!\!\!=\!\!\! & \C[\dd\ell(C_{\tilG}),\dd r(E_K)] \,=\, \C[\dd\ell(C_G),\dd r(E_K)].
  \end{array}\right.$
\end{enumerate}
\end{proposition}

Thus the algebra $\D_G(X)$ is generated by $\D_{\tilG}(X)$ and $\dd r(Z(\kk_{\C}))$, and also by $\dd\ell(Z(\g_{\C}))$ and $\dd r(Z(\kk_{\C}))$, but \emph{not} by $\D_{\tilG}(X)$ and $\dd\ell(Z(\g_{\C}))$.
The subalgebra generated by $\D_{\tilG}(X)$ and $\dd\ell(Z(\g_{\C}))$, which is isomorphic to the polynomial ring $\C[\dd\ell(C_{\tilG}),\dd\ell(C_G)]$, has index two in $\D_G(X)$.

We now identify
\begin{equation} \label{eqn:HCZg-ix}
\Hom_{\C\text{-}\mathrm{alg}}(Z(\g_{\C}),\C) \ \simeq\ \jj_{\C}^*/W(\g_{\C}) \ \simeq\ \C^3/W(D_3)
\end{equation}
via the standard basis and use again the Harish-Chandra isomorphism \eqref{eqn:HCX-viii} for $X=\SO(7)/G_{2(-14)}$.
The set $\Disc(K/H)$ consists of the representations of $K=\U(3)$ of the form $\tau=\chi_k$ for $k\in\Z$.

\begin{proposition}[Transfer map] \label{prop:nu-ex(ix)}
Let
$$X = \tilG/\tilH = \SO(7)/G_{2(-14)} \simeq \SO(6)/\SU(3) = G/H$$
and $K=\U(3)$.
For $\tau=\chi_k\in\Disc(K/H)$ with $k\in\Z$, the affine map
\begin{eqnarray*}
S_{\tau} :\ \tila_{\C}^* \simeq \C & \longrightarrow & \hspace{0.8cm} \C^2\oplus\C \hspace{0.8cm} \simeq \jj_{\C}^*\\
\lambda & \longmapsto & \Big(\lambda+\frac{1}{2}, \lambda-\frac{1}{2}, k\Big)
\end{eqnarray*}
induces a transfer map
$$\nnu(\cdot,\tau) : \Hom_{\C\text{-}\mathrm{alg}}(\D_{\tilG}(X),\C) \longrightarrow \Hom_{\C\text{-}\mathrm{alg}}(Z(\g_{\C}),\C)$$
as in Theorem~\ref{thm:nu-tau}.
\end{proposition}

In order to prove Propositions \ref{prop:ex(ix)} and~\ref{prop:nu-ex(ix)}, we use the following results on finite-dimensional representations.

\begin{lemma}\label{lem:ex(ix)}
\begin{enumerate}
  \item Discrete series for $\tilG/\tilH$, $G/H$, and $F=K/H$:
  \begin{eqnarray*}
  \Disc(\SO(7)/G_{2(-14)}) & = & \{ \Rep(\SO(7),(j,j,j)) \ :\ j\in\N\} ;\\
  \Disc(\SO(6)/\SU(3)) & = & \{ \Rep(\SO(6),(j,j,k)) \ :\ |k|\leq j,\ j,k\in\Z\} ;\\
  \Disc(\U(3)/\SU(3)) & = & \{ \chi_k \ :\ k\in\Z\} ,
  \end{eqnarray*}
  where $\chi_k(g)=(\det g)^k$.
  \item Branching laws for $\SO(7)\downarrow\SO(6)$: For $j\in\N$,
  $$\Rep\big(\SO(7),(j,j,j)\big)|_{\SO(6)}\ \simeq\ \bigoplus_{\substack{k\in\Z\\ |k|\leq j}} \Rep\big(\SO(6),(j,j,k)\big).$$
  \item Irreducible decomposition of the regular representation of~$G$: For~$k\in\nolinebreak\Z$,
  $$L^2\big(\SO(6)/\U(3),\chi_k\big)\ \simeq\ \sumplus{\substack{j\in\N\\ j\geq |k|}}\ \Rep\big(\SO(6),(j,j,k)\big) .$$
  \item The ring $S(\g_{\C}/\h_{\C})^H=S(\C^6\oplus\C)^{\SU(3)}$ is generated by two algebraically independent homogeneous elements of respective degrees $1$ and~$2$.
\end{enumerate}
\end{lemma}

\begin{proof}[Proof of Lemma~\ref{lem:ex(ix)}]
(1) For $\Disc(\SO(7)/G_{2(-14)})$, see Lemma~\ref{lem:ex(viii)}.(1).
The equality for $\Disc(\U(3)/\SU(3))$ is clear.
The equality for $\Disc(\SO(6)/\SU(3))$ is given in \cite{kra79}, but we now provide an alternative approach for later purposes.
The isomorphism of Lie groups $\U(4)/\Diag(\U(1))\simeq\SO(6)$ induces a bijection between the two sets
\begin{eqnarray*}
\widehat{\SO(6)} & \simeq & \{ (x,y,z)\in\Z^3 \ :\ x\geq y\geq |z|\},\\
(\U(4)/\Diag(\U(1)))^{\widehat{\,\,}} & \simeq & \{ (a,b,c,d)\in\Z^4 \ :\ a\geq b\geq c\geq d\}/\Z(1,1,1,1),
\end{eqnarray*}
via the map
\begin{equation} \label{eqn:isom-dual-SO(6)}
(x,y,z) \longmapsto (x+y,x+z,y+z,0).
\end{equation}
Now the description of $\Disc(\SO(6)/\SU(3))$ follows from the classical branching law for $\SU(N)\downarrow\SU(N-1)$, see \eg \cite[Th.\,8.1.1]{gw09}, for $N=4$ and from the Frobenius reciprocity.

(2) This is a special case of the classical branching law for $\SO(N)\downarrow\SO(N-1)$, see \eg \cite[Th.\,8.1.3]{gw09}, for $N=7$.

(3) By using \eqref{eqn:isom-dual-SO(6)}, the proof is reduced to the classical branching law for $\U(N)\downarrow\U(N-1)$ for $N=4$.

(4) This is immediate from the symmetric case $\SO(6)/\U(3)$ because $\g_{\C}/\h_{\C}\!\simeq (\so(6,\C)/\gl(3,\C))\oplus\C$, and $H$ acts trivially on the second component.
\end{proof}

\begin{proof}[Proof of Proposition~\ref{prop:ex(ix)}]
(1) By Lemma~\ref{lem:ex(ix)}, the map $\vartheta\mapsto (\pi(\vartheta),\tau(\vartheta))$ of Proposition~\ref{prop:pi-tau-theta} is given by
\begin{equation} \label{eqn:pitau-ix}
\Rep(\SO(6),(j,j,k)) \longmapsto \big(\Rep(\SO(7),(j,j,j)),\chi_k\big)
\end{equation}
for $k\in\Z$ and $j\in\N$ with $|k|\leq j$.
The Casimir operators for $\tilG$, $G$, and~$K$ act on these irreducible representations as the following scalars.
\begin{center}
\begin{tabular}{|c|c|c|}
\hline
Operator & Representation & Scalar\\
\hline\hline
$C_{\tilG}$ & $\Rep(\SO(7),(j,j,j))$ & $3 (j^2 + 3j)$\\
\hline
$C_G$ & $\Rep(\SO(6),(j,j,k))$ & $2 (j^2 + 3j) + k^2$\\
\hline
$C_K$ & \multirow{2}{*}{$\chi_k$} & $k^2$\\
\cline{1-1} \cline{3-3}
$E_K$ & & $\sqrt{-1}\,k$\\
\hline
\end{tabular}
\end{center}
This implies $2\,\dd\ell(C_{\tilG}) = 3\,\dd\ell(C_G) - 3\,\dd r(C_K)$.

(2) For $\D_{\tilG}(X)$, see Proposition~\ref{prop:ex(viii)}.(2).
For $\D_K(F)$, the statement is obvious since $H$ is a normal subgroup of~$K$ and $K/H$ is isomorphic to the toral group $\mathbb{S}^1$.
We now focus on $\D_G(X)$.
We only need to prove the first equality, since the other one follows from the relations between the generators.
For this, using Lemmas \ref{lem:struct-DGX}.(3) and~\ref{lem:ex(ix)}.(4), it suffices to show that the two differential operators $\dd\ell(C_{\tilG})$ and $\dd r(E_K)$ on~$X$ are algebraically independent.
Let $f$ be a polynomial in two variables such that $f(\dd\ell(C_{\tilG}),\dd r(E_K))=0$ in $\D_G(X)$.
By letting this differential operator act on the $G$-isotypic component $\vartheta = \Rep(\SO(6),(j,j,k))$ in $C^{\infty}(X)$, we obtain
$$f\big( 3(j^2 + 3j), - \sqrt{-1}\,k \big) = 0$$
for all $j,k\in\Z$ with $|k|\leq j$, hence $f$ is the zero polynomial.
\end{proof}

\begin{proof}[Proof of Proposition~\ref{prop:nu-ex(ix)}]
We use Proposition~\ref{prop:Stau-transfer} and the formula \eqref{eqn:pitau-ix} for the map $\vartheta\mapsto (\pi(\vartheta),\tau(\vartheta))$ of Proposition~\ref{prop:pi-tau-theta}.
Let $\tau=\chi_k\in\Disc(K/H)$ with $k\in\Z$.
If $\vartheta\in\Disc(G/H)$ satisfies $\tau(\vartheta)=\tau$, then $\vartheta$ is of the form $\vartheta = \Rep(\SO(6),(j,j,k))$ for some $j\in\N$ with $j\geq |k|$, by \eqref{eqn:pitau-ix}.
The algebra $\D_{\tilG}(X)$ acts on the irreducible $\tilG$-submodule $\pi(\vartheta) = \Rep(\SO(7),(j,j,j))$ by the scalars
$$\lambda(\vartheta) + \rho_{\tila} = \frac{1}{2} (2j+3) \in \C/(\Z/2\Z)$$
via the Harish-Chandra isomorphism \eqref{eqn:HCX-viii}, whereas the algebra $Z(\g_{\C})$ acts on the irreducible $G$-module $\vartheta = \Rep(\SO(6),(j,j,k))$ by the scalars
$$\nu(\vartheta) + \rho = (j+2, j+1, k) \in \C^3/W(D_3)$$
via \eqref{eqn:HCZg-ix}.
Thus the affine map $S_{\tau}$ in Proposition~\ref{prop:nu-ex(ix)} sends $\lambda(\vartheta)+\rho_{\tila}$ to $\nu(\vartheta)+\rho$ for any $\vartheta\in\Disc(G/H)$ such that $\tau(\vartheta)=\tau$, and we conclude using Proposition~\ref{prop:Stau-transfer}.
\end{proof}

\begin{remark}
The only noncompact real form of $\tilG_{\C}/\tilH_{\C}$ is $X_{\R}=\linebreak\SO(4,3)_0/G_{2(2)}$.
There are exactly two real forms of $G_{\C}/H_{\C}$ isomorphic to~$X_{\R}$ \cite[Ex.\,5.2]{kob94}:
$$X_{\R} \,\simeq\, \SO(3,3)_0/\SL(3,\R) \,\simeq\, \SO(4,2)_0/\SU(2,1).$$
Discrete series representations for $X_{\R}$ in both cases were classified in \cite{kob94} via branching laws for $\tilG\downarrow G$, in the same spirit as in the present paper.
In these cases the isotropy group is noncompact, which means that $\SO(3,3)_0$ and $\SO(4,2)_0$ do \emph{not} act properly on~$X_{\R}$.
This explains why the case (ix) of Table~\ref{table1} does not appear in our application \cite{kkII} to spectral analysis.
\end{remark}

\subsection{The case $(\tilG,\tilH,G)=(\SO(7),\SO(6),G_{2(-14)})$}\label{subsec:ex(x)}

Here $H=\SU(3)$ is a maximal connected proper subgroup of~$G$, so that $K=H$ and $F$ is a point.

\begin{proposition}[Generators and relations] \label{prop:ex(x)}
For
$$X = \tilG/\tilH = \SO(7)/\SO(6) \simeq G_{2(-14)}/\SU(3) = G/H$$
and $K=H=\SU(3)$, we have
\begin{enumerate}
  \item $\dd\ell(C_{\tilG}) = \dd\ell(C_G)$;
  \item $\D_G(X) = \D_{\tilG}(X) = \C[\dd\ell(C_{\tilG})] = \C[\dd\ell(C_G)]$.
\end{enumerate}
\end{proposition}

Let $\omega_1,\omega_2$ be the fundamental weights with respect to the simple roots $\alpha_1,\alpha_2$ of~$G_2$, respectively, labeled as follows:
\begin{picture}(30,6)
\put(3,3){\circle{6}}
\put(25,3){\circle{6}}
\put(5,1){\line(1,0){18}}
\put(6,3){\line(1,0){16}}
\put(5,5){\line(1,0){18}}
\put(10,0.2){$>$}
\put(0,-10){$\alpha_1$}
\put(20,-10){$\alpha_2$}
\end{picture}.
Then $\omega_1=3\alpha_1+2\alpha_2$ and $\omega_2=\alpha_1+2\alpha_2$.
We identify
\begin{align}
\Hom_{\C\text{-}\mathrm{alg}}(Z(\g_{\C}),\C) & \ \simeq\ \jj_{\C}^*/W(\g_{\C}) \ \simeq\ (\C\omega_1\oplus\C\omega_2)/W(G_2),\label{eqn:HCZg-x}\\
\Hom_{\C\text{-}\mathrm{alg}}(\D_{\tilG}(X),\C) & \ \simeq\ \tila_{\C}^*/\widetilde{W} \ \simeq\ \C/(\Z/2\Z)\label{eqn:HCX-x}
\end{align}
by the standard bases.
In this case, $\Disc(K/H)$ is the singleton $\{\mathbf{1}\}$.

\begin{proposition}[Transfer map] \label{prop:nu-ex(x)}
Let
$$X = \tilG/\tilH = \SO(7)/\SO(6) \simeq G_{2(-14)}/\SU(3) = G/H$$
and $K=H=\SU(3)$.
For $\tau=\mathbf{1}\in\Disc(K/H)$, the affine map
\begin{eqnarray*}
S_{\tau} :\ \tila_{\C}^* \simeq \C & \longrightarrow & \hspace{1cm} \C^2 \hspace{1cm} \simeq \jj_{\C}^*\\
\lambda & \longmapsto & \omega_1 + \Big(\lambda-\frac{3}{2}\Big)\,\omega_2
\end{eqnarray*}
induces a transfer map
$$\nnu(\cdot,\tau) : \Hom_{\C\text{-}\mathrm{alg}}(\D_{\tilG}(X),\C) \longrightarrow \Hom_{\C\text{-}\mathrm{alg}}(Z(\g_{\C}),\C)$$
as in Theorem~\ref{thm:nu-tau}.
\end{proposition}

In order to prove Propositions \ref{prop:ex(x)} and~\ref{prop:nu-ex(x)}, we use the following results on finite-dimensional representations.

\begin{lemma}\label{lem:ex(x)}
\begin{enumerate}
  \item Discrete series for $\tilG/\tilH$ and $G/H$:
  \begin{eqnarray*}
  \Disc(\SO(7)/\SO(6)) & = & \{ \mathcal{H}^k(\R^7) \ :\ k\in\N\} ;\\
  \Disc(G_{2(-14)}/\SU(3)) & = & \{ \Rep(G_{2(-14)},k\omega_2) \ :\ k\in\N\} .
  \end{eqnarray*}
  \item Branching laws for $\SO(7)\downarrow G_{2(-14)}$: For $k\in\N$,
  $$\mathcal{H}^k(\R^7)|_{G_{2(-14)}}\ \simeq\ \Rep(G_{2(-14)},k\omega_2).$$
  \item The ring $S(\g_{\C}/\h_{\C})^H=S(\C^3\oplus (\C^3)^{\vee})^{\SU(3)}$ is generated by a single homogeneous element of degree~$2$.
\end{enumerate}
\end{lemma}

\begin{proof}[Proof of Lemma~\ref{lem:ex(x)}]
In~(1), the first equality follows from the Cartan--Helgason theorem (Fact~\ref{fact:CartanHelgason}) and the second from \cite{kra79}.
For~(2), see for instance \cite{kq78}; the formula can also be obtained directly from the Borel--Weil theorem, applied to the isomorphism
$$G_{2(-14)}/\U(2) \simeq \OO(7)/(\SO(2)\times\OO(5))$$
of generalized flag varieties.
Statement~(3) follows from the isomorphism
$$S(\C^3 \oplus (\C^3)^{\vee}) \simeq \bigoplus_{i,j\in\N} S^i(\C^3) \otimes S^j(\C^3)^{\vee}$$
and from the fact that the $S^i(\C^3)$, for $i\in\N$, are irreducible and mutually nonisomorphic.
\end{proof}

\begin{proof}[Proof of Proposition~\ref{prop:ex(x)}]
(1) Since the restriction $\mathcal{H}^k(\R^7)|_{G_{2(-14)}}$ remains irreducible by Lemma~\ref{lem:ex(x)}.(2), the map $\vartheta\mapsto (\pi(\vartheta),\tau(\vartheta))$ of Proposition~\ref{prop:pi-tau-theta} reduces to
\begin{equation} \label{eqn:pitau-x}
\Rep(G_{2(-14)},k\omega_2) \longmapsto (\mathcal{H}^k(\R^7),\mathbf{1}).
\end{equation}
We normalize $C_G$ so that the short root of $\g_2$ has length~$1$.
Then the Casimir operators for $\tilG$ and~$G$ act on these irreducible representations as the following scalars.
\begin{center}
\begin{tabular}{|c|c|c|}
\hline
Operator & Representation & Scalar\\
\hline\hline
$C_{\tilG}$ & $\mathcal{H}^k(\R^7)$ & $k^2 + 5k$\\
\hline
$C_G$ & $\Rep(G_{2(-14)},k\omega_2)$ & $k^2 + 5k$\\
\hline
\end{tabular}
\end{center}
This implies $\dd\ell(C_{\tilG})=\dd\ell(C_G)$.

(2) This follows from the fact that $\D_G(X)$ is generated by a single differential operator of degree~$2$, by Lemma~\ref{lem:ex(x)}.(3).
\end{proof}

\begin{proof}[Proof of Proposition~\ref{prop:nu-ex(x)}]
We use Proposition~\ref{prop:Stau-transfer} and the formula \eqref{eqn:pitau-x} for the map $\vartheta\mapsto (\pi(\vartheta),\tau(\vartheta))$ of Proposition~\ref{prop:pi-tau-theta}.
Let $\tau=\mathbf{1}\in\Disc(K/H)$.
If $\vartheta\in\Disc(G/H)$ satisfies $\tau(\vartheta)=\tau$, then $\vartheta$ is of the form $\vartheta = \Rep(G_{2(-14)},k\omega_2)$ for some $k\in\N$, by \eqref{eqn:pitau-x}.
The algebra $\D_{\tilG}(X)$ acts on the irreducible $\tilG$-submodule $\pi(\vartheta) = \mathcal{H}^k(\R^7)$ by the scalars
$$\lambda(\vartheta) + \rho_{\tila} = \frac{1}{2} (2k+5) \in \C/(\Z/2\Z)$$
via the Harish-Chandra isomorphism \eqref{eqn:HCX-x}, whereas the algebra $Z(\g_{\C})$ acts on the irreducible $G$-module $\vartheta = \Rep(G_{2(-14)},k\omega_2)$ by the scalars
$$\nu(\vartheta) + \rho = k\omega_2+\rho = \omega_1+(k+1)\omega_2 \in (\C\omega_1+\C\omega_2)/W(G_2)$$
via \eqref{eqn:HCZg-x}.
Thus the affine map $S_{\tau}$ in Proposition~\ref{prop:nu-ex(x)} sends $\lambda(\vartheta)+\rho_{\tila}$ to $\nu(\vartheta)+\rho$ for any $\vartheta\in\Disc(G/H)$ such that $\tau(\vartheta)=\tau$, and we conclude using Proposition~\ref{prop:Stau-transfer}.
\end{proof}

\subsection{The case $(\tilG,\tilH,G)=(\SO(8),\Spin(7),\SO(7))$}\label{subsec:ex(xi)}

Here $H=G_{2(-14)}$ is a maximal connected proper subgroup of~$G$, so that $K=H$ and $F$ is a point.

\begin{proposition}[Generators and relations] \label{prop:ex(xi)}
For
$$X = \tilG/\tilH = \SO(8)/\Spin(7) \simeq \SO(7)/G_{2(-14)} = G/H$$
and $K=H=G_{2(-14)}$, we have
\begin{enumerate}
  \item $3\,\dd\ell(C_{\tilG}) = 4\,\dd\ell(C_G)$;
  \item $\D_G(X) = \D_{\tilG}(X) = \C[\dd\ell(C_{\tilG})] = \C[\dd\ell(C_G)]$.
\end{enumerate}
\end{proposition}

We identify $\D_{\tilG}(X)$ with $\C/(\Z/2\Z)$ as in \eqref{eqn:HCX-vii}, and $Z(\g_{\C})$ with $\C^3/W(B_3)$ as in \eqref{eqn:HCZg-iia} with $m=2$.
In this case, $\Disc(K/H)$ is the singleton $\{\mathbf{1}\}$.

\begin{proposition}[Transfer map] \label{prop:nu-ex(xi)}
Let
$$X = \tilG/\tilH = \SO(8)/\Spin(7) \simeq \SO(7)/G_{2(-14)} = G/H$$
and $K=H=G_{2(-14)}$.
For $\tau=\mathbf{1}\in\Disc(K/H)$, the affine map
\begin{eqnarray*}
S_{\tau} :\ \tila_{\C}^* \simeq \C & \longrightarrow & \hspace{1.2cm} \C^3 \hspace{1.2cm} \simeq \jj_{\C}^*\\
\lambda & \longmapsto & \frac{1}{2}\,(\lambda+2, \lambda, \lambda-2)
\end{eqnarray*}
induces a transfer map
$$\nnu(\cdot,\tau) : \Hom_{\C\text{-}\mathrm{alg}}(\D_{\tilG}(X),\C) \longrightarrow \Hom_{\C\text{-}\mathrm{alg}}(Z(\g_{\C}),\C)$$
as in Theorem~\ref{thm:nu-tau}.
\end{proposition}

In order to prove Propositions \ref{prop:ex(xi)} and~\ref{prop:nu-ex(xi)}, we use the following results on finite-dimensional representations.

\begin{lemma}\label{lem:ex(xi)}
\begin{enumerate}
  \item Discrete series for $\tilG/\tilH$ and $G/H$:
  \begin{eqnarray*}
  \Disc(\SO(8)/\Spin(7)) & = & \{ \Rep(\SO(8),(k,k,k,k)) \ :\ k\in\N\} ;\\
  \Disc(\SO(7)/G_{2(-14)}) & = & \{ \Rep(\SO(7),(k,k,k)) \ :\ k\in\N\} .
  \end{eqnarray*}
  \item Branching laws for $\SO(8)\downarrow\SO(7)$: For $k\in\N$,
  $$\Rep\big(\SO(8),(k,k,k,k)\big)|_{\SO(7)}\ \simeq\ \Rep\big(\SO(7),(k,k,k)\big).$$
  \item The ring $S(\g_{\C}/\h_{\C})^H=S(\C^7)^{G_2}$ is generated by a single homogeneous element of degree~$2$.
\end{enumerate}
\end{lemma}

\begin{proof}[Proof of Lemma~\ref{lem:ex(xi)}]
(1) We consider the double covering $\Spin(8)/\Spin(7)\to\SO(8)/\Spin(7)$ and apply the triality of~$D_4$ to the covering space, which shows that
\begin{eqnarray*}
\Disc(\Spin(8)/\Spin(7)) & = & \{ \varsigma\cdot\mathcal{H}^{\ell}(\R^8) \,:\, \ell\in\N\}\\
& \simeq & \Big\{ \Rep\Big(\Spin(8),\frac{1}{2}\,(\ell,\ell,\ell,\ell)\Big) \,:\, \ell\in\N\Big\}.
\end{eqnarray*}
The representations which contribute to $\Disc(\SO(8)/\Spin(7))$ are those with even parity, which yields the description of $\Disc(\SO(8)/\Spin(7))$.
For\linebreak $\Disc(\SO(7)/G_{2(-14)})$, see Lemma~\ref{lem:ex(viii)}.(1).

(2) This is a special case of the classical branching law for $\SO(N)\downarrow\SO(N-1)$, see \eg \cite[Th.\,8.1.2]{gw09}, for $N=8$.

(3) By \eqref{eqn:decomp-S(C7)} and Lemma~\ref{lem:ex(x)}.(2), we have
$$S(\C^7) \,\simeq\, \bigoplus_{j\in\N}\ \Rep(G_2,j\omega_2) \otimes \C[s^2]$$
as graded $G_2$-modules, where $s^2$ is an $\SO(7)$-invariant quadratic form on~$\C^7$ and $\Rep(G_2,j\omega_2)\otimes\C\cdot 1$ is realized in $S^j(\C^7)$ for $j\in\N$.
Therefore, $S(\C^7)^{G_2}$ is isomorphic to $\C[s^2]$ as a graded algebra, proving Lemma~\ref{lem:ex(xi)}.(3).
\end{proof}

\begin{proof}[Proof of Proposition~\ref{prop:ex(xi)}]
(1) Since the restriction $\mathcal{H}^k(\R^n)|_{G_{2(-14)}}$ remains irreducible by Lemma~\ref{lem:ex(xi)}.(2), the map $\vartheta\mapsto (\pi(\vartheta),\tau(\vartheta))$ of Proposition~\ref{prop:pi-tau-theta} reduces to
\begin{equation} \label{eqn:pitau-xi}
\Rep(\SO(7),(k,k,k)) \longmapsto \big(\Rep(\SO(8),(k,k,k,k)),\mathbf{1}\big).
\end{equation}
The Casimir operators for $\tilG$ and~$G$ act on these irreducible representations as the following scalars.
\begin{center}
\begin{tabular}{|c|c|c|}
\hline
Operator & Representation & Scalar\\
\hline\hline
$C_{\tilG}$ & $\Rep(\SO(8),(k,k,k,k))$ & $4 (k^2 + 3k)$\\
\hline
$C_G$ & $\Rep(\SO(7),(k,k,k))$ & $3 (k^2 + 3k)$\\
\hline
\end{tabular}
\end{center}
This implies $3\dd\ell(C_{\tilG})=4\dd\ell(C_G)$.

(2) The description of $\D_{\tilG}(X)$ from the classical result for the symmetric space $\SO(8)/\SO(7) \simeq \Spin(8)/\Spin(7)$ (see Proposition~\ref{prop:ex(i)}.(2) with $n=3$) by the triality of~$D_4$.
The description of $\D_G(X)$ follows from the fact that it is generated by a single differential operator of degree~$2$, by Lemma~\ref{lem:ex(xi)}.(3).
\end{proof}

\begin{proof}[Proof of Proposition~\ref{prop:nu-ex(xi)}]
We use Proposition~\ref{prop:Stau-transfer} and the formula \eqref{eqn:pitau-xi} for the map $\vartheta\mapsto (\pi(\vartheta),\tau(\vartheta))$ of Proposition~\ref{prop:pi-tau-theta}.
Let $\tau=\mathbf{1}\in\Disc(K/H)$.\linebreak
If $\vartheta\in\Disc(G/H)$ satisfies $\tau(\vartheta)=\tau$, then $\vartheta$ is of the form $\vartheta =\linebreak \Rep(\SO(7),(k,k,k))$ for some $k\in\N$, by \eqref{eqn:pitau-xi}.
The algebra $\D_{\tilG}(X)$ acts on the irreducible $\tilG$-submodule $\pi(\vartheta) = \Rep(\SO(8),(k,k,k,k))$ by the scalars
$$\lambda(\vartheta) + \rho_{\tila} = 2k + 3 \in \C/(\Z/2\Z)$$
(see \eqref{eqn:Hk-tri} for the triality of~$D_4$), whereas the algebra $Z(\g_{\C})$ acts on the irreducible $G$-module $\vartheta = \Rep(\SO(7),(k,k,k))$ by the scalars
$$\nu(\vartheta) + \rho = \Big(k+\frac{5}{2}, k+\frac{3}{2}, k+\frac{1}{2}\Big) \in \C^3/W(B_3).$$
Thus the affine map $S_{\tau}$ in Proposition~\ref{prop:nu-ex(xi)} sends $\lambda(\vartheta)+\rho_{\tila}$ to $\nu(\vartheta)+\rho$ for any $\vartheta\in\Disc(G/H)$ such that $\tau(\vartheta)=\tau$, and we conclude using Proposition~\ref{prop:Stau-transfer}.
\end{proof}

\subsection{Application of the triality of~$D_4$}\label{subsec:ex-triality}

The cases (xii), (xiii), (xiii)$'$, (xiv) of Table~\ref{table1} can all be reduced to cases discussed earlier by using the triality of the Dynkin diagram~$D_4$, described in Section~\ref{subsec:ex(vii)}.

\subsubsection{The case $(\tilG,\tilH,G)=(\SO(8),\SO(7),\Spin(7))$ (see (xii) in Table~\ref{table1})}\label{subsubsec:ex(xii)}

We realize $\tilH$ and~$G$ in such a way that $H:=\tilH\cap G\simeq G_{2(-14)}$.
Up to applying an inner automorphism of $\tilg_{\C}=\so(8,\C)$, we may assume that $\jj_{\C}\cap\tilh_{\C}=e_1^{\perp}$ and $\jj_{\C}\cap\g_{\C}=\omega_+^{\perp}$ or~$\omega_-^{\perp}$, where for $\lambda\in\tilh_{\C}^{\ast}$ we set
$$\lambda^{\perp} := \{ x\in\tilh_{\C} :\ \lambda(x)=0\} .$$
Then the automorphism group of the Dynkin diagram~$D_4$, switching $e_4$ and $\omega_+$ or~$\omega_-$, induces an automorphism of $\tilg_{\C}=\so(8,\C)$ switching $\so(7,\C)$ and $\spin(7,\C)$.
Thus this case reduces to the case~(xi) of Table~\ref{table1}.

\subsubsection{The case $(\tilG,\tilH,G)=(\SO(8),\Spin(7),\SO(6)\times\SO(2))$ (see (xiii) in Table~\ref{table1})} \label{subsubsec:ex(xiii)}

Here $H=\tilH\cap\nolinebreak G$ is isomorphic to the double covering
$$\widetilde{\U(3)} \,\simeq\, \{ (Z,s)\in\U(3)\times\C^* :\, \det Z=s^2\}$$
of~$\U(3)$.
The only maximal connected proper subgroup of~$G$ containing~$H$ is $K=\U(3)\times\SO(2)$.
The group $H$ is the image of the embedding
\begin{eqnarray*}
\iota_{13} :\ \widetilde{\U(3)} & \longrightarrow & \hspace{1cm} \U(3)\times\SO(2)\\
(Z,s) & \longmapsto & \left(s^{-1}Z,\begin{pmatrix} \mathrm{Re}(s) & -\mathrm{Im}(s)\\ \mathrm{Im}(s) & \mathrm{Re}(s)\end{pmatrix}\right).
\end{eqnarray*}
Up to applying an inner automorphism of $\tilg_{\C}=\so(8,\C)$, we may assume that $\jj_{\C}\cap\tilh_{\C}=\omega_+^{\perp}$ or~$\omega_-^{\perp}$ and $\g_{\C}=Z_{\tilg_{\C}}(e_1)$, where we identify $\widetilde{\jj}_{\C}^{\ast}$ with~$\widetilde{\jj}_{\C}$ via the Killing form and write $Z_{\tilg_{\C}}(\lambda)$ for the centralizer in~$\tilg_{\C}$ of $\lambda\in\widetilde{\jj}_{\C}^{\ast}\simeq\widetilde{\jj}_{\C}$.
Then the automorphism group of the Dynkin diagram~$D_4$, switching $e_1$ and $\omega_+$ or~$\omega_-$, induces an automorphism $\tau$ of $\tilg_{\C}=\so(8,\C)$ such that $\tau(\tilh_{\C})\cap\widetilde{\jj}_{\C}=e_1^{\perp}$ and $\tau(\g_{\C})=Z_{\tilg_{\C}}(\omega_+)$ or $Z_{\tilg_{\C}}(\omega_+)$.
Then $\tau(\tilh_{\C})\simeq\so(7,\C)$ and $\tau(\g_{\C})\simeq\gl(4,\C)$, and $\tau$ takes the triple $(\SO(8),\Spin(7),\SO(6)\times\SO(2))$ to $(\SO(8),\SO(7),\U(4))$.
Thus this case reduces to the case~(i) of Table~\ref{table1} with $n=3$.

Likewise, the case (xiii)$'$ reduces to the case~(i)$'$ of Table~\ref{table1} with $n=3$.

\subsubsection{The case $(\tilG,\tilH,G)=(\SO(8),\SO(6)\times\SO(2),\Spin(7))$ (see (xiv) in Table~\ref{table1})} \label{subsubsec:ex(xiv)}

Here $H=\tilH\cap\nolinebreak G$ is again isomorphic to the double covering $\widetilde{\U(3)}$ of~$\U(3)$.
The only maximal connected proper subgroup of~$G$ containing~$H$ is $K=\Spin(6)$.
The embedding of $H\simeq\widetilde{\U(3)}$ into $\tilH=\SO(6)\times\SO(2)$ factors through the embedding $\iota_{13} : \widetilde{\U(3)}\to\U(3)\times\SO(2)$ of Section~\ref{subsubsec:ex(xiii)}.
The automorphism $\tau$  of $\tilg_{\C}=\so(8,\C)$ that we just introduced in Section~\ref{subsubsec:ex(xiii)} takes the triple $(\SO(8),\SO(6)\times\SO(2),\Spin(7))$ to $(\SO(8),\U(4),\SO(7))$.
Thus this case reduces to the case~(ii) of Table~\ref{table1} with $n=3$.

\counterwithout{theorem}{subsection}
\counterwithin{theorem}{section}
\counterwithout{equation}{subsection}
\counterwithin{equation}{section}

\section{Explicit generators and relations when $\tilG$ is a product} \label{sec:ex(*)}

Let us recall that the basic setting \ref{setting} of this paper is a triple $(\tilG,\tilH,G)$ such that $\tilG$ is a connected compact Lie group, $\tilH$ and~$G$ are connected closed subgroups, and $\tilG_{\C}/\tilH_{\C}$ is $G_{\C}$-spherical.
Our main results have been proved in this setting under an additional assumption, namely
\begin{itemize}
  \item[$\bullet$] $\tilG$ is simple,
\end{itemize}
based on the classification of such triples (Table~\ref{table1}).
In this section, we consider the basic setting \ref{setting} with another assumption, namely
\begin{itemize}
  \item[$\bullet$] $\tilG$ is isomorphic to $G\times G$ where $G$ is simple, and $\tilH=H_1\times H_2$ where $H_1,H_2$ are subgroups of~$G$.
\end{itemize}
In Section~\ref{subsec:classif-nonsimple}, we begin with a classification of such triples $(\tilG,\tilH,G)$: Proposition~\ref{prop:class-triple-prod} states that, up to coverings and automorphisms, the triple \eqref{eqn:nonsimple-case}, which is described more precisely as Example~\ref{ex:nonsimple-triple-precise} below, is essentially the only one.
Then, for the rest of the section, we examine to what extent analogous results to our main theorems hold for this triple.

\begin{example} \label{ex:nonsimple-triple-precise}
Let $\varsigma$ be the lift to $\Spin(8)$ of the outer automorphism $\varsigma$ of order three of the Lie algebra $\so(8)$ described in Section~\ref{subsec:ex(vii)}.
We consider the triple
$$(\tilG,\tilH,G) = \big(\Spin(8)\times\Spin(8),\Spin(7)\times\Spin(7),\Spin(8)\big),$$
where $H$ is embedded into~$\tilG$ using the covering of the standard embedding $\SO(7)\hookrightarrow\SO(8)$ in both components, and $G$ is embedded into~$\tilG$ by $g\mapsto (g,\varsigma(g))$.
\end{example}

\subsection{Classification of triples} \label{subsec:classif-nonsimple}

In contrast with the classification of the triples $(\tilG,\tilH,G)$ for simple~$\tilG$ in Table~\ref{table1}, the following proposition states that there are very few triples $(\tilG,\tilH,G)$ with $\tilG$ a product in the setting \ref{setting}.

\begin{proposition} \label{prop:class-triple-prod}
In the setting \ref{setting}, suppose that $\tilG$ is isomorphic to $G\times G$ where $G$ is simple, and that $\tilH=H_1\times H_2$ where $H_1,H_2$ are subgroups of~$G$.
Then the triple $(\tilG,\tilH,G)$ is isomorphic to
\begin{equation} \label{eqn:nonsimple-case-ij}
\big(\Spin(8)\times\Spin(8),\Spin(7)\times\Spin(7),(\varsigma^i,\varsigma^j)(\Spin(8))\big)
\end{equation}
for some $0\leq i\neq j\leq 2$, up to coverings and (possibly outer) automorphisms of~$\tilG$.
\end{proposition}

For $(i,j)=(0,1)$, the triple \eqref{eqn:nonsimple-case-ij} is the triple \eqref{eqn:nonsimple-case} described in Example~\ref{ex:nonsimple-triple-precise}.
Up to applying the outer automorphism $(\varsigma^{-i},\varsigma^{-j})$ of $\tilG=G\times G$, the triple \eqref{eqn:nonsimple-case-ij} is isomorphic to
$$\big(\Spin(8)\times\Spin(8),\varsigma^{-i}(\Spin(7))\times\varsigma^{-j}(\Spin(7)),\Spin(8)\big),$$
where $\Spin(8)$ is embedded into $\Spin(8)\times\Spin(8)$ diagonally, by $g\mapsto (g,g)$.
Thus the proof of Proposition~\ref{prop:class-triple-prod} reduces to the following lemma.

\begin{lemma} \label{lem:class-intersect}
Let $G$ be a connected compact simple Lie group, and $H_1$ and~$H_2$ connected closed subgroups such that $G=H_1H_2$ and that $G_{\C}/((H_1)_{\C}\cap (H_2)_{\C})$ is $G_{\C}$-spherical.
Then the triple $(G,H_1,H_2)$ is isomorphic to
$$\big(\Spin(8), \varsigma^i(\Spin(7)), \varsigma^j(\Spin(7))\big)$$
for some $0\leq i\neq j\leq 2$, up to coverings and conjugations.
\end{lemma}

\begin{proof}[Proof of Lemma~\ref{lem:class-intersect}]
The triples $(G,H_1,H_2)$ where $G$ is a connected compact simple Lie group, $H_1$ and~$H_2$ are connected closed subgroups, and $G=H_1H_2$, were classified by Oni\v{s}\v{c}ik \cite{oni69}.
Among them, we find the triples $(G,H_1,H_2)$ such that $G_{\C}/((H_1)_{\C}\cap (H_2)_{\C})$ is $G_{\C}$-spherical by using Kr\"amer's classification \cite{kra79} of spherical homogeneous spaces $G_{\C}/H_{\C}$ with $G_{\C}$ simple.
\end{proof}

\subsection{Differential operators and transfer map for the triple \eqref{eqn:nonsimple-case}}

For the rest of the section, we examine the algebra $\D_G(X)$ and its subalgebras for the triple \eqref{eqn:nonsimple-case}.
In this case, since $\tilG/\tilH\simeq\mathbb{S}^7\times\mathbb{S}^7$ is simply connected, the transitive $G$-action on $\tilG/\tilH$ via $G\hookrightarrow\tilG$, $g\mapsto (g,\varsigma(g))$, has connected stabilizer $H:=\tilH\cap G$; it is isomorphic to $G_{2(-14)}$.
The only maximal connected proper subgroup of $G$ containing~$H$ is $K=\Spin(7)$.
For $i\in\{1,2\}$, we see the Casimir element of the $i$-th factor of $\tilG=\Spin(8)\times\Spin(8)$ as an element of $Z(\g_{\C})$, and denote it by $C_{\tilG}^{(i)}$.
Clearly $C_{\tilG}=C_{\tilG}^{(1)}+C_{\tilG}^{(2)}$.

\begin{proposition}[Generators and relations] \label{prop:ex(*)}
For
$$X = \tilG/\tilH = (\Spin(8)\times\Spin(8))/(\Spin(7)\times\Spin(7)) \simeq \Spin(8)/G_{2(-14)} \!= G/H$$
and $K=\Spin(7)$, we have
\begin{enumerate}
  \item $3\,\dd\ell(C_{\tilG}) = 3 \big(\dd\ell(C_{\tilG}^{(1)})+\dd\ell(C_{\tilG}^{(2)})\big) = 6\,\dd\ell(C_G) - 4\,\dd r(C_K)$;
  \item $\left \{
\begin{array}{l}
    \D_{\tilG}(X) = \C[\dd\ell(C_{\tilG}^{(1)}),\dd\ell(C_{\tilG}^{(2)})];\\
    \D_K(F) = \C[\dd r(C_K)];\\
    \D_G(X) = \C\big[\dd\ell(C_{\tilG}^{(1)}),\dd\ell(C_{\tilG}^{(2)}),\dd r(C_K)\big] = \C\big[\dd\ell(C_{\tilG}^{(1)}),\dd\ell(C_{\tilG}^{(2)}),\dd\ell(C_G)\big].
    \end{array} \right.$
\end{enumerate}
\end{proposition}

Proposition~\ref{prop:ex(*)}.(2) states that
$$\D_G(X) = \langle\D_{\tilG}(X),\dd\ell(Z(\g_{\C}))\rangle = \langle\D_{\tilG}(X),\dd r(Z(\kk_{\C}))\rangle.$$
In particular, condition ($\widetilde{\mathrm{B}}$) of Section~\ref{subsec:intro-applic} holds.
However, unlike in the previous cases where $\tilG_{\C}$ is simple, here the subalgebra
\begin{equation} \label{eqn:RZgZk-*}
R := \langle\dd\ell(Z(\g_{\C})),\dd r(Z(\kk_{\C}))\rangle
\end{equation}
is strictly contained in $\D_G(X)$, namely condition ($\widetilde{\mathrm{A}}$) fails.
More precisely, the following holds.

\begin{proposition} \label{prop:ex(*)-RDGX}
For
$$X = \tilG/\tilH = (\Spin(8)\times\Spin(8))/(\Spin(7)\times\Spin(7)) \simeq \Spin(8)/G_{2(-14)} \!= G/H$$
and $K=\Spin(7)$, and for any $i\in\{1,2\}$, we have
\begin{enumerate}
  \item $\dd\ell(C_{\tilG}^{(i)})\notin R$;
  \item $\D_G(X) = R + R\,\dd\ell(C_{\tilG}^{(i)})$ as $R$-modules.
\end{enumerate}
\end{proposition}

We identify
\begin{align}
\Hom_{\C\text{-}\mathrm{alg}}(Z(\g_{\C}),\C) & \ \simeq\ \jj_{\C}^*/W(\g_{\C}) && \hspace{-1.9cm} \simeq\ \C^4/W(D_4),\label{eqn:HCZg-*}\\
\Hom_{\C\text{-}\mathrm{alg}}(\D_{\tilG}(X),\C) & \ \simeq\ \hspace{0.3cm} \tila_{\C}^*/\widetilde{W} && \hspace{-1.9cm} \simeq\ \C^2/(\Z/2\Z)^2\label{eqn:HCX-*}
\end{align}
by the standard bases.
More precisely, let $\jj_{\C}$ be a Cartan subalgebra of $\g_{\C}=\so(8,\C)$ and $\{e_1,e_2,e_3,e_4\}$ the standard basis of~$\jj_{\C}^*$.
For later purposes, we fix a positive system $\Delta^+(\g_{\C},\jj_{\C}) = \{ e_i\pm e_j \,:\, 1\leq i<j\leq 4\}$.
Let $\varsigma$ be the outer automorphism of order three of $\so(8)$ leaving $\jj_{\C}$ invariant and
$$\varsigma^*(e_i) = \omega_+ := \frac{1}{2}(1,1,1,1),$$
as in \eqref{eqn:triality-basis}.
We view $G=\Spin(8)$ as a subgroup of $\tilG=\Spin(8)\times\Spin(8)$ via $\mathrm{id}\times\varsigma$.

The set $\Disc(K/H)$ consists of the representations of $K=\Spin(7)$ of the form $\tau = \Rep\big(\Spin(7),\frac{1}{2}(a,a,a)\big)$ for $a\in\N$.
In this case, $\tilG$ is not a simple Lie group, but Theorem~\ref{thm:nu-tau}.(4) still holds as follows.

\begin{proposition}[Transfer map] \label{prop:nu-ex(*)}
Let
$$X = \tilG/\tilH = (\Spin(8)\times\Spin(8))/(\Spin(7)\times\Spin(7)) \simeq \Spin(8)/G_{2(-14)} \!= G/H$$
and $K=\Spin(7)$.
For $\tau=\Rep\big(\Spin(7),\frac{1}{2}(a,a,a)\big)\in\Disc(K/H)$ with $a\in\N$, the affine map
\begin{eqnarray}
\hspace{1cm}S_{\tau} :\ \tila_{\C}^* \simeq\quad \C^2\hspace{0.2cm} & \longrightarrow & \hspace{2.8cm} \C^4 \hspace{3cm} \simeq \jj_{\C}^*\label{eqn:Stau-*}\\
\hspace{1.3cm}(\lambda,\lambda') & \longmapsto & \frac{1}{2} \, \big(\lambda+a+3, \lambda'+1, \lambda'-1, \lambda-a-3\big)\nonumber
\end{eqnarray}
induces a transfer map
$$\nnu(\cdot,\tau) : \Hom_{\C\text{-}\mathrm{alg}}(\D_{\tilG}(X),\C) \longrightarrow \Hom_{\C\text{-}\mathrm{alg}}(Z(\g_{\C}),\C)$$
as in Theorem~\ref{thm:nu-tau}.
\end{proposition}

Proposition~\ref{prop:nu-ex(*)} will be proved in Section~\ref{subsec:proof-prop-ex(*)}.

\subsection{Representation theory for $\Spin(8)/G_{2(-14)}$ with overgroup $\Spin(8)\times\Spin(8)$}

In order to prove Propositions \ref{prop:ex(*)} and~\ref{prop:nu-ex(*)}, we use the following results on finite-dimensional representations.

\begin{lemma}\label{lem:ex(*)}
\begin{enumerate}
  \item Discrete series for $\tilG/\tilH$, $G/H$, and $F=K/H$:
  \begin{align*}
  & \Disc\big(\Spin(8)\!\times\!\Spin(8)/\Spin(7)\!\times\!\Spin(7)\big) = \big\{ \mathcal{H}^j(\R^8)\boxtimes\mathcal{H}^{j'}(\R^8) : j,j'\in\N\big\} ;\\
  & \Disc(\Spin(8)/G_{2(-14)}) = \Big\{ \Rep\Big(\Spin(8),\frac{1}{2}\big(j+a,j',j',a-j\big)\Big) :\\
  & \hspace{6.2cm} |j-j'|\leq a\leq j+j',\ j+j'-a\in 2\N\Big\} ;\\
  & \Disc(\Spin(7)/G_{2(-14)}) = \Big\{ \Rep\Big(\Spin(7),\frac{1}{2} (a,a,a)\Big) \,:\, a\in\N\Big\} .
  \end{align*}
  \item Branching laws for $\Spin(8)\times\Spin(8)\downarrow (\mathrm{id}\times\varsigma)(\Spin(8))$: For $j,j'\in\nolinebreak\N$,
  $$\big(\mathcal{H}^j(\R^8)\boxtimes\mathcal{H}^{j'}(\R^8)\big)|_{\Spin(8)} \,\simeq \bigoplus_{\substack{|j-j'|\leq a\leq j+j'\\ a\equiv j+j' \mod 2}} \Rep\Big(\Spin(8),\frac{1}{2}\big(j+a,j',j',a-j\big)\Big).$$
  \item Irreducible decomposition of the regular representation of~$G$: For~$a\in\nolinebreak\N$,
  \begin{align*}
  & L^2\Big(\Spin(8)/\Spin(7),\Rep\Big(\Spin(7),\frac{1}{2}(a,a,a)\Big)\Big)\\
  & \simeq\ \sumplus{\substack{j,j'\in\N\\ j+j'-a\in 2\N}}\ \Rep\Big(\Spin(8),\frac{1}{2}\big(j+a,j',j',a-j\big)\Big).
  \end{align*}
  \item The ring $S(\g_{\C}/\h_{\C})^H=S(\spin(8,\C)/\g_{2,\C})^{G_2}$ is generated by three algebraically independent homogeneous elements of degree~$2$.
\end{enumerate}
\end{lemma}

\begin{proof}[Proof of Lemma~\ref{lem:ex(*)}]
(1) The description of $\Disc(\tilG/\tilH)$ follows from the classical theory of spherical harmonics as in $\Spin(8)/\Spin(7)\simeq\SO(8)/\SO(7)$.
The description of $\Disc(G/H)$ and $\Disc(K/H)$ is given by Kr\"amer \cite{kra79}.

(2) Let $\varsigma^*\mathcal{H}^j(\R^8)$ be the irreducible representation of $\Spin(8)$ obtained by precomposing $\mathcal{H}^j(\R^8)$ by $\varsigma : \Spin(8)\to\Spin(8)$.
We first observe the following isomorphisms of $\Spin(8)$-modules:
\begin{eqnarray*}
\mathcal{H}^j(\R^8) & \simeq & \Rep(\Spin(8),je_1),\\
\varsigma^* \mathcal{H}^j(\R^8) & \simeq & \Rep(\Spin(8),j\omega_+),
\end{eqnarray*}
because $\varsigma^*(e_1)=\omega_+$.
Let $P_{\C}$ and~$Q_{\C}$ be the parabolic subgroups of $G_{\C}=\Spin(8,\C)$ given by the characteristic elements $(1,0,0,0)$ and $(1,1,1,1)\in\C^4\simeq\jj_{\C}$, respectively.
The nilradicals $\n_{\C}$ and~$\mathfrak{u}_{\C}$ of the parabolic subalgebras $\p_{\C}$ and~$\q_{\C}$ have the following weights:
\begin{eqnarray*}
\Delta(\n_{\C},\jj_{\C}) & = & \{ e_1\pm e_j ~:~ 2\leq j\leq 4\},\\
\Delta(\mathfrak{u}_{\C},\jj_{\C}) & = & \{ e_i+e_j ~:~ 1\leq i<j\leq 4\},
\end{eqnarray*}
and so
$$\Delta(\n_{\C}\cap\mathfrak{u}_{\C},\jj_{\C}) = \{ e_1+e_j ~:~ 2\leq j\leq 4\}.$$
The Levi subgroup of the standard parabolic subgroup $P_{\C}\cap Q_{\C}$ is a (connected) double covering of $\GL(1,\C)\times\GL(3,\C)$, and its Lie algebra acts on $\n_{\C}\cap\mathfrak{u}_{\C}\simeq\C^3$ by $\Rep(\gl(1,\C)+\gl(3,\C),(1,1,0,0))$, hence on the $\ell$-th symmetric tensor $\mathcal{S}^{\ell}(\n_{\C}\cap\mathfrak{u}_{\C})$ by
$$\Rep(\gl(1,\C)+\gl(3,\C),(\ell,\ell,0,0)).$$
Applying Proposition~\ref{prop:tensornormal}, we obtain an upper estimate for the possible irreducible summands of the tensor product representation $\mathcal{H}^j(\R^8) \otimes \varsigma^* \mathcal{H}^{j'}(\R^8)$ by
$$\bigoplus_{\ell=0}^{+\infty} \mathcal{O}\big(G_{\C}/(P_{\C}^-\cap Q_{\C}^-), \mathcal{S}^{\ell}(\n_{\C}^-\cap\mathfrak{u}_{\C}^-)\otimes\mathcal{L}_{je_1+j'\omega_+}\big).$$
By the Borel--Weil theorem, we have an isomorphism of $\Spin(8)$-modules:
\begin{eqnarray*}
& & \mathcal{O}\big(G_{\C}/(P_{\C}^-\cap Q_{\C}^-), \mathcal{S}^{\ell}(\n_{\C}^-\cap\mathfrak{u}_{\C}^-)\otimes\mathcal{L}_{je_1+j'\omega_+}\big)\\
& \simeq & \left\{ \begin{array}{ll}
\Rep\big(\Spin(8),\frac{1}{2}(2j+j'-2\ell, j', j', j'-2\ell)\big) & \text{if $0\leq\ell\leq\min(j,j')$},\\
\{0\} & \text{otherwise}.
\end{array}\right.
\end{eqnarray*}
Via the change of variables $a=j+j'-2\ell$, the condition $0\leq\ell\leq\min(j,j')$ on $\ell\in\N$ amounts to the following conditions on $a\in\Z$:
$$|j-j'| \leq a \leq j+j' \quad\mathrm{and}\quad a\equiv j+j' \mod 2.$$
Thus we have
\begin{equation} \label{eqn:HjR8tensorupper}
\mathcal{H}^j(\R^8) \otimes \varsigma^* \mathcal{H}^{j'}(\R^8) \subset\! \bigoplus_{\substack{|j-j'|\leq a\leq j+j'\\ a\equiv j+j'\mod 2}} \Rep\Big(\Spin(8),\frac{1}{2}(j+a,j',j',a-j)\Big).\hspace{-0.25cm}
\end{equation}
By comparing \eqref{eqn:decomp1} with \eqref{eqn:decomp2}, we see that the set of all irreducible $\Spin(8)$-modules occurring in $\mathcal{H}^j(\R^8)\otimes\varsigma^*\mathcal{H}^{j'}(\R^8)$ for some $j,j'$ coincides with\linebreak $\Disc(\Spin(8)/G_{2(-14)})$, counting multiplicities.
Therefore the description of $\Disc(\Spin(8)/G_{2(-14)})$ in~(1) forces \eqref{eqn:HjR8tensorupper} to be an equality.
This completes the proof of~(2).

(3) By the classical branching law for the \emph{standard} embedding $\so(7)\subset\so(8)$, see \eg \cite[Th.\,8.1.2]{gw09}, we have
$$\Hom_{\so(7)}\Big(\Rep\Big(\so(7),\frac{1}{2}(a,a,a)\Big),\Rep\big(\so(8),(x_1,x_2,x_3,x_4)\big)|_{\so(7)}\Big) = \{0\}$$
if and only if $(x_1,x_2,x_3,x_4)=\frac{1}{2}(j+j', a, a, j-j')$ for some $j,j'\in\Z$ with $|j-j'|\leq a\leq j+j'$ and $j+j'-a\in 2\Z$.
Using \eqref{eqn:triality-basis}, we have
$$\varsigma^{-1}\Big(\frac{1}{2}(j+j', a, a, j-j')\Big) = \frac{1}{2} (j+a, j', j', a-j).$$
We conclude using the Frobenius reciprocity.

(4) There is a unique $7$-dimensional irreducible representation of~$G_2$, and we have an isomorphism of $G_2$-modules
$$\spin(8,\C)/\g_{2,\C} \simeq \C^7 \oplus \C^7.$$
Let $Q\in S^2(\C^7)$ be the quadratic form defining $\SO(7,\C)$.
Then we have a decomposition as $\SO(7,\C)$-graded modules:
$$S(\C^7) \simeq \bigoplus_{j\in\N} \mathcal{H}^j(\C^7) \otimes \C[Q],$$
where $\SO(7,\C)$ acts trivially on the second factor.
The self-dual representations $\mathcal{H}^j(\C^7)$ of $\SO(7,\C)$ remain irreducible when restricted to~$G_2$, and they are pairwise inequivalent.
Therefore $(\mathcal{H}^i(\C^7)\otimes\mathcal{H}^j(\C^7))^{G_2}\neq\{ 0\}$ if and only if $i=j$, and in this case its dimension is one.
Hence the graded $\C$-algebra of $G_2$-invariants
\begin{eqnarray*}
S\big(\spin(8,\C)/\g_{2,\C}\big)^{G_2} & \simeq & \big(S(\C^7) \otimes S(\C^7)\big)^{G_2}\\
& \simeq & \bigoplus_{i,j\in\N} \big(\mathcal{H}^i(\C^7)\otimes\mathcal{H}^j(\C^7)\big)^{G_2} \otimes \C[Q] \otimes \C[Q']
\end{eqnarray*}
is isomorphic to a polynomial ring generated by three homogeneous elements of degree~$2$.
\end{proof}

\subsection{Generators and relations: proof}

In this section we give a proof of Proposition~\ref{prop:ex(*)}.
For this we observe from Lemma~\ref{lem:ex(*)} that the map $\vartheta\mapsto (\pi(\vartheta),\tau(\vartheta))$ of Proposition~\ref{prop:pi-tau-theta} is given by
\begin{equation} \label{eqn:pitau-*}
\begin{array}{l}
\Rep\Big(\Spin(8),\frac{1}{2}\big(j+a,j',j',a-j\big)\Big)\\
\longmapsto \Big(\mathcal{H}^j(\R^8) \boxtimes \mathcal{H}^{j'}(\R^8),\Rep\Big(\Spin(7),\frac{1}{2} (a,a,a)\Big)\Big)
\end{array}
\end{equation}
for $j,j',a\in\N$ with $|j-j'|\leq a\leq j+j'$ and $j+j'-a\in 2\N$.

\begin{proof}[Proof of Proposition~\ref{prop:ex(*)}]
(1) The first equality follows from the identity $C_{\tilG}=\dd\ell(C_{\tilG}^{(1)})+\dd\ell(C_{\tilG}^{(2)})$.
For the second equality, we use the fact that the Casimir operators for $\tilG$, $G$, and~$K$ act on these irreducible representations as the following scalars.
\begin{center}
\begin{changemargin}{-0.6cm}{0cm}
\begin{tabular}{|c|c|c|}
\hline
Operator & Representation & Scalar\\
\hline\hline
$C_{\tilG}$ & $\mathcal{H}^j(\R^8)\boxtimes\mathcal{H}^{j'}(\R^8)$ & $j^2+{j'}^2+6(j+j')$\\
\hline
$C_G$ & $\Rep\big(\Spin(8),\frac{1}{2}(j+a, j', j', a-j)\big)$ & $\frac{1}{2}(j^2 + {j'}^2 + 6(j+j') + a(a+6))$\\
\hline
$C_K$ & $\Rep\big(\Spin(7),\frac{1}{2}(a,a,a)\big)$ & $\frac{3}{4}\,a(a+6)$\\
\hline
\end{tabular}
\end{changemargin}
\end{center}
This, together with the identity
$$3 \big(j^2+{j'}^2+6(j+j')\big) = 3 \big(j^2 + {j'}^2 + 6(j+j') + a(a+6)\big) - 3a(a+6),$$
implies $3\,\dd\ell(C_{\tilG}) = 6\,\dd\ell(C_G) - 4\,\dd r(C_K)$.

(2) The description of $\D_{\tilG}(X)$ follows from the fact that $\Spin(8)/\Spin(7)$ is a symmetric space of rank one, see also Proposition~\ref{prop:ex(i)}.(2) with $n=3$.
For the fiber $F=K/H=\Spin(7)/G_{2(-14)}$, the description of $\D_K(F)$ was given in Proposition~\ref{prop:ex(viii)}.(2).
We now focus on $\D_G(X)$.
We only need to prove the first equality, since the second one follows from the linear relations between the generators in~(1).
For this, using Lemmas \ref{lem:struct-DGX}.(3) and~\ref{lem:ex(*)}.(4), it suffices to show that the three differential operators $\dd\ell(C_{\tilG}^{(1)})$, $\dd\ell(C_{\tilG}^{(2)})$, and $\dd r(C_K)$ on~$X$ are algebraically independent.
Let $f$ be a polynomial in three variables such that $f(\dd\ell(C_{\tilG}^{(1)}),\dd\ell(C_{\tilG}^{(2)}),\dd r(C_K))=0$ in $\D_G(X)$.
By letting this differential operator act on the $G$-isotypic component $\vartheta = \Rep\big(\Spin(8),\frac{1}{2}(j+a, j', j', a-j)\big)$ in $C^{\infty}(X)$, we obtain
$$f\Big(j^2+6j, {j'}^2+6j', \frac{3}{4}\,a(a+6)\Big) = 0$$
for all $j,j',a\in\N$ with $|j-j'|\leq a\leq j+j'$ and $j+j'-a\in 2\N$, hence $f$ is the zero polynomial.
\end{proof}

\subsection{The subalgebra $R=\langle\dd r(Z(\kk_{\C})),\dd\ell(Z(\g_{\C}))\rangle$}

Unlike Theorem~\ref{thm:main}.(2) for simple~$\tilG$, here $\D_G(X)$ is strictly larger than the subalgebra $R$ generated by $\dd r(Z(\kk_{\C}))$ and $\dd\ell(Z(\g_{\C}))$.
In this section, we give a proof of Proposition~\ref{prop:ex(*)-RDGX} on the subalgebra~$R$.
For this, we describe $\D_G(X)$ as a function on $\Disc(G/H)$, as in Proposition~\ref{prop:D-spec-Hom}.
Recall from Lemma~\ref{lem:ex(*)}.(1) that
\begin{align*}
& \Disc(G/H)\\
& \simeq \Big\{ \Rep\Big(\Spin(8),\frac{1}{2}\big(j+a,j',j',a-j\big)\Big) \,:\, a\geq |j-j'|,\ j+j'-a\in 2\N\Big\}.
\end{align*}
Setting $x:=(j+3)^2$ and $y:=(j'+3)^2$ and $z:=(a+3)^2$, we may regard the polynomial ring $\C[x,y,z]$ as a subalgebra of $\mathrm{Map}(\Disc(G/H),\C)$.

\begin{lemma} \label{lem:ex(*)-DGX-polyn}
The map $\tilde{\psi}$ of Proposition~\ref{prop:D-spec-Hom} gives an algebra isomorphism
$$\D_G(X) \overset{\sim}{\longrightarrow} \C[x,y,z].$$
\end{lemma}

\begin{proof}
We take generators $R_1,\dots,R_4$ of $Z(\g_{\C})$ as follows.
Recall from Section~\ref{subsec:symmsp} the notation $\chi_{\nu}^G : Z(\g_{\C})\to\C$ for the infinitesimal character.
There exist unique elements $R_1,\dots,R_4\in Z(\g_{\C})$ such that
$$\left\{ \begin{array}{ccl}
\chi_{\nu}^G(R_k) & = & 2^{2k-1} \big(\nu_1^{2k} + \nu_2^{2k} + \nu_3^{2k} + \nu_4^{2k}\big) \quad\quad \mathrm{for}\ 1\leq k\leq 3,\\
\chi_{\nu}^G(R_4) & = & 2^4 \, \nu_1\nu_2\nu_3\nu_4
\end{array}\right.$$
for $\nu = (\nu_1,\nu_2,\nu_3,\nu_4) \in \jj_{\C}^*/W(D_4) \simeq \C^4/\mathfrak{S}_4\ltimes (\Z/2\Z)^3$ via the standard basis of the Cartan subalgebra $\jj_{\C}$ of $\g_{\C}=\so(8,\C)$.
Then $Z(\g_{\C})$ is the polynomial algebra $\C[R_1,R_2,R_3,R_4]$.

Let us set
$$\left\{ \begin{array}{ccl}
r_k & := & \tilde{\psi}(\dd\ell(R_k)) \hspace{0.98cm} \text{for }1\leq k\leq 4,\\
q & := & \tilde{\psi}(\dd r(C_K)),\\
p_i & := & \tilde{\psi}(\dd\ell(C_{\tilG}^{(i)})) \hspace{0.8cm} \text{for }1\leq i\leq 2.
\end{array} \right.$$
By Proposition~\ref{prop:D-spec-Hom}.(2), these are maps from $\Disc(G/H)$ to~$\C$ sending any $\vartheta = \Rep\big(\Spin(8),\frac{1}{2}\,(j+a,j',j',a-j)\big) \in \Disc(G/H)$ to $\psi(\vartheta,\dd\ell(R_k))$ (for $1\leq k\leq\nolinebreak 4$), $\psi(\vartheta,\dd r(C_K))$, and $\psi(\vartheta,\dd\ell(C_{\tilG}^{(i)}))$, which are the scalars by which $R_k\in Z(\g_{\C})$ acts on~$\vartheta$ and $C_K\in Z(\kk_{\C})$ acts on $\tau(\vartheta)=\Rep\big(\Spin(7),\frac{1}{2} (a,a,a)\big)$ and $C_{\tilG}^{(i)}\in Z(\tilg_{\C})$ acts on $\pi(\vartheta)=\mathcal{H}^j(\R^8)\otimes\mathcal{H}^{j'}(\R^8)$, respectively, by \eqref{eqn:pitau-*}.
These scalars are given as follows:
$$\left\{ \begin{array}{ccl}
r_1 & = & x + y + z + 1,\\
r_2 & = & x^2 + 6zx + z^2 + y^2 + 6y + 1,\\
r_3 & = & x^3 + 15zx^2 + 15z^2x + z^3 + y^3 + 15y^2 + 15y + 1,\\
r_4 & = & (x-1)(y-z),\\
q & = & \frac{3}{4}\,(z-9),\\
p_1 & = & x-9,\\
p_2 & = & y-9.
\end{array} \right.$$
Therefore, the algebra homomorphism $\tilde{\psi} : \D_G(X)\to\mathrm{Map}(\Disc(G/H),\C)$ takes values in $\C[x,y,z]$.
The image is exactly $\C[x,y,z]$ since $p_1,p_2,q$ generate it.
\end{proof}

From now on, we identify $\D_G(X)$ with the polynomial ring $\C[x,y,z]$.
Since the algebra $\dd r(Z(\kk_{\C}))$ is generated by $\dd r(C_K)$ and $\dd\ell(Z(\g_{\C}))$ is generated by the $\dd\ell(R_k)$ for $1\leq k\leq 4$, we may view $R$ in \eqref{eqn:RZgZk-*} as the subalgebra of $\C[x,y,z]$ generated by $q,r_1,r_2,r_3,r_4$.

\begin{lemma} \label{lem:ex(*)-xyz}
In this setting,
\begin{enumerate}
  \item $z,x+y,xz+y,xy\in R$;
  \item $x^n,y^n\in R+Rx$ for all $n\in\N$.
\end{enumerate}
\end{lemma}

\begin{proof}
(1) We have $z\in\C q+\C$, hence $z\in R$.
Similarly, $x+y\in\C q+\C r_1+\C$, hence $x+y\in R$.
The inclusions $xz+y,xy\in R$ follow from the equalities
\begin{align*}
4(xz+y) & \,=\, r_2 + 2 r_4 - (x+y)^2 - (z+1)^2,\\
xy & \,=\, r_4 + (xz+y) + z.
\end{align*}

(2) Let us prove $x^n\in R+Rx$ by induction on~$n$.
The cases $n=0,1$ are clear, and we have $x^2=-xy+(x+y)x\in R+Rx$ by~(1).
Assuming $x^n\in R$, we have $x^{n+1}=xx^n\in x(R+Rx)=Rx+Rx^2$, hence $x^{n+1}\in R+Rx$ by the case $n=2$.
The assertion for~$y^n$ is clear from $y=(x+y)-x$.
\end{proof}

\begin{proof}[Proof of Proposition~\ref{prop:ex(*)-RDGX}]
We again identify $\D_G(X)$ with the polynomial ring $\C[x,y,z]$ as in Lemma~\ref{lem:ex(*)-DGX-polyn}.
Since $x+y\in R$, it is sufficient to show:
\begin{enumerate}
  \item $x\notin R$;
  \item $\C[x,y,z]=R+Rx$ as $R$-modules,
\end{enumerate}
where $R$ is again the subalgebra of $\C[x,y,z]$ generated by $q,r_1,r_2,r_3,r_4$.

(1) Suppose by contradiction that there is a polynomial $f$ in five variables such that
\begin{equation} \label{eqn:fqrp}
f(r_1,r_2,r_3,r_4,q) = x.
\end{equation}
Taking $z=1$ in the identity \eqref{eqn:fqrp} of polynomials in $x,y,z$, we see that the left-hand side is symmetric in $x$ and~$y$, whereas the right-hand side is not, yielding a contradiction.

(2) Since $xy,z\in R$ by Lemma~\ref{lem:ex(*)-xyz}.(1), any monomial of the form $x^{\ell}y^mz^n$ with $\ell,m,n\in\N$ belongs to $x^{\ell-m}R$ if $\ell\geq m$, and to $y^{m-\ell}R$ if $\ell\leq m$.
In both cases we see, using Lemma~\ref{lem:ex(*)-xyz}.(2), that $x^{\ell}y^mz^n\in R+Rx$.
Therefore $R+Rx=\C[x,y,z]$.
\end{proof}

\subsection{Transfer map: proof of Proposition~\ref{prop:nu-ex(*)}} \label{subsec:proof-prop-ex(*)}

As we have seen in the previous section, in our setting the algebra $R = \langle\dd\ell(Z(\g_{\C})),\dd r(Z(\kk_{\C}))\rangle$ does not contain $\D_{\tilG}(X)$, and an analogous statement to Proposition~\ref{prop:qI} does not hold for \emph{all} maximal ideals $\mathcal{I}$ of $Z(\kk_{\C})$.
Nevertheless, the following holds for certain specific maximal ideals~$\mathcal{I}$, which include all ideals we need to define transfer maps.
We denote by $\mathcal{I}_{\tau}$ be the annihilator of the irreducible $K$-module $\tau^{\vee}$ ($\simeq\tau$) in $Z(\kk_{\C})$, and $q_{\mathcal{I}_{\tau}} : \D_G(X) \to \D_G(X)_{\mathcal{I}_{\tau}} := \D_G(X)/\langle\mathcal{I}_{\tau}\rangle$ the quotient map as in Section~\ref{subsec:intro-transfer}.

\begin{proposition} \label{prop:qItau-*}
For any $\tau\in\Disc(K/H)$, the map $q_{\mathcal{I}_{\tau}}$ induces algebra isomorphisms
$$\D_{\tilG}(X) \overset{\sim}{\longrightarrow} \D_G(X)_{\mathcal{I}_{\tau}}$$
and
$$Z(\g_{\C})/\Ker(q_{\mathcal{I}_{\tau}}\circ\dd\ell) \overset{\sim}{\longrightarrow} \D_G(X)_{\mathcal{I}_{\tau}}.$$
\end{proposition}

These isomorphisms combine into an algebra isomorphism
$$\varphi_{\mathcal{I}_{\tau}} : Z(\g_{\C})/\Ker(q_{\mathcal{I}_{\tau}}\circ\dd\ell) \,\overset{\sim}{\longrightarrow}\, \D_{\tilG}(X),$$
which induces a natural map
\begin{eqnarray*}
\varphi_{\mathcal{I}_{\tau}}^* : \Hom_{\C\text{-}\mathrm{alg}}(\D_{\tilG}(X),\C) & \overset{\sim}{\longrightarrow} & \Hom_{\C\text{-}\mathrm{alg}}\big(Z(\g_{\C})/\Ker(q_{\mathcal{I}_{\tau}}\circ\dd\ell),\C\big)\\
& & \hspace{0.5cm} \subset\ \Hom_{\C\text{-}\mathrm{alg}}(Z(\g_{\C}),\C)
\end{eqnarray*}
Proposition~\ref{prop:qItau-*} implies Theorem~\ref{thm:nu-tau}.(1)--(3) in our setting, see Section~\ref{subsec:strategy-transfer}.

\begin{proof}[Proof of Proposition~\ref{prop:qItau-*}]
We identify $\D_G(X)$ with the polynomial ring\linebreak $\C[x,y,z]$ via~$\tilde{\psi}$ using Lemma~\ref{lem:ex(*)-DGX-polyn}.
Write $\tau = \Rep\big(\Spin(7),\frac{1}{2}(a,a,a)\big)$.
Under the isomorphism $\D_G(X)\simeq\C[x,y,z]$, the ideal $\langle\mathcal{I}_{\tau}\rangle$ is generated by $z-(a+3)^2$, and the map $q_{\mathcal{I}_{\tau}}$ identifies with the evaluation $q_a$ at $z=(a+3)^2$, sending $f(x,y,z)\in\C[x,y,z]$ to $f(x,y,(a+3)^2)\in\C[x,y]$.
This induces an algebra isomorphism $\D_G(X)_{\mathcal{I}_{\tau}}\simeq\C[x,y]$, and we obtain the following commutative diagram for each $\tau=\Rep\big(\Spin(7),\frac{1}{2}(a,a,a)\big)$.
$$\xymatrixcolsep{3pc}
\xymatrix{
\D_G(X) \ar[d]^{q_{\mathcal{I}_{\tau}}} \ar[r]^{\sim}_{\tilde{\psi}} & \C[x,y,z] \ar[d]^{q_a}\\
\D_G(X)_{\mathcal{I}_{\tau}} \ar[r]^{\sim} & \C[x,y]
}$$

We now examine the restriction of $q_a$ to the subalgebras $\D_{\tilG}(X)$ and $\dd\ell(Z(\g_{\C}))$ of $\D_G(X)$.
The restriction of $q_a$ to $\D_{\tilG}(X)=\C[p_1,p_2]=\C[x,y]$ is clearly an isomorphism.
On the other hand, a simple computation shows
\begin{eqnarray*}
q_a(r_1) & = & x + y + (a+3)^2 + 1,\\
q_q(-r_1^2 + r_2 + 2r_4) & = & 2\big((a+3)^2 - 1\big) (x-y).
\end{eqnarray*}
Since $(a+3)^2\neq 1$, we conclude that $q_a(\dd\ell(Z(\g_{\C})))=\C[x,y]$, hence the second isomorphism.
\end{proof}

\begin{proof}[Proof of Proposition~\ref{prop:nu-ex(*)}]
We use Proposition~\ref{prop:Stau-transfer} and the formula \eqref{eqn:pitau-*} for the map $\vartheta\mapsto (\pi(\vartheta),\tau(\vartheta))$ of Proposition~\ref{prop:pi-tau-theta}.
Let $\tau=\Rep\big(\Spin(7),\frac{1}{2}(a,a,a)\big)\linebreak\in\Disc(K/H)$ with $a\in\N$.
If $\vartheta\in\Disc(G/H)$ satisfies $\tau(\vartheta)=\tau$, then $\vartheta$ is of the form $\vartheta=\Rep(\Spin(8),\frac{1}{2}\,(j+a,j',j',a-j))$ for some $j,j'\in\N$ with $|j-j'|\leq a\leq j+j'$ and $j+j'-2a\in\N$, by \eqref{eqn:HjR8tensorupper}.
The algebra $\D_{\tilG}(X)$ acts on the irreducible $\tilG$-submodule $\pi(\vartheta) = \mathcal{H}^j(\R^8) \boxtimes \mathcal{H}^{j'}(\R^8)$ of $C^{\infty}(X)$ by the scalars
$$\lambda(\vartheta) + \rho_{\tila} = (j+3,j'+3) \in \C^2/(\Z/2\Z)^2$$
via the Harish-Chandra isomorphism \eqref{eqn:HCX-*}, whereas the algebra $Z(\g_{\C})$ acts on the irreducible $G$-module $\vartheta = \Rep\big(\Spin(8),\frac{1}{2}\big(j+a,j',j',a-j\big)\big)$ by the scalars
$$\nu(\vartheta) + \rho = \frac{1}{2} \, \big(\lambda+a+3, \lambda'+1, \lambda'-1, \lambda-a-3\big) \in \C^4/W(D_4)$$
via \eqref{eqn:HCZg-*}.
Thus the affine map $S_{\tau}$ in Proposition~\ref{prop:nu-ex(*)} sends $\lambda(\vartheta)+\rho_{\tila}$ to $\nu(\vartheta)+\rho$ for any $\vartheta\in\Disc(G/H)$ such that $\tau(\vartheta)=\tau$, and we conclude using Proposition~\ref{prop:Stau-transfer}.
\end{proof}

\vspace{0.5cm}

\end{document}